\documentclass{amsart}
\usepackage{graphicx}
\usepackage[all]{xy}
\usepackage{tikz}

\usepackage{amssymb}
\usepackage{bold-extra}
\usepackage{float}
\usepackage{color}
\usepackage{blkarray}
\usepackage{accents}
\usepackage{lmodern,blindtext}
\newcommand{\Ad}{\mathrm{Ad}}
\newcommand{\C}{\mathbb C}
\newcommand{\D}{\mathbb D}

\newcommand{\R}{\mathbb R}
\newcommand{\N}{\mathbb N}
\newcommand{\Z}{\mathbb Z}
\newcommand{\Q}{\mathbb Q}
\newcommand{\T}{\mathbb T}

\newcommand{\imply}{\Longrightarrow}

\newcommand{\im}{\mathrm{Im\, }}

\newcommand{\Span}{\mathrm{span }}
\newcommand{\Hom}{\mathrm{Hom}}
\newcommand{\Ext}{\mathrm{Ext}}
\newcommand{\supp}{\mathrm{supp }}

\newcommand{\open}[1]{\overset{\circ}{#1}}
\newcommand{\vsE}[2]{#1 | #2}
\newcommand{\hsE}[2]{\frac{#2}{#1}}
\newcommand{\sVf}[4]{\frac{#4}{#1}\!\frac||\!\frac{#3}{#2}}

\newcommand{\Supp}{\mathrm{supp\,}}
\newcommand{\sF}{\mathrm{sF}}
\newcommand{\sE}{\mathrm{sE}}
\newcommand{\sV}{\mathrm{sV}}
\newcommand{\tensor}{\otimes}
\newcommand{\dlim}{\lim\limits_{\to}}
\newcommand{\eps}{\varepsilon}

\newcommand{\coker}{\mathrm{coker\,}}

\definecolor{redish}{rgb}{.92, .26, .47}
\definecolor{greenI}{rgb}{0, .71, .02}
\definecolor{greenish}{rgb}{.35, .73, .51}

\newtheorem{thm}{Theorem}[section]
\newtheorem{pro}[thm]{Proposition}
\newtheorem{cor}[thm]{Corollary}
\newtheorem{lem}[thm]{Lemma}
\theoremstyle{definition}
\newtheorem{rem}[thm]{Remark}
\newtheorem{defn}[thm]{Definition}
\newtheorem{exam}[thm]{Example}
\usepackage[hidelinks]{hyperref}

\numberwithin{equation}{section}

\begin{document}

\title[\tiny On the $K$-theory of $C^*$-algebras for substitution tilings (a pedestrian version) ] 
{\large O\MakeLowercase{n the }K\MakeLowercase{-theory of }$C^*$\MakeLowercase{-algebras for substitution tilings}\\
   \MakeLowercase{(a pedestrian version)} 
}

\author[\tiny D. Gon\c{c}alves \qquad M. Ramirez-Solano\qquad ]{ 
Daniel Gon\c{c}alves$^*$\,\,\,\,\,\,\, Maria Ramirez-Solano$^{**}$\\\\\\
{\normalfont \scriptsize\large\emph{I\MakeLowercase{n memory of }U\MakeLowercase{ffe }H\MakeLowercase{aagerup$^{\dag}$ }}}
}

\date{2 September 2018}

\thanks{$^{*}$ Partially supported by CNPq.}
\thanks{$^{**}$ Supported by the Villum Foundation under the project ``Local and global structures of groups
and their algebras'' at University of Southern Denmark, and by CNPq at Universidade Federal de Santa Catarina.}
\thanks{$^\dag$ Dedicated to a special friend whose presence will never be forgotten.}

\newcommand{\Addresses}{{
  \bigskip
  \footnotesize
  Daniel Gon\c{c}alves, \textsc{Departamento de Matem\'atica, Universidade Federal de Santa Catarina, Florian\'opolis, 88.040-900, Brazil.}\par\nopagebreak
  \textit{E-mail address}: \texttt{daemig@gmail.com}
  \medskip
  \medskip

  Maria Ramirez-Solano, \textsc{Department of Mathematics, University of Southern Denmark, Odense, Denmark.}\par\nopagebreak
  \textit{E-mail address}: \texttt{maria.r.solano@gmail.com}

}}

\begin{abstract}
Under suitable conditions, a substitution tiling gives rise to a Smale space, from which three equivalence relations can be constructed, namely  the stable, unstable, and asymptotic equivalence relations. 
We denote with $S$, $U$, and $A$ their corresponding $C^*$-algebras in the sense of Renault. 
In this article we show that the $K$-theories of $S$ and $U$ can be computed from the cohomology and homology of a single cochain complex with connecting maps for tilings of the line and of the plane.
Moreover, we provide formulas to compute the $K$-theory for these three $C^*$-algebras. 
Furthermore, we show that the $K$-theory groups for tilings of dimension 1 are always torsion free.
For tilings of dimension 2, only $K_0(U)$  and $K_1(S)$ can contain torsion.
\end{abstract}


%
\maketitle
\section{\textbf{Introduction}}\label{s:one}
The study of aperiodic order has gained impetus after the discovery of quasicrystals by Dan Shechtman in $1982$, for which he was awarded the 2011 Nobel Prize in Chemistry. The quasicrystal structure can be explained by an aperiodic tiling \cite{Steinhardt}, \cite{Janot}, \cite{Sadun2006}, and physical properties of quasicrystals are related to mathematical aspects of the corresponding tiling. An important mathematical aspect to be understood is the $K$-theory of $C^*$-algebras associated with an aperiodic tiling, since $K$-theory enters physics through Bellissard's formulation of gap labelling \cite{Bellissard1}, \cite{Bellissard2}, \cite{BBG}. Moreover, $K$-theory has applications to deformations, spaces of measures, and exact regularity \cite{CS}, \cite{Sadun2011}. One can associate the so-called stable and unstable $C^*$-algebras to a tiling. The first author computed the $K$-theory for the stable $C^*$-algebra for a few examples in an ad-hoc manner in \cite{Gon2}, and from that the first author raised the conjecture that there is a relationship between the stable and unstable $C^*$-algebras. In this paper we show that there is indeed a relationship. To this end, we introduce the so-called stable cohomology and stable-transpose homology, from which one can compute the $K$-theories of the stable and unstable $C^*$-algebras, respectively. 
In particular, for tilings of dimension 1, and for tilings of dimension 2, in the absence of torsion, we show that the cohomology and homology can be computed as a direct limit of a matrix and its transpose.
Our new method for computing the K-theories of the stable and unstable C*-algebras is highly computable, and we have implemented this new method in Mathematica. 
The stable-transpose homology arguably provides a simpler method than other methods found in the literature for computing the $K$-theory of the unstable $C^*$-algebra.
We start with some background theory.

A tiling of $\R^d$ is a subdivision of $\R^d$ into so-called tiles.
These tiles can only intersect on their boundaries, and they are homeomorphic to the closed unit disk.
There are several ways to construct these tilings, see for example \cite{Sadun}.
In this paper, we consider only those built via a substitution rule on a finite number of tiles.  
For such tilings, one can construct a tiling space $(\Omega,d)$ and a substitution map $\omega:\Omega\to\Omega$ with inflation factor $\lambda>1$. 
We will assume that the tiling space has FLC (finite local complexity), and that the substitution map $\omega$ is primitive and injective. In such case, the tiling space, which is also called the continuous hull, is a nonempty compact metric space,  and it has no periodic tilings. Moreover, $\omega$ is a homeomorphism, and the triple $(\Omega,d,\omega)$ is a topologically mixing Smale space. For more details see \cite{AP}.
Using the Smale space structure of $(\Omega, d, \omega)$, one can construct three equivalence relations on $\Omega$, namely, $R_s, R_u, R_a$ -- the stable, unstable, and asymptotic equivalence relations. Each of these equivalence relations are defined as an increasing union of subspaces of $\Omega\times\Omega$ with the subspace product topology, and $R_s, R_u, R_a$ are given the inductive limit topology. They are locally compact and Hausdorff. Moreover, the equivalence classes $[T]_{R_s}, [T]_{R_u}, [T]_{R_a}$ of $T\in\Omega$ are each dense in $\Omega$. A characterization of the stable and unstable equivalence relations is given in the following proposition

\begin{pro}[Characterization of the equivalence relations $R_s, R_u$] For tilings $T,T'\in \Omega$, the following holds
  \begin{itemize}
    \item $T \sim_{R_s} T'$ $\Longleftrightarrow$ $\exists\, n\in\N_0$: $\omega^n(T)(0)=\omega^n(T')(0)$
    \item$T \sim_{R_u} T'$ $\Longleftrightarrow$ $\exists\, x\in\R^d$: $T'=T+x$,
  \end{itemize}
  where $\N_0:=\N\cup\{0\}$, and $T(0)$ is the set of tiles in $T$ that contain the origin.
  \end{pro}
 \noindent(cf.~Section \ref{s:dynamics}).
Recall that an equivalence relation is in particular a groupoid, and that the definition of a groupoid $C^*$-algebra \cite{Renault} requires a topology on the groupoid and a Haar system for the groupoid. 
$R_s$ and $R_u$ have the already-mentioned inductive limit topologies and the Lebesgue measures. 
The inductive limit topology of the asymptotic equivalence relation $R_a$ is an \'etale topology and thus $R_a$ has the Haar system of counting measures. 
Let 
$$S:=C_r^*(R_s),\qquad U:=C_r^*(R_u), \qquad A:=C_r^*(R_a)$$
 denote their $C^*$-algebras, in the sense of Renault. 
These are all UCT, simple, and amenable -- 
UCT holds for the $C^*$-algebra of any amenable groupoid, by J. L. Tu's \cite[Proposition~10.7~and~Lemma~3.5]{Tu}.
Moreover, $A$ is strongly Morita equivalent to $U\otimes S$.
For more details, see \cite{PutnamSpielberg}, \cite{Putnam96SmaleSpaces}.
We should note that "the $C^*$-algebra of a tiling" in the literature usually refers to the $C^*$-algebra $U$.

The $K$-theory of $U$ is related to the \v{C}ech cohomology of $\Omega$ (cf.~\cite[Theorem~6.3]{AP}).
This cohomology can be computed in many different ways \cite{AP}, \cite{Kel-PE}, \cite{KelPut-PE}, \cite{Sadun-PE}, \cite{WaltonHom}, \cite{BargeDiamond}.
An alternative way of computing the $K$-groups of $U$ is as follows: 
The transversal equivalence relation $R_{\mathrm{punc}}$ is an \'etale equivalence relation. 
Moreover, the $C^*$-algebras $U$ and  $C^*_r(R_{\mathrm{punc}})$ are strongly Morita equivalent by Theorem 2.8 in \cite{MRW}, and hence their $K$-theories coincide.
For more details see \cite{Kel1}, \cite{Kel2}, \cite{Kel3}, \cite{KelPut}. 
One can also compute the $K$-theory of $U$ for some tilings via AF-algebras \cite{JulienSavinien}.
In general, there exists a homology theory for Smale spaces \cite{PutnamH} and a duality theory for the Ruelle algebras associated to a Smale space \cite{KPW}, but there is no general theory linking the K-theory of the stable and unstable $C^*$-algebras associated to a Smale space.
We show nonetheless in this article that the $K$-theories of $S$ and $U$ can be computed from the cohomology and homology of a single cochain complex with connecting maps.

The $K$-theory of $S$ is computed via a transversal to $R_s$ as follows: Let $T\in\Omega$ be a fixed tiling.
The equivalence relation 
\begin{equation}\label{e:R_sp} 
R_s':=\bigcup\limits_{n\in\N_0} R_n
\end{equation}
is equipped with the inductive limit topology, where 
$$  R_n:=\frac{1}{\lambda^n} R(\omega^n(T))=\frac1{\lambda^n}\{(x,y)\in \R^d\times \R^d \mid \omega^n(T)(x)-x =\omega^n(T)(y)-y\}$$
is given the subspace topology of $\R^d\times\R^d$.
Then $R'_s$ is topologically groupoid isomorphic to a transversal of $R_s$.
Since $R'_s$ is an \'etale equivalence relation, i.e.~the range and source maps are local homeomorphisms, we can define the $C^*$-algebra
$$S':=C_r^*(R'_s)$$  
in the sense of Renault.
Since $S$ is strongly Morita equivalent to $S'$, their $K$-theory coincide. For more details see Section \ref{s:dynamics}.

In this paper we only focus on tilings of the line and of the plane.
So assume that the fixed tiling $T$ is of dimension $d\in\{1,2\}$.
Although many of our definitions can be generalized immediately to higher dimensions, we will keep them in their simplest form, for the purpose of clarity.

The equivalence relation $R_0$ induces an equivalence relation on the cells.
The $R_0$-equivalence classes of the vertices, edges, and faces are called  the stable vertices, stable edges, and stable faces, respectively, and there is a finite number of them. These numbers are denoted by $sV$, $sE$, $sF$, respectively.
It is convenient to identify the stable cells with its representatives. This proves quite useful when doing explicit computations. 
We thus make the following informal definition 
\begin{defn}[Stable cells]\label{d:stablecellsT0}
If $v$ is a vertex, $e$ an edge, $f$ a face of the fixed tiling $T$ of dimension $d\le 2$, then 
\begin{itemize}
\item stable vertex: $sv:=T(v)$. 
\item stable edge: $se:=T(\open{e})$,\,\, where $\open{e}$ is the edge minus its vertices.
\item stable face: $sf:=f$.
\end{itemize}
Here $T(X)$, for $X\subset \R^d$, is the set of tiles that contain  one or more points of $X$. 
\end{defn}
\noindent
We remark that the stable vertices are collared vertices (cf.~Definition~\ref{d:collaredcells}). 

The cochain map of abelian groups with connecting maps is given as the following commutative diagram
\begin{equation}\label{e:HSdiagram}
\xymatrix{0\ar[r]&\Z^{sV}\ar[r]^{\delta^0}\ar[d]_{W_V}&\Z^{sE}\ar[r]^{\delta^1}\ar[d]_{W_E}&\Z^{sF}\ar[r]\ar[d]_{W_F}&0\\
  0\ar[r]&\Z^{sV}\ar[r]^{\delta^0}&\Z^{sE}\ar[r]^{\delta^1}&\Z^{sF}\ar[r]&0,}
\end{equation}
where the coboundary maps $\delta^0$, $\delta^1$, and the connecting maps $W_V$, $W_E$, $W_F$ are given below, right after the statement of Theorem \ref{t:KTheoryofA_intro}. 
The stable cohomology groups for $T$ are given by
 \begin{eqnarray} \label{e:Hhomology}
 H^2_S(T)&=&\lim_{\to}\xymatrix{\coker\delta^1 \ar[r]^{W_F}&\coker\delta^1\ar[r]^{W_F}&\coker\delta^1\ar[r]^{W_F}&\cdots}\\ \nonumber
 H^1_S(T)&=&\lim_{\to}\xymatrix{\frac{\ker\delta^1}{\im\delta^0}\ar[r]^{W_E}&\frac{\ker\delta^1}{\im\delta^0}\ar[r]^{W_E}&\frac{\ker\delta^1}{\im\delta^0}\ar[r]^{W_E}&\cdots}\\ \nonumber
 H^0_S(T)&=&\lim_{\to}\xymatrix{\ker\delta^0\ar[r]^{W_V}&\ker\delta^0\ar[r]^{W_V}&\ker\delta^0\ar[r]^{W_V}&\cdots}.
 \end{eqnarray}
and the stable-transpose homology groups for $T$ are given by
\begin{eqnarray}\label{e:ST-homology}
H^{ST}_2(T)&=&\lim\limits_{\to}\xymatrix{\ker ({\delta^1}^t) \ar[r]^{W_F^t}&\ker({\delta^1}^t)\ar[r]^{W_F^t}&\ker({\delta^1}^t)\ar[r]^{W_F^t}&\cdots}\\
\nonumber H^{ST}_1(T)&=&\lim\limits_{\to}\xymatrix{\frac{\ker({\delta^0}^t)}{\im({\delta^1}^t)}\ar[r]^{W_E^t}&\frac{\ker({\delta^0}^t)}{\im ({\delta^1}^t)}\ar[r]^{W_E^t}&\frac{\ker({\delta^0}^t)}{\im({\delta^1}^t)}\ar[r]^{W_E^t}&\cdots}\\
\nonumber H^{ST}_0(T)&=&\lim\limits_{\to}\xymatrix{\coker({\delta^0}^t)\ar[r]^{W_V^t}&\coker({\delta^0}^t)\ar[r]^{W_V^t}&\coker({\delta^0}^t)\ar[r]^{\,\,\,W_V^t}&\cdots}.
\end{eqnarray}
Then by Section \ref{s:preliminaries}, Section \ref{s:S} and Section \ref{s:S-U-relationship} 
we have
\begin{thm}[$K$-theory of S]\label{t:stableKtheory}
For tilings of dimension 1 or 2 the following holds
  \begin{itemize}
  \item[line:]  \qquad\qquad\qquad $K_0(S)=H^0_S(T)$ \qquad\qquad\qquad $K_1(S)=\Z$ 
  \item[plane:] \,\, $\xymatrix@C=0.6cm{0\ar[r]&\Z\ar@{^{(}->}[r]&K_0(S)\ar@{->>}[r]&H^0_S(T)\ar[r]&0}$  \qquad $K_1(S)=H^1_S(T)$.
\end{itemize}
  where the sequence is a short exact sequence. Moreover, $H^0_S(T)$ is torsion free for both dimensions.
  In particular, tilings of the line always have torsion free $K$-theory groups. 
  \end{thm}
    
\begin{thm}[$K$-theory of $U$]\label{t:unstableKtheory}
For tilings of the plane with convex tiles, or tilings of the line, the following holds.
\begin{itemize}
 \item[line:] \,\, $K_0(U)=H^{ST}_0(T)$ \qquad\quad $K_1(U)=\Z$ 
 \item[plane:] \,\, $K_0(U)=\Z\oplus H^{ST}_0(T)$ \quad   $K_1(U)=H^{ST}_1(T)$.
\end{itemize}
Moreover, $H^{ST}_1(T)$ is torsion free for both dimensions.
For tilings of the line $H^{ST}_0(T)$ is torsion free.
Thus tilings of the line always have torsion free $K$-theory groups.
\end{thm}
It is worth noting that in the absence of torsion, e.g.~for tilings of dimension 1, $H_S^k(T)$ and $H_k^{ST}(T)$ are given alone by the direct limit of a single matrix and its transpose! (cf.~Corollary \ref{c:A-At}).

Since $S$ is $UCT$, and $A$ is  strongly Morita equivalent to $U\otimes S$, we can use the K\"unneth formula to express the $K$-theory of $A$ in terms of the $K$-theory of $U$ and $S$ (cf.~Section \ref{s:A})
\begin{thm}[$K$-theory of $A$]\label{t:KTheoryofA_intro} For tilings of dimension 1 or 2, the following holds
\begin{itemize}
\item [line:] 
  $K_0(A)=\!\!\Big(\!K_0(S)\otimes K_0(U)\Big) \oplus \Z\qquad$\\
  $K_1(A)=\,K_0(S) \oplus K_0(U)$ 
\item [plane:] If $K_i(S)$ and $K_i(U)$, $i=1,2$ are torsion free then\\
  $K_0(A)=\Big(K_0(S)\otimes K_0(U)\Big) \oplus \Big(K_1(S)\otimes K_1(U)\Big)$\\
  $K_1(A)=\Big(K_0(S)\otimes K_1(U)\Big) \oplus \Big(K_1(S)\otimes K_0(U)\Big).$
\end{itemize}
\end{thm}
\noindent

To define the coboundary maps $\delta^0$, $\delta^1$ we need to assume 
that the stable faces and stable edges have orientation, e.g.~one can put counterclockwise orientation on the faces.
By translation, all cells of $T$ have orientation.
For vertex $v$, edge $e$, face $f$ of $T$ let
\begin{equation}\label{e:deltaev}
\delta_{e,v}:=\left\{
                      \begin{array}{ll}
                        -1, & \text{$v$ is the initial vertex of $e$} \\
                        1, & \text{$v$ is the final vertex of $e$}\\
                        0, & \text{else}
                      \end{array}
                    \right.
\end{equation}
\begin{equation}\label{e:deltafe}
\delta_{f,e}:=\left\{
                      \begin{array}{ll}
                        1, & e\in \partial f \text{ same orient.} \\
                        -1, & e\in \partial f \text{ opps. orient.} \\
                        0, & \text{else}
                      \end{array}
                    \right.
\end{equation}                    
\noindent
Then $\delta^0: \Z^{sV}\to\Z^{sE}$ is given by
\begin{equation}\label{e:delta0-intro}
\delta^0([v]):=\sum_{e\in T(v)} \delta_{e,v}\,\, [e]
\end{equation}
and $\delta^1:\Z^{sE}\to\Z^{sF}$ is given by
\begin{equation}\label{e:delta1-intro}
\qquad\qquad\qquad\quad\delta^1([e]):=\sum_{f\in T(\open{e}) } \delta_{f,e}\,\,\, [f]\,\,=\,\,[f_1]-[f_2],
\end{equation}
where $f_1$ and $f_2$ are the faces in $T$ that contain the edge $e$, and such that $e$ and the corresponding edge in $f_1$ have same orientation.

To define the connecting maps, we need a homotopy $h$ on the prototiles which extends to $\R^d$ by translation of the prototiles (cf.~Definition \ref{d:homotopy}). For edges $e\in T$, $e'\in\omega(T)$, let
$$\delta_{e,e'}:=\left\{
                      \begin{array}{ll}
                        1, & h_1(\frac1{\lambda}e')=e \text{ (matching orientation)} \\
                        -1, & h_1(\frac1{\lambda}e')=e \text{ (opposite orientation)} \\
                        0, & \text{else}
                      \end{array}
                    \right.
$$                         
Then the connecting maps, which are also named substitution-homotopy matrices, are given by
\begin{eqnarray*}
W_F([f])&:=&\sum_{\substack{f'\in\omega(f)\\ h_1(\frac1{\lambda}f')= f\\ \text{$ $}}} [f']\\
W_E([e])&:=&\sum_{\substack{e'\in\omega(T(\open{e}))\\ h_1(\frac1{\lambda}e')=e\\ \text{$ $}}} \delta_{e,e'}\,\, [e']\\
W_V([v])&:=&\sum_{\substack{v'\in\omega(T(v))\\ h_1(\frac1{\lambda}v')=v}} [v'].
\end{eqnarray*}
Compare these formulas with their unstable counterparts given in Section \ref{s:preliminaries}.
(The reason that we use the letter $W$ is that it looks similar to $\omega$).

In Section \ref{s:examples} we computed the $K$-theory of a number of tilings of dimension 1 and dimension 2 using the above formulas. The results are listed in
Tables \ref{table:K0groups-d1}, \ref{table:K1groups-d1}, \ref{table:K0groups-d2}, \ref{table:K1groups-d2}. 
It is interesting to note that in all of our examples of dimension 1, we had the equality $K_i(U)=K_i(S)$, $i=0,1$.
Some of these groups are direct sums of groups of the form $\Z[\frac{1}{\lambda}]$,  $\Z[\frac{1}{\det A}]$, where $\lambda$ denotes the eigenvalues of a matrix $A$ coming from the substitution matrices.
But we also include one example, namely, the tiling that we named Pathologic, whose stable and unstable $K_0$-groups are rank-3 subgroups of $\Z\oplus\Z[1/17]^2$.
Here it is impossible to write the rank-2 subgroup of $\Z[1/17]^2$ nontrivially as a direct sum! Yet those two $K_0$-groups are isomorphic.
For tilings of the plane, it is worth noting that the torsion of the stable and unstable $K$-groups occurred precisely in the groups predicted by Corollary \ref{c:torsionmovement}.
 
We are in the process of extending this article to higher dimensions using spectral sequences. 
It is also our hope that we can remove the homotopy.

The paper is organized in the following way: In the next section we summarize a description of the tiling space as an inverse limit using an alternative definition of the so-called collared tiles. Moreover, we present some remarks on cohomology. In Section \ref{s:dynamics}, we prove that the stable equivalence relation $R_s$ has a transversal which is isomorphic to the inductive equivalence relation $R'_s$ on $\R^d$.
In Section \ref{s:S}, we introduce the $C^*$-algebras $C^*(R_n)$, and calculate their $K$-theory groups in terms of the compacts.
Then we introduce a homotopy of the tiling and use it to construct connecting maps which we show are $*$-homomorphisms. 
We calculate the generators of the $K$-theory groups and formulas for the connecting maps in $K$-theory, and the $K$-theory of $S$.
In Section \ref{s:S-U-relationship}, we show a relationship between the $K$-theory of the stable $C^*$-algebra $S$ (as constructed in this paper) and the $K$-theory of the unstable $C^*$-algebra $U$ via \v{C}ech cohomology, PE-cohomology, PE-homology, 
stable-transpose homology and stable homology, plus the homotopy introduced in Section \ref{s:S}. Moreover, we use a universal coefficient theorem to derive properties of the stable cohomology and the stable-transpose homology.
In Section \ref{s:A} we investigate the asymptotic $C^*$-algebra $A$ and show that the stable $C^*$-algebra $S$ is amenable.
In Section \ref{s:examples} we give explicit calculations of the $K$-theory for a number of tilings of the line and of the plane.
In Appendix \ref{a.s:direct-limits-AG}, we provide methods for calculating the direct limits of automorphisms of finitely generated free abelian groups. In particular, we provide formulas to compute the kernel, cokernel, and homology of integer matrices. 
Such formulas are based on the Smith normal form of integer matrices, and they can easily be implemented for example in Mathematica.
For convenience to the reader we provide in the rest of the appendix some background results on topological spaces and $C^*$-algebras, which we assume are well-known.

This paper is a continuation of the work initiated by the first author in the papers \cite{Gon1}, \cite{Gon2}.
On the one hand, we provide detailed proofs not found in those papers, for example that $R'_s$ is isomorphic to a transversal of $R_s$.
We furthermore give a solid foundation to the concept of a homotopy defined on a tiling, which can be used to define $*$-homomorphisms called connecting maps. These connecting maps are used to give explicit formulas for the $K$-theory of the stable $C^*$-algebra.
We should note that the idea of the homotopy was borrowed from an example in \cite{Gon2}.
Of course, we avoid repeating results from \cite{Gon1}, \cite{Gon2}, but in a few instances we provide alternative short proofs using machinery found in the literature for the full $C^*$-algebra (instead of the reduced $C^*$-algebra as was done in \cite{Gon1}, \cite{Gon2}).  
Moreover, we work mostly with connected components of $R_n$, instead of working directly with $R_n$, in order to make the proofs simpler.

On the other hand, we show the results mentioned in the introduction such as a relationship between the $K$-theories of the stable and unstable $C^*$-algebras, and we provide interesting examples and formulas to compute the $K$-theory groups of the $C^*$-algebras $S$, $U$ and $A$. 

\section{\textbf{Preliminaries}}\label{s:preliminaries}
We start with some remarks on the definition of a tiling $T$ of $\R^d$.
On the one hand, $T$ is a CW-complex with underlying space $\R^d$, such that the closed $d$-cells are homeomorphic to the closed unit disk.
The closed $d$-cells are called tiles, and we allow the possibility that our tiles carry labels.
On the other hand, we will consider $T$ to be the collection of its tiles, which is a cover of $\R^d$ satisfying that two tiles can only intersect on their boundaries.
(cf.~\cite[p.~6]{AP}).
More generally, we consider the closed $k$-cells of $T$ as a cover of the $k$-skeleton of $T$, where $k=0,\ldots,d$.
We will always assume that the cells of $T$ are closed.

The tiling space $\Omega$ is the set of all the tilings whose tiles are translated copies of a finite (fixed) set of prototiles $t_1,\ldots,t_N$, such that
if $T\in\Omega$, and $P\subset T$ is a patch (i.e.~a finite subset) then $P$ is contained in a translated copy of $\omega^k(t_j)$ for some $k\in\N$, $j\in\{1,\ldots,N\}$.
Since the equivalence class $[T]_{R_u}$ is dense in $\Omega$, it holds as well that
$$\Omega=\overline{T+\R^d}^{d}$$
 i.e.~$\Omega$ is the completion under the metric $d$ of the set of all the translations of $T$.
Moreover, one can also write the tiling space as an inverse limit
\begin{equation}\label{e:inverselimit}
\Omega = \lim\limits_{\leftarrow} \xymatrix{\ar@{<<-}[r]\Gamma&\ar@{<<-}[r]\Gamma&\ar@{<<-}[r]\Gamma&\cdots},
\end{equation}
where $\Gamma:=\frac{\Omega\times \R^d}{R_c}$, with $R_c$ being given as follows:
\begin{defn}[Collared equivalence relation $R_c$]
Two pointed tilings $(T,x)$, $(T',x')\in \Omega\times \R^d$ are said to be collared equivalent (or $R_c$-equivalent) if 
 $$T(\sigma) -x = T'(\sigma') - x',$$
 where $\sigma\in T$ (resp.~$\sigma'\in T'$) is the (necessarily unique) closed cell whose interior contains $x$ (resp.~$x'$).
\end{defn}
We remark that by a pointed tiling we just mean a pair consisting of a tiling and a point. 
It is easy to see that $R_c$ is indeed an equivalence relation, and so there is no need to take the transitive closure.  
All the proofs in \cite[Section~4]{AP} carry out almost immediately using $R_c$ instead of the relation $\sim_1$ defined in that section, and thus 
Eq.~(\ref{e:inverselimit}) holds (using $R_c$ instead of $\sim_1$).
We should remark that our definition of $R_c$ is equivalent to the one in \cite[Definition~2.2]{Gahler} and they also show Eq.~(\ref{e:inverselimit}) using their definition.
Since  \v{C}ech cohomology behaves well under inverse limits we get
\begin{equation}\label{e:Hinverselimit}
\breve{H}(\Omega) = \lim\limits_{\rightarrow} \xymatrix{\ar[r]\breve{H}^{k}(\Gamma)&\ar[r]\breve{H}^{k}(\Gamma)&\ar[r]\breve{H}^{k}(\Gamma)&\cdots},
\end{equation}
and since $\Gamma$ is a finite CW-complex, $\breve{H}^{k}(\Gamma)=H_{\mathrm{cell}}^k(\Gamma)$, and cellular cohomology is much easier to compute.

The equivalence relation $R_c$ induces an equivalence relation on the cells.
The $R_c$-equivalence classes of the vertices, edges, and faces are called  the collared vertices, collared edges, and collared faces, respectively, and there is a finite number of them. These numbers are denoted by $cV$, $cE$, $cF$, respectively. 
We make though the following informal definition, just as we did for their stable counterparts.
\begin{defn}[Collared cells]\label{d:collaredcells}
If $v$ is a vertex, $e$ an edge, $f$ a face of a tiling $T$, then 
\begin{itemize}
\item collared vertex: $cv := T(v)$.\\ (i.e.~$cv$ is the set of tiles that contain the vertex $v$).
\item collared edge: $ce := T(e)$.\\ (i.e.~$ce$ is the set of tiles that contain the edge $e$ or its vertices). 
\item collared face: $cf := T(f)$.\\ (i.e.~$cf$ is the set of tiles that share an edge or vertex with the tile $f$).
\end{itemize}
\end{defn}
Using that $\Gamma$ is a finite CW-complex, we can define the following commutative diagram
\begin{equation}\label{e:HUdiagram}
\xymatrix{0\ar[r]&\Z^{cV}\ar[r]^{\partial_1^t}\ar[d]_{\omega_V^t}&\Z^{cE}\ar[r]^{\partial_2^t}\ar[d]_{\omega_E^t}&\Z^{cF}\ar[r]\ar[d]_{\omega_F^t}&0\\
  0\ar[r]&\Z^{cV}\ar[r]^{\partial_1^t}&\Z^{cE}\ar[r]^{\partial_2^t}&\Z^{cF}\ar[r]&0}
\end{equation}
where the exponent $t$ stands for the transpose, the boundary maps $\partial_1$, $\partial_2$ are the standard cellular boundary maps,  and the substitution matrices $\omega_V$, $\omega_E$, $\omega_F$ are given below.
The \v{C}ech cohomology groups are given by
\begin{eqnarray}\label{e:Hcohomology}
\breve H^2(\Omega)&=&\lim_{\to} \xymatrix{\coker\partial_1^t \ar[r]^{\omega_F^t}&\coker\partial_1^t\ar[r]^{\omega_F^t}&\coker\partial_1^t\ar[r]^{\,\,\,\omega_F^t}&\cdots}\\ \nonumber
\breve H^1(\Omega)&=&\lim_{\to}\xymatrix{\frac{\ker\partial_2^t}{\im\partial_1^t}\ar[r]^{\omega_E^t}&\frac{\ker\partial_2^t}{\im\partial_1^t}\ar[r]^{\omega_E^t}&\frac{\ker\partial_2^t}{\im\partial_1^t}\ar[r]^{\omega_E^t}&\cdots}\\ \nonumber
\breve H^0(\Omega)&=& \xymatrix{\Z}.
\end{eqnarray}
By Theorems 6.1, 6.3, 7.1 in \cite{AP} we get 
\begin{thm}[$K$-theory of $U$]\label{t:unstableKtheory-AP}
For tilings of dimension 1 or 2, the following holds.
\begin{itemize}
  \item[line:]  $K_0(U)=\breve H^1(\Omega)$ \qquad \quad $K_1(U)=\Z$ 
  \item[plane:] $K_0(U)=\Z\oplus \breve H^2(\Omega)$ \quad $K_1(U)=\breve H^1(\Omega)$.
\end{itemize}
\end{thm}
To define the boundary maps $\partial_1$, $\partial_2$, we need to assume that the prototiles and the collared edges have orientation.
By translation, all the cells of $T\in\Omega$ have orientation.

\noindent 
The boundary map $\partial_1:\Z^{cE}\to \Z^{cV}$ is given by
$$\partial_1([e]):=\sum_{\substack{v\in e\\ T(v)\subset T(e)\\ \text{$ $}}} \delta_{e,v}\,\,[v],$$
and the boundary map $\partial_2:\Z^{cF}\to \Z^{cE}$ is given by
$$\partial_2([f]):=\sum_{\substack{e\in f\\ T(e)\subset T(f)\\ \text{$ $}}} \delta_{f,e}\,\,[e],$$
where $\delta_{e,v}$ and $\delta_{f,e}$ are defined in Eq.~(\ref{e:deltaev}) and Eq.~(\ref{e:deltafe}).

The collared substitution matrices are given by
$$\omega_F([f]):=\sum_{\substack{f'\in \omega(f)\\ T(f')\subset \omega(T(f))\\ \text{$ $}}} [f'],$$
$$\omega_E([e]):=\sum_{\substack{e'\in \omega(e)\\ T(e')\subset \omega(T(e))\\ \text{$ $}}} [e'],$$
$$\omega_V([v]):=\sum_{\substack{v'\in \omega(v)\\ T(v')\subset \omega(T(v))\\ \text{$ $}}} [v'].$$

\subsection{Cohomology notation}\label{ss:cohomologynotation}
Given a cochain complex 
$$\xymatrix{
&C^\bullet:=\cdots\ar[r]^{\delta^{-1}}&C^0\ar[r]^{\delta^0}&C^1\ar[r]^{\delta^1}&C^2\ar[r]^{\delta^2}&\cdots,\\
}$$
we will denote the cohomology groups by
$$H^n(C^\bullet):=\frac{\ker\delta^n}{\im\delta^{n-1}}\qquad n\in\Z.$$
\section{\textbf{Dynamics}}\label{s:dynamics}
The construction of the Smale space $(\Omega,d,\omega)$ for a substitution tiling, and the dynamics for the unstable equivalence relation $R_u$, were developed in \cite[Section~3]{AP}.
In this section, we consider the stable equivalence relation $R_s$, describe a transversal to it, and set up the ground for computation of the $K$-theory of $S$.

Recall that two tilings $T$, $T'$ are said to be stable equivalent if 
$$\lim_{n\to\infty} d(\omega^n(T),\omega^n(T'))=0.$$
The stable equivalence class for tiling $T$ can be written as (see \cite[p.181]{Putnam96SmaleSpaces})
\begin{equation}\label{e:Vs}
[T]_{R_s}=V^s(T)=\bigcup_{n\in\N_0} \omega^{-n}(V^s(\omega^n(T),\varepsilon_0)),
\end{equation}
where $V^s(\omega^n(T),\varepsilon)$ is the stable local canonical coordinate of $T$ and it is defined in
\cite[p.~9]{AP} as
$$V^s(T,\varepsilon)=\{T'\in\Omega\mid [T,T']=T'\text{\, and\, } d(T,T')<\varepsilon\}\qquad\text{for } 0<\eps\le \eps_0.$$
Moreover, the constant $\varepsilon_0<\frac{1}{\sqrt{2}}$ is fixed, and it depends on the inflation factor $\lambda$ and on some 
separation constant of the diagonal of $R_0$ to the other components of $R_0$,
and it is defined right after Lemma~3.2 in \cite{AP}. Note that $\frac{1}{\eps_0}-\eps_0>0$. 
We remark that the union in (\ref{e:Vs}) is an increasing union (see \cite[p.181]{Putnam96SmaleSpaces})
$$V^s(T,\varepsilon_0)\subset\omega^{-1}(V^s(\omega(T),\varepsilon_0))\subset \omega^{-{2}}(V^s(\omega^{2}(T),\varepsilon_0))\subset\cdots.$$
The equivalence relation $R_s$ is considered as a topological groupoid as follows.
Define
$$G_0^s\subset G_1^s\subset\cdots$$
by
\begin{eqnarray*}
G_0^s&:=&\{(T,T')\in \Omega\times \Omega\mid T'\in V^s(T,\varepsilon_0)\}\\
G_n^s&:=&(\omega\times\omega)^{-n}(G_0^s),
\end{eqnarray*}
and equip $G_n^s\subset\Omega\times\Omega$ with the subspace topology, and give the increasing union
$$R_s=\bigcup_{n\in\N_0} G_n^s$$
the inductive limit topology.
A simple characterization of the stable equivalence relation is

\begin{pro}\label{p:Rs-characterization}
Two tilings $T$, $ T'$are stable equivalent if and only if their tiles containing the origin agree after a finite number of substitutions. That is,
$$T \sim_{R_s} T' \,\,\Longleftrightarrow\,\,\exists\, n\in\N_0\,:\,\omega^n(T)(0)=\omega^n(T')(0).$$
\end{pro}
\begin{proof}
$(\Rightarrow).$ By \cite[p.~9]{AP}, we have
$$V^s(T,\varepsilon)\subset\{T'\in\Omega\mid T\text{ and } T' \text{ agree on } B_{\frac{1}{\varepsilon}-\varepsilon}(0)\},$$
for any $0<\varepsilon\le\varepsilon_0$.
(We would like to point out that we replaced the strict inequality $\varepsilon<\varepsilon_0$ given in \cite[p.~9]{AP} with
$\varepsilon\le\varepsilon_0$. This can be done since a Smale space with constant $\eps_0$ is also a Smale space with smaller constant $\tilde\eps_0$.
We work with $0<\eps\le \tilde \eps_0$, and we rename $\tilde\eps_0$ as $\eps_0$.)
 Recall that two tilings $T,T'\in\Omega$ are said to agree on a set $U\subset \R^d$ if $T(U)=T'(U)$, where $T(U)$ denotes the smallest patch in $T$ containing the set $U$. (cf.~\cite[p.~4]{AP}).
Suppose that $T'\in [T]_{R_s}$. 
Then by Eq.~(\ref{e:Vs}) there exists a $n\in\N_0$ such that
$$\omega^n(T')\in V^s(\omega^n(T),\eps_0).$$
Thus $\omega^n(T')$ and $\omega^n(T)$ agree on a ball of radius $\frac1{\eps_0}-\eps_0>0$ centered at the origin. In particular they agree at the origin, i.e.~
$$\omega^{n}(T')(0)=\omega^{n}(T)(0).$$
$(\Leftarrow).$ 
Suppose that $\omega^n(T)(0)=\omega^n(T')(0)$ for some $n\in\N_0$, and define the patch $P:=\omega^n(T)(0)$. Let $r:=d(0,\partial P)$ be the largest radius  such that $B_r(0)\subset P$.
Note that $0$ is an interior point of $P$, and hence $r>0$.
Then for $k\in\N$, $\omega^{n+k}(T)$ and $\omega^{n+k}(T')$ agree on the patch $\omega^k(P)$. 
Thus they agree on the ball of radius $\lambda^k r$. Hence 
$$d(\omega^{n+k}(T),\omega^{n+k}(T'))\le \frac{1}{\lambda^k r}\to0\quad \text{for}\quad k\to\infty.$$
Thus $T$ and $T'$ are stable equivalent.

\end{proof}
By \cite{{PutnamSpielberg}} the unstable equivalence class $[T]_{R_u}$ is dense in $\Omega$ and $[T]_{R_u}=T+\R^d$ is a transversal to $\Omega$ with respect to $R_s$. We elaborate on this in the rest of the paragraph. The unstable equivalence class $[T]_{R_u}$ is considered as a topological space as follows.
Define 
$$V^u_n(T):=\omega^n(V^u(\omega^{-n}(T),\varepsilon_0))=T+B_{2\varepsilon_0\lambda^n}(0)\subset \Omega,$$
where the equality is given in \cite[p.9]{AP}, and the set $V^u(T,\varepsilon_0)$ is known as the unstable local canonical coordinate of $T$.
By \cite[p.10]{AP}, we can write
\begin{equation}\label{e:transversal}
[T]_{R_u}=V^u(T)=\bigcup_{n\in\N_0} V^u_n(T)=T+\R^d.
\end{equation}
Equip $V^u_n(T)\subset \Omega$ with the subspace topology, and give the increasing union in Eq.~(\ref{e:transversal}) the inductive limit topology.
Define the equivalence relation 
\begin{eqnarray*}
H_s(T)&:=&R_s\cap\big((T+\R^d)\times (T+\R^d)\big)\\
&=&\{(T-x,T-y)\in R_s\mid x,y\in\R^d\}\\
&=&\{(T-x,T-y)\in\Omega\times\Omega\mid\exists n\in\N_0:\, \omega^n(T-x)(0)=\omega^n(T-y)(0)\}.
\end{eqnarray*}
By \cite[Theorem~4.2,~Theorem~3.7]{PutnamSpielberg} $H_s(T)$ is groupoid equivalent to $R_s$ in the sense of \cite{MRW}. Moreover, 
by \cite[Theorem~3.7]{PutnamSpielberg} $H_s(T)$ is \'etale.
Note that the set 
$T+\R^d$ is a (generalized) transversal to $\Omega$ with respect to $R_s$ since the set $[T']_{R_s}\cap(T+\R^d)$ is countable for any $T'\in\Omega$.

By \cite[Theorem~2.8]{MRW} the $C^*$-algebras $C^*(H_s(T))$ and  $C^*(R_s)$ are strongly Morita equivalent, and hence their $K$-theories coincide.
It is worth noting that $T+\R^d\subset\Omega$ has been given the inductive limit topology in Eq.~(\ref{e:transversal}) and not the subspace topology of $\Omega$.
Moreover, $H_s(T)$ is equipped with the topology stated in \cite[Lemma~3.3]{PutnamSpielberg}.
Convergence of sequences in $H_s(T)$ can be stated as follows by \cite[p.10]{PutnamSpielberg}: 
$$(T-x_n,T-y_n)\to_{H_s(T)} (T-x,T-y)$$
if and only if
\begin{eqnarray*}
&&(T-x_n,T-y_n)\to_{R_s} (T-x,T-y)\\
&\text{and}&T-x_n\to_{T+\R^d} T-x\\
&\text{and}&T-y_n\to_{T+\R^d} T-y.
\end{eqnarray*}

\begin{pro}\label{p:Rd-TplusRd}
Equip $T+\R^d$ with the inductive limit topology as in Eq.~(\ref{e:transversal}).
Then the map $\R^d\to T+\R^d$ given by
$$x\mapsto T-x$$
is a homeomorphism.
\end{pro}
\begin{proof}
The map $\alpha:\R^d\to T+\R^d$ defined by
$$\alpha(x):=T-x$$
is bijective because $T-x=T-y$ implies $x=y$ as the tiling $T$ is aperiodic.
Moreover, $\alpha$ is continuous since $d(T-x,T-y)\le\frac12||x-y||$, where the inequality follows since $T-x+\frac{x-y}2$ and $T-y-\frac{x-y}2$ agree everywhere.
Let $B_n:=B_{2\eps_0 \lambda^n}(0)\subset \R^d$, $n\in\N_0$ be the open ball of radius $2\eps_0 \lambda^n$.
Since $B_n\subset \R^d$ is bounded and $\alpha$ is continuous, the restriction map  
 $$\alpha_n:=\,\,\,\,\alpha:B_n\to T+B_n$$
 is a homeomorphism, where $B_n$ and $T+B_n$ have the subspace topology of $\R^d$ and $\Omega$, respectively. (It is the restriction of the homeomorphism $\bar\alpha_n:\overline{B_n}\to \overline{T+B_n}$ defined in the same way).
 Then, since the following diagram commutes,
 \begin{equation*}
\xymatrix{ B_0\ar@{^(->}[r]^{ }\ar[d]_{\alpha_0}^{\cong}& B_1\ar@{^(->}[r]^{ }\ar[d]_{\alpha_1}^{\cong}&B_2\ar@{^(->}[r]\ar[d]_{\alpha_2}^{\cong}&\\
  T+B_0\ar@{^(->}[r]^{ }&T+B_1\ar@{^(->}[r]^{ }&T+B_2\ar@{^(->}[r]&}
\end{equation*}
we get that
$$\R^d=\bigcup_{n\in\N_0}B_n\cong \bigcup_{n\in\N_0} T+B_n= T+\R^d,$$
where the first equality is by Lemma \ref{l:Rd-ind}.
\end{proof}
We should remark that the above proposition would be false if we had equipped $T+\R^d$ with the subspace topology of $\Omega$, as shown in Proposition \ref{p:TplusxToxNotContinuous}.

Using Proposition \ref{p:Rd-TplusRd}, we can "see" the transversal $H_s$ as an equivalence relation on $\R^d$. We provide the details below.
\subsection{Equivalence relation $R_s'$ on $\R^d$.}\label{ss:Rs'}
For a tiling $T$ of $\R^d$, $d\in\N$, let $R(T)$ be the associated equivalence relation on $\R^d$ given by
$$R(T):=\{(x,y)\in \R^d\times\R^d\mid T(x)-x=T(y)-y\}.$$
Let $T\in\Omega$ be a fixed tiling, and define the sequence of tilings with shrunk prototiles
$$T_n:=\frac{1}{\lambda^n}\omega^n(T)\lambda^n,\qquad n\in\N_0$$
and equivalence relations
$$R_n:= R(T_n)=\frac1{\lambda^n}R(\omega^n(T)),\qquad n\in \N_0,$$
where $\omega$ is the substitution map with inflation factor $\lambda>1$.
Note that $T_n$, as a function on $\R^d$, is given by
$$T_n(x):=\frac{1}{\lambda^n}(\omega^n(T)(\lambda^n x)),\quad\text{for } x\in \R^d.$$
It is easy to check that 
$$R_1\subset R_2\subset R_3\subset\ldots,$$
so we can define the equivalence relation $R_s'$ on $\R^d$ by
$$R_s':=\bigcup_{n=0}^\infty R_n.$$
We equip $R_n\subset \R^{2d}$ with the subspace topology, and $R_s'$ with the inductive limit topology, which, as we show next, is \'etale.
Since $R_n$ is an equivalence relation on $\R^d$, it is in particular a groupoid with unit space $\R^d$. 
Every connected component $C$ in $R_n$ is homeomorphic onto its image in $\R^d$ via both the range $r$ and source $s$ maps given by $r(x,y):=x$ and $s(x,y):=y$.
In particular, we have that $r$ and $s$ are local homeomorphisms. Hence $R_n$ is \'etale.
Moreover, $r(C)$ and $s(C)$ are both open in $\R^d$. 
It follows that $R_n$ is open in $R_{n+1}$:  if $x\in R_n$, and $C_n$ (resp.~$C_{n+1}$) is the connected component of $x$ in $R_n$ (resp.~$R_{n+1}$), then $r(C_n)$ is open in $r(C_{n+1})$ since $r(C_n)$ is open in $\R^d$ and $C_n\subset C_{n+1}$.
Thus, $C_n$ is open in $C_{n+1}$ because $r:C_{n+1}\to  r(C_{n+1})$ is a homeomorphism.
As a consequence we have that $R_s'$ is \'etale:
If $x\in R'_s$ then $x\in R_n$ for some $n\in\N_0$. 
As above, let $C_n$ be the connected component of $x$ in $R_n$.
Since $R_n$ is open in $R_{k}$ for $k\ge n$ and since $R_s'$ has the inductive limit topology, then 
$C_n$ is open in $R'_s$.
Furthermore, since  $R_n$ is open in $R_{n+1}$, $C_n$ as a subspace of $R_n$ has the same topology as $C_n$ as a subspace of $R_s'$.
Since $r$ and $s$ restricted to $C_n$ (as a subspace of $R_n$ and hence also as a subspace of $R_s'$) are homeomorphisms onto their images in $\R^d$, 
we see that $r$ and $s$ are local homeomorphisms on $R_s'$.

It is convenient to write $R_n$ in different ways:
 \begin{eqnarray*}
  R_n&=&R(T_n)\\
  &=&\{(x,y)\in\R^{d}\times \R^d\mid T_n(x)-x=T_n(y)-y\}\\
&=&\{(x,y)\in\R^{d}\times \R^d\mid \omega^n(T-x)(0)=\omega^n(T-y)(0)\}.
  \end{eqnarray*} 
Moreover, $R_n$ induces an equivalence relation on the cells of the shrunk tiling $T_n$, namely
\begin{defn}[Stable cells]\label{d:stablecells}
Let $n\in\N_0$ be fixed. If $v$ is a vertex, $e$ an edge, and $f$ a face of the tiling $T_n$, then we define the equivalence relations: 
\begin{itemize}
      \item vertices: $v \sim_{R_n} v'$ $\iff$ $ \exists x\in\R^d: T_n(\open{v}\,\!')=T_n(\open{v})+x$
      \item edges: $e \sim_{R_n} e'$ $\iff$ $ \exists x\in\R^d: T_n(\open{e}\,\!')=T_n(\open{e})+x$
      \item faces: $f \sim_{R_n} f'$ $\iff$ $ \exists x\in\R^d: T_n(\open{f}\,\!')=T_n(\open{f})+x$,
\end{itemize}
where $\open{\sigma}$ denotes the interior of a cell $\sigma$. We remark that the interior of a vertex is by convention the vertex itself.
\end{defn}
The number of equivalence classes of $k$-cells, $k=0,1,2$, is always finite and independent of $n$, and is denoted by sV, sE, sF, respectively. In this paper we are only considering tilings of dimension $d\le 2$, and thus we have that $T_n(\open{f})=f$.
The equivalence classes of vertices, edges, and faces are called stable vertices, stable edges, and stable faces, respectively, but informally it is more convenient to define them as follows
\begin{defn}[Stable cells]
Let $n\in\N_0$ be fixed. If $v$ is a vertex, $e$ an edge, $f$ a face of tiling $T_n$, then 
\begin{itemize}
\item stable vertex: $sv:=T_n(\open{v})=T_n(v)$. 
\item stable edge: $se:=T_n(\open{e})$,\,\, where $\open{e}$ is the edge minus its vertices.
\item stable face: $sf:=T_n(\open{f})=f$.
\end{itemize}
\end{defn}

Note that $\frac{1}{\lambda^n}\omega^n(T)$ and $T_n$ are the same as tilings, i.e.~one can draw $T_n$ by first drawing the tiling $\omega^n(T)$ and then shrinking all its cells by the factor $\frac{1}{\lambda^n}$. Also $\frac{1}{\lambda^n}\omega^n(T)(0)= T_n(0)$, but in general $\frac{1}{\lambda^n}\omega^n(T)(x)\ne T_n(x)$ for $x\in\R^d$, $n\in\N$. 
\subsection{$R_s'$ - $H_s$ relationship.}\label{ss:Rs'Hs}
Given a fixed tiling $T\in\Omega$, the following proposition shows that the transversal equivalence relation $H_s:=H_s(T)$ is groupoid isomorphic to the equivalence $R'_s$.
\begin{pro}\label{p:Rp-Hs}
 The map $R'_s\to H_s$ given by
$$(x,y)\mapsto (T-x,T-y)$$
is a topological groupoid isomorphism. 
\end{pro}
\begin{proof}
By the proof of Proposition \ref{p:Rd-TplusRd}, the map $\alpha:\R^d\to T+\R^d$ given by $\alpha(x):=T-x$ is bijective.
Since 
$$x\sim_{R'_s} y \iff \alpha(x)\sim_{H_s} \alpha(y),$$
 $\alpha$ induces the (non-topological) groupoid isomorphism $\phi:R'_s\to H_s$ given by $\phi(x,y):=(T-x,T-y)$.
 It remains to show that $\phi$ is a homeomorphism.
 We start by showing that $\psi:=\phi^{-1}$ is continuous.
Suppose that $(T-x_k,T-y_k)\to (T-x,T-y)$ in $H_s$. That is, suppose that
$$(T-x_k,T-y_k)\to (T-x,T-y)\text{ in }R_s$$
and
$$T-x_k\to T-x\text{ in }T+\R^d\text{ (ind)}$$
and
$$T-y_k\to T-y\text{ in }T+\R^d\text{ (ind)},$$
where 'ind' stands for convergence in the inductive limit topology of $T+\R^d$.
By Proposition \ref{p:ind:Xn0-eventually}, there exists a $n_0,k_0\in\N_0$ such that 
 $$(T-x,T-y),\,(T-x_k,T-y_k)\in G^s_{n_0},\quad k\ge k_0.$$
Define the set 
$$G_n':=G_n^s\cap\big( (T+\R^d)\times (T+\R^d)\big),\,\, n\in\N_0.$$
Since also $(T-x_k,T-y_k)\in (T+\R^d)\times (T+\R^d)$, then  
$$(T-x,T-y),\,\,(T-x_k,T-y_k)\in G'_{n_0}\subset  G'_{n_0,max},\quad k\ge k_0,$$
where $\eps_1:=\frac{1}{\eps_0}-\eps_0$ and 
$$G'_{n_0,max}:=\{(T-x,T-y)\in(T+\R^d)^2\mid \text{ $T_{n_0}-x$ and $T_{n_0}-y$ agree on $B_{\frac{1}{\eps_1\lambda^n}}(0)$}\}.$$
Since $\eps_1>0$ (as $\eps_0<\frac1{\sqrt2}$), we get  $(\alpha\times\alpha)^{-1}(G'_{n_0,max})\subset R_{n_0}$ as sets. Thus
$$(x,y),\,\,(x_k,y_k)\in R_{n_0},\quad k\ge k_0.$$
Now, since by Proposition \ref{p:Rd-TplusRd}
$$x_k=\alpha^{-1}(T-x_k)\to \alpha^{-1}(T-x)=x$$
and
$$y_k=\alpha^{-1}(T-y_k)\to \alpha^{-1}(T-y)=y$$
in $\R^d$, and $R_{n_0}\subset \R^{2d}$ has the subspace topology, then, by Lemma \ref{l:seqInSubspace},
$$(x_k, y_k)\to (x,y) \text{ in } R_{n_0}$$
and hence in $R_s'$, because the inclusion $R_{n_0}\hookrightarrow R_s'$ is continuous.
This completes the proof that $\psi$ is continuous.

We now show that $\phi$ is continuous.
Suppose that $(x_k,y_k)\to (x,y)$ in $R'_s$.
We need to show that 
$(T-x_k,T-y_k)\to (T-x,T-y)$ in $H_s$.
That is, we need to show the following three statements
\begin{equation}\label{e:Hs-0}
(T-x_k,T-y_k)\to (T-x,T-y) \text{\quad in\quad } R_s,
\end{equation} 
\begin{equation}\label{e:Hs-1}
T-x_k\to  T-x \text{\quad in\quad } T+\R^d\text{ (ind),}
\end{equation}
\begin{equation}\label{e:Hs-2}
 T-y_k\to T-y \text{\quad in\quad } T+\R^d\text{ (ind)}.
\end{equation}
Note that $R_s'\hookrightarrow \R^d\times\R^d$ is continuous since all the inclusions $R_n\hookrightarrow \R^d\times\R^d$ are continuous.
By this and since $(x_k,y_k)\to (x,y)$ in $R'_s$, we get that
$$x_k\to x\qquad\text{and}\qquad y_k\to y\text{\quad in\quad} \R^d.$$
Hence, Eq.~(\ref{e:Hs-1}) and Eq.~(\ref{e:Hs-2}) hold by Proposition \ref{p:Rd-TplusRd}.
It remains to show Eq.~(\ref{e:Hs-0}).
By Proposition \ref{p:ind:Xn0-eventually}, there exists $m_0,k_0\in\N_0$ such that
$$(x,y),(x_k,y_k)\in R_{m_0}\qquad k\ge k_0,$$
and by Lemma \ref{l:seqInSubspace} $(x_k,y_k)\to (x,y)$ in $R_{m_0}$.
Let 
$$G'_{\min}:=\bigcup_{n\in\N_0} G'_{n,\text{min}},$$
where
$$G'_{n,\min}:=\{(T-x,T-y)\in(T+\R^d)^2\mid \text{ $T_{n}-x$ and $T_{n}-y$ agree on $B_{\frac{1}{\eps_0\lambda^n}}(0)$}\}.$$
For $n\in\N_0$, give $G'_{n,\min}$ the subspace topology of $(T+\R^d)\times(T+\R^d)$, where $T+\R^d$ has the inductive limit topology.
By a proof similar to the one of Proposition \ref{p:Rs-characterization}, we have $H_s=G'_{\min}$ as sets. Hence there is an $n'_0\in\N_0$ such that
$(T-x,T-y)\in G'_{n_0',\text{min}}$. Let 
$$n_0:=\max(n_0'+1,m_0+1),$$
and let $C_G^{n_0}$ be the connected component of $(T-x,T-y)$ in $G'_{n_0,\min}$. 
Let $C_R^{n_0}$ be the connected component of $(x,y)$ in $R_{n_0}$.
Before we proceed we need the following auxiliary result
\begin{lem} Let $\alpha:\R^d\to T+\R^d$ be the homeomorphism from Proposition \ref{p:Rd-TplusRd}.
Let $C_G^{n}$ be the connected component of $(T-x,T-y)$ in $G'_{n,\min}$, and  
let $C_R^{n}$ be the connected component of $(x,y)$ in $R_{n}$.
Then
$(\alpha\times \alpha)^{-1}(C^n_G)$ is a subset of $C^n_R$. 
Moreover, $(\alpha\times\alpha)^{-1}(C_G^n)$ is a subset of the 
$C_R^{n+1}$-interior of $(\alpha\times\alpha)^{-1}(C_G^{n+1})$, $n\in\N_0$.
\end{lem}
\begin{figure}[b]
\centerline{
\begin{tikzpicture}[scale=2]


\draw[fill,yellow!30] (0,0)--(6,0)--(6,2)--(0,2)--(0,0);
\draw[dashed] (0,0)--(6,0)--(6,2)--(0,2)--(0,0);
\draw[fill,blue!20] (.25,.25)--(2.25,.25)--(2.25,1.75)--(.25,1.75)--(.25,.25);
\draw (.25,.25)--(2.25,.25)--(2.25,1.75)--(.25,1.75)--(.25,.25);

\draw[fill,green!30] (.5,.5)--(2,.5)--(2,1.5)--(.5,1.5)--(.5,.5);
\draw (.5,.5)--(2,.5)--(2,1.5)--(.5,1.5)--(.5,.5);

\draw[<->,ultra thick,black!40!green] (0,1)--(.5,1);
\node[black!60!green,ultra thick] at (.17,1.15) {$r_n$};

\node[green,ultra thick] at (1.25,1) {$r'(C^n_G)\subseteq r'(\tilde C^n_G)$};
\node[black,ultra thick] at (1.25,1) {$r'(C^n_G)\subseteq r'(\tilde C^n_G)$};

\draw[<->,blue, ultra thick] (0,.5)--(.25,.5);
\node[blue,ultra thick] at (.2,.65) {$r_{n+1}$};

\node[blue,ultra thick] at (1.5,1.66) {$r'(C^{n+1}_G)$};
\node[black,ultra thick] at (1.5,1.66) {$r'(C^{n+1}_G)$};

\node[yellow,ultra thick] at (3.5,1.75) {$r(C^n_R)$};
\node[black,ultra thick] at (3.5,1.75) {$r(C^n_R)$};
 
\draw[fill,white] (2.5,.15)--(3,.15)--(3,1.9)--(2.5,1.9)--(2.5,.15);
\draw[dashed] (2.5,.15)--(3,.15)--(3,1.9)--(2.5,1.9)--(2.5,.15);
\draw[fill,green!30] (3.5,.5)--(5.5,.5)--(5.5,1.5)--(3.5,1.5)--(3.5,.5);
\draw (3.5,.5)--(5.5,.5)--(5.5,1.5)--(3.5,1.5)--(3.5,.5);

\node[green,ultra thick] at (4.5,1) {$r'(\tilde C^n_G)$};
\node[black,ultra thick] at (4.5,1) {$r'(\tilde C^n_G)$};


\end{tikzpicture}
}
 \caption{Subsets $r(C_R^n)$, $r'(C_G^{n})$, $r'(C_G^{n+1})$, $r'(\tilde C_G^{n})$ of $\R^d$, where $r':=r\circ(\alpha\times\alpha)^{-1}$.\label{f:CGs}}
\end{figure}
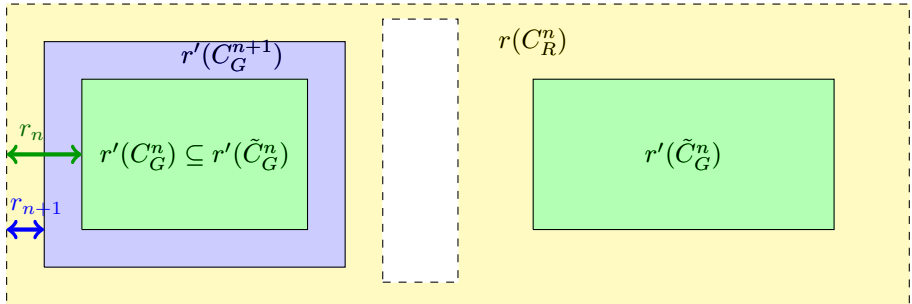
\begin{proof}
The proof of the lemma is illustrated in Figure \ref{f:CGs}.
Since $\eps_0>0$, we have that $(\alpha\times \alpha)^{-1}(G'_{n,\min})\subset R_n$ as sets.
Since $(\alpha\times\alpha)^{-1}:G'_{n,\min}\to R_n$ is continuous by restriction and corestriction, and since $C^n_G$ is connected,
we get that $(\alpha\times\alpha)^{-1}(C^n_G)$ is connected.
Thus $(\alpha\times\alpha)^{-1}(C^n_G)\subset C^n_R$. This proves the first statement of the lemma.
Define
\begin{eqnarray*}
\tilde C^n_G&:=&\{(T-x,T-y)\in (T+\R^d)^2\mid (x,y)\in C^n_R,\,\, B_{r_n}(x)\subset r(C^n_R)\}\\
&=&\{(T-x,T-y)\in (T+\R^d)^2\mid (x,y)\in C^n_R,\,\, d(x,\partial_{\R^d}(r(C^n_R)))\ge r_n\},
\end{eqnarray*}
where $r_n:=\frac{1}{\eps_0\lambda^n}$ and $B_{r_n}(x)$ is the open ball of $\R^d$ of radius $r_n$ and with center $x\in \R^d$.
Roughly speaking, $\tilde C^n_G$ is obtained from $C^n_R$ by removing a neighborhood of the boundary.
We claim that 
\begin{equation}\label{e:claim:Gmin-CR}
(\alpha\times \alpha)^{-1}(\tilde C^n_G)=(\alpha\times\alpha)^{-1}(G'_{n,\min})\cap C^n_R\quad\text{(as sets)}.
\end{equation}
Then, since $C^n_G\subset G'_{n,\min}$ and $(\alpha\times \alpha)^{-1}(C^n_G)\subset C^n_R$, we get by the claim that 
\begin{equation}\label{e:CG-CGtilde}
(\alpha\times \alpha)^{-1}(C_G^n)\subset (\alpha\times \alpha)^{-1}(\tilde C^n_G)\subset C^n_R.
\end{equation}
We now show our claim. We start by showing Eq.~(\ref{e:claim:Gmin-CR})$(\subseteq):$
Let $(x',y')\in(\alpha\times\alpha)^{-1}(\tilde C^n_G)$. Then $(T-x',T-y')\in \tilde C^n_G$.
Then $(x',y')\in C^n_R$ and $B_{r_n}(x')\subset r(C^n_R)$.
Let $x''\in B_{r_n}(x')\subset r(C^n_R)$. 
Let $(x'',y''):=(r|_{C^n_R})^{-1}(x'')\in C^n_R\subset R_n$.
Thus $(T_n-x'')(0)=(T_n-y'')(0)$.
Since $(x'',y'')\in C^n_R$, $y''=x''+y'-x'$. 
The last two equalities 
together with the identity (cf.~\cite[p.9]{AP})
$$(T_n-x'')(0)+x''=T_n(x'')=(T_n-x')(x''-x')+x'$$
give
\begin{eqnarray}\label{e:balltranslation}
(T_n-x')(x''-x')&=&(T_n-x'')(0)+(x''-x')\\
\nonumber&=&(T_n-y'')(0)+(y''-y')\\
\nonumber&=&(T_n-y')(y''-y').
\end{eqnarray}
i.e.~$(T_n-x')$ and $T_n-y'$ agree on $B_{r_n}(0)$.
Thus $(x',y') \in (\alpha\times\alpha)^{-1}(G'_{n,\min})$.
Since $(x',y')\in C^n_R$, the inclusion $\subseteq$ holds. 

We now show Eq.~(\ref{e:claim:Gmin-CR})$(\supseteq)$ by contraposition:
Suppose that $(x',y')\in C^n_R$ but $(x',y')\not\in(\alpha\times\alpha)^{-1}(\tilde C_G)$, we want to show that $(x',y')\not\in (\alpha\times \alpha)^{-1}(G'_{n,\min})$.
By this assumption we have $B_{r_n}(x')\not\subset r(C^n_R)$.
We claim that this implies that 
$$B_{r_n}(x')\cap \partial r(C^n_R)\ne \emptyset.$$
Suppose for contradiction that $B_{r_n}(x')\cap \partial r(C^n_R)= \emptyset$.
Then
\begin{eqnarray*}
B_{r_n}(x')&=&\big(B_{r_n}(x')\cap r(C^n_R)\big)\cup \big(B_{r_n}(x')\cap \partial r(C^n_R)\big) \cup\big(B_{r_n}(x')\cap \big(\R^d\backslash r(C^n_R)\big)^\circ\big)\\
      &=&\big(B_{r_n}(x')\cap r(C^n_R)\big)\cup \big(B_{r_n}(x')\cap \big(\R^d\backslash r(C^n_R)\big)^\circ\big).
\end{eqnarray*}
Moreover, $B_{r_n}(x')\cap r(C^n_R)$ and $B_{r_n}(x')\cap \big(\R^d\backslash r(C^n_R)\big)^\circ$ are both open in $B_{r_n}(x')$ with the subspace topology of $\R^d$.
Since they are complement of each other, they are clopen in $B_{r_n}(x')$.
Moreover, $B_{r_n}(x')\cap r(C^n_R)$ is not empty since it contains $x'$, and $B_{r_n}(x')\cap \big(\R^d\backslash r(C^n_R)\big)^\circ$ is not empty because $B_{r_n}(x')\not\subset r(C^n_R)$.
Hence $B_{r_n}(x')$ is not connected, a contradiction, which proves our claim.

We illustrate this paragraph in Figure \ref{f:CRxpp}. Pick $x''\in B_{r_n}(x')\cap \partial r(C^n_R)$, and let $y''=x''+y'-x'$. 
\begin{figure}[t]
\centerline{\begin{tikzpicture}[scale=2.5]
\draw[->] (0,0) -- (2.2,0);
\draw[dotted,gray] (0,1) -- (2,1);
\draw[dotted,gray] (0,2) -- (2,2);
\draw[->] (0,0) -- (0,2.1);
\draw[dotted,gray] (1,0) -- (1,2);
\draw[dotted,gray] (2,0) -- (2,2);
\draw[ultra thick] (1.02,0.02)--(1.98,.98);
\draw (1,0) circle (.04cm);
\draw (2,1) circle (.04cm);
\draw[thick] (0.02,1.02)--(0.98,1.98);
\draw (0,1) circle (.04cm);
\draw (1,2) circle (.04cm);
\draw[thick] (0,0) -- (2.1,2.1);
\draw[ultra thick] (1.02,0) -- (1.98,0);
\draw (2,0) circle (.04cm);

\draw[dashed](1.5,-.05)--(1.5,.5)--(-.05,.5);
\draw[dashed](2,-.05)--(2,1)--(-.05,1);
\node at (1.5,-.15) {$x'$};
\node at (-.1,0.5) {$y'$};
\node at (2,-.15) {$x''$};
\node at (-.15,1) {$y''$};
\node at (1.5,.7) {$C_R^n$};
\node at (1.75,0.1) {$r(C_R^n)$};
\node at (2.2,2) {$R_n$};
 
\end{tikzpicture}}
 \caption{$(x'',y'')\not\in R_n$.\label{f:CRxpp}}
\end{figure}
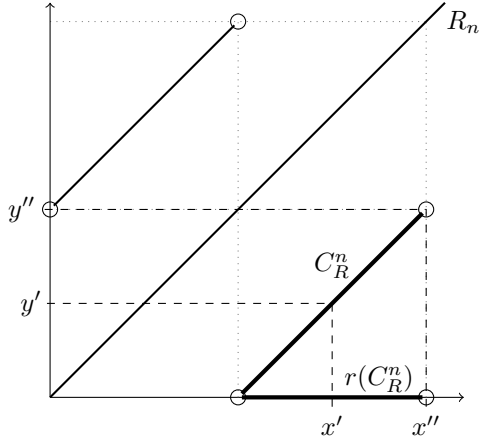
Then since $r(C_R^n)$ is open and since 
$$(\tilde x, \tilde y)\in C_R^n\iff \big(\tilde y-y'=\tilde x-x'\text{ \,and\, } \tilde x\in r(C^n_R)\big)$$
we get that $(x'',y'')\not\in C_R^n$.  
Since also $x''\in\partial r(C_R^n)$ we have $(x'',y'')\not\in R_n$.
Thus $(T_n-x'')(0)\ne (T_n-y'')(0)$.
Hence 
$(T_n-x')$ and $(T_n-y')$ do not agree on $B_{r_n}(0)$ as $y''-y'=x''-x'\in B_{r_n}(0)$ and by Eq.~(\ref{e:balltranslation}).
Hence $(x',y')\not\in (\alpha\times\alpha)^{-1}(G'_{n,\min})$, and thus Eq.~(\ref{e:claim:Gmin-CR})$(\supseteq)$ has been shown.

We will now show the second statement of the lemma.
Let $(x',y')\in (\alpha\times \alpha)^{-1}(C^n_G)\subset (\alpha\times\alpha)^{-1}(\tilde C^n_G)$. 
Then $(x',y')\in C_R^{n}\subset C_R^{n+1}$ and $B_{r_n}(x')\subset r(C^n_R)\subset r(C_R^{n+1})$, where the last inclusion is because $R_n\subset R_{n+1}$.
We need to show that $(x',y')$ is in the interior of $(\alpha\times \alpha)^{-1}(C_G^{n+1})$ as subsets of $C_R^{n+1}$.
Let  $\tilde r:=r_n-r_{n+1}$. Then for any $x''\in B_{\tilde r}(x')$ we get by the triangle inequality that
$$B_{r_{n+1}}(x'')\subset B_{r_n}(x')\subset r(C^{n+1}_R),$$
which shows that $x''\in r'(\tilde C_G^{n+1})$ and hence that $x''\in r'(C^{n+1}_G)$ as $B_{\tilde r}(x')$ is connected, where $r':=r\circ (\alpha\times \alpha)^{-1}$.
Thus
$$U:=(r|_{C^{n+1}_R})^{-1}(B_{\tilde r}(x'))\subset (\alpha\times\alpha)^{-1}(C_G^{n+1})$$
is a neighborhood of $(x',y')$ and is open in $C_R^{n+1}$.

\end{proof}
\proof[Continuation of the proof of Proposition \ref{p:Rp-Hs}]
Since $(x_k,y_k)\to (x,y)$ in $R_{n_0}$, and since $C_R^{n_0}$ is open, we get that 
$(x_k,y_k)\to (x,y)$ in $C_R^{n_0}$ for $k\ge k_1$ for some $k_1\ge k_0$.
By the lemma above, the point $(x,y)$, which is in $(\alpha\times\alpha)^{-1}(C_G^{n_0-1})$, is also in the $C_R^{n_0}$-interior of $(\alpha\times\alpha)^{-1}(C_G^{n_0})$, and thus
$(x_k,y_k)\in (\alpha\times\alpha)^{-1}(C_G^{n_0})$ for $k\ge k_2$ for some $k_2\ge k_1$.
By Lemma \ref{l:seqInSubspace} $(x_k,y_k)\to (x,y)$  in $(\alpha\times\alpha)^{-1}(C_G^{n_0})\subset C_R^{n_0}$ for $k\ge k_2$.
Applying the homeomorphism $\alpha\times\alpha$ we get
$(T-x_k,T-y_k)\to (T-x,T-y)$ in $C_G^{n_0}$ for $k\ge k_2$.
Since the inclusion $f:T+\R^d\to \Omega$ is injective and continuous by \cite[Lemma~4.1,~p.16]{PutnamSpielberg}, where $T+\R^d$ has the inductive limit topology,
then the restriction $f\times f: C_G^{n_0}\to G_{n_0}$ is continuous, and hence 
$(T-x_k,T-y_k)\to (T-x,T-y)$ in $G_{n_0}$ for $k\ge k_2$. Since the inclusion $G_{n_0}\hookrightarrow R_s$ is continuous, we get
$(T-x_k,T-y_k)\to (T-x,T-y)$ in $R_s$.
\end{proof}
\begin{cor}
Let $\alpha:\R^d\to T+\R^d$ be the homeomorphism from Proposition \ref{p:Rd-TplusRd}.
For $n\in \N_0$ let
\begin{eqnarray*}
H_n&:=&(\alpha\times\alpha)(R_n)\\
&=&\{(T-x,T-y)\in (T+\R^d)\times (T+\R^d)\mid (T_n-x)(0)=(T_n-y)(0)\}.
\end{eqnarray*}
Give $H_n$ the subspace topology of $(T+\R^d)\times(T+\R^d)$, where $T+\R^d$ has the inductive limit topology.
Equip the increasing union 
$$H_0\subset H_{1}\subset \cdots$$
with the inductive limit topology.
Then 
$$H_s\,=\bigcup_{n\in\N_0} H_n$$
 as topological spaces.
\end{cor}
\begin{proof}
%
Let $$H_s'':=\bigcup_{n\in\N_0}H_n.$$
 Clearly $H_s=H_s''$ as sets.
 Since  $\phi_n: R_n\to H_n$ given by the restriction
 $$\phi_n:= \alpha\times \alpha:R_n\to H_n$$
 is a homeomorphism, and since the following is a commutative diagram
 \begin{equation*}
\xymatrix{ R_0\ar@{^(->}[r]^{ }\ar[d]_{\phi_0}^{\cong}& R_1\ar@{^(->}[r]^{ }\ar[d]_{\phi_1}^{\cong}& R_2\ar@{^(->}[r]\ar[d]_{\phi_2}^{\cong}&\\
  H_0\ar@{^(->}[r]^{ }&H_1\ar@{^(->}[r]^{ }&H_2\ar@{^(->}[r]&}
\end{equation*}
we get
$$H_s\cong R_s'=\lim\limits_{\to}  R_n \cong \lim\limits_{\to} H_n= H_s''.$$
Since the map from $H_s$ to $H_s''$ is the identity map, which by the above expression is a homeomorphism, the corollary follows.
\end{proof}
%

\section{\textbf{Stable $C^*$-algebra $S$}}\label{s:S}
\noindent
In this section we build a cochain complex with connecting maps from which we define the so-called stable cohomology. We describe the $K$-theory of $S$ in terms of this stable cohomology. 

For fixed tiling $T\in \Omega$, recall the definition of the tiling $T_n$ with shrunk prototiles given by
$$T_n:=\frac{1}{\lambda^n}\omega^n(T)\lambda^n\qquad n\in\N_0.$$
That is $T_n(x):=\frac{1}{\lambda^n}(\omega^n(T)(\lambda^n x))$ for $x\in \R^d$.

Recall as well the associated equivalence relation $R_n$ (for $T_n$) on $\R^d$ given by 
$$R_n:=\{(x,y)\in\R^{2d}\mid T_n(x)-x=T_n(y)-y\}.$$
Then the increasing union 
$$R_s':=\bigcup_{n=0}^\infty R_n$$
is an \'etale equivalence relation, where $R_n\subset \R^{2d}$ is given the subspace topology and $R_s'$ the inductive limit topology.
For more details see subsection \ref{ss:Rs'}. That subsection also comments on the connected components of $R_n$, which play an important role in some of the proofs of this section.

Let $X_k^{T_n}$ be the $k$-skeleton of $T_n$, $n\in\N_0$, $k=0,\dots,d$.
 Then $X_k^{T_n}$ is $R_n$-invariant, i.e.~if $x\in X_k^{T_n}$ and $x\sim_{R_n}y$ then $y\in X_k^{T_n}$, or more succinctly,
 $$s(r^{-1}(X_k^{T_n}))\subset X_k^{T_n},$$
 where $r$ and $s$ denote the range and source map of the groupoid $R_n$.
 The restriction
 \begin{eqnarray*}
 R_n\mid_{_{X_k^{T_n}}}&:=&r^{-1}(X_k^{T_n})\cap s^{-1}(X_k^{T_n}),\qquad n\in\N_0,\, k=0,\ldots,d\\
 &=&\{(x,y)\in R_n\mid x,y\in X_k^{T_n}\}
 \end{eqnarray*}
 is a sub-equivalence relation of $R_n$ and hence a subgroupoid of $R_n$ with unit space $X_k^{T_n}$.
 Moreover, the complement
 $$X_k^{T_n}-X_{k-1}^{T_n}:=X_k^{T_n}\backslash X_{k-1}^{T_n}, \quad n\in\N_0,$$
 is also $R_n$-invariant, and the restriction
 $$R_n\mid_{X_k^{T_n}-X_{k-1}^{T_n}}, \qquad n\in\N_0,\, k=1,\ldots d$$
 is a sub-equivalence relation of $R_n$ and thus a subgroupoid of $R_n$ with unit space $X_k^{T_n}-X_{k-1}^{T_n}$.
Since the equivalence relations $R_n$ are \'etale, i.e.~$r:R_n\to \R^d$ and $s:R_n\to \R^d$ are local homeomorphisms, 
by restriction these maps are still local homeomorphisms, and hence the above sub-equivalence relations are also \'etale.
We can now define the following reduced groupoid $C^*$-algebras in the sense of Renault \cite{Renault}.
\begin{defn}[Groupoid $C^*$-algebras]\label{d:ABCIJ}
 For $n\in\N_0$, let
  \begin{eqnarray*}
    &&A_n:=C_r^*(R_n)\\
    &&B_n:=C_r^*(R_n|_{X^{T_n}_1})\\
    &&C_n:=C_r^*(R_n|_{X^{T_n}_0})\\
    &&I_{n}:=C_r^*(R_n|_{X^{T_n}_2-X^{T_n}_1})\\
    &&J_{n}:=C_r^*(R_n|_{X^{T_n}_1-X^{T_n}_0}).
  \end{eqnarray*}
\end{defn}
\noindent
By the following proposition the algebras $C_n$, $I_n$, $J_n$, $n\in\N_0$, can be written in terms of the compacts.
Using the theory for the compact operators, we can then compute their corresponding $K$-theory groups. 
For convenience to the reader, we provide in the appendix a short proof to the following proposition.
 \begin{pro}{\cite[Prop.~3.7-9]{Gon1}} \label{p:compacts}
 For $n\in\N_0$,
  \begin{enumerate}
  \item $C_n\cong \bigoplus\limits_{k=1}^{\mathrm{sV}} K(\ell^2([v_k]))$
\item $J_n \cong \bigoplus\limits_{k=1}^{\mathrm{sE}} C_0(\open{e}_k, K(\ell^2([e_k])))$
\item $I_{n}\cong \bigoplus\limits_{k=1}^{\mathrm{sF}} C_0(\open{f}_k, K(\ell^2([f_k])))$
\end{enumerate}
where $\mathrm{sV}, \mathrm{sE}, \mathrm{sF}$ are the number of stable vertices, stable edges, stable faces, respectively, and
$v_k, e_k, f_k$ are  representatives of the $R_n$-equivalence classes.
\end{pro}
\begin{proof}
Proposition \ref{p:a:compacts}.
\end{proof}
Taking the $K$-theory of these algebras, we get (cf.~\cite{Rordam})
\begin{cor}For $n\in \N_0$
\begin{enumerate}
\item $K_0(C_n)\cong\Z^\mathrm{sV}$ \qquad$K_1(C_n)=0$.
\item $K_0(J_n)=0$\qquad\quad\, $K_1(J_n)\cong \Z^{\mathrm{sE}}$
\item $K_0(I_n)\cong \Z^{\mathrm{sF}}$ \qquad\,\,$K_1(I_n)=0$
\end{enumerate}
\end{cor}

Since $I_n$, $J_n$ and $C_n$ are given in terms of the compacts
, they are type I and hence amenable. Using that $C_n$ is amenable we get by \cite[Remark~4.10~p.32]{Ren-IdealStructure} the short exact sequence
\begin{equation}\label{e:JBC}
  \xymatrix{0\ar[r]&J_n\ar[r]&B_n\ar[r]&C_n\ar[r]&0.}
\end{equation}
Since the class of type I $C^*$-algebras is closed under extensions we see that $B_n$ is type I, hence amenable. 
Using that $B_n$ is amenable we get by arguing as above the short exact sequence
\begin{equation}\label{e:IAB}
  \xymatrix{0\ar[r]&I_{n}\ar[r]&A_n\ar[r]&B_n\ar[r]&0,}
\end{equation}
and that $A_n$ is type I, hence amenable.
We would like to remark that it is crucial that the quotient algebras be amenable for the above statements to hold, otherwise, one can find counterexamples in \cite[Remark~4.10~p.32]{Ren-IdealStructure} and \cite{Willett}.

We are interested in computing the $K$-theory of the $C^*$-algebra (cf.~Proposition \ref{ap:Sisdirlim}) 
$$S'=\lim_{\to} \xymatrix{A_0\ar[r]^{\iota_0}&A_1\ar[r]^{\iota_1}&A_2\ar[r]^{\iota_2}&},$$
where the inclusion map $\iota_n:A_n\hookrightarrow A_{n+1}$, $n\in\N_0$, is a $*$-homomorphism by Remark \ref{r:iota-cont}.
We will first compute the abelian groups $K_0(A_n)$, $K_1(A_n)$, and then we will construct a homotopy of $*$-homomorphisms in order 
to compute the group homomorphisms $K_0(\iota_n)$, $K_1(\iota_n)$.
\subsection{The $K$-theory groups of $A_n$}
To compute the $K$-theory groups $K_0(A_n)$ and $K_1(A_n)$, we need to put an orientation on the cells of the tiling.
For tilings of the plane, we can put on the faces the counterclockwise orientation, and for 
the edges we can put the orientation of the vectors 
$$(\cos\theta,\sin\theta),\quad \theta\in(-\frac{\pi}{2},\frac{\pi}{2}].$$
Each short exact sequence of $C^*$-algebras induces in $K$-theory a six-term exact sequence of abelian groups. 
From Eq.~(\ref{e:JBC}) we get  the six term exact sequence
\begin{equation}\label{e:K-JBC}
\xymatrix{
0\ar[r]&K_0(B_n)\ar@{^{(}->}[r]&\ar[d]^{\delta^0} K_0(C_n)      &\!\!\!\!\!\!\!\!\!\!\!\!\!\!\!\cong \Z^{\mathrm{sV}}\\
0\ar[u]      &K_1(B_n)\ar[l]&\ar@{>>}[l] K_1(J_n)&\!\!\!\!\!\!\!\!\!\!\!\!\!\!\!\cong \Z^{\mathrm{sE}}.
}
\end{equation}
because $K_0(J_n)=0$ and $K_1(C_n)=0$. For $n=0$, the exponential map $\delta^0:\Z^{\sV}\to\Z^{\sE}$ is computed to be (cf.~\cite[Prop.~2.1]{Gon2})
\begin{equation}\label{e:delta0}
\delta^0([v]):=\sum_{e\in T(v)} \delta_{e,v}\,\, [e],
\end{equation}
where  $\delta_{e,v}$ was defined in Eq.~(\ref{e:deltaev}), and
$T(v)$ is the set of all tiles in $T$ that contain the vertex $v$, and $e\in T(v)$ means all the edges in the patch $T(v)$.
We should remark that we are using a different convention for $\delta^0$ than the one in \cite{{Gon2}} (one is the negative of the other
). Our $\delta^0$ corresponds to the standard differential defining cellular cohomology.
There is no ambiguity in calling $\delta^0:K_0(C_n)\to K_1(J_n)$ and $\delta^0:\Z^{\sV}\to\Z^{\sE}$ with the same name, as the latter is a computation of 
the former via the isomorphisms $K_0(C_n)\cong\Z^{\sV}$, $K_1(J_n)\cong\Z^{sE}$.
For convenience to the reader we provide a proof of the description of the exponential map for tilings of dimension 1 in Proposition \ref{p:delta0-d1}.

From Eq.~(\ref{e:IAB}) we get the six term exact sequence
\begin{equation}\label{e:K-IAB-1}
\xymatrix{
\Z^{sF}\cong&\!\!\!\!\!\!\!\!\!\!\!\!\!\!\!K_0(I_n)\ar[r]&K_0(A_n)\ar@{>>}[r]&\ar[d] K_0(B_n) \\
     &\!\!\!\!\!\!\!\!\!\!\!\!\!\!\!K_1(B_n)\ar[u]^{\tilde\delta^1}       &K_1(A_n)\ar@{_{(}->}[l]&\ar[l] 0.
}
\end{equation}
because $K_1(I_n)=0$, where $\tilde\delta^1$ is the index map.

From Eq.~(\ref{e:K-JBC}), we get 
$$K_0(B_n)\cong \ker\delta^0$$
\begin{equation}\label{e:K1ofBn}
K_1(B_n)\cong \frac{\Z^{sE}}{\im \delta^0}
\end{equation}
Thus, Eq.~(\ref{e:K-IAB-1}) can be rewritten as 
\begin{equation}\label{e:K-IAB-2}
\xymatrix@C=.5pc @R=.55pc{
			  & 			  &&\Z^{\sF}\ar@{}[d]|*[@]{\cong}\\
			  &				  &&K_0(I_n)\ar[rr]&&K_0(A_n)\ar@{>>}[rr]&&\ar[dd]K_0(B_n)        &\!\!\!\cong \ker\delta^0\subset \Z^{\sV} \\
			  &				  &&               &&                    &&                       &\\
\Z^{\sE}\cong &\!\!\!\! K_1(J_n)\ar@{>>}[rr]_{q}\ar[rruu]^{\delta^1}          
                              &&K_1(B_n)\ar@{}[d]|*[@]{\cong}\ar[uu]_{\tilde\delta^1}       
                                               &&K_1(A_n)\ar@{_{(}->}[ll]&&\ar[ll] 0,&&\\
              &   			  &&\frac{\Z^{\sE}}{\im \delta^0}	                                                    
}
\end{equation}
where $q$ is the quotient map, and $\delta^1:=\tilde\delta^1\circ q$. Since $q$ is surjective, $\im \delta^1=\im \tilde\delta^1$, and therefore 
by Lemma \ref{l:es-ses} we get the short exact sequence
\begin{equation}\label{e:K0A0}
\xymatrix{0\ar[r]&\frac{\Z^{\sF}}{\im \delta^1}\ar@{^{(}->}[r]&K_0(A_n)\ar@{>>}[r]&\ker\delta^0\ar[r]&0}.
\end{equation}
Since $\ker\delta^0$ is free abelian, the short exact sequence splits
\begin{equation}\label{e:K0AnDirectSum}
K_0(A_n)\cong \frac{\Z^{\sF}}{\im \delta^1}\oplus \ker\delta^0.
\end{equation}
Note that $\delta^1(\im \delta^0)=0$, i.e.~
\begin{equation}\label{e:d12iszero}
\delta^1\circ\delta^0=0.
\end{equation}
From Eq.~(\ref{e:K-IAB-2}) we get
\begin{equation}\label{e:K1A0}
K_1(A_n)\cong \ker\tilde\delta^1\cong \frac{\ker\delta^1}{\ker q}=\frac{\ker\delta^1}{\im \delta^0}.
\end{equation}
For $n=0$, the (extended) index map $\delta^1:\Z^{\sE}\to\Z^{\sF}$ is computed in \cite[Prop.~2.4]{Gon2}  to be
\begin{equation}\label{e:delta1}
\delta^1([e]):=\sum_{f\in T(\open{e}) } \delta_{f,e}\,\,\, [f],
\end{equation}
where  $\delta_{f,e}$ was defined in Eq.~(\ref{e:deltafe}), and
where $T(\open{e})$ means all the tiles that contain the edge $e$ without its two vertices, and $f\in T(\open{e})$ means all the faces in the patch $T(\open{e})$.

\begin{lem}\label{l:imDelta1isZ}
For tilings of the plane,
$$\im\delta^1=\{(x_1,\ldots,x_{\sF})\in \Z^{\sF}\mid x_1+\ldots+x_{\sF}=0\}.$$ 
\end{lem}
\begin{proof}
For each edge $e$,
$$\delta^1([e])=[f_1]-[f_2],$$
where $f_1$, $f_2$ are the two faces adjacent to $e$. Thus, whenever two faces $f_1$, $f_2$ share an edge, $[f_1]-[f_2]\in \im\delta^1$, i.e.~$[f_1]\sim_{\im\delta^1} [f_2]$.
By connectedness of the tiling, all faces are mapped to the same element in $\frac{\Z^{\sF}}{\im\delta^1}$, i.e.~all faces are $\im\delta^1$-equivalent,
i.e.~if $f_1$ and $f_2$ are arbitrary faces then $[f_1]-[f_2]\in \im\delta^1$. For instance if $f_1 f_2 f_3$ is a chain of faces (i.e.~$f_1$ shares an edge with $f_2$, and $f_2$ shares an edge with $f_3$) then 
$$\im\delta^1\ni([f_1]-[f_2])+([f_2]-[f_3])=[f_1]-[f_3].$$
Thus
\begin{eqnarray*}
\im\delta^1&=&\Span_{\Z}\{(1,-1,0,\ldots,0),(0,1,-1,0,\ldots,0),(0,\ldots,0,1,-1)\}\\
&=&\{(x_1,\ldots,x_{\sF})\in \Z^{\sF}\mid x_1+\ldots+x_{\sF}=0\}.
\end{eqnarray*}
The second equality is because
\begin{eqnarray*}
&&y_1(1,-1,0,\ldots,0)+y_2(0,1,-1,0,\ldots,0)+y_{_{\sF-1}}(0,\ldots,0,1,-1)\\
&=&(y_1,y_2-y_1,y_3-y_2,y_4-y_3,\ldots,y_{\sF-1}-y_{_{\sF-2}},-y_{_{\sF-1}})\\
&=&(x_1,x_2,x_3,\ldots,x_{_{\sF}}),
\end{eqnarray*}
and thus
$$x_1+x_2+\cdots+x_{_{\sF}}=y_1+y_2-y_1+y_3-y_2+\cdots +y_{_{\sF-1}}-y_{_{\sF-1}}=0.$$
Finally, notice that from the $x_i$'s one gets the $y_i$'s and vice-versa.
\end{proof}

\begin{lem}\label{l:plane-ZsFoverImDelta1}
For tilings of the plane, it holds
$$\frac{\Z^{\sF}}{\im\delta^1}\cong\Z$$
$$K_0(A_n)\cong \Z\oplus \ker\delta^0$$
\end{lem}
\begin{proof}
Let $\phi:\Z^{\sF}\to \Z$ be the linear map $\phi(x_1,\ldots,x_{\sF})=x_1+\ldots+x_{\sF}$. Note that $\phi$ has matrix $(1,\ldots,1)$.
By Lemma \ref{l:imDelta1isZ}, $\ker\phi=\im\delta^1$. Hence since $\phi$ is surjective,
$$\frac{\Z^{\sF}}{\im\delta^1}=\frac{\Z^{\sF}}{\ker\phi}\cong\im\phi=\Z.$$
The second equation in the lemma follows by Eq.~(\ref{e:K0AnDirectSum}).
\end{proof}
\subsection{Homotopy}\label{ss:homotopy}

Recall that $T$ is a fixed tiling of dimension $d$, and that $T_0:=T$, and $T_1:=\frac1{\lambda}\omega(T)\lambda$. Let $t_1,\ldots, t_N\in T$ be the prototiles (then for dimension $d=1$, N=sE is the number of stable edges, and for $d=2$, $N=\sF$ is the number of stable faces).
  
Roughly speaking, one can say that the tiles of $T_0$ can be obtained from the shrunk tiles of $T_1$ simply by joining them, i.e.~for a tile $t\in T$, one joins the shrunk tiles of the patch $\frac{1}{\lambda}(\omega(t))\subset T_1$ to get $t$.
We need a homotopy $h_s$ of $\R^d$ such that the restriction $h_s:t\to t$ homotopes a unique shrunk tile $t'\in \frac{1}{\lambda}(\omega(t))$ to $t$, and the rest of the cells will collapse cellularly, 
i.e.~the vertices of $\frac1{\lambda}\omega(t)$ collapse to vertices of $t$,
the edges of $\frac1{\lambda}\omega(t)$ collapse to edges or vertices of $t$,
and the faces of $\frac1{\lambda}\omega(t)\backslash\{t'\}$ collapse to edges or vertices of $t$.
Since $T$ is a substitution tiling, we define $h:t_i\to t_i$ on the prototiles and extend it to the whole $\R^d$ by translation. 
Of course, one has to ensure that the extension by translation to the whole $\R^d$ is possible, for one can easily construct a homotopy on the prototiles which will not extend by translation, e.g.~if ``pulling'' occurs in opposite directions when gluing together two copies of the patches $\frac1{\lambda}\omega(t_i)$.

Formally, we assume that there is a homotopy satisfying the following assumptions
\begin{defn}[Homotopy $h_s$]\label{d:homotopy}
Suppose that $h:\R^d\times[0,1]\to \R^d$, $d\le 2$, is a homotopy that satisfies the following items with the assumption that the homotopy ignores the orientation of the cells, i.e.~as if the cells had no orientation,
\begin{itemize}
\item[(1)] $h_0=id$
\item[(2)] $h_s$, $0<s<1$ is a homeomorphism where for each prototile $t_i\in T_0$
\begin{itemize} 
\item[($a$)] $h_s:t_i\to t_i$ is a homeomorphism and is cellular. Hence, in particular, $h_s(v)=v$ for every vertex $v\in t_i$, and $h_s(e)=e$ for every edge $e\in t_i$.
\end{itemize}
\item[(3)] $h_1:T_1\to T_0$, where for each prototile $t_i\in T_0$ the following holds
\begin{itemize}
\item[($i$)] $h_1:\frac1{\lambda}\omega(t_i)\to t_i$ is cellular and surjective. Moreover, for every edge $e\in t_i$, and every edge $e'\in \frac1{\lambda}\omega(t_i)$, if $e\in h_1(e')$ then $h_1(e')=e$,  i.e.~$h_1(e')$ cannot contain more than one edge $e$.
\item[($ii$)] there is a unique tile $t'_i\in \frac1{\lambda}\omega(t_i)$ which homotopes to $t_i$ and such that
$h_1:\open{t}\,\!{'_{i}}\to \open{t}_i$ is a homeomorphism.
\item[($iii$)] for each vertex $v\in t_i$, we require that $h_1(v)=v$. \\
For each edge $e\in \partial t_i$, there is a unique shrunk edge $e'\in \frac1{\lambda}\omega(\partial t_i)$ that homotopes to $e$.
 For such shrunk edge $e'$ we require that $e'\subset e$, and that 
$h_1:\open{e}\,\!{'}\to \open{e}$ is a homeomorphism.
\item[($iv$)] for each shrunk tile $t''\in \frac{1}{\lambda}\omega(t_i)$, and each edge $e\in t_i$, if $e\in h_1(t'')$ then there is exactly two edges $e'_1, e'_2\in t''$ with $h_1(e'_1)=e=h_1(e'_2)$ when $t''\ne t'_i$, and there is exactly one edge when $t''=t'_i$.
Moreover, $h_1^{-1}(e)$ is connected.  
\end{itemize}
\item[(4)] Items (2)  and (3) extend to the rest of the tiles of $T_0$ by translation of the scaled substituted prototiles $\frac1{\lambda}\omega(t_i+x_0)=\frac1{\lambda}\omega(t_i)+x_0$.
\end{itemize}
\end{defn}
\noindent
\begin{rem}
We should note that some parts of item $3$ follow from item $2(a)$ by continuity, but for clarity we prefer to spell them out.
Note that for $s\in[0,1]$, the underlying space of $h_s(\frac1{\lambda}\omega(t_i))$ is $t_i$.
Hence for any patch $P\subset T$ and any  $s\in[0,1]$, the underlying space of $h_s(\frac1{\lambda}\omega(P))$ is $P$. 
For $s\in[0,1]$, the continuous map $h_s:T\to T$ is by construction cellular, i.e.~$h_s(X_1^T)\subset X_1^T$ and $h_s(X_0^T)\subset X_0^T$, where $X_1^T$ and $X_0^T$ are the 1- and 0-skeleton of $T$, respectively.
Item (3)(iii) plays an important role in Section \ref{s:S-U-relationship}.
We include item $(3)(iv)$ in the definition of our homotopy to simplify our proofs of Section \ref{s:S-U-relationship} even though it is not strictly needed. 
For tilings of the line, items $(3)(iii)$ and $(3)(iv)$, though still true, can be ignored. 
\end{rem}
\begin{defn}
Define tiling 
$$T_s:=h_{1-s}(T_1),\qquad0\le s\le 1.$$
\end{defn}
Note that the family of tilings $T_s$ continuously deform the tiling $T_{1}=h_{0}(T_1)$ to the tiling $T_{0}=h_{1}(T_1)$. 
We illustrate this in Figure \ref{f:tilingsTs}.

\begin{figure}[h]
\centerline{\begin{tikzpicture}[scale=1]
\draw[fill=yellow] (0,0)--(0.5,0)--(0.5,0.5)--(0,0.5)--(0,0);
\node at (.25,.25) {$t'$};
\draw (-.1,0) -- (2.1,0);
\draw (-.1,1) -- (2.1,1);
\draw (-.1,2) -- (2.1,2);
\draw (0,-.1) -- (0,2.1);
\draw (1,-.1) -- (1,2.1);
\draw (2,-.1) -- (2,2.1);
\draw[red] (-.1,0.5) -- (2.1,0.5);
\draw[red] (-.1,1.5) -- (2.1,1.5);
\draw[red] (0.5,-.1) -- (0.5,2.1);
\draw[red] (1.5,-.1) -- (1.5,2.1);
\node at (1,-.4) {$T_1$};

\draw (2.9,0) -- (5.1,0);
\draw (2.9,1) -- (5.1,1);
\draw (2.9,2) -- (5.1,2);
\draw (3,-.1) -- (3,2.1);
\draw (4,-.1) -- (4,2.1);
\draw (5,-.1) -- (5,2.1);
\draw[red] (2.9,0.75) -- (5.1,0.75);
\draw[red] (2.9,1.75) -- (5.1,1.75);
\draw[red] (3.75,-.1) -- (3.75,2.1);
\draw[red] (4.75,-.1) -- (4.75,2.1);
\node at (4,-.4) {$T_{0.5}$};

\draw (5.9,0) -- (8.1,0);
\draw (5.9,1) -- (8.1,1);
\draw (5.9,2) -- (8.1,2);
\draw (6,-.1) -- (6,2.1);
\draw (7,-.1) -- (7,2.1);
\draw (8,-.1) -- (8,2.1);
\draw[red] (5.9,0.875) -- (8.1,0.875);
\draw[red] (5.9,1.875) -- (8.1,1.875);
\draw[red] (6.875,-.1) -- (6.875,2.1);
\draw[red] (7.875,-.1) -- (7.875,2.1);
\node at (7,-.4) {$T_{0.25}$};

\draw[fill=yellow] (9,0)--(9,1)--(10,1)--(10,0)--(9,0);
\node at (9.5,.5) {$t$};
\draw (8.9,0) -- (11.1,0);
\draw (8.9,1) -- (11.1,1);
\draw (8.9,2) -- (11.1,2);
\draw (9,-.1) -- (9,2.1);
\draw (10,-.1) -- (10,2.1);
\draw (11,-.1) -- (11,2.1);
\draw[red] (8.9,0.99) -- (11.1,0.99);
\draw[red] (8.9,1.99) -- (11.1,1.99);
\draw[red] (9.99,-.1) -- (9.99,2.1);
\draw[red] (10.99,-.1) -- (10.99,2.1);
\node at (10,-.4) {$T_{0}$};
\end{tikzpicture}}
 \caption{Tilings $T_s=:h_{1-s}(T_1)$, $s\in[0,1].$\label{f:tilingsTs}}
\end{figure}
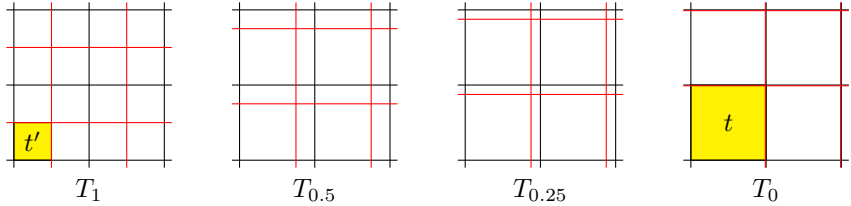

Since the $C^*$-algebra $A_n$ is amenable, it is equal to the full $C^*$-algebra $C^*(R_n)$.
Given a locally compact space $X$, we denote by $C_c(X)$ the space of complex-valued, continuous functions on $X$ with compact support.
In what follows we consider groupoids and their $C^*$-algebras. The results concern the full $C^*$-algebra $C^*(G)$, not the reduced $C^*$-algebra $C^*_r(G)$.
The goal is to construct a homotopy of $*$-homomorphisms $\phi_s$. 
\begin{rem}
The fact that the $\phi_s$ defined below is a $*$-homomorphism relies on the properties of our special homotopy.
The construction of $\phi_s$ is inspired by the fact that any groupoid homomorphism $\pi: G\to H$, which is continuous and proper, induces the map $\pi^*: C_c(H)\to C_c(G)$ given by $\pi^*(f)=f\circ \pi$.
However, this map is usually not a $*$-homomorphism. For instance, consider the case when $G:=\Delta_{R_0}$ is the diagonal of $H:=R_0$ for the Fibonacci tiling with proto-edges $a$, $b$.
Then $G$ is a connected component and a subgroupoid of the \'etale equivalence relation $H$. 
Note that the product in $C_c(G)$ is just pointwise multiplication, and $G$ is trivially \'etale.
Let $i: G\to H$ be the inclusion map, which is obviously continuous and proper.
Then the restriction map $i^*$ is not a $*$-homomorphism since the product is not preserved. For example, choose $a_1=[0,1]$, $a_2=[1,2]$ with label $a$, and let $f,g\in C_c(H)$
be such that $f$ is zero outside $(a_1\times a_2)^\circ\cap H$, and $f(0.5,1.5)\ne 0$, and $g$ is zero outside $(a_2\times a_1)^\circ\cap H$ and $g(1.5,0.5)\ne 0$.
Then $i^*(f)=f\circ i=0$ and thus $i^*(f)*i^*(g)=0$. On the other hand, $i^*(f*g)(0.5,0.5)=f(0.5,1.5)g(1.5,0.5)\ne 0$, and hence $i^*(f*g)\ne i^*(f)*i^*(g)$.
\end{rem}

\begin{defn}[Homotopy $\phi_s$]\label{d:phi_s}
Let $\phi_s:C_c(R_0)\to C_c(R_0)$, $0\le s\le 1$, be defined as
$$\phi_s(g)(x,y):= \left\{ \begin{array}{ccl}
    g\circ (h_s\times h_s)(x,y) &\quad& h_s(x)\sim_{R_0} h_s(y) \\ 
    0 &\quad&  \text{else},\\ 
  \end{array}\right.$$
where $g\in C_c(R_0)$ and  $x\sim_{R_0}y$.
\end{defn}

In \cite[p.92-93]{Gon0}, it was shown that $\phi_s$ is well-defined as stated in the next lemma.
For the convenience of the reader, we provide an alternative proof of this lemma in terms of connected components.
Moreover, many of the ideas in the proof of the next lemma will also be used in later proofs of this section.
\begin{lem}\label{l:phi_s-welldefined}
The map $\phi_s:C_c(R_0)\to C_c(R_0)$, $s\in [0,1]$ is well-defined.
\end{lem}
\begin{proof}
Since $g\in C_c(R_0)$ is evaluated only when $h_s(x)\sim_{R_0} h_s(y)$, the map $\phi_s(g):R_0\to\C$ is well-defined. 
For $s\in[0,1]$ we have  
\begin{equation}\label{e:supp-phi-s-homotopy}
\Supp \phi_s(g)=(h_s\times h_s)^{-1}(\Supp g)
\end{equation} because
\begin{eqnarray*}
(x',y')\in \Supp \phi_s(g) &\iff& g(h_s(x'),h_s(y'))\ne 0 \text { and } h_s(x')\sim_{R_0} h_s(y')\\
&\iff& (h_s(x'),h_s(y'))\in \Supp g\\
&\iff&(x',y')\in (h_s\times h_s)^{-1}(\Supp g).
\end{eqnarray*}
By assumption 
$$\tilde K:=\overline{\supp(g)}^{R_0}\subset R_0$$
 is compact in $R_0$.
For showing that $\phi_s(g)$ is continuous with compact support,
it is enough to consider the restrictions $g|_C$, $\phi_s(g)|_C$ to a connected component $C$ of $R_0$.
This is because the support of $g$ is in $\tilde K$, and by compactness, $\tilde K$ intersects only a finite number of connected components.
Moreover, $\tilde K\cap C$ is compact in $C$ because $C$ is clopen. Also, a compact set in $C$ is compact in $R_0$ by continuity of the inclusion map $C\hookrightarrow R_0$.
Furthermore, a finite union of compact sets is compact.\\
Since the range map $r:R_0\to\R^d$ restricted to the component $C$ is a homeomorphism onto its image, it is enough to consider the maps 
$$f:r(C)\to \C,\qquad \phi'_s(f):r(C)\to\C$$
given by
 $$f:=g|_{_C}\circ r^{-1},\qquad \phi'_s(f):=\phi_s(g)|_{_C}\circ r^{-1}.$$
Note that 
\begin{itemize}
\item $r(C)$ can be unbounded (e.g.~$r(C)=\R^d$ when $C$ is the diagonal of $R_0$). Abusing language, $r(C)$ is a (not necessarily finite) connected patch  minus its boundary; the patch is made of tiles of $T$.
\item For $s<1$, $h_s$ maps $r(C)$ to $r(C)$ homeomorphically. 
\item $h_1$ maps $r(C)$ to $\overline{r(C)}^{\R^d}$, thus $\overline{r(C)}^{\R^d}\backslash h_1(r(C))$ can be nonempty. 
Actually for any patch $P$ of $T$, $h_1:\open{P}\to P$, and $h_1(\open{P})\backslash \open{P}$ can be nonempty.
\item $h_1:r(C)\,\backslash\, h_1^{-1}(\partial_{_{\R^d}}\,r(C))\to r(C)$ is surjective.
\end{itemize} 
Let $K$ be the closure of the support of $f$ in $r(C)$. Then, by Eq.~(\ref{e:supp-phi-s-homotopy}),  $K'_s:=h_s^{-1}(K)$ contains the closure of the support of $\phi'_s(f)$  in $r(C)$ i.e.
\begin{equation}\label{e:K's}
 K=\overline{\Supp f\,}^{r(C)} \qquad\text{and}\qquad \overline{h_s^{-1}(\Supp f)}^{r(C)}\subset h_s^{-1}( K)=K'_s.
\end{equation}

If $s<1$, then $\overline{h_s^{-1}(\Supp f)}^{r(C)}=h_s^{-1}( K)=K'_s$ is compact in $r(C)$ because $h_s$ is a homeomorphism of $r(C)$ when restricted to $r(C)$.
Moreover, since 
$$\phi'_s(f)=f\circ h_s,\qquad s<1$$
is a composition of continuous functions, $\phi'_s(f)$ is continuous. This completes the proof for $s<1$.\\
Now suppose that $s=1$. Then 
$$\phi'_1(f)(x)=\left\{  \begin{array}{ll}
    f\circ h_1(x)\quad & h_1(x)\in r(C)\\
    0 & \mathrm{else}\\ 
  \end{array}\right.
$$

We start by showing that $\phi'_1(f)$ is continuous.
Suppose that $(x_n)_{n=1}^\infty \subset r(C)$ is a sequence that converges to $x\in r(C)$.
We need to show that $\phi'_1(f)(x_n)$ converges to $\phi'_1(f)(x)\in\C$.\\
Assume that $h_1(x)\in r(C)$.
Then by continuity of $h_1$ we have $h_1(x_n)\to h_1(x)$.
Since $r(C)$ is open, and $h_1(x)\in r(C)$, $h_1(x_n)\in r(C)$ for $n\ge n_0$ for some $n_0\in\N$.
Thus by continuity of $f$ on $r(C)$, 
$$\phi'_1(f)(x_n)=f(h_1(x_n))\,\,\to\,\, f(h_1(x))=\phi'_1(f)(x), \qquad (n\ge n_0)$$
which proves that $\phi'_1(f)$ is continuous at $h_1(x)$ whenever $h_1(x)\in r(C)$.\\
Assume that $h_1(x)\not\in r(C)$. (i.e.~$h_1(x)$ is in the boundary of $\overline{r(C)}^{\R^d}$.)
Then $h_1(x)\not\in K$ because $K\subset r(C)$. Thus $\phi'_1(f)(x)=0$. 
We claim that $h_1(x_n)\not\in K$, for all $n\ge n_0$ for some $n_0\in\N$.  Hence $\phi'_1(f)(x_n)=0\to 0=\phi'_1(f)(x)$, and thus $\phi_1'(f)$ is continuous.
We proceed to the claim. For contradiction, assume that the claim is false. Then by compactness of $K$, there exists a convergent subsequence $h_1(x_{n_k})\in K$. 
Since $h_1(x_n)\to h_1(x)$, we have that $h_1(x_{n_k})\to h_1(x)$. Since $K$ is closed, $h_1(x)\in K$, a contradiction, which proves the claim.
This completes the proof of continuity of $\phi'_1(f)$.\\\\
We will now show that $K'_1=h_1^{-1}(K)$ is $r(C)$-compact.

Since $K$ is compact in $r(C)$, it is compact in $\R^d$ by continuity of the inclusion map $r(C)\hookrightarrow \R^d$. 
Hence $K$ is closed and bounded in $\R^d$. Let $P:=T( K)$ be the smallest patch of $T$ that contains $K$, i.e.~$P$ is the necessarily finite set of tiles in $T$ that intersect $K$.
By definition of our homotopy, $h_1^{-1}(K)\subseteq P$, as the homotopy moves points from a tile of $T$ to the tile itself. See Figure \ref{f:TKh1}.
More precisely, if $x\not \in P$ then $h_1(x)\not\in \open{P}$ (although it could be in the boundary of $P$), and since $K\subset \open{P}$ then $h_1(x)\not\in K$, i.e.~$x\not\in h_1^{-1}(K)$.
 Hence $T(h_1^{-1}( K))\subseteq T(P)$. 
Hence $K'_1\subset P$ is bounded. 
Since $h_1$ is $\R^d$-continuous, $h_1^{-1}(K)$ is $\R^d$-closed and hence $K'_1=h_1^{-1}(K)$ is $\R^d$-compact.

Since $h_1$ is cellular, (e.g.~$h_1$ maps the $\R^d$-boundary of $r(C)$ to itself), and $K\subset\open{T(K)}\subset r(C)$, we have that 
$$K'_1=h_1^{-1}(K)\subset r(C).$$
This is because 
\begin{eqnarray*}
&&h_1(\partial r(C))\subset \partial r(C)\\
& {\Leftrightarrow}&\forall x\in \overline{r(C)}: x\in \partial r(C) \imply h_1(x)\in \partial(r(C))\\
& {\Leftrightarrow}&\forall x\in \overline{r(C)}: h_1(x)\not\in \partial(r(C))\imply x\not\in\partial(r(C))\\
& {\Leftrightarrow}&\forall x\in \overline{r(C)}: h_1(x)\in r(C)\imply x\in r(C)\\
& {\Leftrightarrow}&\forall x\in \overline{r(C)}: x\in h_1^{-1}(r(C))\imply x\in r(C)\\
& {\Leftrightarrow}& h_1^{-1}(r(C))\subset r(C),
\end{eqnarray*}
where the boundary and closure are done in $\R^d$.
Thus we have shown that $K'_1$ is $\R^d$-compact and that $K'_1$ is contained in $r(C)$.
By Lemma \ref{l:X-compact}, $K'_1$ is $r(C)$-compact.
\begin{figure}[t]
\centerline{\begin{tikzpicture}[scale=2]

%


\draw[fill,gray!20] (0,1.5) -- (1.5,1.5) -- (1.5,0) -- (2,0) -- (2,2) -- (0,2) -- (0,1);
\draw[fill,yellow] (.5,.5) circle (.25cm);
\draw[fill,yellow] (.5,1) circle (.25cm);
\draw[fill,yellow] (1,1) circle (.25cm);
\draw[fill,yellow] (1,0.5) circle (.25cm);
\draw[] (.5,.5) circle (.25cm);
\draw[] (.5,1) circle (.25cm);
\draw[] (1,1) circle (.25cm);
\draw[] (1,0.5) circle (.25cm);
\draw[fill,yellow!70] (0.5,0.25) -- (1,0.25) -- (1,1.25) -- (.5,1.25) -- (.5,0.25);
\draw[fill,yellow!70] (0.25,0.5) -- (1.25,0.5) -- (1.25,1) -- (0.25,1) -- (0.25,0.5);
\draw (.5,0.25) -- (1,0.25);
\draw (1.25,0.5) -- (1.25,1);
\draw (.5,1.25)--(1,1.25);
\draw (.25,.5)--(.25,1);
\draw[dashed] (0,0) -- (2,0);
\draw[dashed] (0,2) -- (2,2);
\draw (0,1) -- (2,1);
\draw[red] (0,0.5) -- (2,0.5);
\draw[red] (0,1.5) -- (2,1.5);
\draw[dashed] (0,0) -- (0,2);
\draw[dashed] (2,0) -- (2,2);
\draw ((1,0) -- (1,2);
\draw[red] (0.5,0) -- (0.5,2);
\draw[red] (1.5,0) -- (1.5,2);

\draw[thick,dashed] (0,0) -- (1.5,0) -- (1.5,1.5) -- (0,1.5);  
\node at (0.7,.75){$h_1^{-1}(K)$};
\node at (1,1.75) {maps to boundary};
\node at (4,-.2){$T(K)$};
\node[red] at (.125,.125){$t'_i$};
\node at (3.5,.5){$t_i$};


\draw[fill,yellow] (4,1) circle(0.5cm);
\draw[] (4,1) circle(0.5cm);
\node at (4.2,1.2){$K$};

\draw[dashed] (3,0) -- (5,0);
\draw[dashed] (3,2) -- (5,2);
\draw (3,1) -- (5,1);
\draw[dashed] (3,0) -- (3,2);
\draw[dashed] (5,0) -- (5,2);
\draw ((4,0) -- (4,2);

\end{tikzpicture}}

 \caption{ $h_1^{-1}(K)\subset T(K)$.\label{f:TKh1}}
\end{figure}
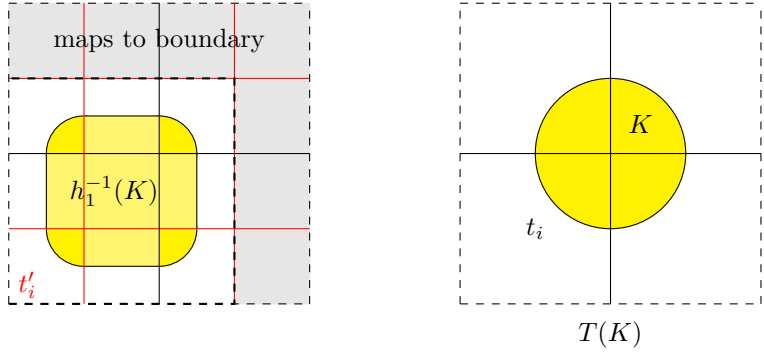
\end{proof}

\begin{rem} Using the notation from the above proof (of Lemma \ref{l:phi_s-welldefined}),
the map $\phi'_1(f)$ shrinks the graph of $f$ on prototile $t_i\in T$ to the shrunk tile $t'_i\in T_1$ and expands continuously (cylindrically) the graph of $f$ on the boundary $\partial t_i$ to the set $t_i\backslash t'$.
We illustrate this in Figure \ref{f:phi1}.
\begin{figure}[h]
\centerline{
\includegraphics[scale=.25]{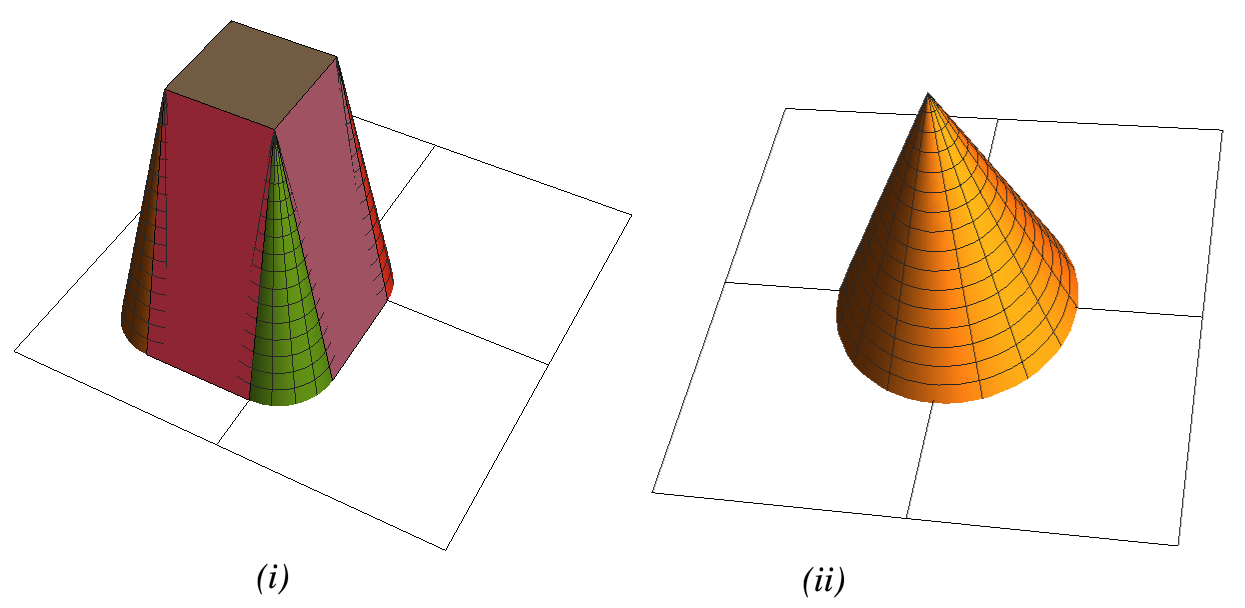}
}
 \caption{$(i)$ $\phi_1'(f)$,  $(ii)$ $f$.\label{f:phi1}}
\end{figure}
\end{rem}

We next introduce the auxiliary map $\psi$ whose extension (see Lemma \ref{l:tilde-psi}) is used in the proof of Proposition \ref{p:phi-star-homo} given later in this section.
The proof of the next lemma is similar to the one of Lemma \ref{l:phi_s-welldefined}.
\begin{lem}\label{l:psi-welldefined}
The map $\psi:C_c(R_0)\to C_c([0,1]\times R_0)$ given by
\begin{equation}\label{e:psi}
\psi(f)(s,(x,y)):= \left\{ \begin{array}{cc}
    f(h_s(x),h_s(y)) & h_s(x)\sim_{R_0} h_s(y) \\ 
    0 & \text{else} \\ 
  \end{array}\right.
\end{equation}
 is well-defined.
\end{lem}
\begin{proof}
Recall that the product of two groupoids is a groupoid, and $[0,1]$ is the trivial groupoid where each $x\in [0,1]$ is a unit.
As explained in the proof of Lemma \ref{l:phi_s-welldefined}, it is enough to consider the restriction 
to a connected component $C$ of $R_0$ via the range map $r:R_0\to \R^d$, because a $R_0$-compact set can only intersect a finite number of connected components, and because the range map restricted to a connected component is a homeomorphism onto its image. Thus, it suffices to show that 
$\psi':C_c(r(C))\to C_c([0,1]\times r(C))$ given by
$$
\psi'(f)(s,x):= \left\{ \begin{array}{cc}
    f(h_s(x)) & h_s(x)\in r(C) \\ 
    0 & \text{else} \\ 
  \end{array}\right.
$$
is well-defined.
The case $s<1$ is easy as $h_s$ is a homeomorphism of $r(C)$, so we just show the case $s=1$.
We start by showing continuity of $\psi'(f)$ for $f\in C_c(r(C))$ with compact support $K=\overline{\Supp f}^{r(C)}$.
Suppose that the sequence $(s_n,x_n)$ converges to $(1,x)$ and that $h(1,x)\in r(C)$.
Then by continuity of $h$ we have  $h(s_n,x_n)\to h(1,x)$ for $n\to \infty$.
Since $r(C)$ is open and $h(1,x)\in r(C)$, $h(s_n,x_n)\in r(C)$ for $n\ge n_0$, for some $n_0\in \N$.
Then by continuity of $f$ on $r(C)$, we get
$$\psi'(f)(s_n,x_n)=f(h(s_n,x_n))\to f(h(1,x))=\psi'(f)(1,x),\quad (n\ge n_0).$$
Thus $\psi'(f)$ is continuous at $(1,x)$ whenever $h(1,x)\in r(C)$.
Now assume that the sequence $(s_n,x_n)$ converges to $(1,x)$, and that $h_1(x)\not \in r(C)$. (i.e.~$h_1(x)\in \partial_{\R^d}(r(C))$).
Thus $h(1,x)\not\in K$ because $K\subset r(C)$.
We claim that $h(s_n,x_n)\not\in K$ for $n\ge n_0$ for some $n_0\in\N$. Then
$$\psi'(f)(s_n,x_n)=0\to 0=\psi'(1,x),\qquad (n\ge n_0),$$
and hence $\psi'(f)$ is continuous at $(1,x)$ whenever $h_1(x)\not\in r(C)$.
We now prove the claim.
Suppose that it is false, i.e. suppose that for all $n_0$, there exists an $n\ge n_0$ such that $h(s_n,x_n)\in K$.
Since $K$ is compact, there exists a convergent subsequence $h(s_{n_j},x_{n_j})\to h(1,x)\in K$, which contradicts $ h(1,x)\not\in K$.

\end{proof}

\begin{lem}\label{l:psi-cont}
Equip $C_c(R_0)$ and $C_c([0,1]\times R_0)$ with the inductive limit topology (cf.~\cite[Example~5.10,~p~.118]{Conway}). Then $\psi:C_c(R_0)\to C_c([0,1]\times R_0)$ is continuous.
\end{lem}
\begin{proof}
As explained in the proof of Lemma \ref{l:phi_s-welldefined}, and used in the previous lemma, it is enough to consider the restriction 
to a connected component $C$ of $R_0$ via the range map $r:R_0\to \R^d$.
Thus, it suffices to show that 
$\psi':C_c(r(C))\to C_c([0,1]\times r(C))$ given by
$$
\psi'(f)(s,x):= \left\{ \begin{array}{cc}
    f(h_s(x)) & h_s(x)\in r(C) \\ 
    0 & \text{else} \\ 
  \end{array}\right.
$$
is continuous in the inductive limit topology.
Suppose that $f_n\in C_c(r(C))$ converges to $f\in C_c(r(C))$ in the inductive limit topology.
That is, there is a compact subset $K\subset r(C)$ such that for all $n\in \N$
$$\Supp f_n,\,\Supp f\subset K,$$
and such that
$$||f_n-f||_{\sup}\to 0\quad\text{as} \quad n\to \infty.$$
Note that the supremum is taken over $K$. 
We claim that 
$$K':=h^{-1}(K)\subset r(C)$$
 is compact and that 
 $$\Supp\, \psi'(f_n),\, \Supp\, \psi'(f)\subset K'.$$
 Then
\begin{eqnarray*}
||\psi'(f_n)-\psi'(f)||_{\sup}&=&\sup_{(s,x')\in K'} |\psi'(f_n)(s,x')-\psi'(f)(s,x')|\\
&=&\sup_{\substack{(s,x')\in K'\\ h_s(x')\in r(C)}} |f_n(h_s(x'))-f(h_s(x'))|\\
&\le&\sup_{x\in K} |f_n(x)-f(x)|\\
&=&||f_n-f||_{\sup}\to 0 \text{ as } n\to \infty,
\end{eqnarray*}
where the inequality holds because we can only get nonzero values if $h_s(x')\in K$.
To check that the claim holds, notice that, by the proof of Lemma \ref{l:phi_s-welldefined}, 
$$K'_s:=h_s^{-1}(K)\subset r(C),\qquad s\in[0,1].$$
Since 
$$(s,x)\in h^{-1}(K)\iff h_s(x)\in K,$$
we have
$$K'=h^{-1}(K)=\bigcup_{s\in[0,1]} \{s\}\times h_s^{-1}(K)=\bigcup_{s\in[0,1]} \{s\}\times K'_s\subset [0,1]\times r(C).$$
Moreover, $\Supp\, \psi'(f)\subset K'$ because if 
$\psi'(f)(s,x)\ne 0$ then $f(h_s(x))\ne 0$ and thus $h_s(x)\in \Supp f$, i.e.~$x\in h_s^{-1}(\Supp f)$. Hence $(s,x)\in \{s\}\times K'_s\subset K'$.
Similarly, it holds that $\Supp\, \psi'(f_n)\subset K'$.
By definition of our homotopy $K'$ is $\R^{d+1}$-bounded, and by continuity of $h$, it is $\R^{d+1}$-closed. Hence $K'$ is $\R^{d+1}$-compact, and since it is a subset of $ [0,1]\times r(C)$, $K'$ is $([0,1]\times r(C))$-compact by Lemma \ref{l:X-compact}.
\end{proof}
We show next that $\phi_s$ is continuous, the proof of which is a ``pointwise" version of the previous proof (continuity of $\psi$).
\begin{lem}\label{l:phi_s-Cc-cont}
Equip $C_c(R_0)$ with the inductive limit topology (cf.~\cite[Ex.~5.10,p.118]{Conway}). Then $\phi_s:C_c(R_0)\to C_c(R_0)$ is continuous for any $s\in[0,1]$.
\end{lem}
\begin{proof}
As before, it is enough to consider the restriction 
to a connected component $C$ of $R_0$ via the range map $r:R_0\to \R^d$.

Suppose that $f_n\in C_c(r(C))$ converges to $f\in C_c(r(C))$ in the inductive limit topology.
That is, there is a compact subset $K\subset r(C)$ such that for all $n\in \N$
$$\Supp f_n,\,\Supp f\subset K,$$
and such that
$$||f_n-f||_{\sup}\to 0\quad\text{as} \quad n\to \infty.$$
Note that the supremum is taken over $K$. 
We claim that 
$$K':=h_s^{-1}(K)\subset r(C)$$
 is compact and that 
 $$\Supp\, \phi'_s(f_n),\, \Supp\, \phi'_s(f)\subset K',$$
 where $\phi_s'$ was given in the proof of Lemma \ref{l:phi_s-welldefined}.
 Then
\begin{eqnarray*}
||\phi'_s(f_n)-\phi'_s(f)||_{\sup}&=&\sup_{x'\in K'} |\phi'_s(f_n)(x')-\phi'_s(f)(x')|\\
&=&\sup_{\substack{x'\in K'\\ h_s(x')\in r(C)}} |f_n(h_s(x'))-f(h_s(x'))|\\
&\le&\sup_{x\in K} |f_n(x)-f(x)|\\
&=&||f_n-f||_{\sup}\to 0 \text{ as } n\to \infty,
\end{eqnarray*}
where the inequality holds because we can only get nonzero values if $h_s(x')\in K$.\\
To check the claim notice that,
since by Eq.~(\ref{e:supp-phi-s-homotopy}) $\Supp \phi'_s(f_n)=h_s^{-1}(\Supp f_n)$ and $\Supp \phi'_s(f)=h_s^{-1}(\Supp f)$, it is clear that
$$\Supp \phi'_s(f_n),\,\Supp \phi'_s(f)\,\subset\, h_s^{-1}(K)\,=\,K'.$$
Now,
$K'=h_s^{-1}(K)\subset r(C)$ is compact,
by the same argument as in the proof of Lemma \ref{l:phi_s-welldefined}.
\end{proof}

The next lemma shows that $\phi_s$ preserves the involution and the product. Note that the product is convolution, not pointwise multiplication. 
As a corollary it will follow that $\psi$ is also a $*$-homomorphism.
\begin{lem}\label{l:phi_s-Cc-star-homo}
Let $s\in[0,1]$. Then 
the map $\phi_s:C_c(R_0)\to C_c(R_0)$ given in Definition \ref{d:phi_s}
is a $*$-homomorphism.
\end{lem}

\begin{proof}
The $*$-operation holds because
$$\phi_s(g)^*(x,y)=\overline{\phi_s(g)}(y,x)=\overline{g}(h_s(y),h_s(x))=\phi_s(g^*)(x,y)$$
 when $h_s(x)\sim_{R_0} h_s(y)$.
Now
\begin{equation}\label{e:lhs}
\phi_s(g\star g')(x,y)=\sum_{h_s(x)\sim_{R_0} z'} g(h_s(x),z')g'(z',h_s(y)),  
\end{equation}
when $h_s(x)\sim_{R_0} h_s(y)$, else the sum is zero. On the other hand,
\begin{equation*}
\phi_s(g)\star \phi_s(g')(x,y)=\sum_{x\sim_{R_0} z} g(h_s(x),h_s(z))g'(h_s(z),h_s(y)),
\end{equation*}
when $h_s(x)\sim_{R_0} h_s(z)$ and  $h_s(z)\sim_{R_0}h_s(y),$ else the sum is zero. i.e.
\begin{equation}\label{e:rhs}
\phi_s(g)\star \phi_s(g')(x,y)=\sum_{\substack{x\sim_{R_0} z\\ h_s(x)\sim_{R_0} h_s(z)}} g(h_s(x),h_s(z))g'(h_s(z),h_s(y)),
\end{equation}
when $h_s(x)\sim_{R_0} h_s(y)$ else the sum is zero.

$ $
\\
Assume $s<1$. Since $h_s$ is a homeomorphism, there is a unique $z\in \R^d$ such that $z'=h_s(z)$. Moreover, by item $2$($a$) in the definition of our homotopy, $h_s(x)\sim_{R_0} h_s(z)$ if and only if $x\sim_{R_0} z$. Hence $\phi_s(g\star g')=\phi_s(g)\star \phi_s(g')$ for $s<1$. 
(See picture \ref{f:hsx-iff-x}).\\
 
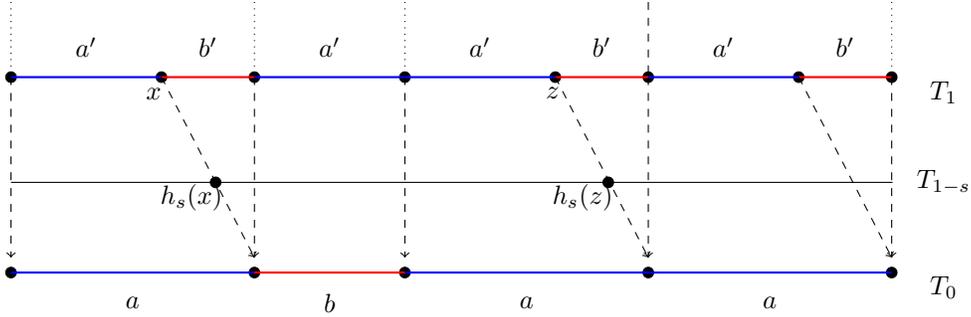
\begin{figure}[h]
\centerline{\begin{tikzpicture}[scale=2]

\draw[black,thin,dotted] (0.,0.6)--(0,0.1);
\draw[black,thin,dotted] (1.61803,0.6)--(1.61803,0.1);
\draw[black,thin,dotted] (2.61803,0.6)--(2.61803,0.1);
\draw[black,thin,dashed] (4.23607,0.6)--(4.23607,0.1);
\draw[black,thin,dotted] (5.8541,0.6)--(5.8541,0.1);

\node at (6.2,0){$T_1$};

\draw plot[mark=*,mark size=1] coordinates{(0.,0.1)};
\draw plot[mark=*,mark size=1] coordinates{(1.,0.1)};
\draw plot[mark=*,mark size=1] coordinates{(1.61803,0.1)};
\draw plot[mark=*,mark size=1] coordinates{(2.61803,0.1)};
\draw plot[mark=*,mark size=1] coordinates{(3.61803,0.1)};
\draw plot[mark=*,mark size=1] coordinates{(4.23607,0.1)};
\draw plot[mark=*,mark size=1] coordinates{(5.23607,0.1)};
\draw plot[mark=*,mark size=1] coordinates{(5.8541,0.1)};
\draw[blue,thick] (0.,0.1)--(1.,0.1);
\draw[red,thick] (1.,0.1)--(1.61803,0.1);
\draw[blue,thick] (1.61803,0.1)--(2.61803,0.1);
\draw[blue,thick] (2.61803,0.1)--(3.61803,0.1);
\draw[red,thick] (3.61803,0.1)--(4.23607,0.1);
\draw[blue,thick] (4.23607,0.1)--(5.23607,0.1);
\draw[red,thick] (5.23607,0.1)--(5.8541,0.1);

\node at (0.5,0.3){$a'$};
\node at (1.30902,0.3){$b'$};
\node at (2.11803,0.3){$a'$};
\node at (3.11803,0.3){$a'$};
\node at (3.92705,0.3){$b'$};
\node at (4.73607,0.3){$a'$};
\node at (5.54508,0.3){$b'$};

\node at (6.2,-.6){$T_{1-s}$};
\draw[->,black,thin,dashed] (0.,0.1)--(0,-1.1);
\draw[->,black,thin,dashed] (1.,0.1)--(1.61803,-1.1);
\draw[->,black,thin,dashed] (1.61803,0.1)--(1.61803,-1.1);
\draw[->,black,thin,dashed] (2.61803,0.1)--(2.61803,-1.1);
\draw[->,black,thin,dashed] (3.61803,0.1)--(4.23607,-1.1);
\draw[->,black,thin,dashed] (4.23607,0.1)--(4.23607,-1.1);
\draw[->,black,thin,dashed] (5.23607,0.1)--(5.8541,-1.1);
\draw[->,black,thin,dashed] (5.8541,0.1)--(5.8541,-1.1);


\node at (6.2,-1.3){$T_{0}$};
\draw plot[mark=*,mark size=1] coordinates{(0.,-1.2)};
\draw plot[mark=*,mark size=1] coordinates{(1.61803,-1.2)};
\draw plot[mark=*,mark size=1] coordinates{(2.61803,-1.2)};
\draw plot[mark=*,mark size=1] coordinates{(4.23607,-1.2)};
\draw plot[mark=*,mark size=1] coordinates{(5.8541,-1.2)};
\draw[blue,thick] (0.,-1.2)--(1.61803,-1.2);
\draw[red,thick] (1.61803,-1.2)--(2.61803,-1.2);
\draw[blue,thick] (2.61803,-1.2)--(4.23607,-1.2);
\draw[blue,thick] (4.23607,-1.2)--(5.8541,-1.2);

\node at (0.809017,-1.4){$a$};
\node at (2.11803,-1.4){$b$};
\node at (3.42705,-1.4){$a$};
\node at (5.04508,-1.4){$a$};

\draw (0,-0.6)--(5.86,-0.6);
\draw plot[mark=*,mark size=1] coordinates{(1.36,-.6)};
\draw plot[mark=*,mark size=1] coordinates{(3.97,-.6)};
\node at (1.2,-.7){$h_s(x)$};
\node at (3.8,-.7){$h_s(z)$};
\node at (0.95,-.01){$x$};
\node at (3.6,-.005){$z$};

\end{tikzpicture}}
 \caption{$h_s(x)\sim_{R_0} h_s(z)$ if and only if $x\sim_{R_0} z$ for $s<1$.\label{f:hsx-iff-x}}
\end{figure}
Now assume that $s=1$. If $h_1(x)\not\sim_{R_0} h_1(y)$ then both sums are zero, so assume  $h_1(x)\sim_{R_0} h_1(y)$.  Define 
$$S_1:=\{z'\mid h_1(x)\sim_{R_0} z'\}$$
$$S_2:=\{z\mid x\sim_{R_0} z,\,\, h_1(x)\sim_{R_0} h_1(z)\}.$$
Let $\Phi:S_2\to S_1$ be given by
$$\Phi(z):=h_1(z).$$
Then clearly, $S_1=[h_1(x)]_{R_0}$, $S_2=[x]_{R_0}\cap h_1^{-1}\big([h_1(x)]_{R_0}\big)$ and
\begin{equation}\label{e:lhs-2}
\phi_1(g\star g')(x,y)=\sum_{z'\in S_1} g(h_1(x),z')g'(z',h_1(y)),  
\end{equation}
\begin{equation}\label{e:rhs-2}
\phi_1(g)\star \phi_1(g')(x,y)=\sum_{z\in S_2} g(h_1(x),h_1(z))g'(h_1(z),h_1(y)).
\end{equation}
We claim that $\Phi$ is bijective. Then $\phi_1(g\star g')=\phi_1(g)\star \phi_1(g')$ and  hence $\phi_1$ is a $*$-homomorphism.
We start by showing surjectivity.
Suppose that $z'\in\R^d$ satisfies that $h_1(x)\sim_{R_0} z'$. We thus have $T(h_1(x))-h_1(x)=T(z')-z'$. By definition of our homotopy we also have $x\in T(h_1(x))$, and note that $h_1(x)\in T(x)\subseteq T(h_1(x))$, where the inclusion holds because the map $h_1$ is cellular.  
Thus $z\in T(z')$ given by 
$$h_1(x)-x=z'-z$$
satisfies $h_1(z)=z'$, since the definition of the homotopy on $T(z')$ is the same as that on $T(h_1(x))$ (up to translation).
We illustrate this in the following figure.\\\\
\centerline{
\centerline{\begin{tikzpicture}[scale=3.5]
\draw[] (0,0) -- (1,0) -- (1,1) -- (0,1) -- (0,0);
\draw[] (0.5,0)--(.5,1);
\draw[] (0,0.5)--(1,0.5);
\node [black] at (.5,.3) {\textbullet};
\node at (.55,.3) {$x$};
\node [black] at (.5,.5) {\textbullet};
\node at (.65,.55) {$h_1(x)$};
\draw [line width=1pt,->] (.5,.3) -- (.5,.48);
\node at (.5,-.1) {$T(h_1(x))$};

\draw (2,0) -- (3,0) -- (3,1) -- (2,1) -- (2,0);
\draw (2,.5) -- (3,.5);
\draw (2.5,0) -- (2.5,1);
\node [black] at (2.5,.3) {\textbullet};
\node at (2.55,.3) {$z$};
\node [black] at (2.5,.5) {\textbullet};
\node at (2.75,.55) {$z'=h_1(z)$};
\draw [line width=0.3mm,->] (2.5,.3) -- (2.5,0.48);
\node at (2.5,-.1) {$T(z')$};

\end{tikzpicture}}
}
By construction of $z$, we have $T(x)-x=T(z)-z$, i.e.~$x\sim_{R_0} z$, and hence $\Phi$ is surjective.

We now show that $\Phi$ is injective.
Suppose that $\Phi(z)=h_1(z)=z'=h_1(\tilde z)=\Phi(\tilde z)$ and that $x\sim z$,\,\, $x\sim \tilde z$,\,\, $h_1(x)\sim h_1(z)$,\,\,
$h_1(x)\sim h_1(\tilde z)$. Thus we have,
$$ T(z)-z=T(x)-x=T(\tilde z)-\tilde z.$$
In particular,  the definition of $h_1$ on $T(z)$ is the same as that on $T(\tilde z)$.
By definition of our homotopy, we must have $z,z'\in T(z)$,\,\, $\tilde z,z'\in  T(\tilde z)$, from where we note that $z'\in T(z)\cap T(\tilde z)$. 
From the two last statements we get $z'-z=z'-\tilde z$, i.e.~$z=\tilde z$.
(Compare with Figure \ref{f:h1-G0-claim}.) Thus $\Phi$ is injective.
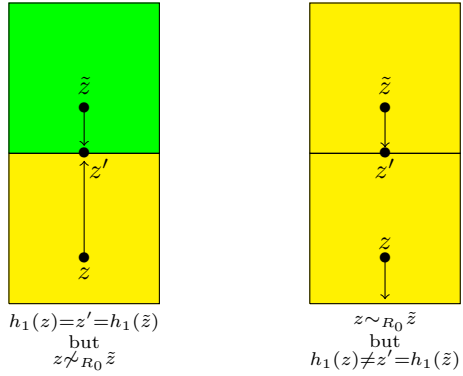
\begin{figure}[h]
\centerline{
\centerline{\begin{tikzpicture}[scale=2]
\draw[fill,yellow] (0,0) -- (1,0) -- (1,1) -- (0,1) -- (0,0);
\draw[fill,green] (0,1) -- (1,1) -- (1,2) -- (0,2) -- (0,1);
\draw (0,0) -- (1,0) -- (1,1) -- (0,1) -- (0,0);
\draw (0,1) -- (1,1) -- (1,2) -- (0,2) -- (0,1);
\node [black] at (.5,.3) {\textbullet};
\node at (.5,.2) {$z$};
\node [black] at (.5,1) {\textbullet};
\node at (.6,.9) {$z'$};
\node [black] at (.5,1.3) {\textbullet};
\node at (.5,1.45) {$\tilde z$};
\draw [->] (.5,.3) -- (.5,.95);
\draw [->] (.5,1.3) -- (.5,1.05);
\node at (.5,-.25) {$\substack{h_1(z)= z'=h_1(\tilde z)\\\text{but}\\z\not\sim_{R_0} \tilde z}$};

\draw[fill,yellow] (2,0) -- (3,0) -- (3,1) -- (2,1) -- (2,0);
\draw[fill,yellow] (2,1) -- (3,1) -- (3,2) -- (2,2) -- (2,1);
\draw (2,0) -- (3,0) -- (3,1) -- (2,1) -- (2,0);
\draw (2,1) -- (3,1) -- (3,2) -- (2,2) -- (2,1);
\node [black] at (2.5,.3) {\textbullet};
\node at (2.5,.4) {$z$};
\node [black] at (2.5,1) {\textbullet};
\node at (2.5,.9) {$z'$};
\node [black] at (2.5,1.3) {\textbullet};
\node at (2.5,1.45) {$\tilde z$};
\draw [->] (2.5,.3) -- (2.5,0.03);
\draw [->] (2.5,1.3) -- (2.5,1.04);
\node at (2.5,-.25) {$\substack{z\sim_{R_0} \tilde z\\\text{but}\\h_1(z)\ne z'=h_1(\tilde z)}$};

\end{tikzpicture}}
}
 \caption{$h_1(z)=z'=h_1(\tilde z) \text{ and } z\sim_{R_0} \tilde z$ implies $z=\tilde z$.\label{f:h1-G0-claim}}
\end{figure}

\end{proof}


\begin{cor}\label{c:psi-Cc-star-homo}
The map $\psi:C_c(R_0)\to C_c([0,1]\times R_0)$ given in Eq.~(\ref{e:psi}) is a $*$-homomorphism.
\end{cor}
\begin{proof}
Note that $\psi(f)(s,(x,y))=\phi_s(f)(x,y)$ for $f\in C_c(R_0)$, $s\in[0,1]$, $(x,y)\in R_0$.
We have $\psi(f^*)=\psi(f)^*$ because 
$$\psi(f^*)(s,(x,y))=\phi_s(f^*)(x,y)=\overline{\phi_s(f)}(y,x),\quad\text{and}$$
$$\psi(f^*)(s,(x,y))=\overline{\psi(f)}(s^{-1},(x,y)^{-1})=\overline{\psi(f)}(s,(y,x))=\overline{\phi_s(f)}(y,x).$$
We also have $\psi(f\star g)= \psi(f)\star \psi(g)$ because 
$$\psi(f\star g)(s,(x,y))=\phi_s(f\star g)(x,y)=\phi_s(f)\star \phi_s(g)(x,y)$$
and
\begin{eqnarray*}
\psi(f)\star \psi(g)((s,s),(x,y))&=&\sum_{(s',z)\sim(s,x)}\psi(f)((s,s'),(x,z))\psi(g)((s',s),(z,y))\\
&=&\sum_{z\sim x}\psi(f)(s,(x,z))\psi(g)(s,(z,y))\\
&=&\sum_{z\sim x}\phi_s(f)(x,z)\phi_s(g)(z,y)\\
&=&\phi_s(f)\star \phi_s(g)(x,y),
\end{eqnarray*}
where we identify the interval $[0,1]$ with the ``trivial" equivalence relation, where each $s\in[0,1]$ is equivalent only to itself, i.e.~$[s]=\{s\}$ for $s\in[0,1]$.
We also see it as the ``trivial" groupoid where each $s\in[0,1]$ is a unit.
\end{proof}
In the next two lemmas we apply Theorem \ref{t:tilde-psi-general} to extend the $*$-homomorphisms $\psi$ and $\phi_s$ to their completions under the full $C^*$-norm.
\begin{lem}\label{l:tilde-psi}
The map $\psi$ defined in Eq.~(\ref{e:psi}) extends to a (continuous) $*$-homomorphism
 $$\hat\psi:C^*(R_0)\to C^*([0,1]\times R_0).$$
\end{lem}
\begin{proof}
By Lemma \ref{l:psi-cont} and Corollary \ref{c:psi-Cc-star-homo},  the map $\psi: C_c(R_0)\to C_c([0,1]\times R_0)$ is continuous and it is a $*$-homomorphism.
The statement of the lemma follows by Theorem \ref{t:tilde-psi-general}.
\end{proof}
\begin{lem}\label{l:phi-s-star-homo}
The map $\phi_s:C_c(R_0)\to C_c(R_0)$  in Definition \ref{d:phi_s} extends to $*$-homomorphism
 $$\hat\phi_s:C^*(R_0)\to C^*(R_0),\qquad s\in[0,1].$$
\end{lem}
\begin{proof}
By Lemma \ref{l:phi_s-Cc-cont} and Lemma \ref{l:phi_s-Cc-star-homo},  the map $\phi: C_c(R_0)\to C_c(R_0)$ is continuous and it is a $*$-homomorphism.
By Theorem \ref{t:tilde-psi-general} the lemma follows.
\end{proof}

We have now arrived to the main result of this subsection, namely that $\phi_s$ gives a homotopy of $*$-homomorphisms of $C^*$-algebras.
We should note that in the statement of the proposition, we use the same notation for $\phi_s$ as for its extension.
\begin{pro}\label{p:phi-star-homo}
The map $\phi:[0,1]\times C_c(R_0)\to C_c(R_0)$ given by
$$\phi(s,g):=\phi_s(g)$$
extends to a homotopy of $*$-homomorphisms
 $$\phi_s:C^*(R_0)\to C^*(R_0),$$
where $\phi_s$ is given in Definition \ref{d:phi_s}, $C_c(R_0)$ has the inductive limit topology, and multiplication in $C_c(R_0)$ is convolution.
 In particular, $id\sim_h \phi_1$ (a path of $*$-homomorphisms).
\end{pro}
\begin{proof}
Let $\hat\phi:[0,1]\times C^*(R_0)\to C^*(R_0)$ be the map given by
$$\hat\phi(s,a):=\hat\phi_s(a),$$
where $\hat\phi_s$ was constructed in Lemma \ref{l:phi-s-star-homo}.
By Lemma \ref{l:tilde-psi}, 
$$\hat \psi:C^*(R_0)\to C^*([0,1]\times R_0)$$
 is a (continuous) $*$-homomorphism.
It is well known that
$$C^*([0,1]\times R_0)\cong C([0,1], C^*(R_0))$$
Via the above isomorphism, we have 
$$\hat \psi=(a\mapsto\hat \phi(\cdot, a))$$
(up to the isomorphism), and hence $\hat \phi(\cdot, a)\in C([0,1],C^*(R_0))$ for $a\in C^*(R_0)$.
It follows that the map $\rho:a\mapsto\hat \phi(\cdot, a)$ is a $*$-homomorphism from $C^*(R_0)$ to $C([0,1],C^*(R_0))$. 
We are now ready to show that $\hat \phi$ is continuous. Suppose that $(s_n, a_n)$ converges to $(s,a)$. Then
$$\hat\phi(s_n,a_n)-\hat \phi(s,a)=\hat\phi(s_n,a_n)-\hat\phi(s_n,a)+\hat\phi(s_n,a)-\hat\phi(s,a),$$
and
$$||\hat\phi(s_n,a_n)-\hat\phi(s_n,a)||_{C^*(R_0)}\le ||\rho(a_n-a)||_{\sup}\le ||a_n-a||\to 0,$$
because $\rho$ is bounded, and
$$||\hat\phi(s_n,a)-\hat\phi(s,a)||_{C^*(R_0)}\to 0,$$
because $\hat\phi(\cdot, a)\in C([0,1],C^*(R_0))$ and $s_n\to s$.
\end{proof}

%
\subsection{Connecting maps}\label{ss:connectingmaps}
In this subsection we do the preparations for computing the group homomorphisms $K_0(\iota_n)$, $K_1(\iota_n)$, $n\in\N_0$.
To accomplish this task, we need to use the filtration of $R_n$ by the skeletons of the tiling $T$, i.e.~
we need to use the short exact sequences in Eq.~(\ref{e:JBC}) and Eq.~(\ref{e:IAB}).
However, we cannot simply ``restrict" the inclusion map $\iota_n$ to the ideal and quotient $C^*$-algebras of these short exact sequences.
The reason is that we would get a non-commutative diagram, as it is going to be explained in Remark \ref{r:connectingmaps}.
But, one gets a commutative diagram when one instead ``restricts" the $*$-homomorphism $\iota_0\circ\phi_1$. 
Since $\phi_s$ is a homotopy of $*$-homomorphisms, and $\phi_0$ is the identity, we get by homotopy invariance of $K$-theory (cf.~\cite[Prop.~4.1.4, p.~61]{Rordam}), that $K_i(\iota_0)=K_i(\iota_0\circ\phi_1)$, $i=1,2$.
Recall that homotopy invariance of $K$-theory uses the property that the homotopy is a homotopy of $*$-homomorphisms $\phi_s$, since we need the fact that $\phi_s$, extended to the unitization, preserves projections and unitaries.
We only need to consider $\iota:=\iota_0$, since, in the natural bases in terms of stable cells, each matrix for the $K$-theory of $\iota_n\circ\phi_1$ ``restricted" to the ideal and quotient $C^*$-algebras is equal to that of $\iota_0\circ\phi_1$.
We end this subsection with Lemma \ref{l:K-maindiagram-1}, where we show the construction of a ``computable" cochain complex with connecting maps. It is ``computable" since it is given in terms of the compacts.

Similarly, from Eq.~(\ref{e:IAB}) we get the diagram (which will be shown to be commutative in Lemma \ref{l:connecting-maps})
\begin{equation}\label{e:IAB-connectingmaps}
\xymatrix{
0\ar[r]&I_0\ar@{^{(}->}[r]\ar[d]_{\alpha'}&A_0\ar[d]_{\alpha}\ar@{->>}[r]&B_0\ar[d]^{\beta}\ar[r]&0\\
0\ar[r]&I_1\ar@{^{(}->}[r]&A_1\ar@{->>}[r]&B_1\ar[r]&0,
}
\end{equation}
where the connecting maps $\alpha$, $\beta$, $\alpha'$ are defined in terms of $\phi_1$ as follows:\\
The map $\alpha:=\iota_0\circ\phi_1:A_0\to A_1$ is given by
$$\alpha(g)(x',y'):= \left\{ \begin{array}{ccl}
    g\circ (h_1\times h_1)(x',y') &\quad& h_1(x')\sim_{R_0} h_1(y') \\ 
    0 &\quad&  \text{else},\\ 
  \end{array}\right.$$
for $g\in C_c(R_0)$ and  $x'\sim_{R_1}y'$.\\
The map $\alpha':=\alpha|_{I_0}:I_0\to I_1$ is given by
$$\alpha'(g)(x',y'):= \left\{ \begin{array}{ccl}
    g\circ (h_1\times h_1)(x',y') &\quad& h_1(x')\sim_{R_0} h_1(y') \\ 
    0 &\quad&  \text{else},\\ 
  \end{array}\right.$$
for $g\in C_c(R_0|_{X_2^{T_0}-X_1^{T_0}})$ and  $x'\sim_{R_1}y'$. Note that $g$ vanishes on the edges of $T_0$.\\
The map $\beta:B_0\to B_1$ is given by 
$$\beta(g)(x',y'):= \left\{ \begin{array}{ccl}
    g\circ (h_1\times h_1)(x',y') &\quad& h_1(x')\sim_{R_0} h_1(y') \\ 
    0 &\quad&  \text{else},\\ 
  \end{array}\right.$$
for $g\in C_c(R_0|_{X_1^{T_0}})$ and  $x'\sim_{R_1}y'$ and $x',y'\in X_1^{T_1}$ lie in edges of $T_1$.
Note that $g$ is defined only on the edges of $T_0$.\\

From Eq.~(\ref{e:JBC}) we get the diagram (which will be shown to be commutative in Lemma \ref{l:connecting-maps})
\begin{equation}\label{e:JBC-connectingmaps}
\xymatrix{
0\ar[r]&J_0\ar@{^{(}->}[r]\ar[d]_{\beta'}&B_0\ar[d]_{\beta}\ar@{->>}[r]&C_0\ar[d]^{\gamma}\ar[r]&0\\
0\ar[r]&J_1\ar@{^{(}->}[r]&B_1\ar@{->>}[r]&C_1\ar[r]&0,
}
\end{equation}
where the connecting maps $\beta'$, $\gamma$ are defined as follows:\\
The map $\beta':=\beta|_{J_0}$ is given by
$$\beta'(g)(x',y'):= \left\{ \begin{array}{ccl}
    g\circ (h_1\times h_1)(x',y') &\quad& h_1(x')\sim_{R_0} h_1(y') \\ 
    0 &\quad&  \text{else},\\ 
  \end{array}\right.$$
for $g\in C_c(R_0|_{X_1^{T_0}-X_0^{T_0}})$ and  $x'\sim_{R_1}y'$ and $x',y'$ are in the edges of $T_1$. 
Note that $g$ is defined only on the edges of $T_0$ and vanishes on the vertices of $T_0$.\\
The map $\gamma$ is given by 
$$\gamma(g)(x',y'):= \left\{ \begin{array}{ccl}
    g\circ (h_1\times h_1)(x',y') &\quad& h_1(x')\sim_{R_0} h_1(y') \\ 
    0 &\quad&  \text{else},\\ 
  \end{array}\right.$$
for $g\in C_c(R_0|{X^{T_0}_0})$ and  $x'\sim_{R_1}y'$ and $x',y'\in X_0^{T_1}$ are vertices of $T_1$.
Note that $g$ is defined only on the vertices of $T_0$.
\begin{rem}\label{r:connectingmaps}
Roughly speaking, the connecting maps of the above two diagrams are just the inclusion maps composed with the homotopy $\phi_1$. By the next lemma, these diagrams commute.
Had we defined the above diagrams solely in terms of the inclusion map (without homotopy) then the restriction maps $\alpha'$, $\beta'$ would not map into the ideals $I_1$, $J_1$ respectively, and the connecting maps $\beta$, $\gamma$ would make the diagrams to not commute (to make sense of this see for example Figure \ref{f:betagammaNotCommutative}).  
\end{rem}

\begin{figure}[h]
\centerline{
\includegraphics[scale=.23]{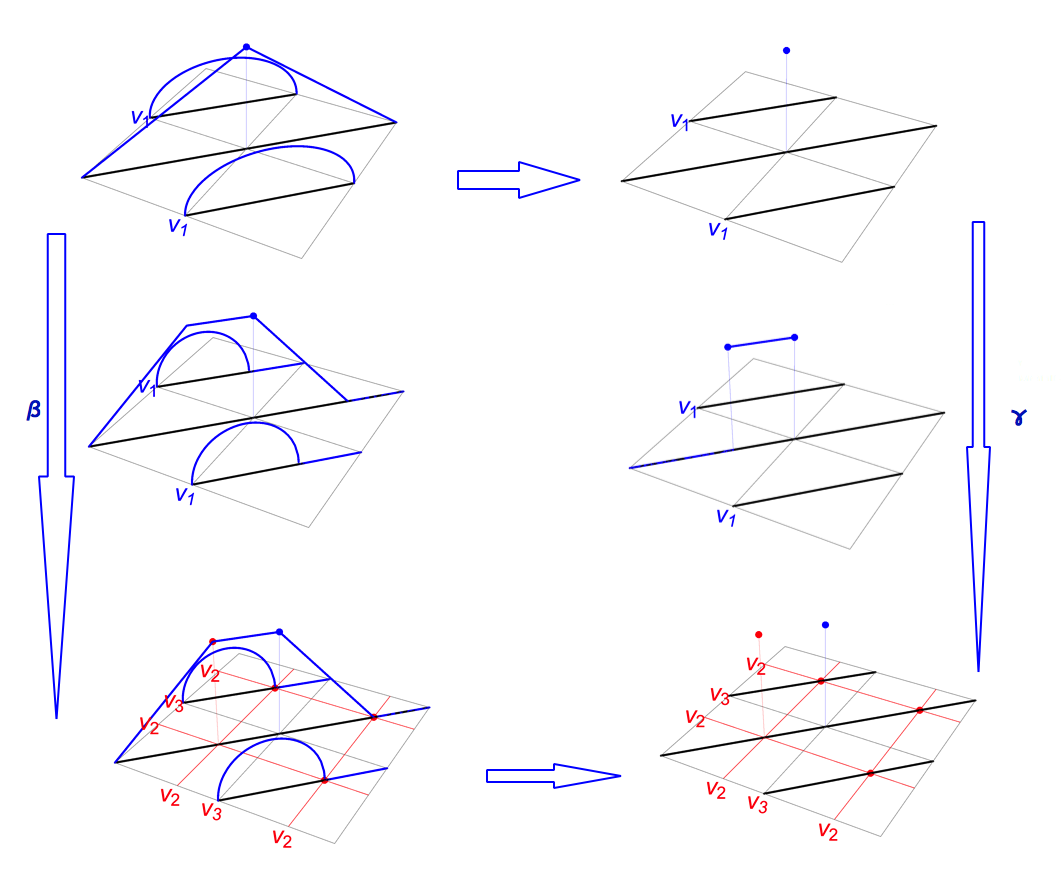}
}
 \caption{
 A commutative diagram for the fibonacci tiling.
 \label{f:betagammaNotCommutative}}
\end{figure}

\begin{lem}[Connecting maps]\label{l:connecting-maps}
The diagrams in Eq.~(\ref{e:IAB-connectingmaps}) and Eq.~(\ref{e:JBC-connectingmaps}) are commutatative diagrams of $*$-homomorphisms.
\end{lem}
\begin{proof}
The map $\alpha|_{I_0}$ is a $*$-homomorphism because it is the composition $I_0 \hookrightarrow A_0 \xrightarrow{\alpha} A_1$ of $*$-homomorphisms.
The left square of the diagram in Eq.~(\ref{e:IAB-connectingmaps}) commutes because $\alpha'$ is just a corestriction of the $*$-homomorphism $\alpha|_{I_0}$. 
That is, $\alpha'$ maps $I_0$ into $I_1$.
This follows since the map $h_1:T_1\to T_0$ is cellular. Indeed,
if $g\in C_c(R_0|_{X_2^{T_0}-X_1^{T_0}})$ vanishes on the edges of $T_0$, then $\alpha'(g)\in C_c(R_1|_{X_2^{T_1}-X_1^{T_1}})$ vanishes on the edges of $T_1$.
That is, $\alpha'(g)(x',y')=g(h_1(x'),h_1(y'))=0$ for $x',y'\in X_1^{T_1}$ because $h_1(x'),h_1(y')$ are in the edges of $T_0$.

We will now show that $\beta$ is a $*$-homomorphism making the right square in Eq.~(\ref{e:IAB-connectingmaps}) commute.

We start by showing that 
$$\xymatrix{0\ar[r]& I'_n\ar@{^(->}[r]^{i_n}& A'_n\ar@{>>}[r]^{q_n}& B'_n\ar[r]& 0}$$
is a short exact sequence, where $I'_n:=C_c(R_n|_{X_2^{T_n}-X_1^{T_n}})$, $A'_n:=C_c(R_n)$, $B'_n:=C_c(R_n|_{X_1^{T_n}})$.
To show that $q_n$ is surjective one uses Tietze's extension theorem applied to functions with compact support, cf.~\cite[Theorem 3.4.8~p.124]{PutnamCnotes}.
That $q_n$ is a $*$-homomorphism follows from the fact that the 1-skeleton $X_1^{T_n}$ is $R_n$-invariant.
The map $i_n$ extends functions with compact support to be zero outside $R_n|_{X_2^{T_n}-X_1^{T_n}}$ such that the extended function is continuous, which can be done as $R_n|_{X_2^{T_n}-X_1^{T_n}}$ is open in $R_n$.
It follows that $i_n$ is injective. 
That $i_n$ is a $*$-homomorphism follows by the fact that $X_2^{T_n}-X_1^{T_n}$ is $R_n$-invariant.
It is easy to see that $\im i_n = \ker q_n$. 

If $g,g'\in C_c(R_0)$ are equal on the 1-skeleton of $T_0$, 
i.e.~$q_0(g)=q_0(g')$ then $(q_1\circ \alpha)(g)=(q_1\circ\alpha)(g')$ because $h_1$ is cellular.
More precisely, for $x',y' \in X_1^{T_1}$ we get by cellularity of $h_1$ that $h_1(x'),h_1(y')$ are in the 1-skeleton of $T_0$.
Then 
$$(q_1\circ\alpha)(g)(x',y')=g(h_1(x'),h_1(y')) = g'(h_1(x'),h_1(y')) = (q_1\circ\alpha)(g')(x',y')$$
if $h_1(x')\sim_{R_0} h_1(y')$ else they are both zero. 
Hence $q_1\circ \alpha:A'_0\to B'_1$ descends to the quotient giving a unique $*$-homomorphism $\beta:B'_0\to B'_1$, making the diagram commute. The evaluation of $\beta$ on $B'_0$ is the map shown above. 
(In this proof, $\beta$ is defined on $B'_0$, while its extension to the $C^*$-algebra is denoted by $\hat{\hat\beta}$. We do this for notational reasons).
It is a general fact that $\beta$ preserves multiplication:
$$\beta(q_0(g)q_0((g'))=\beta(q_0(gg'))=q_1(\alpha(gg'))=
q_1(\alpha(g))q_1(\alpha(g'))=\beta(q_0(g))\beta(q_0(g')),$$
 for $g,g'\in C_c(R_0)$.
It is clear that $\beta$ is $*$-preserving.

To continue with our proof we need the following result.

\begin{lem}\label{l:beta-Cc-cont}
Equip $C_c(R_0)$ with the inductive limit topology (cf.~\cite[Ex.~5.10,p.118]{Conway}). Then $\beta:C_c(R_0|_{X_1^{T_0}})\to C_c(R_1|_{X_1^{T_1}})$ is continuous.
\end{lem}
\begin{proof}
As explained in the proof of Lemma \ref{l:phi_s-welldefined}, it is enough to consider the restriction 
to a connected component $C$ of $R_0$ via the range map $r:R_0\to \R^d$ because a $R_0$-compact set can only intersect a finite number of connected components, because the range map restricted to a connected component is a homeomorphism onto its image, and because $\beta(g)$ is 0 outside $R_0$ for $g\in C_c(R_0|_{X_1^{T_0}})$.
Thus it suffices to show that the maps $\hat \beta$ and $\hat\iota$ defined below are continuous in the inductive limit topology.
The map $\hat\iota:C_c(R_0|_{X_1^{T_1}})\to C_c(R_1|_{X_1^{T_1}})$ is given by
$$
\hat\iota(g)(x',y'):= \left\{ \begin{array}{cc}
    g(x',y') & x'\sim_{R_0} y' \\ 
    0 & x'\not\sim_{R_0} y' \\ 
  \end{array}\right.,
$$
where $x',y'\in X_1^{T_1}$ and $x'\sim_{R_1} y'$.
Since $R_0$ is open in $R_1$, we get by the gluing lemma for open subsets that one can always extend a function with compact support to be zero outside $R_0$ such that the extended function is continuous. Hence $\hat\iota$ is well-defined.
Moreover, $\hat\iota$ is continuous in the inductive limit topology because the support of a function is unchanged when extending it by zeroes.
 
The map $\hat\beta:C_c(r(C)\cap X_1^{T_0})\to C_c(r(C)\cap X_1^{T_1})$ is given by 
$$
\hat\beta(f)(x'):= \left\{ \begin{array}{cc}
    f(h_1(x')) & h_1(x')\in r(C) \\ 
    0 & h_1(x')\not\in r(C) \\ 
  \end{array}\right.
$$
where $x'\in X_1^{T_1}$ (and thus $h_1(x')\in X_1^{T_0}$).

Suppose $f_n\in C_c(r(C)\cap X_1^{T_0})$ converges to $f\in C_c(r(C)\cap X_1^{T_0})$ in the inductive limit topology.
That is, there is a compact subset $K\subset r(C)\cap X_1^{T_0}$ such that for all $n\in \N$
$$\Supp f_n,\,\Supp f\subset K$$
and such that
$$||f_n-f||_{\sup}\to 0\quad\text{as} \quad n\to \infty.$$
Note that the supremum is taken over $K$. 
We claim that 
$$K':=h_1^{-1}(K)\cap X_1^{T_1} \subset r(C)\cap X_1^{T_1}$$
 is compact and that 
 $$\Supp\, \hat\beta(f_n),\, \Supp\, \hat\beta(f)\subset K'.$$
 Then
\begin{eqnarray*}
||\hat\beta(f_n)-\hat\beta(f)||_{\sup}&=&\sup_{x'\in K'} |\hat\beta(f_n)(x')-\hat\beta(f)(x')|\\
&=&\sup_{\substack{x'\in K'\\ h_1(x')\in r(C)}} |f_n(h_1(x'))-f(h_1(x'))|\\
&\le&\sup_{x\in K} |f_n(x)-f(x)|\\
&=&||f_n-f||_{\sup}\to 0 \text{ as } n\to \infty,
\end{eqnarray*}
where the inequality holds because we can only get nonzero values if $h_1(x')\in K$.

Next we check the claim.
By an argument similar to Eq.~(\ref{e:supp-phi-s-homotopy}),  $\Supp \hat\beta(f_n)=h_1^{-1}(\Supp f_n)\cap X_1^{T_1}$ and $\Supp \hat\beta(f)=h_1^{-1}(\Supp f)\cap X_1^{T_1}$. 
It follows that
$$\Supp \hat\beta(f_n),\,\Supp \hat\beta(f)\,\subset\,K'.$$
Now, $K$ is compact in $r(C)$ since the inclusion $r(C)\cap X_1^{T_0}\hookrightarrow r(C)$ is continuous.
By the same argument as in the proof of Lemma \ref{l:phi_s-welldefined}, $h_1^{-1}(K)$ is compact in $r(C)$.
Since the 1-skeleton $X_1^{T_1}$ is $\R^d$-closed, 
$K'\subset r(C)\cap X_1^{T_1}$ is compact in $\R^d$ hence also compact in $r(C)\cap X_1^{T_1}$ by Lemma \ref{l:X-compact}.
\end{proof}

We continue with the proof of Lemma \ref{l:connecting-maps}.
By the above lemma, Lemma \ref{l:beta-Cc-cont}, $\beta$ is continuous in the inductive limit topology.
By Theorem \ref{t:tilde-psi-general}, $\beta$ extends to a $*$-homomorphism $\hat{\hat\beta}:B_0\to B_1$.
By continuity of $\hat{\hat\beta}$, we get $\hat{\hat\beta}(\lim b_i)=\lim \beta(b_i)$, $b_i\in C_c(R_0|_{X_1^{T_0}})$, from which it follows that the right square commutes.

A similar argument shows that the diagram in Eq.~(\ref{e:JBC-connectingmaps}) commutes, and that the maps are $*$-homomorphisms.
\end{proof} 

The inclusion map $\iota_n$ was shown to be a $*$-homomorphism in \cite[Prop.~IV.9, p.~57]{Gon0} using the reduced norm.
In the next remark, we include an alternative proof using the full norm.
\begin{rem}\label{r:iota-cont}
Recall that $\iota_n:C_c(R_n)\to C_c(R_{n+1})$ is given by 
$$
\iota_n(g)(x',y'):= \left\{ \begin{array}{cc}
    g(x',y') & x'\sim_{R_n} y' \\ 
    0 & x'\not\sim_{R_n} y' \\ 
  \end{array}\right.,
$$
where $x'\sim_{R_{n+1}} y'$.
The map $\iota_n$ is well-defined and continuous in the inductive limit topology by the same proof as for $\hat\iota$ in the proof of Lemma \ref{l:beta-Cc-cont}.
It is easy to check that $\iota_n$ is a $*$-homomorphism.
By Theorem \ref{t:tilde-psi-general}, $\iota_n$ extends to a continuous $*$-homomorphism $\iota_n:C^*(R_n)\to C^*(R_{n+1})$.
\end{rem}

\begin{lem}\label{l:K-maindiagram-1}
Each row in the following diagram is a cochain complex of abelian groups (not a short exact sequence). 
\begin{equation}\label{e:K-maindiagram-1}
\xymatrix{
0\ar[r]&K_0(C_0)\ar[r]^{\delta^0}\ar[d]_{K_0(\gamma)}&K_1(J_0)\ar[d]_{K_1(\beta')}\ar[r]^{\delta^1}&K_0(I_0)\ar[d]^{K_0(\alpha')}\ar[r]&0\\
0\ar[r]&K_0(C_1)\ar[r]^{\delta^0}&K_1(J_1)\ar[r]^{\delta^1}&K_0(I_1)\ar[r]&0.
}
\end{equation}
Moreover the diagram commutes, (i.e.~the diagram is a cochain map).
\end{lem}
\begin{proof}
From the six-term exact sequence in Eq.~(\ref{e:K-JBC}) we get the diagram with exact rows
\begin{equation*}
\xymatrix{
K_0(C_0)\ar[r]^{\delta^0}\ar[d]_{K_0(\gamma)}&K_1(J_0)\ar[d]_{K_1(\beta')}\ar@{>>}[r]^{}&K_1(B_0)\ar[d]^{K_1(\beta)}\ar[r]&0\\
K_0(C_1)\ar[r]^{\delta^0}&K_1(J_1)\ar@{>>}[r]^{ }&K_1(B_1)\ar[r]&0.
}
\end{equation*}
The diagram commutes because of the naturality of the exponential map, because the diagram in Eq.~(\ref{e:JBC-connectingmaps}) commutes, and because $K_j(\phi\circ\psi)=K_j(\phi)\circ K_j(\psi)$, $j=1,2$.
Similarly, from the six-term exact sequence in Eq.~(\ref{e:K-IAB-1}) we get the commutative diagram
\begin{equation*}
\xymatrix{
K_1(B_0)\ar[r]^{\tilde\delta^1}\ar[d]_{K_1(\beta)}&K_0(I_0)\ar[d]^{K_0(\alpha')}\\
K_1(B_1)\ar[r]^{\tilde\delta^1}&K_0(I_1).
}
\end{equation*} 
Putting these two diagrams together and using Eq.~(\ref{e:K-IAB-2}) we get
\begin{equation}\label{e:delta0-delta1-deltat1}
\xymatrix{
K_0(C_0)\ar[r]^{\delta^0}\ar[d]_{K_0(\gamma)}&K_1(J_0)\ar[d]_{K_1(\beta')}\ar@{>>}[r]^{}\ar@/^2pc/[rr]^{\delta^1}&K_1(B_0)\ar[d]^{K_1(\beta)}\ar[r]^{\tilde\delta^1}&K_0(I_0)\ar[d]^{K_0(\alpha')}\\
K_0(C_1)\ar[r]^{\delta^0}&K_1(J_1)\ar@{>>}[r]^{ }\ar@/_2pc/[rr]_{\delta^1}&K_1(B_1)\ar[r]^{\tilde\delta^1}&K_0(I_1).
}
\end{equation}
By Eq.~(\ref{e:d12iszero}), $\delta^1\circ\delta^0=0$. Thus each row in Eq.~(\ref{e:K-maindiagram-1}) is a cochain complex of abelian groups.
\end{proof}
%
%
\subsection{Computation of $K_0(\gamma)$, $K_0(\alpha')$, $K_1(\beta')$}
In this subsection, we compute the connecting maps of the diagram in Lemma \ref{l:K-maindiagram-1}.
That is, we compute the connecting maps $K_0(\gamma)$, $K_0(\alpha')$, $K_1(\beta')$ in terms of the generators (i.e.~the natural basis elements) of the $K$-groups.
This can be done since they are expressed in terms of the compact operators.
The following lemma will be given up to the isomorphism given in Proposition \ref{p:compacts}, which we denote by $\lambda$.
\color{black}
\begin{lem}
The generators of $K_0(C_0)$ are 
$$[e_{uu}]_0,$$
where $u$ is a vertex in $T$. Take a vertex $u\in T$ from each $R_0$-equivalence class (i.e.~each stable class) to get all generators.
\end{lem}
\begin{proof}
Let 
$$e_{uv}(\delta_w):= \left\{ \begin{array}{cc}
    \delta_u & v=w \\ 
    0 & else \\ 
  \end{array}\right.$$
be a matrix unit (e.g.~it sends $\delta_v$ to $\delta_u$), where $u\sim_{R_0}v\sim_{R_0}w$ are vertices in $T$, and $\delta_w\in \ell^2([u]_{R_0})$ is the standard basis element of the Hilbert space $\ell^2([u]_{R_0})$, i.e.
for $w'\in [u]_{R_0}$,
$$\delta_w(w'):=
 \left\{ \begin{array}{cc}
    1 & w=w' \\ 
    0 & else. \\ 
  \end{array}\right.
 $$
 
Recall that a 1-dimensional projection generates the $K_0$-group of the compact operators $K$ 
and that $K_0(K)$ is isomorphic to $\Z$ (cf.~\cite[Cor.~6.4.2,~p.103]{Rordam}).
One can choose any 1-dimensional projection because they are all equivalent (cf.~Lemma \ref{l:onedimprojareequiv}).
Since $e_{uu}$ is a 1-dimensional projection, and since $C_0$ can be written in terms of the compacts by Proposition~\ref{p:compacts}(1), the lemma follows by taking a vertex $u\in T$ from each $R_0$-equivalence class.
\end{proof}

\begin{lem}\label{l:Wv} With notation as in the previous lemma, $K_0(\gamma):K_0(C_0)\to K_0(C_1)$ evaluated on the generators of $K_0(C_0)$ is given by
\begin{equation*}
K_0(\gamma)([e_{uu}]_0)=[\gamma(e_{uu})]_0=\sum_{h_1(u')=u}[e_{u'u'}]_0,
\end{equation*}
where the sum is over all vertices $u'$ in $T_1=\frac{1}{\lambda}\omega(T)\lambda$.
\end{lem}
\begin{proof}
Take vertex $u\in T$. Then by definition of $K_0$ on $*$-homomorphisms we get (cf.~\cite[p.~61]{Rordam}) 
$$K_0(\lambda\gamma\lambda^{-1})([(e_{uu})]_0)=[\lambda\gamma\lambda^{-1}(e_{uu})]_0,$$
where $\lambda$ is the isomorphism in Proposition \ref{p:compacts}(1).
We have $\lambda^{-1}(e_{uu}) = g$, where
$$g(x,y):= \left\{ \begin{array}{cc}
    1 & x=y=u \\ 
    0 & else. \\ 
  \end{array}\right.
$$
Then since 
$$\gamma(g)(x,y)=\left\{  
\begin{array}{ll}
    1 & x=y\text{ and } h_1(x)=u \\ 
    0 & else \\ 
  \end{array}\right.$$
where $x$, $y$ are vertices in $T_1$ such that $x\sim_{R_1} y$, we get
 $$\lambda\gamma\lambda^{-1}(e_{uu})=\sum_{h_1(u')=u}e_{u'u'},$$
 where the sum is over all vertices $u'\in T_1$.

\end{proof}

The following lemma will be given up to the isomorphism given in Proposition \ref{p:compacts}, which we denote by $\lambda$.
\begin{lem}
The generators of $K_0(I_0)$ are 
$$[b\otimes e_{uu}+s(b)\tensor (I-e_{uu})]_0-[s(b)\tensor I]_0,$$
where $u\in T$ is a face, 
$b\in M_2\big(\,\widetilde{C_0(\open{u})}\,\big)$ is the (reparametrized) Bott projection, $s(b)$ is its scalar part, and 
$$I=\sum_{v} e_{vv},$$ where the sum is over all faces $v\in T$.\\
Choose face $u\in T_0$ from each $R_0$-equivalence class (i.e.~from each stable face) to get all generators.
\end{lem}
\begin{proof}
Choose protoface $f_n\in T$, ($n$ fixed), and let  $K_n:=K(\ell^2([f_n]))$ be the $C^*$-algebra of compact operators on the Hilbert space $H_n:=\ell^2([f_n])$.
Since $[f_n]_{R_0}$ is an infinite set, $K_n$ is non-unital, i.e.~$I_n:=I_{H_n}\not \in K_n$.
Let 
$$\D:=\{z\in \C\mid |z|<1\}$$ be the open unit disk.
Since $\D$ is homeomorphic to $\open{f_n}$, we have $C_0(\open{f_n},K_n)\cong C_0(\D,K_n)$.
Recall that $C_0(\D,K_n)$ is isomorphic to  $C_0(\D)\tensor K_n$ via the map 
$$f\tensor a\mapsto (t\mapsto f(t) a).$$
(cf~\cite[p.173]{WeggeOlsen}).
Let $e_{uv}\in K_n$ be the standard matrix unit
$$e_{uv}(\delta_w)= \left\{\begin{array}{cc}
    \delta_u & w=v \\ 
    0 & else \\ 
  \end{array}\right.
$$
where $u,v\in [f_n]$, i.e.~$u,v\in T$ are faces $R_0$-equivalent to face $f_n$.
Consider the projection $e_{uu}\in K_n$, and let
$\phi_u:{C_0(\D)}\to {C_0(\D)\tensor K_n}$ be the $*$-homomorphism 
$$\phi_u(a):= a \tensor e_{uu}.$$
Then  $K_0(C_0(\D)\tensor K_n)$ is canonically isomorphic to $K_0(C_0(\D))$ via the homomorphism $K_0(\phi_u)$, (but the inverse is not explicit).
Moreover, $K_0(C_0(\D))$ is isomorphic to $\Z$ via the map
$1\mapsto [b]_0-[s(b)]_0$ ,
where $b\in M_2(\widetilde{C_0(\D)})$ is the Bott projection, and $s(b)$ is its scalar part (cf.~\cite[p.156,173]{WeggeOlsen}, \cite[Example~11.3.2~p.200]{Rordam}, \cite[Exercise~9.5,~p.171]{Rordam}).
We remark that the Bott projection and its scalar part correspond to the elements in $M_2(C(\bar \D))$  given by 
$$b=(z\mapsto
  \left[\begin{array}{cc}
    |z|^2 & z\sqrt{1-|z|^2} \\ 
    \bar z\sqrt{1-|z|^2} & 1-|z|^2 \\ 
  \end{array}\right]),\qquad
s(b)=
  \left[\begin{array}{cc}
    1_{C(\bar \D)} & 0 \\ 
    0 & 0 \\ 
  \end{array}\right]=
(z\mapsto
  \left[\begin{array}{cc}
    1 & 0 \\ 
    0 & 0 \\ 
  \end{array}\right]).
$$
Define the non-unital $C^*$-algebras
$$A:=\bigoplus_{i=1}^{\sF} C_0(\D),\qquad A':=\bigoplus_{i=1}^{\sF} C_0(\D)\tensor K_n.$$
 As usual,  $\tilde A$, $\widetilde{A'}$ will denote their unitization, respectively.
Let $p\in M_2(\widetilde{A'})$ be the projection (assuming n=1)
\begin{eqnarray*}
p&:=&\big((b-s(b))\tensor e_{uu},0,\ldots,0\big)+\big(s(b)\tensor I_1,\ldots,s(b)\tensor I_{\sF}\big)\\
&=&b\tensor e_{uu}+s(b)\tensor (I-e_{uu}),
\end{eqnarray*}
where 
$$I:=\sum_{i=1}^{\sF} I_i=\sum_{i=1}^{\sF}\sum_{v\sim f_i} e_{vv}=\sum_{v} e_{vv},$$
and the last sum is over all faces $v$ of $T$. (We use the rule of tensors $a\tensor x+a\tensor y=a\tensor(x+y)$).
Thus the scalar part of $p$ is $s(p)=s(b)\tensor I$.
Note that the unit element of $\widetilde{A'}$ is
$$1_{\widetilde{A'}}=1_{\widetilde{C_0(\D)}}\tensor I,$$
where $1_{\widetilde{C_0(\D)}}$ is the constant function $z\mapsto 1$, $z\in \bar \D$.
Let $b_A\in M_2(\tilde A)$ be the projection
\begin{eqnarray*}
b_A&:=&\big((b-s(b)),0,\ldots\big)+\big(s(b),s(b),\ldots,s(b))\\
&=&(b,s(b),\ldots,s(b)).
\end{eqnarray*}
We claim that
\begin{equation}\label{e:bp}
\tilde\phi_2(b_A)=p,
\end{equation}
where $\tilde\phi_2$ is the usual extension of $\phi_u$ to $M_2(\tilde A)\to M_2(\widetilde{A'})$.
Then by \cite[Prop.~4.2.2,~p.~63]{Rordam} we have the following identifications
$$1\mapsto [b]_0-[s(b)]_0
\mapsto [b_A]_0-[s(b_A)]_0
\mapsto [p]_0-[s(p)]_0,
$$
from which the lemma follows.

The proof of the claim, which is inspired by the splitting lemma, follows from the facts that 
\begin{equation}\label{e:rs}
\phi_2\big(b_A-s(b_A)\big)=p-s(p)\qquad\text{and}\qquad \psi_2(\pi_2(b_A))=\pi'_2(p),
\end{equation}
where $\phi_2$ is the extension of $\phi_u$ to $M_2(A)\to M_2(A')$, and $\psi_2,\pi_2,\pi'_2$ are defined below.
Here we use the fact that we have a commutative diagram of short exact sequences that split as $\C$-vector spaces:
$$
\xymatrix{0\ar[r]&A\ar[r]^{i }\ar[d]_{\phi}&\ar@/_1.0pc/[l]_{r}\tilde A\ar[r]^{\pi}\ar[d]_{\tilde\phi}&\C\ar@/_1.0pc/[l]_{\sigma}\ar[r]\ar[d]_{\psi}&0\\
  0\ar[r]&A'\ar[r]^{i' }&\ar@/_1.0pc/[l]_{r'}\widetilde{A'}\ar[r]^{\pi'}&\C\ar@/_1.0pc/[l]_{\sigma'}\ar[r]\ar[r]&0,}
$$
where $\phi:=\phi_u$, and the retractions are $r(a)=a-s(a)$ and $r'(a')=a'-s(a')$, and the sections $\sigma(\pi(a))=s(a)$ and $\sigma'(\pi'(a'))=s(a')$, and $\psi=id_{\C}$, and we have the identities
\begin{eqnarray}\label{e:rispi}
&&r\circ i=id_A,\quad \pi\circ \sigma=id_\C,\quad i\circ r+\sigma\circ \pi=id_{\tilde A}\\
\nonumber&& r'\circ i'=id_{A'},\quad \pi'\circ \sigma'=id_\C,\quad i'\circ r'+\sigma'\circ \pi'=id_{\widetilde{A'}}.
\end{eqnarray}
By functoriality and additivity of $M_2$ we get the commutative diagram
$$
\xymatrix{0\ar[r]&M_2(A)\ar[r]^{ }\ar[d]_{\phi_2}&\ar@/_1.0pc/[l]_{r_2}M_2(\tilde A)\ar[r]^{\pi_2}\ar[d]_{\tilde\phi_2}&M_2(\C)\ar@/_1.0pc/[l]_{\sigma_2}\ar[r]\ar[d]_{\psi_2}&0\\
  0\ar[r]&M_2(A')\ar[r]^{ }&\ar@/_1.0pc/[l]_{r'_2}M_2(\widetilde{A'})\ar[r]^{\pi'_2}&M_2(\C)\ar@/_1.0pc/[l]_{\sigma_2'}\ar[r]&0,}
$$
where the operations are done coordinatewise. In particular, the equations corresponding to Eq.~(\ref{e:rispi}) hold here as well.
Since $b_A\in M_2(\tilde A)$ and $p\in M_2(\widetilde{A'})$ and assuming $\phi_2\circ r_2(b_A)=r'_2(p)$ and $\psi_2(\pi_2(b_A))=\pi'_2(p)$ then $\tilde\phi_2(b_A)=p$ because
$$\tilde\phi_2(b_A)=\tilde\phi_2 (i_2\circ r_2(b_A)+\sigma_2\circ \pi_2(b_A))$$
$$=i'_2\circ\phi_2\circ r_2(b_A) +\sigma'_2\circ \psi_2\circ\pi_2(b_A)=i'_2\circ r'_2(p)+\sigma'_2\circ\pi'_2(p)=id_{\widetilde{A'}}(p)=p.$$
The converse, although not needed in this proof, follows by commutativity of the diagram.
Thus Eq.~(\ref{e:bp}) is equivalent to Eq.~(\ref{e:rs}). 
It remains to show that Eq.~(\ref{e:rs}) holds.  The first equality follows immediately:
\begin{eqnarray*}
\phi_2(b_A-s(b_A))&=&\phi_2\big(b-s(b),0,\ldots,0\big)\\
&=&\big((b-s(b))\tensor e_{uu},0,\ldots,0\big)\\
&=&p-s(p).
\end{eqnarray*}

 Let $\hat 1:=1_{\widetilde{\C_0(\D)}}$. Then $s(b_A)\in M_2(\tilde A)$ is
\begin{eqnarray*}
s(b_A)&=&(s(b),\ldots,s(b))\\
&=&(
\left[
  \begin{array}{cc}
    \hat 1 & 0 \\ 
    0 & 0 \\ 
  \end{array}
  \right],
\ldots
,
\left[
  \begin{array}{cc}
    \hat 1 & 0 \\ 
    0 & 0 \\ 
  \end{array}
  \right]
)\\
&=&
\left[
  \begin{array}{cc}
    (\hat 1,\ldots,\hat 1) & 0 \\ 
    0 & 0 \\ 
  \end{array}
  \right]\\
&=&
\left[
  \begin{array}{cc}
    1_{\tilde A} & 0 \\ 
    0 & 0 \\ 
  \end{array}
  \right],
\end{eqnarray*}
where we use the identifications 
$$\bigoplus_{i=1}^{\sF} M_2\big(\widetilde{C_0(\D)}\big)\cong M_2\big(\bigoplus_{i=1}^{\sF}\,\widetilde{C_0(\D)}\,\big)\supset M_2(\tilde A).$$ 
Hence 
$$\pi_2(b_A)=\pi_2(s(b_A))=
\left[
  \begin{array}{cc}
    1 & 0 \\ 
    0 & 0 \\ 
  \end{array}
  \right]\in M_2(\C).
$$
Similarly,
\begin{eqnarray*}
s(p)&=&(
\left[
  \begin{array}{cc}
    \hat 1\tensor I_1 & 0 \\ 
    0 & 0 \\ 
  \end{array}
  \right],
\ldots
,
\left[
  \begin{array}{cc}
    \hat 1\tensor I_{\sF} & 0 \\ 
    0 & 0 \\ 
  \end{array}
  \right]
)\\
&=&
\left[
  \begin{array}{cc}
    (\hat 1\tensor I_1,\ldots,\hat 1\tensor I_{\sF}) & 0 \\ 
    0 & 0 \\ 
  \end{array}
  \right]\\
&=&
\left[
  \begin{array}{cc}
    1_{\widetilde{A'}} & 0 \\ 
    0 & 0 \\ 
  \end{array}
  \right],
\end{eqnarray*}
yields
$$\pi_2(p)=\pi_2(s(p))=
\left[
  \begin{array}{cc}
    1 & 0 \\ 
    0 & 0 \\ 
  \end{array}
  \right]\in M_2(\C).
$$
Hence the second equality in  Eq.~(\ref{e:rs}) holds as well.
\end{proof}
\begin{lem}\label{l:Wf} With notation as in the previous lemma,
$K_0(\alpha'):K_0(I_0)\to K_0(I_1)$ evaluated on the generators of $K_0(I_0)$ is 
$$K_0(\alpha')\Big([b\otimes e_{uu}+s(b)\tensor (I-e_{uu})]_0-[s(b)\tensor I]_0\Big)$$
$$=[\alpha'\Big( b\otimes e_{uu}+s(b)\tensor (I-e_{uu})\Big) ]_0-[\alpha'\Big(s(b)\tensor (I-e_{uu})\Big)]_0$$
$$=[b\otimes e_{u'u'}+s(b)\tensor (I-e_{u'u'})]_0-[s(b)\tensor I]_0.$$
In the above expression $u'\in T_1$ is the unique face of $T_1=\frac{1}{\lambda}\omega(T)\lambda$, which homotopes to $u$.
\end{lem}

\begin{proof}
Let $g:R_0\to M_2(\C)$ be given by $g = {\tilde\lambda}^{-1}\big(b\otimes e_{uu}+s(b)\tensor (I-e_{uu})\big)$ where $\tilde\lambda$ is the $*$-isomorphism induced by $\lambda$. Note that here $b$ denotes the Bott projection reparametrized to face $u$ and note that $g$ is (up to isomorphism) a projection in $M_2$ of the unitization of $I_0=C^*(R_0|_{X_2-X_1})$. Evaluating $g$ we get
$$g(x,y)=\left\{
  \begin{array}{cc}
    s(b)(x) & x=y\not\in u \\ 
    b(x) & x=y\in u \\ 
    0 & x\ne y \\ 
  \end{array}\right..
$$
Note that $g$ is zero outside the diagonal of $R_0$, and it is constant on the diagonal intersected with $X_1\times X_1$.
Next, we evaluate $\tilde\alpha'(g)$ and get
$$\tilde\alpha'(g)(x,y)=\left\{
  \begin{array}{cc}
    b(h_1(x)) & x=y\in u' \\
    c & x=y\not \in u' \\
    0 & x\ne y \\ 
  \end{array}\right.,
$$
where $c:=
  \left[\begin{array}{cc}
    1 & 0 \\ 
    0 & 0 \\ 
  \end{array}\right]\in M_2(\C)$ is the constant value of the constant function $s(b)$.
Then, denoting the constant function with value $c$ by $\zeta(x):=c$ (i.e.~$\zeta$ is a reparametrization of $s(b)$), we have that
$$\tilde\lambda(\tilde\alpha'(g))=b\circ h_1\circ \gamma_{u_i',u'}\otimes e_{u'u'}+\zeta\tensor (I-e_{u'u'})$$
is a projection in the unitization of $\bigoplus_{i=1}^{\sF} C_0(\open{u_i}\,\!\!')\tensor K(\ell^2([u'_i]_{R_1}))$, ($u'_i$ are the representatives of the $R_1$-equivalence classes, i.e., the stable faces of $T_1$)  and its scalar part is
$$s(\tilde\lambda(\tilde\alpha'(g)))=\zeta\tensor I.$$
The map $\gamma_{u_i',u'}:u_i'\to u'$ is the translation map given by $\gamma_{u_i',u'}(x):=x+x_0$, with $x_0$ being the translation of the two faces $u'=u_i'+x_0$.
Note that the constant function $\zeta$ has different faces as domain depending on the context.
That is
$$
\zeta\tensor I=\sum_{i=1}^{\sF} \zeta_{u'_i}\tensor (\sum_{v'\sim u'_i}e_{v'v'}),\qquad
\zeta\tensor (I-e_{u'u'})=\sum_{i=1}^{\sF} \zeta_{u'_i}\tensor (\sum_{\substack{v'\sim u'_i\\ v'\ne u'}}e_{v'v'}),
$$
where the inner sum is over all faces $v'\in T_1$ equivalent to $u_i'$, and $\zeta_{u'_i}$ is the constant function $c$ with domain $u'_i$. 
By definition of $K_0$ on $*$-homomorphisms and Proposition 4.2.2 p.63 in \cite{Rordam}, we get  
$$[\tilde\lambda(\tilde\alpha'(g))]_0-[s(\tilde\lambda(\tilde\alpha'(g)))]_0=[b\circ h_1\circ \gamma_{u_i',u'}\otimes e_{u'u'}+\zeta\tensor (I-e_{u'u'})]_0-[\zeta\tensor I)]_0.$$
Moreover, the scalar part map $s$ and $\tilde\lambda$ commute.
Since all faces of $T_1$ and $T_0$ are oriented counter clockwise, the map $h_0|_u = id_u$ is orientation preserving and since $h$ is a homotopy, also $h_1: u' \to u$ is orientation preserving. Hence (up to orientation preserving reparametrization) we can replace $b \circ h_1\circ\gamma_{u_i',u'}$ with $b$ and $\zeta$ with $s(b)$ and we get
$$[\tilde\lambda(\tilde\alpha'(g))]_0-[s(\tilde\lambda(\tilde\alpha'(g)))]_0=[b\otimes e_{u'u'}+s(b)\tensor (I-e_{u'u'})]_0-[s(b)\tensor I]_0.$$

\end{proof}

The following lemma will be given up to the isomorphism given in Proposition \ref{p:compacts}, which we denote by $\lambda$.
Since the proofs of the next two lemmas are similar to the previous two, we omit some technical details.
\begin{lem}The generators of $K_1(J_0)$ are 
\begin{eqnarray*}
[ b\otimes e_{uu} +\sum_{v\ne u} 1\otimes e_{vv}]_1\
\end{eqnarray*}
where $u$ is an edge in $T$, the sum is over all vertices of $v$ in $T$, and $b$ is the map $s\mapsto e^{i2\pi s}, s\in[0,1]$ (up to reparametrization from domain $u$ to domain $[0,1]$). Choose edge $u\in T$ from each $R_0$-equivalence class to get all generators.
\end{lem}
\begin{proof}
Let $\T:=\{z\in \C\mid |z|=1\}$ be the unit circle, and $u\in T$ be an edge.
Recall that $\Z$ is isomorphic to $K_1(C(\T))$ mapping $1\mapsto[ b']_1$, where $b'$ is the identity map $z\mapsto z$, $z\in \T$.
Moreover, $K_1(C(\T))$ is isomorphic to $K_1(C_0(\open{u}))$ mapping $[b']_1\mapsto [b]_1$. Note that $b$ is in the unitization of $C_0(\open{u})$.
Recall that $C_0((0,1),K)$ is isomorphic to  $C_0((0,1))\tensor K$ via the map 
$$f\tensor a\mapsto (t\mapsto f(t) a).$$
Recall that $K_1(C_0(\open{u}))$ is isomorphic to $K_1(C_0(\open{u})\tensor K)$, where $K$ is the $C^*$-algebra of compact operators on the Hilbert space $\ell^2([u]_{R_0})$.
Moreover, $C_0(\open{u})\tensor K$ embeds in $J_0$ via the isomorphism given in Proposition \ref{p:compacts}(2),
where for simplicity, we assume that $u$ is one of the representatives $e_i$ in the proposition.
Let $e_{\tilde uv}\in K(\ell^2([u]))$  be the standard matrix unit
$$e_{\tilde uv}(\delta_w)= \left\{\begin{array}{cc}
    \delta_{\tilde u} & w=v \\ 
    0 & else \\ 
  \end{array}\right.,
$$
where $\tilde u,v,w\in T$ are edges $R_0$-equivalent to edge $u$.
Thus we have the following identifications 
$$1\mapsto [b']_1\mapsto [b]_1\mapsto[b\tensor e_{uu}+1\tensor(I-e_{uu})]_1,$$
where the third map follows by the same argument as for the faces, the $1$ in the expression $1\tensor(I-e_{uu})$ denotes the constant function $1$, and 
$$
I:=\sum_{i=1}^{\sE}I_{B(H_i)}=\sum_{v} e_{vv},\qquad
1\tensor I:=\sum_{i=1}^{\sE}1_{e_i}\tensor \left(\sum_{v\sim e_i} e_{vv}\right),
$$ 
where the sum is over all edges $v\in T$, and $1_{e_i}\in \widetilde{C_0(\open{e_i})}$ is the unit element -- the constant function $1$.
Note that 
$b\tensor e_{uu}+1\tensor(I-e_{uu})$ is a unitary in the unitization of $\bigoplus_{i=1}^{\sE}C_0(\open{e_i})\tensor K(\ell^2([e_i]))$. 
\end{proof}

\begin{lem}\label{l:We} With notation as in the previous lemma, $K_1(\beta'):K_1(J_0)\to K_1(J_1)$ evaluated on the generators of $K_1(J_0)$ is
\begin{eqnarray*}
K_1(\beta')\Big([ b\otimes e_{uu} +\sum_{v\ne u} 1\otimes e_{vv}]_1\Big)=
\end{eqnarray*}
$$\sum_{h_1(u')=u}[b\otimes e_{u'u'}+\sum_{\substack{v'\ne u' }} 1\otimes e_{v'v'}]_1\,-\,\sum_{h_1(u')=\breve{u}}[ b\otimes e_{u'u'} +\sum_{\substack{v'\ne u'}} 1\otimes e_{v'v'}]_1.$$
where the sums are over all edges $u',v'$ in $T_1$.
Here $\breve u$ is the edge $u$ with reversed orientation, and for this lemma equality $h_1(u') = u$ is understood to include that $h_1: u' \to u$ is orientation preserving.

We remark that by Whitehead's lemma
$$[b\tensor e_{u'u'}+\sum_{w' \ne u'} 1\otimes e_{w'w'}]_1=[b\tensor e_{v'v'}+\sum_{w' \ne v'} 1\otimes e_{w'w'}]_1\qquad \text{for } u'\sim_{R_1} v'.$$

\end{lem}
\begin{proof}
Take edge $u\in T$. Then by definition of $K_1$ on $*$-homomorphisms we get 
$$K_1(\lambda\beta'\lambda^{-1})\Big([ b\otimes e_{uu} +1\otimes(I-e_{uu})]_1\Big)=[\tilde\lambda\tilde\beta'\tilde\lambda^{-1}( b\otimes e_{uu} +1\otimes(I-e_{uu})) ]_1,$$
where $1$ is the constant function 1.

Let $g$ be the unitary in the unitization of $J_0$ given by 
$$g:=\tilde\lambda^{-1}\big(b\tensor e_{uu}+1\tensor(I-e_{uu})\big),$$
 where $\tilde\lambda$ is the isomorphism in Proposition \ref{p:compacts} extended to the unitization. We evaluate $g$ and get
\begin{equation}\label{e:edges-g}
g(x,y):=\left\{  
\begin{array}{cc}
    b(x) & x=y\in u \\ 
    1 & x=y\not\in u \\ 
    0 & else. \\ 
  \end{array}\right.
\end{equation}

Since
$$\tilde \beta'(g)(x,y)=\left\{
  \begin{array}{cc}
    b(h_1(x)) & x=y\in u', h_1(u')=u \\ 
    b(h_1(x)) & x=y\in u', h_1(u')=\breve u \\ 
    1 &  x=y\in u',\breve u\ne h_1(u')\ne u\\ 
    0 & x\ne y \\ 
  \end{array}\right.
$$
we get
$$\tilde \lambda\tilde \beta'(g)=
 \sum_{h_1(u')=u} b\otimes e_{u'u'}+\sum_{h_1(u')=\breve{u}} b^*\otimes e_{u'u'} +\sum_{\substack{h_1(v')\ne u \\ h_1(v')\ne \breve u}} 1\otimes e_{v'v'},
$$
where the sums is over all edges $u',v'\in T_1=\frac{1}{\lambda}\omega(T)\lambda$ (and $b^*$ is the function $e^{-2\pi i s}$, $s\in[0,1]$, up to reparametrization).
Thus
\begin{eqnarray*}
[\tilde\lambda\tilde\beta'\tilde\lambda^{-1}(b\otimes e_{uu} +\sum_{v\ne u} 1\otimes e_{vv}) ]_1=
\end{eqnarray*}
$$[ \sum_{h_1(u')=u} b\otimes e_{u'u'}+\sum_{h_1(u')=\breve{u}} b^*\otimes e_{u'u'} +\sum_{\substack{h_1(v')\ne u \\ h_1(v')\ne \breve u}} 1\otimes e_{v'v'}]_1=$$
$$[ \sum_{h_1(u')=u} b\otimes e_{u'u'}+\sum_{\substack{h_1(v')\ne u }} 1\otimes e_{v'v'}]_1\,-\,[\sum_{h_1(u')=\breve{u}} b\otimes e_{u'u'} +\sum_{\substack{h_1(v')\ne \breve u}} 1\otimes e_{v'v'}]_1=$$
$$\sum_{h_1(u')=u}[ b\otimes e_{u'u'}+\sum_{\substack{v'\ne u' }} 1\otimes e_{v'v'}]_1\,-\,\sum_{h_1(u')=\breve{u}}[ b\otimes e_{u'u'} +\sum_{\substack{v'\ne u'}} 1\otimes e_{v'v'}]_1,$$
where the equalities hold by \cite[Prop.~8.1.4(iv),~p.135]{Rordam}.

Note that by Whitehead's lemma
$$[b\tensor e_{u'u'}+c_2]_1=[b\tensor e_{v'v'}+c_3]_1,$$ 
whenever $u'\sim_{R_1} v'$.
This equality is equivalent to $[b\tensor e_{u'u'}+ b^*\tensor e_{v'v'}+c_1]_1=0$  by \cite[Prop.~8.1.4($iv$), p.135]{Rordam}, which can easily be verified:
\begin{eqnarray*}
[b\tensor e_{u'u'}+b^*\tensor e_{v'v'}+c_1]_1&=&[bb^*\otimes e_{u'u'}+c_2]_1=[(1\otimes e_{u'u'}+c_2)]_1=[1]_1=0,
\end{eqnarray*}
where the first equality is by Whitehead's lemma, and
$$
c_1:=\sum_{\substack{w'\ne u' \\ w'\ne v'}} 1\otimes e_{w'w'}\qquad
c_2:=\sum_{w' \ne u'} 1\otimes e_{w'w'}\qquad
c_3:=\sum_{w' \ne v'} 1\otimes e_{w'w'}.
$$

\end{proof}

\subsection{Stable cohomology and $K$-theory}
In this subsection we introduce the stable cohomology groups, from which we compute the $K$-theory of the $C^*$-algebra $S'$.
We start with the following theorem, which is one of our main results. It shows that 
\begin{equation}\label{e:CSbullet}
C_S^\bullet:=\,\,\,\,\xymatrix{0\ar[r]&\Z^{sV}\ar[r]^{\delta^0}&\Z^{sE}\ar[r]^{\delta^1}&\Z^{sF}\ar[r]&0}
\end{equation}
is a cochain complex, and the collection of connecting maps
\begin{equation}\label{e:WvWeWf} 
W_V:=K_0(\gamma),\qquad W_E:=K_1(\beta'),\qquad W_F:=K_0(\alpha')
\end{equation}
form a cochain map, which we denote by $W^\bullet: C_S^\bullet\to C_S^\bullet$, and whose  explicit computations are given in Lemma \ref{l:Wv}, Lemma \ref{l:We}, Lemma \ref{l:Wf}.
Moreover, it relates the group homomorphisms $K_0(\iota)$, $K_1(\iota)$ with the connecting maps.
\begin{thm} \label{t:W-VEF}
With the above notation, the rows of the following commutative diagram are cochain complexes of abelian groups
\begin{equation}
\xymatrix{0\ar[r]&\Z^{sV}\ar[r]^{\delta^0}\ar[d]_{W_V}&\Z^{sE}\ar[r]^{\delta^1}\ar[d]_{W_E}&\Z^{sF}\ar[r]\ar[d]_{W_F}&0\\
  0\ar[r]&\Z^{sV}\ar[r]^{\delta^0}&\Z^{sE}\ar[r]^{\delta^1}&\Z^{sF}\ar[r]&0.}
\end{equation}

Moreover, the group homomorphism $K_1(\iota):K_1(A_0)\to K_1(A_1)$ is given by 
\begin{equation}
K_1(\iota)=H^1(W_E), 
\end{equation}
where $H^1(W_E):{\frac{\ker\delta^1}{\im\delta^0}}\to{\frac{\ker\delta^1}{\im\delta^0}}$ is given by $H^1(W_E)([z]):=[W_E(z)]$, $z\in \ker\delta^1$.

The group homomorphism $K_0(\iota):K_0(A_0)\to K_0(A_1)$ fits into the commutative diagram 
\begin{equation}
\xymatrix{
0\ar[r]&H^2(C_S^\bullet)\ar@{^{(}->}[r]\ar[d]_{H^2(W_F)}&K_0(A_0)\ar[d]_{K_0(\iota)}\ar@{->>}[r]&H^0(C_S^\bullet)\ar[d]^{H^0(W_V)}\ar[r]&0\\
0\ar[r]&H^2(C_S^\bullet)\ar@{^{(}->}[r]&K_0(A_{1})\ar@{->>}[r]&H^0(C_S^\bullet)\ar[r]&0,
}
\end{equation}
where $H^0(C_S^\bullet)=\ker\delta^0$,\, $H^0(W_V)=W_V|_{\ker\delta^0}$.
For tilings of dimension 1, $H^2(C_S^\bullet)=0$. For tilings of dimension 2, $H^2(C_S^\bullet)=\Z$ and $H^2(W_F)=id$.

\end{thm}
\begin{proof}
The first diagram is simply restating the one in Lemma \ref{l:K-maindiagram-1} but with explicit formulas for the connecting maps, which were computed in the lemmas \ref{l:Wv}, \ref{l:We}, \ref{l:Wf}.
We now show the second statement of the theorem, namely that $K_1(\iota)=H^1(W_E)$. 
From the six-term exact sequence in Eq.~(\ref{e:K-IAB-2}) we get the following diagram with exact rows
\begin{equation*}
\xymatrix{
0\ar[r]&K_1(A_0)\ar[d]_{K_1(\alpha)}\ar@{^{(}->}[r]&K_1(B_0)\ar[d]^{K_1(\beta)}\ar[r]^{\tilde\delta^1}&K_0(I_0)\ar[d]^{K_0(\alpha')}\\
0\ar[r]&K_1(A_{1})\ar@{^{(}->}[r]&K_1(B_1)\ar[r]^{\tilde\delta^1}&K_0(I_1)
}
\end{equation*}
as $K_1(I_n)=0$.
The diagram commutes because of naturality of $\tilde\delta^1$, because the diagram in Eq.~(\ref{e:IAB-connectingmaps}) commutes, and because $K_j(\phi\circ\psi)=K_j(\phi)\circ K_j(\psi)$, $j=1,2$. (cf.~\cite[Prop.~9.1.5, p.~157]{Rordam}).
Hence by this diagram 
\begin{eqnarray*}
K_1(\iota)&=&K_1(\alpha)=K_1(\beta)|_{\ker\tilde\delta^1},
\end{eqnarray*}
where the first equality follows from the invariance under homotopy of $K$-theory, and the second equality is up to the isomorphism $K_1(A_n)\cong \ker\tilde\delta^1$.
By the diagram in Eq.~(\ref{e:delta0-delta1-deltat1})  we get
\begin{eqnarray*}
K_1(\beta)|_{\ker\tilde\delta^1}&=&[K_1(\beta')|_{\ker\delta^1}]_{\im\delta^0}\\
&=&[W_{E}|_{\ker\delta^1}]_{\im\delta^0}\\
&=&H^1(W_E),
\end{eqnarray*}
where the first equality is up to the isomorphism $\ker\tilde\delta^1\cong\frac{\ker\delta^1}{\im\delta^0}$, and $[W_{E}|_{\ker\delta^1}]_{\im\delta^0}$ denotes the map $[x]_{_{\im\delta^0}}\mapsto [W_{_E}(x)]_{_{\im\delta^0}}$ for $x\in \ker\delta^1$.
  
We now show the last statement of the theorem.
By Eq.~(\ref{e:K-IAB-2}) and Eq.~(\ref{e:K0A0}) and naturality of the index map (and hence naturality of $\delta^1$) we get the following commutative diagram with exact rows
\begin{equation*}
\xymatrix{
0\ar[r]&\frac{K_0(I_0)}{\im\delta^1}\ar@{^{(}->}[r]\ar[d]^{[K_0(\alpha')]_{_{\im\delta^1}}}&K_0(A_0)\ar[d]^{K_0(\iota)}\ar@{->>}[r]& K_0(B_0)\ar[d]^{K_0(\beta)}\ar[r]&0\\
0\ar[r]&\frac{K_0(I_1)}{\im\delta^1}\ar@{^{(}->}[r]&K_0(A_{1})\ar@{->>}[r]&K_0(B_1)\ar[r]&0,
}
\end{equation*}
and note that 
$$[K_0(\alpha')]_{_{\im\delta^1}}=[W_F]_{_{\im\delta^1}}=H^2(W_F).$$
By the six-term exact sequence in Eq.~(\ref{e:K-JBC}) and naturality of the exponential map we get the following commutative diagram with exact rows
\begin{equation*}
\xymatrix{
0\ar[r]&K_0(B_0)\ar[d]_{K_0(\beta)}\ar@{^{(}->}[r]&K_0(C_0)\ar[d]^{K_0(\gamma)}\ar[r]^{\delta^0}&K_1(J_0)\ar[d]^{K_1(\beta')}\\
0\ar[r]&K_0(B_{1})\ar@{^{(}->}[r]&K_0(C_1)\ar[r]^{\delta^0}&K_1(I_1)
}
\end{equation*}
Thus 
$$K_0(\beta)=K_0(\gamma)|_{\ker\delta^0}=W_V|_{\ker\delta^0}=H^0(W_V),$$
where the first equality is up to the isomorphism $K_0(B_n)\cong \ker\delta^0$.

\noindent For tilings of dimension 2, it holds, by Lemma \ref{l:plane-ZsFoverImDelta1}, that
\begin{eqnarray*}
&&\frac{K_0(I_n)}{\im\tilde\delta^1}\cong\frac{\Z^{\sF}}{\im\tilde\delta^1}=\frac{\Z^{\sF}}{\im\delta^1}\cong \Z\\
&&[W_F]_{_{\im\delta^1}}= 1.
\end{eqnarray*}
Note that in the last isomorphism, any face $f$ yields a generator of $\Z$.
Now $W_F([f])=[f']$ in the notation of the definition of homotopy (cf.~Definition \ref{d:homotopy}), corresponds to the identity map $\Z\to\Z$, which we denote by $1$. 
\end{proof}
%

\begin{defn}[Stable(S)]\label{d:stablecohomology}
For $k\in\{0,1,2\}$, we define the stable cohomology groups for the fixed tiling $T$ as 
\begin{eqnarray*}
 H_S^{k}&:=&\lim_{\to}(H^{k}(C_S^\bullet),\,\,H^{k}(W^\bullet)),
\end{eqnarray*}
where the chain complex $C_S^\bullet$ and chain map $W^\bullet$ were defined right above Theorem \ref{t:W-VEF}.
Here the notation is
$$\xymatrix{
&\lim\limits_{\to}(X,A):=&\!\!\!\!\!\!\!\!\!\!\!\!\!\!\!\!\lim\limits_{\to}X\ar[r]^{A}&X\ar[r]^{A}&X\ar[r]^{A}&.\\
}$$
\end{defn}

The $K$-theory groups $K_0(S')$, $K_1(S')$ for the $C^*$-algebra $S'$ can be expressed in terms of the stable cohomology groups $H^0_S$, $H^1_S$, $H^2_S$ as is shown in the following theorem.
\begin{thm}\label{t:Ktheory-tung}
For tilings of dimension 1 or 2, the group $K_1(S')$ is given by
$$K_1(S')=H^1_S,$$
and the group $K_0(S')$ fits into the short exact sequence
$$\xymatrix{0\ar[r]&H_S^2\ar@{^{(}->}[r]&K_0(S')\ar@{->>}[r]&H_S^0\ar[r]&0.}$$
Moreover, in dimension 1 it holds, $H^2_S=0$, $H^1_S=\Z$, and in dimension 2 it holds
$H^2_S=\Z.$
\end{thm}
\begin{proof}
By Proposition \ref{ap:Sisdirlim} we have that
$$S'=\lim_{\to} \xymatrix{A_0\ar[r]^{\iota_0}&A_1\ar[r]^{\iota_1}&\ldots&}.$$
By \cite[Theorem~6.3.2]{Rordam} 
\begin{eqnarray*}
K_j(S')&=&\lim_{\to} \xymatrix{K_j(A_0)\ar[r]^{K_j(\iota_0)}&K_j(A_1)\ar[r]^{K_j(\iota_1)}&\ldots&}\quad j=0,1.
\end{eqnarray*}
Then, by Theorem \ref{t:W-VEF}, 
we get the statements of the theorem, where for the second statement, we use the fact that the direct limit of a directed system of short exact sequences of abelian groups is a short exact sequence because the direct limit is an exact functor in the category of abelian groups.
\end{proof}
Unlike in the unstable case (cf.~Theorem \ref{t:unstableKtheory}), we do not know whether the short exact sequence in Theorem \ref{t:Ktheory-tung} splits, since $H^0_S$ is not necessarily projective, i.e.~ isomorphic to $\Z^\ell$ for some $\ell\in\N_0$.


\subsection{Properties.}
\begin{pro}
The short exact sequence in Eq.~(\ref{e:JBC}) does not split.
\end{pro}
\begin{proof}
Suppose for contradiction that the short exact sequence splits, that is, there exists a section $\sigma:C_n\to B_n$.
Since $\sigma$ is in particular a $*$-homomorphism, $\tilde\sigma(p)$ is a projection for any projection $p\in M_\ell(\tilde C_n)$, $\ell\in\N$.
Then by the continuous functional calculus, $\mathrm{exp}(2\pi i \tilde\sigma(p))=I$. By \cite[Prop.~12.2.2($i$)]{Rordam}, $\delta^0([p]_0-[s(p)]_0)=[I]_1=0$, i.e.~the exponential map $\delta^0$ is the zero map.
This, together with Eq.~(\ref{e:delta0}), implies that there is only 1 prototile, a contradiction since the tiling $T$ is aperiodic.
\end{proof}

\begin{pro}
For tilings of dimension 2, the short exact sequence in Eq.~(\ref{e:IAB}) does not split.
\end{pro}
\begin{proof}
Suppose for contradiction that the short exact sequence splits, that is, there exists a section $\sigma:B_n\to A_n$.
Since $\sigma$ is in particular a $*$-homomorphism, $v:=\tilde\sigma(u)$ is a unitary for any unitary $u\in M_\ell(\tilde B_n)$, $\ell\in\N$.
Thus $p:=1-v^*v=0$ and $q:=1-vv^*=0$. By \cite[Prop.~9.2.2]{Rordam}, $\tilde\delta^1([u]_1)=[p]_0-[q]_0=0$, i.e.~the exponential map $\tilde\delta^1$ is the zero map, and hence also $\delta^1$ is the zero map.
This together with Eq.~(\ref{e:delta1}) implies that there is only 1 prototile, a contradiction since the tiling $T$ is aperiodic.
\end{proof}

\begin{pro}
  The $C^*$-algebra $S'$ is not unital. (But this does not rule out that $S$ might be unital).
\end{pro}
\begin{proof}
It is well-known that the $C^*$-algebra of compact operators $K$ is never unital on an infinite dimensional Hilbert space $H$, since the identity $I\in B(H)$ is not a compact operator.
The quotient of a unital $C^*$-algebra is unital because the class [I] is the identity in the quotient. Since, by Eq.~(\ref{e:JBC}), the compacts $K$ is a quotient of $B_n=C_r^*(R_n|_{X_1})$, $B_n$ is not unital. Similarly, by Eq.~(\ref{e:IAB}), $A_n=C_r^*(R_n)$ is non-unital.
Hence $S'$ must be non-unital, since all the algebras forming the limit are non unital.
Cf.~\cite[Exc.6.7(iii)]{Rordam} which uses that the inclusion map $A_n\to A_{n+1}$ is injective.
This is because an element of a unital $C^*$-algebra with distance less than 1 to the identity is invertible, so if the limit algebra is unital then the individual algebras must be unital from a certain step.
  Tensoring with the compact operators will turn a unital $C^*$-algebra into an non-unital  $C^*$-algebra Morita equivalent to the first. So $S'$ non unital does not rule out the possibility that $S$ might be unital. 
\end{proof}
\begin{pro}Let $n\in\N$. Then,
the ideal $J_n$ of $B_n$ (d=1) in Definition~\ref{d:ABCIJ} does not contain any projections other than the zero projection.
Hence $B_n$ is not AF. (The same holds for d=2). This does not prove though that $S'$ is not AF. However, we can decide this from $K_1(S')$.
\end{pro}
\begin{proof}
If $p\in C_0((0,1), K(\ell^2([\mathrm{edge}])))$ is a projection, then $p:(0,1)\to K(\ell^2([\mathrm{edge}]))$ is continuous vanishing at the endpoints. Since the map  $p(t)\mapsto||p(t)||$ is continuous and since $p(t)$ is a projection, its norm is $0$ or $1$. Since $(0,1)$ is connected, and since the image of connected sets under continuous maps is connected, the map  $p(t)\mapsto||p(t)||$ is either $0$ or $1$. Since it vanishes at the endpoints, it has to be zero. In particular $J_n$ is not AF.
Any ideal in an AF algebra is AF. Hence $B_n$ is not AF.
(Recall that an AF-algebra is the norm closed linear span of its projections. This is because it is an inductive limit of finite dimensional $C^*$-algebras (which are direct sums of matrix algebras), and every finite dimensional $C^*$-algebra is the linear span of its projections.)
\end{proof}



\section{\textbf{Stable-Unstable relationship}}\label{s:S-U-relationship}
In this section we relate the K-theory of $U$ to the $K$-theory of $S$.
On the one hand, the $K$-theory of $S$, as shown in Section \ref{s:S}, is given in terms of the stable cohomology via a skeletal decomposition.
On the other hand, the $K$-theory of $U$ is well-known to be given in terms of the \v{C}ech cohomology also via a skeletal decomposition.
We will relate the so-called stable-transpose homology to \v{C}ech cohomology via PE-cohomology and PE-homology.
In particular, the stable-transpose homology is a simpler method for computing the $K$-theory of $U$.

Recall that $\Omega$ denotes the continuous hull, $\omega:\Omega\to\Omega$ denotes the substitution map, which is a homeomorphism, and $\lambda>1$ denotes the inflation factor.
Define the tiling space 
$$\Omega_n:=\lambda^{-n}\Omega\lambda^{n},\qquad n\in\Z,$$
with inflation factor $\lambda$ and substitution map $\omega_n:\Omega_n\to\Omega_n$ given by
$$\omega_n:=\Ad(\lambda^{-n})\omega \Ad(\lambda^n)\qquad\qquad n\in\Z,$$
where
$\Ad(\lambda):\Omega\to\lambda\Omega\lambda^{-1}$ is defined by
$$\Ad(\lambda)(T):=\lambda T \lambda^{-1}.$$
Fix a substitutional tiling $T \in\Omega$. We can then construct the sequence of tilings 
$$T_n:=\lambda^{-n}\omega^n(T)\lambda^{n}\in\Omega_n,\qquad\qquad n\in\Z,$$
whose tiles shrink as $n$ increases.
Furthermore we can construct the following commutative diagram
\begin{equation*}
\cdots
\!\!\!\!\!\!\!\!\!\!\!\!\!\!\!
\xymatrix@C=1.2cm{
\ar@{|->}@{|->}@{}[r]^(.4){}="a"^(.8){}="b" \ar "a";"b"
&T_{-2}\ar@{|->}[r]^{\color{blue}\tilde\omega_{-2}}\ar@{|->}@/^-1.0pc/@[black][drr]_{\color{blue}\theta_{2}}
&T_{-1}\ar@{|->}[r]^{\color{blue}\quad \tilde\omega_{-1}}\ar@{|->}[dr]_{\color{blue}\theta_{1}}
&T_0\ar@{|->}[r]^{\color{blue}\!\!\!\!\!\tilde\omega_{0}}\ar@{|->}[d]^{\color{blue}id}
&T_1\ar@{|->}[r]^{\color{blue}\tilde\omega_{1}}\ar@{|->}[dl]^{\color{blue}\theta_{-1}}
&T_2\ar@{|->}@/^1.0pc/@[black][dll]^{\color{blue}\theta_{-2}}
\ar@{|->}[r]&\\
&&&T_0&&&
}
\cdots
\end{equation*}
where $T_0:=T$,\,\,\,$\theta_0:=id$, and
\begin{eqnarray}
\label{e:tildeomegan}&&\tilde\omega_n:=\Ad(\frac{1}{\lambda})\omega_{n},\qquad\qquad\qquad n\in\Z.\\
\nonumber&&\theta_n:=\omega^{-n} \Ad(\lambda^n),\qquad\qquad\qquad n\in\Z.
\end{eqnarray}
If $p$ is a prototile of $T_0$ then $\lambda^{-n} p$ is a prototile of tiling $T_n$, $n\in\Z$. 
Moreover, the prototile $\lambda^{n} p$ coincides with $\omega^{n}(p)$ as sets.
\begin{exam} Consider the Fibonacci tiling with protoedges $a,b$ of lengths $|a|=1$, $|b|=1/\phi$, ($\lambda=\phi$=golden ratio) and substitution rule $\omega(a)=ab$, $\omega(b)=a$.
The tilings $T_1$, $T_0$, $T_{-1}$, $T_{-2}$ are shown in the following picture.

\begin{tikzpicture}
\draw[dotted](0,-2)--(0,1);
\draw[dotted](.618,0)--(.618,1);
\draw[dotted](1,-1)--(1,1);
\draw[dotted](1.618,-2)--(1.618,1);
\draw[dotted](2.236,0)--(2.236,1);
\draw[dotted](2.618,-2)--(2.618,1);

\draw[dotted](4,-2)--(4,1);
\draw[dotted](4.618,0)--(4.618,1);
\draw[dotted](5,-1)--(5,1);
\draw[dotted](5.618,-2)--(5.618,1);
\draw[dotted](6.236,0)--(6.236,1);
\draw[dotted](6.618,-2)--(6.618,1);
\draw[dotted](6.618,-2)--(6.618,1);
\draw[black,thick] (0,1) -- (.618,1);
\node at (.3,.7) {$\hat a$};
\draw[red,thin] (.618,1) -- (1,1);
\node at (.8,.7) {$\hat b$};
\draw[black,thin] (1,1) -- (1.618,1);
\node at (1.3,.7) {$\hat a$};
\draw[black,thick] (1.618,1) -- (2.236,1);
\node at (1.8,.7) {$\hat a$};
\draw[red,thin] (2.236,1) -- (2.618,1);
\node at (2.4,.7) {$\hat b$};
\draw[black,thick] (4,1) -- (4.618,1);
\node at (4.3,.7) {$\frac{a}{\lambda}$};
\draw[red,thin] (4.618,1) -- (5,1);
\node at (4.8,.7) {$\frac{b}{\lambda}$};
\draw[black,thin] (5,1) -- (5.618,1);
\node at (5.3,.7) {$\frac{a}{\lambda}$};
\draw[black,thick] (5.618,1) -- (6.236,1);
\node at (5.8,.7) {$\frac{a}{\lambda}$};
\draw[red,thin] (6.236,1) -- (6.618,1);
\node at (6.4,.7) {$\frac{b}{\lambda}$};
\draw[black,thick] (0,0) -- (1,0);
\node at (.5,-.3) {$a$};
\draw[red,thin] (1,0) -- (1.618,0);
\node at (1.3,-.3) {$b$};
\draw[black,thick] (1.618,0) -- (2.618,0);
\node at (2,-.3) {$a$};
\draw[black,thick] (4,0) -- (5,0);
\node at (4.5,-.3) {$a$};
\draw[red,thin] (5,0) -- (5.618,0);
\node at (5.3,-.3) {$b$};
\draw[black,thick] (5.618,0) -- (6.618,0);
\node at (6,-.3) {$a$};
\node[align=left] at (8.5,1) {$T_{1}:=\frac1{\lambda}\omega(T)\lambda$};
\node[align=left] at (8.05,0) {$T_0:=T$};
\node[align=left] at (8.8,-1) {$T_{-1}:=\lambda\omega^{-1}(T)\frac{1}{\lambda}$};
\node[align=left] at (9,-2) {$T_{-2}:=\lambda^2\omega^{-2}(T)\frac{1}{\lambda^{2}}$};
\draw[black,thick] (0,-1) -- (1.618,-1);
\node at (.5,-1.3) {$a'$};
\draw[red,thin] (1.618,-1) -- (2.618,-1);
\node at (2,-1.3) {$b'$};
\draw[black,thick] (4,-1) -- (5.618,-1);
\node at (4.5,-1.3) {$\lambda a$};
\draw[red,thin] (5.618,-1) -- (6.618,-1);
\node at (6,-1.3) {$\lambda b$};
\draw[black,thick] (0,-2) -- (2.618,-2);
\node at (1.5,-2.3) {$a''$};
\draw[black,thick] (4,-2) -- (6.618,-2);
\node at (5.5,-2.3) {$\lambda^2 a$};
\end{tikzpicture}\\
Note that the edge $a'=\lambda a\in T_{-1}$ coincides with the patch $\omega(a)=ab\subset T_0$ as sets, i.e.
$a'$ is ``tiled" with $\omega(a)$. 
\end{exam}
We then define the equivalence relation $R_n$ (for $T_n$) on $\R^d$ as
\begin{eqnarray*}
R_n:=R(T_n)=\frac{1}{\lambda^n}R(\omega^n(T)),\qquad n\in\Z,
\end{eqnarray*}
where
$$R(T_n):=\{(x,y)\in \R^{2d}\mid T_n(x)-x=T_n(y)-y\},$$
and recall that $T_n(x)$ is the patch made of all the tiles in tiling $T_n$ that contain $x$.
We equip these equivalence relations with the subspace topology of $\R^{2d}$. Then
$$\cdots\subset R_{-2}\subset R_{-1}\subset R_{0}\subset R_{1}\subset R_{2}\subset\cdots$$
and $R_n$ is open in $R_{n+1}$, $n\in\Z$ (cf. Subsection \ref{ss:Rs'}).
We have the directed system of chain complexes indexed by $\Z$
$$
\cdots\xymatrix{\ar[r]&C_S^\bullet(R_{-1})\ar[r]^{W_{-1}^\bullet}&C_S^\bullet(R_0)\ar[r]^{W_0^\bullet}&C_S^\bullet(R_1)\ar[r]^{W_1^\bullet}&C_S^\bullet(R_2)\ar[r]^{\qquad W_2^\bullet}&}\cdots,
$$
where $C_S^\bullet(R_{0})$ and $W_0^\bullet$ are defined in the diagram of Lemma \ref{l:K-maindiagram-1} in terms of the compacts and of the stable cells of $T_0$ and simplified in Eq.~(\ref{e:CSbullet}).
The remaining $C_S^\bullet(R_{n})$ and $W_n^\bullet$ are defined similarly in terms of the compacts and of the stable cells of $T_n$.
They are independent of $n$ because $T_n$ has the same stable cells as $T_0$ up to shrinking.
Applying the contravariant $\Hom(_{\text{---}},\Z)$ functor we get
$$\cdots
\xymatrix{\ar[r]&C^{ST}_\bullet(R_{1})\ar[r]^{{W_{0}^\bullet}^t}&C^{ST}_\bullet(R_0)\ar[r]^{{W_{-1}^\bullet}^t}&C^{ST}_\bullet(R_{-1})\ar[r]^{{W_{-2}^\bullet}^t}&C^{ST}_\bullet(R_{-2})\ar[r]^{\qquad{W_{-3}^\bullet}^t}&}\cdots,
$$
where $C^{ST}_\bullet(R_n):=C_S^\bullet(R_n)^t$.  Note that we used the transpose instead of $\Hom(_{\text{---}},\Z)$ since the groups we take the transpose of are finitely generated free abelian groups with a fixed basis (the stable cells of $T_n$).
Recall that the category of abelian groups is cocomplete, i.e.~all small colimits (in particular all direct limits) exist. Hence, we can define the abelian groups
$$
H^k_S:=\lim_{\to}(\cdots\xymatrix@C=1.2cm{\ar[r]^{H^k(W_{-1}^\bullet)\qquad\quad}&H^k(C_S^\bullet(R_0))\ar[r]^{H^k(W_0^\bullet)}&H^k(C_S^\bullet(R_1))\ar[r]^{\qquad\quad H^k(W_1^\bullet)}&}\cdots),
$$
$$
H_k^{ST}:=\lim_{\to}(\cdots\xymatrix@C=1.2cm{\ar[r]^{H_k({W_{0}^\bullet}^t)\qquad\quad}&H_k(C^{ST}_\bullet(R_0))\ar[r]^{H_k({W_{-1}^\bullet}^t)}&H_k(C^{ST}_\bullet(R_{-1}))\ar[r]^{\qquad\quad H^k({W_{-2}^\bullet}^t)}&}\cdots).
$$
Since all chain complexes are independent of $n$, it makes sense to define them more succinctly as follows
\begin{defn}[Stable-Transpose(ST), Stable(S)]\label{d:ST-S-PE}
For $k\in\{0,1,2\}$, define the homology and cohomology groups
\begin{eqnarray*}
 H^{ST}_{k}&:=&\lim_{\to}(H_{k}({C^{ST}_\bullet}),\,\, H_k({W^\bullet}^{t}))\\
 H_S^{k}&:=&\lim_{\to}(H^{k}(C_S^\bullet),\,\,H^{k}(W^\bullet)),
\end{eqnarray*}
where we repeat the definition of $H_S^{k}$ from Definition \ref{d:stablecohomology}. 
When we want to emphasize the tiling $T$ we will write $H^{ST}_{k}(T)$ and $H_S^{k}(T)$ instead.
Here the notation is
$$\xymatrix{
&\lim\limits_{\to}(X,A):=&\!\!\!\!\!\!\!\!\!\!\!\!\!\!\!\!\lim\limits_{\to}X\ar[r]^{A}&X\ar[r]^{A}&X\ar[r]^{A}&.\\
}$$
\end{defn}
In Section \ref{s:S} we related the stable cohomology $H_S^k$ to the $K$-theory of the stable $C^*$-algebra $S$.
Next we present a series of definitions, which are necessary for us to relate the stable-transpose homology $H^{ST}_k$ to the $K$-theory of the unstable $C^*$-algebra $U$ via the \v{C}ech cohomology of $\Omega$ (cf. Section \ref{s:preliminaries}).

Define the relation
\begin{defn}[$R_n$-equivalent sets]
Let $n\in\Z$. Two bounded subsets $\sigma_1,\sigma_2\subset \R^d$ are said to be $R_n$-equivalent, $\sigma_1\sim_{R_n}\sigma_2$, 
if there exists  $x\in \R^d$ such that $\sigma_1=\sigma_2+x$ and  $T_n(\sigma_1^\circ)=T_n(\sigma_2^\circ)+x.$
(Recall that for tiling $T'$, $T'(\sigma)$ denotes the patch made of all the tiles containing at least a point of $\sigma$).
\end{defn}
\noindent Note that $x$ in the definition is unique: 
If $\sigma_1=\sigma_2+x$ then we can write 
$$x_i=\sup p_i(\sigma_1)-\sup p_i(\sigma_2),$$ 
where $p_i:\R^d\to\R$ is the projection onto the $i$-th coordinate.
Since the sets $\sigma_1,\sigma_2$ are bounded, their supremums on each coordinate are unique and hence $x=(x_1,\ldots,x_n)$ is uniquely determined by $\sigma_1$ and $\sigma_2$.  

It is straightforward to check that the above $R_n$-equivalence is an equivalence relation on the set of bounded subsets of $\R^d$.
For instance, transitivity is just addition of vectors.
Moreover, this definition reduces to Definition \ref{d:stablecells}(stable cells) when applying it to the $k$-cells of $T_n$.
We use the standard convention that a vertex has no boundary.

\begin{defn}[Combinatorial ball $T'^m(\sigma)$]
Let $m\in\N_0$, let $T'$ be a tiling, and let $\sigma\subset\R^d$ be a bounded subset.
Define the combinatorial ball $T'^m(\sigma)$ of combinatorial radius $m$ and combinatorial center $\sigma$ to be the patch $T'(\cdots(T'(\sigma)))$ done $m$ times.
\end{defn}
The combinatorial ball $T'^m(\sigma)$ induces the following relation:
\begin{defn}[$T'^m$-equivalent sets]
Let $m\in\N_0$, and $T'$ a tiling. Two bounded subsets $\sigma_1,\sigma_2\subset \R^d$ are said to be $T'^m$-equivalent, which we denote by $\sigma_1\sim_{T'^m}\sigma_2$, 
if there exists  $x\in \R^d$ such that $\sigma_1=\sigma_2+x$ and  $T'^m(\open{\sigma}_1)=T'^m(\open{\sigma}_2)+x$.
\end{defn}
It is straighforward to check that the above relation $\sim_{T'^m}$ is actually an equivalence relation on the set of bounded subsets of $\R^d$ and that the vector $x$ is unique. The proof is the same as for the $R_n$-equivalence relation.

\subsection{Borel-Moore chains}\label{ss:BM}
We start by defining some subchain complexes of the Borel-Moore chain complex on a tiling $T'$.
These definitions are equivalent to definitions in \cite{WaltonHom}.
\begin{defn}[BM $k$-chain for tiling $T'$]
Let $k\in\{0,1,2\}$ be fixed, and let 
$$\xi:=\sum_{\sigma\text{ closed $k$-cell of $T'$}} K_\sigma\, \sigma,$$
where $K_\sigma\in\Z$ for all $k$-cells $\sigma\in T'$.
We say that the $k$-chain $\xi$ is Borel-Moore (BM for short).
\end{defn}
We remark that all integers  $K_\sigma$ in the above sum can be non-zero.
This is in contrast with the definition of a standard cellular k-chain where only a finite number of the $K_\sigma$'s are allowed to be nonzero.
\begin{defn}[PE $k$-chain for tiling $T'$]
Let $k\in\{0,1,2\}$ be fixed, and let 
$$\xi:=\sum_{\sigma\text{ closed $k$-cell of $T'$}} K_\sigma\, \sigma,$$
where $K_\sigma\in\Z$ for all $k$-cells $\sigma\in T'$.
We say that the $k$-chain $\xi$ is pattern-equivariant  (PE for short) if the following condition is satisfied:
\begin{equation}\label{e:PE}
\exists m\in \N_0:\forall \sigma,\sigma' \text{ $k$-cells of T'}:\sigma\sim_{T'^m} \sigma'\imply K_\sigma=K_{\sigma'}.
\end{equation}
\end{defn}

If $\xi$ is a PE $k$-chain, then by definition there exists an $m_0$ such that the  condition in Eq.~(\ref{e:PE}) is satisfied.
We would like to remark that the condition is also satisfied for any integer $m>m_0$.

We define the chain complex
\begin{equation}\label{e:BM-PE}
\xymatrix{
0\ar[r]
&C_2^{\text{BM,PE}}(T')\ar[r]^{\partial_2}
&C_1^{\text{BM,PE}}(T')\ar[r]^{\partial_1}
&C_0^{\text{BM,PE}}(T')\ar[r]
&0
}
\end{equation}
where 
$C_k^{\text{BM,PE}}(T')$, $k\in\{0,1,2\}$, is an abelian group whose elements are exactly all the PE $k$-chains, and where the differentials $\partial_k$ are the standard cellular boundary maps.

\begin{exam}
For the Fibonacci tiling,
an example of a sequence $(K_e)_{e\in T}$ of edges that yield a PE 1-chain with $m\ge0$ is
\begin{center}
\begin{tikzpicture}

\node at (10.7,2){$T$};
\node at (10.7,2.5){$K_e$};
\node at (0.5,2.4){$2$};
\node at (1.30902,2.4){$3$};
\node at (2.11803,2.4){$2$};
\node at (3.11803,2.4){$2$};
\node at (3.92705,2.4){$3$};
\node at (4.73607,2.4){$2$};
\node at (5.54508,2.4){$3$};
\node at (6.3541,2.4){$2$};
\node at (7.3541,2.4){$2$};
\node at (8.16312,2.4){$3$};
\node at (8.97214,2.4){$2$};

\draw plot[mark=*,mark size=1] coordinates{(0.,2.1)};
\draw plot[mark=*,mark size=1] coordinates{(1.,2.1)};
\draw plot[mark=*,mark size=1] coordinates{(1.61803,2.1)};
\draw plot[mark=*,mark size=1] coordinates{(2.61803,2.1)};
\draw plot[mark=*,mark size=1] coordinates{(3.61803,2.1)};
\draw plot[mark=*,mark size=2] coordinates{(4.23607,2.1)};
\draw plot[mark=*,mark size=1] coordinates{(5.23607,2.1)};
\draw plot[mark=*,mark size=1] coordinates{(5.8541,2.1)};
\draw plot[mark=*,mark size=1] coordinates{(6.8541,2.1)};
\draw plot[mark=*,mark size=1] coordinates{(7.8541,2.1)};
\draw plot[mark=*,mark size=1] coordinates{(8.47214,2.1)};
\draw[blue,thick] (0.,2.1)--(1.,2.1);
\draw[red,thick] (1.,2.1)--(1.61803,2.1);
\draw[blue,thick] (1.61803,2.1)--(2.61803,2.1);
\draw[blue,thick] (2.61803,2.1)--(3.61803,2.1);
\draw[red,thick] (3.61803,2.1)--(4.23607,2.1);
\draw[blue,thick] (4.23607,2.1)--(5.23607,2.1);
\draw[red,thick] (5.23607,2.1)--(5.8541,2.1);
\draw[blue,thick] (5.8541,2.1)--(6.8541,2.1);
\draw[blue,thick] (6.8541,2.1)--(7.8541,2.1);
\draw[red,thick] (7.8541,2.1)--(8.47214,2.1);
\draw[blue,thick] (8.47214,2.1)--(9.47214,2.1);

\node at (0.5,1.8){$a$};
\node at (1.30902,1.8){$b$};
\node at (2.11803,1.8){$a$};
\node at (3.11803,1.8){$a$};
\node at (3.92705,1.8){$b$};
\node at (4.73607,1.8){$a$};
\node at (5.54508,1.8){$b$};
\node at (6.3541,1.8){$a$};
\node at (7.3541,1.8){$a$};
\node at (8.16312,1.8){$b$};
\node at (8.97214,1.8){$a$};
\end{tikzpicture}
\end{center}

and with $m\ge1$  is (m=0 does not work here)
\begin{center}
\begin{tikzpicture}

\node at (10.7,2){$T$};
\node at (10.7,2.5){$K_e$};
\node at (0.5,2.4){$?$};
\node at (1.30902,2.4){$4$};
\node at (2.11803,2.4){$5$};
\node at (3.11803,2.4){$3$};
\node at (3.92705,2.4){$4$};
\node at (4.73607,2.4){$2$};
\node at (5.54508,2.4){$4$};
\node at (6.3541,2.4){$5$};
\node at (7.3541,2.4){$3$};
\node at (8.16312,2.4){$4$};
\node at (8.97214,2.4){$?$};

\draw plot[mark=*,mark size=1] coordinates{(0.,2.1)};
\draw plot[mark=*,mark size=1] coordinates{(1.,2.1)};
\draw plot[mark=*,mark size=1] coordinates{(1.61803,2.1)};
\draw plot[mark=*,mark size=1] coordinates{(2.61803,2.1)};
\draw plot[mark=*,mark size=1] coordinates{(3.61803,2.1)};
\draw plot[mark=*,mark size=2] coordinates{(4.23607,2.1)};
\draw plot[mark=*,mark size=1] coordinates{(5.23607,2.1)};
\draw plot[mark=*,mark size=1] coordinates{(5.8541,2.1)};
\draw plot[mark=*,mark size=1] coordinates{(6.8541,2.1)};
\draw plot[mark=*,mark size=1] coordinates{(7.8541,2.1)};
\draw plot[mark=*,mark size=1] coordinates{(8.47214,2.1)};
\draw[blue,thick] (0.,2.1)--(1.,2.1);
\draw[red,thick] (1.,2.1)--(1.61803,2.1);
\draw[blue,thick] (1.61803,2.1)--(2.61803,2.1);
\draw[blue,thick] (2.61803,2.1)--(3.61803,2.1);
\draw[red,thick] (3.61803,2.1)--(4.23607,2.1);
\draw[blue,thick] (4.23607,2.1)--(5.23607,2.1);
\draw[red,thick] (5.23607,2.1)--(5.8541,2.1);
\draw[blue,thick] (5.8541,2.1)--(6.8541,2.1);
\draw[blue,thick] (6.8541,2.1)--(7.8541,2.1);
\draw[red,thick] (7.8541,2.1)--(8.47214,2.1);
\draw[blue,thick] (8.47214,2.1)--(9.47214,2.1);

\node at (0.5,1.8){$a$};
\node at (1.30902,1.8){$b$};
\node at (2.11803,1.8){$a$};
\node at (3.11803,1.8){$a$};
\node at (3.92705,1.8){$b$};
\node at (4.73607,1.8){$a$};
\node at (5.54508,1.8){$b$};
\node at (6.3541,1.8){$a$};
\node at (7.3541,1.8){$a$};
\node at (8.16312,1.8){$b$};
\node at (8.97214,1.8){$a$};
\end{tikzpicture}
\end{center}

where $K_{a\underline{a}b}=3, K_{a\underline{b}a}=4,K_{b\underline{a}a}=5,K_{b\underline{a}b}=2$, and the underlined letter denotes the center of the combinatorial ball.

An example of a sequence $(K_v)_{v\in T}$ of vertices that yield a PE 0-chain with $m\ge0$ is
\begin{center}
\begin{tikzpicture}

\node at (10.7,2){$T$};
\node at (10.7,2.5){$K_v$};
\node at (0.,2.5){$?$};
\node at (1.,2.5){$5$};
\node at (1.61803,2.5){$5$};
\node at (2.61803,2.5){$5$};
\node at (3.61803,2.5){$5$};
\node at (4.23607,2.5){$5$};
\node at (5.23607,2.5){$5$};
\node at (5.8541,2.5){$5$};
\node at (6.8541,2.5){$5$};
\node at (7.8541,2.5){$5$};
\node at (8.47214,2.5){$5$};
\node at (9.47214,2.5){$?$};

\draw plot[mark=*,mark size=1] coordinates{(0.,2.1)};
\draw plot[mark=*,mark size=1] coordinates{(1.,2.1)};
\draw plot[mark=*,mark size=1] coordinates{(1.61803,2.1)};
\draw plot[mark=*,mark size=1] coordinates{(2.61803,2.1)};
\draw plot[mark=*,mark size=1] coordinates{(3.61803,2.1)};
\draw plot[mark=*,mark size=2] coordinates{(4.23607,2.1)};
\draw plot[mark=*,mark size=1] coordinates{(5.23607,2.1)};
\draw plot[mark=*,mark size=1] coordinates{(5.8541,2.1)};
\draw plot[mark=*,mark size=1] coordinates{(6.8541,2.1)};
\draw plot[mark=*,mark size=1] coordinates{(7.8541,2.1)};
\draw plot[mark=*,mark size=1] coordinates{(8.47214,2.1)};
\draw plot[mark=*,mark size=1] coordinates{(9.47214,2.1)};
\draw[blue,thick] (0.,2.1)--(1.,2.1);
\draw[red,thick] (1.,2.1)--(1.61803,2.1);
\draw[blue,thick] (1.61803,2.1)--(2.61803,2.1);
\draw[blue,thick] (2.61803,2.1)--(3.61803,2.1);
\draw[red,thick] (3.61803,2.1)--(4.23607,2.1);
\draw[blue,thick] (4.23607,2.1)--(5.23607,2.1);
\draw[red,thick] (5.23607,2.1)--(5.8541,2.1);
\draw[blue,thick] (5.8541,2.1)--(6.8541,2.1);
\draw[blue,thick] (6.8541,2.1)--(7.8541,2.1);
\draw[red,thick] (7.8541,2.1)--(8.47214,2.1);
\draw[blue,thick] (8.47214,2.1)--(9.47214,2.1);

\node at (0.5,1.8){$a$};
\node at (1.30902,1.8){$b$};
\node at (2.11803,1.8){$a$};
\node at (3.11803,1.8){$a$};
\node at (3.92705,1.8){$b$};
\node at (4.73607,1.8){$a$};
\node at (5.54508,1.8){$b$};
\node at (6.3541,1.8){$a$};
\node at (7.3541,1.8){$a$};
\node at (8.16312,1.8){$b$};
\node at (8.97214,1.8){$a$};
\end{tikzpicture}
\end{center}

and with $m\ge 1$ is (m=0 does not work here)
\begin{center}
\begin{tikzpicture}

\node at (10.7,2){$T$};
\node at (10.7,2.5){$K_v$};
\node at (0.,2.5){$?$};
\node at (1.,2.5){$1$};
\node at (1.61803,2.5){$2$};
\node at (2.61803,2.5){$0$};
\node at (3.61803,2.5){$1$};
\node at (4.23607,2.5){$2$};
\node at (5.23607,2.5){$1$};
\node at (5.8541,2.5){$2$};
\node at (6.8541,2.5){$0$};
\node at (7.8541,2.5){$1$};
\node at (8.47214,2.5){$2$};
\node at (9.47214,2.5){$?$};

\draw plot[mark=*,mark size=1] coordinates{(0.,2.1)};
\draw plot[mark=*,mark size=1] coordinates{(1.,2.1)};
\draw plot[mark=*,mark size=1] coordinates{(1.61803,2.1)};
\draw plot[mark=*,mark size=1] coordinates{(2.61803,2.1)};
\draw plot[mark=*,mark size=1] coordinates{(3.61803,2.1)};
\draw plot[mark=*,mark size=2] coordinates{(4.23607,2.1)};
\draw plot[mark=*,mark size=1] coordinates{(5.23607,2.1)};
\draw plot[mark=*,mark size=1] coordinates{(5.8541,2.1)};
\draw plot[mark=*,mark size=1] coordinates{(6.8541,2.1)};
\draw plot[mark=*,mark size=1] coordinates{(7.8541,2.1)};
\draw plot[mark=*,mark size=1] coordinates{(8.47214,2.1)};
\draw plot[mark=*,mark size=1] coordinates{(9.47214,2.1)};
\draw[blue,thick] (0.,2.1)--(1.,2.1);
\draw[red,thick] (1.,2.1)--(1.61803,2.1);
\draw[blue,thick] (1.61803,2.1)--(2.61803,2.1);
\draw[blue,thick] (2.61803,2.1)--(3.61803,2.1);
\draw[red,thick] (3.61803,2.1)--(4.23607,2.1);
\draw[blue,thick] (4.23607,2.1)--(5.23607,2.1);
\draw[red,thick] (5.23607,2.1)--(5.8541,2.1);
\draw[blue,thick] (5.8541,2.1)--(6.8541,2.1);
\draw[blue,thick] (6.8541,2.1)--(7.8541,2.1);
\draw[red,thick] (7.8541,2.1)--(8.47214,2.1);
\draw[blue,thick] (8.47214,2.1)--(9.47214,2.1);

\node at (0.5,1.8){$a$};
\node at (1.30902,1.8){$b$};
\node at (2.11803,1.8){$a$};
\node at (3.11803,1.8){$a$};
\node at (3.92705,1.8){$b$};
\node at (4.73607,1.8){$a$};
\node at (5.54508,1.8){$b$};
\node at (6.3541,1.8){$a$};
\node at (7.3541,1.8){$a$};
\node at (8.16312,1.8){$b$};
\node at (8.97214,1.8){$a$};
\end{tikzpicture}
\end{center}
where $K_{a.a}=0$, $K_{a.b}=1$, $K_{b.a}=2$.

\end{exam}

\begin{defn}[$(T_{-\ell}, R_{-n})$ $k$-chain]
Let $n\in\Z$ and let $\ell \leq n$. Let $k\in\{0,1,2\}$ be fixed, and let 
$$\xi:=\sum_{\sigma\text{ closed $k$-cell of $T_{-\ell}$}} K_\sigma \sigma,$$
where $K_\sigma\in\Z$ for all $k$-cells $\sigma\in T$.
We say that the $k$-chain $\xi$ is a $(T_{-\ell},R_{-n})$ $k$-chain  if the following condition is satisfied
\begin{equation}
\forall \sigma,\sigma' \text{ $k$-cells of $T_{-\ell}$: } \sigma\sim_{R_{-n}}\sigma'\imply K_\sigma=K_{\sigma'}.
\end{equation}
\end{defn}

A $(T_{-\ell}, R_{-n})$ $k$-chain complex is a PE-chain complex for tiling $T_{-\ell}$. 
This amounts to showing that there exists an $m \geq 0$ such that the map 
$$f_{m,n,k}([\sigma]_{T_{-\ell}^m}):=[\sigma]_{R_{-n}},\qquad\text{$\sigma\in T_{-\ell}$ is a $k$-cell,}$$
is  well-defined.
For instance consider the Fibonacci tiling and $\ell = 0, n=1$:
We get a well-defined map on the edges ($k=1$) if we choose a combinatorial radius $m=1$, and on vertices ($k=0$) if we choose a combinatorial radius $m=2$. Namely, using the notation of Example \ref{ex:Fib-SU}, $f_{m,n,k}$ on the 1-balls of edges is given by
$$
b\underline{a}a\mapsto b'_{e_0}\quad
b\underline{a}b\mapsto a'_{e_0}\quad
a\underline{a}b\mapsto a'_{e_0}\quad
a\underline{b}a\mapsto a'_{e_1},
$$
where the underlined letter denotes the center of the combinatorial ball. On the 2-balls of vertices, $f_{m,n,k}$ is
$$
aa.ba\mapsto a'_{v_1}\quad
ab.aa\mapsto a'.b'\quad
ab.ab\mapsto a'.a'\quad
ba.ab\mapsto b'.a'\qquad
ba.ba\mapsto a'_{v_1}.
$$

\begin{defn}[$C_\bullet(T_{-\ell},R_{-n})$]
Define the abelian group $C_k(T_{-\ell},R_{-n})$ to be the subgroup of $C_k^{BM,PE}(T_{-\ell})$ whose elements are   
the $(T_{-\ell},R_{-n})$ $k$-chains. Here $k\in\{0,1,2\}$, $n\in\Z$, $\ell \leq n$. The chain complex $C_\bullet(T_{-\ell},R_{-n})$ is then a subchain complex of the chain complex  in Eq.~(\ref{e:BM-PE}) (with $T'=T_{-\ell}$) by restricting the differentials.
\end{defn}

For the fixed tiling $T$, we now construct the approximate sequence of chain maps (not a short exact sequence)
\begin{equation}
\begin{gathered}
\xymatrix{
0\ar[r]
&C_2(T,R_0)\ar[r]^{\partial_2}\ar[d]_{q_2}^{i_2}
&C_1(T,R_0)\ar[r]^{\partial_1}\ar[d]_{q_1}^{i_1}
&C_0(T,R_0)\ar[r]\ar[d]_{q_0}^{i_0}
&0\\
0\ar[r]
&C_2(T,R_{-1})\ar[r]^{\partial_2}\ar[d]_{q_2^{(1)}}^{i_2^{(1)}}
&C_1(T,R_{-1})\ar[r]^{\partial_1}\ar[d]_{q_1^{(1)}}^{i_1^{(1)}}
&C_0(T,R_{-1})\ar[r]\ar[d]_{q_0^{(1)}}^{i_0^{(1)}}
&0\\
0\ar[r]
&C_2(T,R_{-2})\ar[r]^{\partial_2}\ar[d]_{q_2^{(2)}}^{i_2^{(2)}}
&C_1(T,R_{-2})\ar[r]^{\partial_1}\ar[d]_{q_1^{(2)}}^{i_1^{(2)}}
&C_0(T,R_{-2})\ar[r]\ar[d]_{q_0^{(2)}}^{i_0^{(2)}}
&0\\
&\ar@{-->}@[blue][d]
&\ar@{-->}@[blue][d]
&\ar@{-->}@[blue][d]
&\\
0\ar[r]
&C_2^{\infty}(T)\ar[r]^{\partial_2}
&C_1^{\infty}(T)\ar[r]^{\partial_1}
&C_0^{\infty}(T)\ar[r]
&0.\\
}
\end{gathered}
\end{equation}
The direct limit chain complex 
$$C_{\bullet}^{\infty}(T):=\lim_{\to}
\xymatrix{
 C_{\bullet}(T,R_0)\ar[r]^{i^{(0)}_\bullet}
& C_{\bullet}(T,R_{-1})\ar[r]^{i^{(1)}_\bullet}
& C_{\bullet}(T,R_{-2})\ar[r]^{\quad i^{(2)}_\bullet}
&
 },
 $$
 where $C^{\infty}_{k}(T):=\cup_{n\in\N_0} C_k(T,R_{-n})$, is quasi-isomorphic to the chain complex in Eq.~(\ref{e:BM-PE}) by \cite[Lemma~4.13]{WaltonHom}. Recall that a quasi-isomorphism is by definition a morphism of chain complexes that becomes an isomorphism after taking homology.
We will only describe in detail the chain maps between $C_{\bullet}(T,R_0)$ and $C_{\bullet}(T,R_{-1})$ 
because the maps for higher $n$ behave similarly.

A basis element for $C_k(T_{-\ell},R_{-n})$, $n\in\Z$, $\ell\le n$, is a $(T_{-\ell},R_{-n})$ $k$-chain
$$I([\sigma]_{R_{-n}}):= \sum_{\sigma'\in[\sigma]_{R_{-n}}}  \sigma',$$
where  $\sigma\in T_{-\ell}$ is a closed $k$-cell. Choose $k$-cell $\sigma\in T_{\ell}$ from each $R_{-n}$ equivalence class to get a whole basis.
Unless ambiguity arises, we will denote $I([\sigma]_{R_{-n}})$ simply as $[\sigma]_{R_{-n}}$.
Since $R_{-(n+1)}\subset R_{-n}$ then for any $k$-cell $\sigma\in T_{-\ell}$, the set $[\sigma]_{R_{-n}}$ is partitioned with the equivalence relation $R_{-(n+1)}$. Thus a basis element of $C_k(T_{-\ell},R_{-n})$ equals a (finite) sum of basis elements of $C_k(T_{-\ell},R_{-(n+1)})$. Hence $C_k(T_{-\ell},R_{-n})$ is a subgroup of $C_k(T_{-\ell},R_{-(n+1)})$.
Define $C^\bullet(T_{-\ell},R_{-n})$ to be the transpose of $C_\bullet(T_{-\ell},R_{-n})$ with respect to the above basis, that is the abelian group $C^k(T_{-\ell},R_{-n})$ can and will be given the basis of $C_k(T_{-\ell},R_{-n})$, but the differentials of $C^\bullet(T_{-\ell},R_{-n})$ are
\begin{equation}\label{e:deltak-SU}
\delta^k:=\partial_{k+1}^t.
\end{equation}
\subsection{Stable Transpose homology (ST)}\label{ss:ST-BM}
In this subsection we relate the stable transpose homology with \v{C}ech cohomology.
Moreover, we show that the stable cohomology and stable transpose homology are always torsion free for tilings of the line.
For tilings of the plane, we show that only $H_S^1$ and $H^{ST}_0$ can contain torsion.
In the absence of torsion, we show that $H_S^k=\lim\limits_{\to}(\Z^{n_k},A_k)$ and $H^{ST}_k=\lim\limits_{\to}(\Z^{n_k},(A_k)^t)$ for some matrix $A_k\in M_{n_k}(\Z)$. 
We start with some technical definitions and lemmas.

\begin{defn}[replacing equivalence relation]
We define the cochain map $r^\bullet:{C^\bullet}(T,R_{-1})\to {C^\bullet}(T,R_0)$ by
$$r^k([\sigma]_{R_{-1}}):=[\sigma]_{R_{0}} \qquad \sigma\in T\text{ is a $k$-cell},$$
for $k\in\{0,1,2\}$.
\end{defn}

\begin{defn}[inclusion]
Let $k\in\{0,1,2\}$.
Let $\sigma\in T$ be a $k$-cell, and suppose that $[\sigma]_{R_{0}}=[\sigma_1]_{R_{-1}}\sqcup\cdots\sqcup [\sigma_j]_{R_{-1}}$ for some $j\in \N$.
The inclusion chain map $i_\bullet:C_\bullet(T,R_0)\to C_\bullet(T,R_{-1})$ in the basis given above and evaluated at $\sigma$ is 
$$i_k(I([\sigma]_{R_{0}})):=I([\sigma_1]_{R_{-1}})+\cdots+I([\sigma_{j}]_{R_{-1}}).$$
\end{defn}
Since $r^k([\sigma_i]_{R_{-1}})=[\sigma_i]_{R_{0}}=[\sigma]_{R_0}$ for $i=1,\ldots, j$, the inclusion $i_k$ is the transpose of $r^k$
\begin{equation}
i_k=(r^k)^t.
\end{equation}

\begin{defn}[relabeled parent map]
Let $k,j\in\{0,1,2\}$.
Since $T_{-1}$ and $\lambda\omega^{-1}(T)$ are the same as tilings, for every $j$-cell $\sigma'\in T_{-1}$ there is a cell $\sigma\in \omega^{-1}(T)$, such that $\sigma'=\lambda \sigma$. In particular $\omega(\sigma)\subset T$. Moreover, $\sigma'$ and $\omega(\sigma)$ are equal as sets. 
We say that $\sigma'$ is the parent of a $k$-cell $\tau$ ($k\le j$) if $\open{\tau}$ is in $\omega(\open{\sigma})$.
Note that the parent of a cell is unique and the parent always has at least the dimension of the cell. We define the cochain map $g'^\bullet:{C^\bullet}(T,R_{-1})\to {C^\bullet}(T_{-1},R_{-1})$ by
$$g'^k([\tau]_{R_{-1}}):=  \left\{\begin{array}{cc}
    \delta_{\sigma,\tau}\,\cdot\,[\lambda\sigma]_{R_{-1}} & \text{if $\lambda\sigma$ is a $k$-cell of $T_{-1}$}  \\ 
    0 & else, \\ 
  \end{array}\right.
$$
where $\tau\in T$ is a $k$-cell such that $\open{\tau}\in \omega(\open{\sigma})$. Informally, and ignoring signs, the map $g'^k$ maps a cell $\tau$ to its parent $\lambda\sigma$ if $\lambda\sigma$ has the same dimension as $\tau$, else it maps $\tau$ to 0 (i.e.~when $\lambda\sigma$ has a higher dimension).

We define the cochain map $g^\bullet:{C^\bullet}(T,R_{-1})\to {C^{\bullet}}(T,R_0)$ by
$$g^k:=\Lambda'^k\circ g'^k,\quad k\in\{0,1,2\},$$
where $(\Lambda'^k)^{-1}:C^k(T_{0},R_0)\to C^{k}(T_{-1},R_{-1})$ is the isomorphism that identifies the $k$-stable cells $\sigma\in T$  with the expanded $k$-stable cells $\lambda\sigma+x\in T_{-1}$
\begin{equation}\label{e:LambdaPrime}
(\Lambda'^k)^{-1}([\sigma]_{R_0}):=[\lambda\sigma+ x]_{R_1},
\end{equation}
where $x$ is a translational vector such that $\lambda\sigma+x\in T_{-1}$.
\end{defn}
\begin{defn}[relabeled children map]
The chain map $q_\bullet:C_\bullet(T,R_0)\to C_\bullet(T,R_{-1})$ is defined as 
$$q_k:=(g^k)^t\qquad k\in\{0,1,2\},$$
i.e.~$q_k$ is the transpose of the relabeled parent map $g^k$. 
\end{defn}
This definition of $q_\bullet$ coincides with that of $q_\bullet$ in \cite[Lemma 4.14]{WaltonHom}, and by the same lemma, $q_\bullet$ is a quasi-isomorphism, i.e.~it is an isomorphism when one takes homology.

\begin{defn}[a section]
Let $h'_1:T\to T_{-1}$ be the homotopy defined in Subsection \ref{ss:homotopy} but for tiling $T_{-1}$ instead of tiling $T_0$. That is,
on the (expanded) prototiles of $T_{-1}$, the definition of $h'_1$ is the same as the definition of $h_1$ on the prototiles of $T_0$ up to expanding by the factor $\lambda$.
We define the cochain map $s'^\bullet:C^\bullet(T_{-1},R_{-1})\to C^\bullet(T,R_{-1})$ by
$$s'^k([\sigma']_{R_{-1}}):=  
    \sum_{\substack{\tau\in \tilde\omega_{_{-1}}(T_{-1}(\open{\sigma}\,\!'))\\\,\, h_1'(\tau)=\sigma'}}\delta_{\sigma',\tau}\,\cdot\,[\tau]_{R_{-1}},
$$
where $\sigma'$ is $k$-cell of $T_{-1}$, and $\tau$ is $k$-cell of $T$, and $\tilde\omega_{_{-1}}:T_{-1}\to T$ is the map defined in Eq.~(\ref{e:tildeomegan}).
Note that $\tilde\omega_{_{-1}}(T_{-1}(\open{\sigma}\,\!'))\subset T_0$ is, as a set, the stable cell $T_{-1}(\open{\sigma}\,\!')$; this set is tiled with prototiles of $T_0$ via $\tilde\omega_{_{-1}}$.

We define the cochain map $s^\bullet:C^\bullet(T,R_0)\to C^\bullet(T,R_{-1})$ as
$$s^k:=s'^k\circ(\Lambda'^k)^{-1},\qquad k\in\{0,1,2\},$$
where ${\Lambda'}^k$ was defined in Eq.~(\ref{e:LambdaPrime}).
\end{defn}
By the following two lemmas, the maps $s^k$ commute with the differentials $\delta^k$.
\begin{lem}
For dimension $d\le 2$, the following diagram commutes
\begin{equation*}
\xymatrix{
C^1(T_{-1},R_{-1})\ar[r]^{s'^1}\ar[d]_{\delta^1}&C^1(T_0,R_{-1})\ar[d]^{\delta^1}\\
C^2(T_{-1},R_{-1})\ar[r]^{s'^2}&C^2(T_0,R_{-1}).}
\end{equation*}
\end{lem}
\begin{proof}
We will assume that $d=2$, otherwise the lemma trivially holds.
 To help the reader with the notation, we will use Fig.~\ref{f:chaintau} as intuition. Moreover, in this proof $[\cdot]$ denotes $[\cdot]_{R_{-1}}$.
Let $T_{-1}(\open{e}\,\!')=\{t'_1,t'_2\}$ be a stable edge, for some edge $e'\in T_{-1}$ of two tiles $t'_1,t'_2\in T_{-1}$, and assume $e'$ has same orientation as one of the edges of $t'_1$, and $e'$ has opposite orientation as one of the edges of $t'_2$.
Then 
$$\delta^{1}([e'])=[t'_1]-[t'_2].$$
We remark that if $t'\in T_{-1}$ is a prototile (i.e.~a stable face), then all the shrunk tiles in $\tilde\omega_{-1}(t')\subset T_0$  are representatives of the $(T_0,R_{-1})$-equivalence classes.
By definition of $h'_1$, there are unique tiles $t_1,t_2\in T_0$ such that $h'_1(t_1)=t_1'$ and $h'_1(t_2)=t_2'$. Hence
$$s'^2(\delta^{1}([e']))=[t_1]-[t_2].$$
On the other hand, by definition of $h'_1$, there is a chain of tiles $\tau_0|\tau_1|\tau_2|\cdots|\tau_n$ from $\tau_0:=t_1$ to $\tau_n:=t_2$ such that the edges 
$\epsilon_i:=\tau_i|\tau_{i+1}\in T_0$ homotope to $e'$, i.e. $h'_1(\epsilon_i)=e'$.
Indeed, 
by definition of $h'_1$ (cf. Definition \ref{d:homotopy}($3(iii)$)) there is a unique edge $e\in\tilde\omega_{-1}(t'_1)\subset T_0$ such that $e\in e'$.
Again, by definition of $h'_1$ (cf. Definition \ref{d:homotopy}($3(iv)$)), a tile $\tau\in\tilde\omega_{-1}(t'_1)\subset T_0$ that contains the edge $e$ contains exactly two edges that homotope to $e'$ whenever $\tau\ne t_1$, else $\tau=t_1$.
There is a finite number of cells to go through, so the chain ends eventually in an edge of $t_1$.
By connectedness of $h_1^{-1}(e')$, all the edges in $\tilde\omega_{-1}(t'_1)$  which homotope to $e'$ are contained in the chain.
A similar argument holds for tile $t_2$. 
We can assume that the edges homotoping to $e'$ preserve the orientation, since we can replace $[\epsilon]$ with $[\check{\epsilon}]=-[\epsilon]$ without changing the resulting value of $\delta^1(s^1([e']))$, where $\check{\epsilon}$ is the edge $\epsilon$ but with opposite orientation. 
Then
$$s'^1([e'])=[\epsilon_1]+\cdots+[\epsilon_{n-1}].$$
Thus
\begin{eqnarray*}
\delta^1(s^1([e']))&=&([t_1]-[\tau_1])+([\tau_1]-[\tau_2])+\cdots+([\tau_{n-1}]-[t_2]])\\
&=&[t_1]-[t_2].
\end{eqnarray*}
\end{proof}
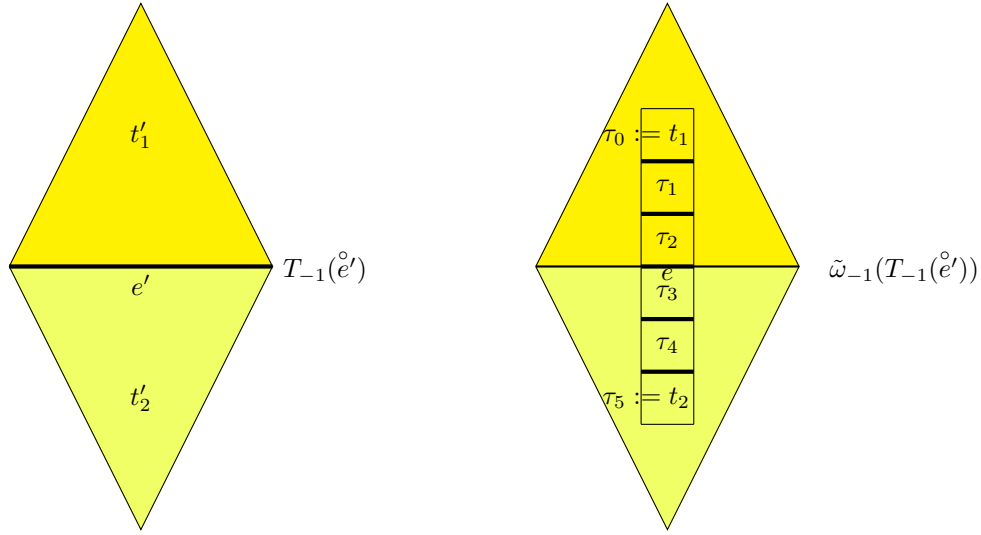
\begin{figure}
\begin{center}
\begin{tikzpicture}[scale=3.5]
\draw[yellow,fill](0,0)--(1,0)--(.5,1)--(0,0);
\draw[green!10!yellow!60,fill](0,0)--(1,0)--(.5,-1)--(0,0);
\draw[] (0,0)--(1,0)--(.5,1)--(0,0);
\draw[green!10!yellow!60,fill](0,0)--(1,0)--(.5,-1)--(0,0);
\draw[] (0,0)--(1,0)--(.5,-1)--(0,0);
\draw[ultra thick] (0,0)--(1,0);
\node at (.5,.5) {$t'_1$};
\node at (.5,-.07) {$e'$};
\node at (.5,-.5) {$t'_2$};
\node at (1.2,0) {$T_{-1}(\open{e}\,\!')$};
%
\draw[yellow,fill](2,0)--(3,0)--(2.5,1)--(2,0);
\draw[green!10!yellow!60,fill](2,0)--(3,0)--(2.5,-1)--(2,0);
\draw[] (2,0)--(3,0)--(2.5,1)--(2,0);
\draw[green!10!yellow!60,fill](2,0)--(3,0)--(2.5,-1)--(2,0);
\draw[] (2,0)--(3,0)--(2.5,-1)--(2,0);
\draw[thick] (2,0)--(3,0);
\draw (2.4,.6)--(2.4,-.6);
\draw (2.6,.6)--(2.6,-.6);
\draw (2.4,.6)--(2.6,.6);
\draw[ultra thick] (2.4,.4)--(2.6,.4);
\draw[ultra thick] (2.4,.2)--(2.6,.2);
\draw[ultra thick] (2.4,0)--(2.6,0);
\draw[ultra thick] (2.4,-.2)--(2.6,-.2);
\draw[ultra thick] (2.4,-.4)--(2.6,-.4);
\draw (2.4,-.6)--(2.6,-.6);
\node at (2.42,.5) {$\tau_0:=t_1$};
\node at (2.5,.3) {$\tau_1$};
\node at (2.5,.1) {$\tau_2$};
\node at (2.5,-.1) {$\tau_3$};
\node at (2.5,-.3) {$\tau_4$};
\node at (2.42,-.5) {$\tau_5:=t_2$};
\node at (3.4,0) {$\tilde\omega_{-1}(T_{-1}(\open{e}\,\!'))$};
\node at (2.5,-.03) {$e$};
\end{tikzpicture}
\end{center}
\caption{The stable edge $T_{-1}(\open{e}\,\!')$ tiled with the patch $\tilde\omega_{-1}(T_{-1}(\open{e}\,\!'))\subset T_0$.}\label{f:chaintau}
\end{figure}

\begin{lem}
For dimension $d\le 2$, the following diagram commutes
\begin{equation*}
\xymatrix{
C^0(T_{-1},R_{-1})\ar[r]^{s'^0}\ar[d]_{\delta^0}&C^0(T_0,R_{-1})\ar[d]^{\delta^0}\\
C^1(T_{-1},R_{-1})\ar[r]^{s'^1}&C^1(T_0,R_{-1}).}
\end{equation*}
\end{lem}
\begin{proof}
In this proof, $[\cdot]$ denotes $[\cdot]_{R_{-1}}$.
Let $T_{-1}(v')$ be a stable vertex, for some vertex $v'\in T_{-1}$.
Suppose that $e'_1,\ldots, e'_n$ are all the edges in $T_{-1}(v')$ for which one of its vertices is $v'$.
We will assume that $e'_j$, $j=1,\ldots, n$,  is oriented so that $v'$ is a final vertex since we can replace $[e'_j]$ with $[\breve e'_j]=-[e'_j]$ without changing the resulting value of $s'^{1}(\delta^0([v']))$, where $\breve e'_j$ is the edge $e'_j$ with reversed orientation. Then 
$$s'^1(\delta^0([v']))=s'^1([e'_1])+\cdots s'^1([e'_n]).$$
Suppose that $T_{-1}(\open{e}_1\,\!')=\{t'_1,t'_2\}$ such that $e_1'$ matches the orientation of one of the edges of $t'_1$, and $e_1'$ with reverse orientation matches one of the edges of $t'_2$.
Suppose that $e_{j1},\ldots e_{jm_j}$ are all the edges whose interior is in $\tilde\omega_{-1}(\open{t}\,\!'_j)\subset T_0$ and that homotope to $e'_1$, where $j=1,2$. 
We assume here as well that the edges homotoping to $e'_1$ preserve the orientation, since the resulting value would not change by similar reasons as above.
Then
$$s'^1([e'_1])=[e_{1}]+[e_{11}]+\cdots+[e_{1m_1}]+[e_{21}]+\cdots+[e_{2m_2}],$$
where $e_1$ is the unique edge in $\tilde\omega_{-1}(e'_1)$ that homotopes to $e'_1$.

On the other hand, consider all the vertices in $\tilde\omega_{-1}(T_{-1}(v'))\subset T_0$ which homotope to $v'$: If two of these vertices are connected by an edge $e$ of $T_0$, then, in the resulting value of $\delta^0(s'^1([v']))$, we get the term $+[e]$ (coming from the final vertex) and the term $-[e]$ (coming from the initial vertex), and hence, after cancellation, $[e]$ does not occur.
After removing such edges, all the edges in $\delta^0(s'^1([v']))$ must homotope to one of the $e'_i$'s. We conclude that the terms in $\delta^0(s'^1([v']))$ coming from edges homotoping to $e'_1$ are
$$[e_1]+[e_{11}]+\cdots+[e_{1m_1}]+[e_{21}]+\cdots+[e_{2m_2}].$$
It follows that $s'^1(\delta^0([v']))=\delta^0(s'^1([v']))$.
\end{proof}

Next we will prove four auxiliary results, that will be necessary in our main theorems.
\begin{lem}\label{l:gs=id}
For $k\in\{0,1,2\}$, the map $s^k$ is a section of the relabeled parent map $g^k$, that is
$$g^k\circ s^k = id.$$
\end{lem}
\begin{proof}
By the standard isomorphism $\Lambda'$ that identifies stable cells $T$ with stable cells of $T_{-1}$, it suffices to show that 
$g'^k\circ s'^k=id$.
We get
\begin{eqnarray*}
g'^k(s'^k([\sigma']_{R_{-1}}))=\sum_{\tau\in \tilde\omega_{_{-1}}(T_{-1}(\open{\sigma}\,\!')),\,\, h_1'(\tau)=\sigma'}\delta_{\sigma',\tau}\,\cdot\,g'^k([\tau]_{R_{-1}})=[\sigma']_{R_{-1}},
\end{eqnarray*}
where the last equality is because of the following:
First, $\sigma'$ is a $k$-cell in $T_{-1}$.
Second, all the $k$-cells $\tau\in\tilde\omega_{_{-1}}(T_{-1}(\open{\sigma}\,\!'))$ with $h_1'(\tau)=\sigma'$ that don't lie inside $\sigma'$ (as sets) vanish under $g'^k$ because the parent of $\tau$  is a $j$-cell of $T_{-1}$ of dimension $j>k$. 
Third, exactly one $k$-cell $\tau_0\subset \sigma'$ homotopes to $\sigma'$ (cf. Subsection \ref{ss:homotopy} adapted to $h'_1$).
Hence $\tau_0$ is the only $k$-cell for which $g'^k$ is nonzero. Since $g'^k([\tau_0]_{R_{-1}})=[\sigma']_{R_{-1}}$ the equality follows.  
\end{proof}

\begin{lem}\label{l:CST-C0}
We have
\begin{equation}\label{e:doubledual}
C_S^\bullet = C^\bullet(T,R_0)\qquad\text{and}\qquad {C^{ST}_\bullet} = C_\bullet(T,R_0).
\end{equation}
\end{lem}
\begin{proof}
Recalling that $C^\bullet(T,R_0) := C_\bullet(T,R_0)^t$ where the right hand side is given the standard basis, it is easy to see that
$$C_S^\bullet = C^\bullet(T,R_0).$$
Namely, the basis elements of $C_S^k$ are the stable classes $[\sigma]_{R_0}$, where $\sigma\in T_0$ is a $k$-cell.
As explained at the end of Subsection \ref{ss:BM}, the basis elements of  $C^k(T,R_0)$ are also the equivalence classes $[\sigma]_{R_0}$ where $\sigma\in T_0$ is a $k$-cell.
Moreover, the differentials for $C_S^\bullet$ in Eq.~(\ref{e:delta0-intro}) and Eq.~(\ref{e:delta1-intro}) agree with those for $C^\bullet(T,R_0)$ given in Eq.~(\ref{e:deltak-SU}).
Applying the transpose to both sides, we get
\begin{equation*}
C^{ST}_\bullet={C_S^\bullet}^t = C^\bullet(T,R_0)^t = (C_\bullet(T,R_0)^t)^t = C_\bullet(T,R_0).
\end{equation*}
\end{proof}

\begin{lem}\label{l:W=rs}
For $k\in\{0,1,2\}$, the connecting maps $W^k:C^k_S\to C^k_S$ can be obtained from the replacing-equivalence-relation chain map $r^{k}$ and the section chain map $s^{k}$ by the following equality:
$$W^k=r^k\circ s^k,\qquad k\in\{0,1,2\}.$$
\end{lem}
\begin{proof}
By the standard isomorphism $\Lambda'^k$ in Eq.~(\ref{e:LambdaPrime}) that identifies stable $k$-cells of $T_{-1}$ with stable $k$-cells of $T$, 
and by $\Lambda^k$ defined similarly, that identifies the stable $k$-cells of $T$ with stable $k$-cells of $T_{1}$, it suffices to show that 
$W_0^k:C^k(T_0,R_0)\to C^k(T_1,R_{1})$ given by
$$W_0^k([\sigma]_{R_0}):=\sum_{\hat{\tau}\in\tilde\omega_0(T(\sigma^\circ)),\,\,\,\, h_1(\hat{\tau})=\sigma} \delta_{\sigma,\hat{\tau}}\,\, [\hat{\tau}]_{R_1}
$$
has the same matrix as 
$r^k\circ s'^k:C^k(T_{-1},R_{-1})\to C^k(T,R_0)$ in their standard basis.
This holds because the composition is
\begin{eqnarray*}
r^k(s'^k([\sigma']_{R_{-1}}))&=&  
    \sum_{\tau\in \tilde\omega_{_{-1}}(T_{-1}(\open{\sigma}\,\!')),\,\, h_1'(\tau)=\sigma'}\delta_{\sigma',\tau}\,\cdot\,r^k([\tau]_{R_{-1}})\\
&=&  
    \sum_{\tau\in\tilde\omega_{_{-1}}(T_{-1}(\open{\sigma}\,\!')),\,\, h_1'(\tau)=\sigma'}\delta_{\sigma',\tau}\,\cdot\,[\tau]_{R_{0}}.
\end{eqnarray*}
\end{proof}

\begin{lem}\label{l:PE-ST}
For $k\in\{0,1,2\}$ 
$$H_k({W^\bullet}^t) = H_k(q_\bullet)^{-1}H_k(i_\bullet).$$ 
\end{lem}
\begin{proof}
Applying the functor $H_k((\_\!\_)^t)$ to the cochain map $W^\bullet:C_S^\bullet\to C_S^\bullet$ we get the following map
\begin{equation*}
H_k({W^\bullet}^t):H_k({C^\bullet_S}^t)\to H_k({C^\bullet_S}^t).
\end{equation*}
Since  $W^\bullet=r^\bullet\circ s^\bullet$ and $g^\bullet\circ s^\bullet=id^\bullet$
and $r^\bullet=i_\bullet^t$ and $g^\bullet=q_\bullet^t$, and $q^\bullet$ is a quasi-isomorphism we have that
$$H_k({W^\bullet}^t)=H_k({s^\bullet}^t)H_k({r^\bullet}^t)=H_k(q_\bullet)^{-1}H_k(i_\bullet).$$
\end{proof}

Recall that the tiling space $\Omega$ is assumed to be FLC, and that the substitution map $\omega$ is primitive and injective. 
\begin{thm}\label{t:PE-ST}
Let $T\in\Omega$ be a tiling of dimension less or equal to two, with convex prototiles.
For $k\in\{0,1,2\}$, the stable-transpose-homology groups are related to the \v{C}ech cohomology groups by
$$H^{ST}_k(T)=\breve{H}^{d-k}(\Omega).$$
\end{thm}

\begin{proof}
Let
$$H^{\infty}_{k}:=\lim_{\to}(H_{k}(C_\bullet(T,R_0)),\,\,  H_k(q_\bullet)^{-1}H_k(i_\bullet)).$$
Then, by Lemma \ref{l:CST-C0} and Lemma \ref{l:PE-ST}, we have that
$$H^{ST}_k= H_k^{\infty}\qquad k\in\{0,1,2\}.$$
The theorem follows by the previous equality and by the isomorphisms
$$  H_k^{\infty}\cong H_k^{PE}\cong H_{PE}^{d-k}\cong \breve{H}^{d-k}(\Omega),$$
where the first isomorphism is by \cite[Thm.~4.5]{WaltonHom}, the second is by \cite[Thm.~2.2]{WaltonHom},
and the third isomorphism is in \cite{Kel-PE}, \cite{KelPut-PE},  \cite{Sadun-PE}.
We should remark that the first isomorphism factors through $H_k(C^\infty_\bullet(T))$. Namely,
the groups $H_k^{\infty}$ and $H_k(C^\infty_\bullet(T))$ are obtained by taking direct limit of the rows in the following commutative diagram whose vertical maps are isomorphisms
\begin{equation}
\begin{gathered}
\xymatrix@C=1.2cm{
H_k(C_\bullet(T,R_0))\ar[r]^{H_k({W^\bullet}^t)}\ar[d]_{id}
&H_k(C_\bullet(T,R_0))\ar[r]^{H_k({W^\bullet}^t)}\ar[d]_{H_k(q_\bullet)}
&H_k(C_\bullet(T,R_0))\ar[r]^{\qquad\quad H_k({W^\bullet}^t)}\ar[d]_{H_k(q^{(1)}_\bullet\circ q_\bullet)}
&\\
H_k(C_\bullet(T,R_0))\ar[r]^{H_k(i_\bullet)}
&H_k(C_\bullet(T,R_{-1}))\ar[r]^{H_k(i^{(1)}_\bullet)}
&H_k(C_\bullet(T,R_{-2}))\ar[r]^{\qquad\quad H_k(i^{(2)}_\bullet)}
&,
}
\end{gathered}
\end{equation}
and one also uses that $C_\bullet^\infty(T)\hookrightarrow C_\bullet^{BM,PE}(T)$ is a quasi-isomorphism.
\end{proof}
\begin{rem}
For tilings with non-convex prototiles, one can either refine the substitution using convex tiles, or adapt a recipe similar to the Kites-Darts Penrose tiling shown in \cite[Example~4.3]{WaltonHom}.
\end{rem}
%


We should remark that we denote cohomology groups in the following theorem using the notational convention given in Subsection \ref{ss:cohomologynotation}.
\begin{thm}\label{t:UCT}
With notation from the beginning of this section, the following diagram commutes, and the rows in the diagram are short exact sequences which split (non-canonically)
\begin{equation*}
\xymatrix{
0\ar[r]
&\Ext_\Z^{1}(H^{k+1}(C_S^\bullet),\Z)\ar[r]\ar[d]^{\Ext_{\Z}^1(H^{k+1}(W^\bullet),\Z)}
&H_k(C_\bullet^{ST})\ar[r]\ar[d]^{H_k({W^\bullet}^t)}
&\Hom(H^k(C^\bullet_S),\Z)\ar[r]\ar[d]^{\Hom(H^k(W^\bullet),\Z)}
&0\\
0\ar[r]
&\Ext_{\Z}^{1}(H^{k+1}(C_S^\bullet),\Z)\ar[r]
&H_k(C_\bullet^{ST})\ar[r]
&\Hom(H^k(C^\bullet_S),\Z)\ar[r]
&0.
}
\end{equation*}
\end{thm}

\begin{proof}
Applying the universal coefficient theorem (UCT) for cohomology to the chain map $W^\bullet:C_S^\bullet\to C_s^\bullet$ (using the notation of Subsection \ref{ss:cohomologynotation}) we get the following commutative diagram, where the rows are short exact sequences which split (non-canonically)
\begin{equation*}
\xymatrix{
0\ar[r]
&\Ext_\Z^{1}(H^{k+1}(C_S^\bullet),\Z)\ar[r]\ar[d]^{\Ext_{\Z}^1(H^{k+1}(W^\bullet),\Z)}
&H_k(\Hom(C^\bullet_S,\Z))\ar[r]\ar[d]^{H_k(\Hom(W^\bullet,\Z))}
&\Hom(H^k(C^\bullet_S),\Z)\ar[r]\ar[d]^{\Hom(H^k(W^\bullet),\Z)}
&0\\
0\ar[r]
&\Ext_{\Z}^{1}(H^{k+1}(C_S^\bullet),\Z)\ar[r]
&H_k(\Hom(C_S^\bullet,\Z))\ar[r]
&\Hom(H^k(C^\bullet_S),\Z)\ar[r]
&0.
}
\end{equation*}
The result follows by identifying $\Hom(C^\bullet_S,\Z)$ with ${C^{ST}_\bullet}$. 
\end{proof}

\begin{defn}[S-tor, ST-torFree] 
Define the following homology groups
\begin{eqnarray*}
 H^{\text{S-tor}}_{k}&:=&\lim_{\to}\big(
									 \Ext_\Z^{1}(H^{k}(C_S^\bullet),\Z) 
 									,\,\, 
 									\Ext_{\Z}^1(H^{k}(W^\bullet),\Z)
 \big)\\
 H^{\text{ST-torFree}}_{k}&:=&\lim_{\to}\big( 
  							   \Hom(H^k(C^\bullet_S),\Z)
  								,\,\, 
								\Hom(H^k(W^\bullet),\Z)
 \big).
 \end{eqnarray*}
\end{defn}
%
\begin{cor}\label{c:torsionmovement} We have
\begin{eqnarray*}
H_k(C_\bullet^{ST})&\cong&\Ext_\Z^{1}(H^{k+1}(C_S^\bullet),\Z)\oplus\Hom(H^k(C^\bullet_S),\Z),
\end{eqnarray*}
where the isomorphism is non-canonical.
Note that the stable-torsion part moves from level $k+1$ to level $k$ of the unstable-torsion part. 

Moreover, $H^{\text{S-tor}}_{k}$ is a torsion group, $H^{\text{ST-torFree}}_{k}$ is a torsion free group, and
\begin{equation*}
\xymatrix{
0\ar[r]
& H^{\text{S-tor}}_{k+1}\ar[r]
&H^{ST}_{k}\ar[r]
& H^{\text{ST-torFree}}_{k}\ar[r]
&0
}
\end{equation*}
is a short exact sequence (might not split).
In particular, 
\begin{itemize}
\item [(1)] if $T$ is a one-dimensional tiling, then there is no stable nor unstable torsion part.
\item [(2)] if $T$ is a two-dimensional tiling, then in the stable case, the groups $H^0_S$ and $H^2_S=\Z$ are torsion free, 
but $H^1_S$ can contain torsion; in the unstable case, the groups
$H_2^{ST}=\Z$, $H_1^{ST}$ are torsion free, but $H_0^{ST}$ can contain torsion.
\end{itemize}
\end{cor}  
\begin{proof}
The direct sum follows from the short exact sequence of the top row of the diagram in the theorem, which splits (non-canonically). 

Recall that the direct limit of a directed system of short exact sequences of abelian groups is a short exact sequence. This is because the direct limit is an exact functor in the category of abelian groups. From this fact and the theorem, the short exact sequence of the corollary follows.

Recall that if $G$ is an abelian torsion group then any quotient of an infinite direct sum of $G$'s is a torsion group.
Since $\Ext_\Z^{1}(H^{k+1}(C_S^\bullet),\Z)$ is a torsion group and the direct limit $H^{\text{S-tor}}_{k+1}$ is in particular a quotient of an infinite direct sum of $\Ext_\Z^{1}(H^{k+1}(C_S^\bullet),\Z)$, we get that $H^{\text{S-tor}}_{k+1}$ is a torsion group (it might be the zero group).

The group $H^{\text{ST-torFree}}_{k}$ is torsion free since it is isomorphic to a subgroup of $\Q^n$ for some $n\in\N$ (here we used that $\Hom(H^k(C^\bullet_S),\Z)$ is a finitely generated free abelian group). (cf.~Appendix \ref{a.s:direct-limits-AG}). 

If the tiling $T$ is of dimension $d=1$ then $H^1(C_S^\bullet)=\Z$ by a similar proof to Lemma \ref{l:plane-ZsFoverImDelta1}, and thus by the direct sum in this corollary, there is no torsion to move to $H_0(C^{ST}_\bullet)$. Furthermore, since $H_{-1}(C^{ST}_\bullet)=0$, there can be no torsion to move from $H^0(C_S^\bullet)$. Finally, $H_1(C^{ST}_\bullet)=\breve{H}^0(\Omega)=\Z$. Since before taking the direct limit there is no torsion in any of the groups, after taking the direct limit there is still no torsion in any of the groups, that is, there is no torsion in $H_\bullet^{ST}$ nor in $H^\bullet_{S}$.
The conclusion for tilings of dimension 2 is deduced similarly from the following table
\begin{equation*}
\xymatrix@R=0.3cm{
\text{\underline{Unstable}}&&\text{\underline{Stable}}\\
H_2(C_\bullet^{ST})=\Z & & H^2(C^\bullet_{S})=\Z  \ar[lld]^{\text{\!\!\!\!\!\!\!\!\!\!torsion free}} \\ 
   H_1(C_\bullet^{ST}) &  & H^1(C^\bullet_{S})\ar[lld]\\ 
    H_0(C_\bullet^{ST}) &  & H^0(C^\bullet_{S}).
 }
\end{equation*}

\end{proof}
\begin{cor}\label{c:A-At}
If $T\in\Omega$ is a tiling of dimension 1, or if $T$ is of dimension 2 with $H^{1}(C_S^\bullet)$ torsion free then
\begin{eqnarray*}
H^{ST}_k&\cong&
\lim_{\to}
\xymatrix{
\,\,\Z^n\ar[r]^{A_k^t}
&\Z^n\ar[r]^{A_k^t}
&\Z^n\ar[r]^{A_k^t}
&
}\qquad k\in\{0,1,2\}\\
H_{S}^k&\cong&
\lim_{\to}
\xymatrix{
\,\,\Z^n\ar[r]^{A_k}
&\Z^n\ar[r]^{A_k}
&\Z^n\ar[r]^{A_k}
&
}\qquad k\in\{0,1,2\}.
\end{eqnarray*}
where $A_k\in M_n(\Z)$ is the matrix of $H^k(W^\bullet)$ with respect to a fixed basis of $H^k(C_S^\bullet)$, and
 $A_k^t$ is the transpose of $A_k$.
\end{cor}
\begin{proof}
By assumption and by Corollary \ref{c:torsionmovement}, $H^k(C^\bullet_S)$ is torsion free. Since taking direct limits cannot introduce torsion, $H^k_{S}$ is also torsion free.
By the short exact sequence in Corollary \ref{c:torsionmovement} we get that
$$H_k^{ST}=\lim_{\to}\big( \Hom(H^k(C_S^\bullet),\Z),\Hom(H^k(W^\bullet),\Z)\big),\quad k\in\{0,1,2\}.$$
Since $H^k(C^\bullet_S)$ is a finitely generated torsion free abelian group, it is isomorphic to $\Z^n$ for some $n\in\N_0$.
Choosing such an isomorphism, that is, giving a basis to $H^k(C_S^\bullet)$, then in the dual basis of $H^k(C_S^\bullet)$,
the matrix for the dual map 
$$\Hom(H^k(W^\bullet),\Z):\Hom(H^k(C_S^\bullet),\Z)\to \Hom(H^k(C_S^\bullet),\Z)$$
is well-known (and easily shown) to be the transpose of the matrix. That is $\Hom(H^k(W^\bullet),\Z)=(H^k(W^\bullet))^t$.
The corollary follows immediately.

Note that,
\begin{eqnarray}\label{e:ST-homology-torsionfree}
 H^{ST}_2&=&\lim\limits_{\to}\xymatrix{(\coker\delta^1)^t \ar[r]^{W_F^t}&(\coker\delta^1)^t\ar[r]^{\quad W_F^t}&\cdots}\\
\nonumber H^{ST}_1&=&\lim\limits_{\to} \xymatrix{(\frac{\ker\delta^1}{\im\delta^0})^t\ar[r]^{W_E^t}&(\frac{\ker\delta^1}{\im\delta^0})^t\ar[r]^{W_E^t}&(\frac{\ker\delta^1}{\im\delta^0})^t\ar[r]^{\quad W_E^t}&\cdots}\\
\nonumber H^{ST}_0&=&\lim\limits_{\to} \xymatrix{(\ker\delta^0)^t\ar[r]^{W_V^t}&(\ker\delta^0)^t\ar[r]^{W_V^t}&(\ker\delta^0)^t\ar[r]^{\quad W_V^t}&\cdots}
\end{eqnarray}
\end{proof}

 
\subsection{Examples}
We denote with $(T_0,R_0)$ the tiling $T_0$ with cells partitioned by the equivalence relation $R_0$.
The equivalence classes $[\sigma]_{R_{0}}$, where $\sigma$ is a $k$-cell of $T$, are the stable vertices if $k=0$, the stable edges if $k=1$, and the stable faces if $k=2$.
For tilings of the line, we denote the stable edges with the labels of the prototiles, e.g. $a$, $b$.
The stable vertices we denote them by pairs of edges with a dot in the middle to denote the vertex, e.g. $a.b$.

Similarly, we denote with $(T_{-1},R_{-1})$ the tiling $T_{-1}$ with cells partitioned by the equivalence relation $R_{-1}$.
The $R_{-1}$-equivalence classes of cells of $T_{-1}$ are the stable cells of $T_{-1}$.
For tilings of the line we denote the stable edges and stable vertices of $T_{-1}$ with primes on the labels of the prototiles, e.g. $a'$, $b'$, $a'.b'$.

Similarly, we denote with $(T_{0},R_{-1})$ the tiling $T_{0}$ with cells partitioned by the equivalence relation $R_{-1}$.
However, the similarity with the above ends here.
Each cell in $T_0$ has a parent cell in $T_{-1}$. 
The $R_{-1}$ equivalence classes of cells in $T_0$ are denoted by the label of the stable parent cell (which lives in $T_{-1}$) together with a unique sub-label to distinguish the siblings from each other.

\begin{exam}[Fibonacci tiling]\label{ex:Fib-SU}
Consider the Fibonnacci tiling with protoedges $a,b$ of lengths $|a|=1$, $|b|=1/\phi$, ($\lambda=\phi$=golden ratio) and substitution rule $\omega(a)=ab$, $\omega(b)=a$.
The pairs $(T_0,R_0)$, $(T_{-1},R_{-1})$ $(T_{0},R_{-1})$ are drawn below:\\
\begin{tikzpicture}[scale=1.1]
\node at (10.7,2){$(T_0,R_0)$};
\draw plot[mark=*,mark size=1] coordinates{(0.,2.1)};
\draw plot[mark=*,mark size=1] coordinates{(1.,2.1)};
\draw plot[mark=*,mark size=1] coordinates{(1.61803,2.1)};
\draw plot[mark=*,mark size=1] coordinates{(2.61803,2.1)};
\draw plot[mark=*,mark size=1] coordinates{(3.61803,2.1)};
\draw plot[mark=*,mark size=2] coordinates{(4.23607,2.1)};
\draw plot[mark=*,mark size=1] coordinates{(5.23607,2.1)};
\draw plot[mark=*,mark size=1] coordinates{(5.8541,2.1)};
\draw plot[mark=*,mark size=1] coordinates{(6.8541,2.1)};
\draw plot[mark=*,mark size=1] coordinates{(7.8541,2.1)};
\draw plot[mark=*,mark size=1] coordinates{(8.47214,2.1)};
\draw[blue,thick] (0.,2.1)--(1.,2.1);
\draw[red,thick] (1.,2.1)--(1.61803,2.1);
\draw[blue,thick] (1.61803,2.1)--(2.61803,2.1);
\draw[blue,thick] (2.61803,2.1)--(3.61803,2.1);
\draw[red,thick] (3.61803,2.1)--(4.23607,2.1);
\draw[blue,thick] (4.23607,2.1)--(5.23607,2.1);
\draw[red,thick] (5.23607,2.1)--(5.8541,2.1);
\draw[blue,thick] (5.8541,2.1)--(6.8541,2.1);
\draw[blue,thick] (6.8541,2.1)--(7.8541,2.1);
\draw[red,thick] (7.8541,2.1)--(8.47214,2.1);
\draw[blue,thick] (8.47214,2.1)--(9.47214,2.1);

\node at (0.5,1.8){$a$};
\node at (1.30902,1.8){$b$};
\node at (2.11803,1.8){$a$};
\node at (3.11803,1.8){$a$};
\node at (3.92705,1.8){$b$};
\node at (4.73607,1.8){$a$};
\node at (5.54508,1.8){$b$};
\node at (6.3541,1.8){$a$};
\node at (7.3541,1.8){$a$};
\node at (8.16312,1.8){$b$};
\node at (8.97214,1.8){$a$};

\draw[black,thin,dotted] (0.,2.1)--(0,0.1);
\draw[black,thin,dotted] (1.61803,2.1)--(1.61803,0.1);
\draw[black,thin,dotted] (2.61803,2.1)--(2.61803,0.1);
\draw[black,thin,dashed] (4.23607,2.1)--(4.23607,0.1);
\draw[black,thin,dotted] (5.8541,2.1)--(5.8541,0.1);
\draw[black,thin,dotted] (6.8541,2.1)--(6.8541,0.1);
\draw[black,thin,dotted] (8.47214,2.1)--(8.47214,0.1);
\draw[black,thin,dotted] (9.47214,2.1)--(9.4721,0.1);

\node at (10.7,0){$(T_0,R_{-1})$};

\draw plot[mark=*,mark size=1] coordinates{(0.,0.1)};
\draw plot[mark=*,mark size=1] coordinates{(1.,0.1)};
\draw plot[mark=*,mark size=1] coordinates{(1.61803,0.1)};
\draw plot[mark=*,mark size=1] coordinates{(2.61803,0.1)};
\draw plot[mark=*,mark size=1] coordinates{(3.61803,0.1)};
\draw plot[mark=*,mark size=2] coordinates{(4.23607,0.1)};
\draw plot[mark=*,mark size=1] coordinates{(5.23607,0.1)};
\draw plot[mark=*,mark size=1] coordinates{(5.8541,0.1)};
\draw plot[mark=*,mark size=1] coordinates{(6.8541,0.1)};
\draw plot[mark=*,mark size=1] coordinates{(7.8541,0.1)};
\draw plot[mark=*,mark size=1] coordinates{(8.47214,0.1)};
\draw[blue,thick] (0.,0.1)--(1.,0.1);
\draw[red,thick] (1.,0.1)--(1.61803,0.1);
\draw[blue,thick] (1.61803,0.1)--(2.61803,0.1);
\draw[blue,thick] (2.61803,0.1)--(3.61803,0.1);
\draw[red,thick] (3.61803,0.1)--(4.23607,0.1);
\draw[blue,thick] (4.23607,0.1)--(5.23607,0.1);
\draw[red,thick] (5.23607,0.1)--(5.8541,0.1);
\draw[blue,thick] (5.8541,0.1)--(6.8541,0.1);
\draw[blue,thick] (6.8541,0.1)--(7.8541,0.1);
\draw[red,thick] (7.8541,0.1)--(8.47214,0.1);
\draw[blue,thick] (8.47214,0.1)--(9.47214,0.1);

\node at (0.5,-0.1){${a'}\!\!_{e_0}$};
\node at (1.30902,-0.1){${a'}\!\!_{e_1}$};
\node at (2.11803,-0.1){${b'}\!\!_{e_0}$};
\node at (3.11803,-0.1){${a'}\!\!_{e_0}$};
\node at (3.92705,-0.1){${a'}\!\!_{e_1}$};
\node at (4.73607,-0.1){${a'}\!\!_{e_0}$};
\node at (5.54508,-0.1){${a'}\!\!_{e_1}$};
\node at (6.3541,-0.1){${b'}\!\!_{e_0}$};
\node at (7.3541,-0.1){${a'}\!\!_{e_0}$};
\node at (8.16312,-0.1){${a'}\!\!_{e_1}$};
\node at (8.97214,-0.1){${b'}\!\!_{e_0}$};

\node at (10.7,-.6){$h_t'$};
\draw[->,black,thin,dashed] (0.,0.1)--(0,-1.1);
\draw[->,black,thin,dashed] (1.,0.1)--(1.61803,-1.1);
\draw[->,black,thin,dashed] (1.61803,0.1)--(1.61803,-1.1);
\draw[->,black,thin,dashed] (2.61803,0.1)--(2.61803,-1.1);
\draw[->,black,thin,dashed] (3.61803,0.1)--(4.23607,-1.1);
\draw[->,black,thin,dashed] (4.23607,0.1)--(4.23607,-1.1);
\draw[->,black,thin,dashed] (5.23607,0.1)--(5.8541,-1.1);
\draw[->,black,thin,dashed] (5.8541,0.1)--(5.8541,-1.1);
\draw[->,black,thin,dashed] (6.8541,0.1)--(6.8541,-1.1);
\draw[->,black,thin,dashed] (7.8541,0.1)--(8.47214,-1.1);
\draw[->,black,thin,dashed] (8.47214,0.1)--(8.47214,-1.1);
\draw[->,black,thin,dashed] (9.47214,0.1)--(9.4721,-1.1);

\node at (10.7,-1.3){$(T_{-1},R_{-1})$};
\draw plot[mark=*,mark size=1] coordinates{(0.,-1.2)};
\draw plot[mark=*,mark size=1] coordinates{(1.61803,-1.2)};
\draw plot[mark=*,mark size=1] coordinates{(2.61803,-1.2)};
\draw plot[mark=*,mark size=2] coordinates{(4.23607,-1.2)};
\draw plot[mark=*,mark size=1] coordinates{(5.8541,-1.2)};
\draw plot[mark=*,mark size=1] coordinates{(6.8541,-1.2)};
\draw plot[mark=*,mark size=1] coordinates{(8.47214,-1.2)};
\draw plot[mark=*,mark size=1] coordinates{(9.47214,-1.2)};
\draw[blue,thick] (0.,-1.2)--(1.61803,-1.2);
\draw[red,thick] (1.61803,-1.2)--(2.61803,-1.2);
\draw[blue,thick] (2.61803,-1.2)--(4.23607,-1.2);
\draw[blue,thick] (4.23607,-1.2)--(5.8541,-1.2);
\draw[red,thick] (5.8541,-1.2)--(6.8541,-1.2);
\draw[blue,thick] (6.8541,-1.2)--(8.47214,-1.2);
\draw[red,thick] (8.47214,-1.2)--(9.47214,-1.2);

\node at (0.809017,-1.5){$a'$};
\node at (2.11803,-1.5){$b'$};
\node at (3.42705,-1.5){$a'$};
\node at (5.04508,-1.5){$a'$};
\node at (6.3541,-1.5){$b'$};
\node at (7.66312,-1.5){$a'$};
\node at (8.97214,-1.5){$b'$};

\node[ultra thick,black!80!green] at (.9,.4) {$a'\!_{v_1}$};
\node[ultra thick,black!80!green] at (1.61803,.4) {$a'\!.b'$};
\node[ultra thick,black!80!green] at (2.61803,.4) {$b'\!.a'$};
\node[ultra thick,black!80!green] at (3.5,.4) {$a'\!_{v_1}$};
\node[ultra thick,black!80!green] at (4.23607,.4) {$a'\!.a'$};
\node[ultra thick,black!80!green] at (5.1,.4) {$a'\!_{v_1}$};
\node[ultra thick,black!80!green] at (5.8541,.4) {$a'\!.b'$};
\node[ultra thick,black!80!green] at (6.8541,.4) {$b'\!.a'$};
\node[ultra thick,black!80!green] at (7.7,.4) {$a'\!_{v_1}$};
\node[ultra thick,black!80!green] at (8.47214,.4) {$a'\!.b'$};

\end{tikzpicture}\\
The matrices for the replacing-the-equivalence-relation maps $r^k$ are
$$
r^0=\begin{blockarray}{ccccc}
 &a'\!.a' & a'\!.b' & b'\!.a' & {a'}\!_{v_1} \\
\begin{block}{c(cccc)}
  a.a & 0 & 0 & 1 & 0  \\
  a.b & 0 & 0 & 0 & 1  \\
  b.a & 1 & 1 & 0 & 0  \\
\end{block}
\end{blockarray}
$$
For instance $r^0(a'.a')=\omega(a).\omega(a)=ab.ab=b.a$.
$$
r^1=\begin{blockarray}{cccc}
 &a'\!\!_{e_0}  &a'\!\!_{e_1} & b'\!\!_{e_0} \\
\begin{block}{c(ccc)}
  a & 1 & 0 & 1  \\
  b & 0 & 1 & 0  \\
\end{block}
\end{blockarray}
$$
For instance $r^1({{a}'}_{e_1})=r^1(\omega(a)_{e_1})=b$ because the second edge of $\omega(a)=ab$ is $b$.

The matrices for the parent maps $g'^k$ are 
$$
g'^0=\begin{blockarray}{ccccc}
 &a'\!.a' & a'\!.b' & b'\!.a' & {a'}\!_{v_1} \\
\begin{block}{c(cccc)}
  a'\!.a' & 1 & 0 & 0 & 0  \\
  a'\!.b' & 0 & 1 & 0 & 0  \\
  b'\!.a' & 0 & 0 & 1 & 0  \\
\end{block}
\end{blockarray}
$$
$$
g'^1=\begin{blockarray}{cccc}
 &{a'}\!\!_{e_0} &{a'}\!\!_{e_1} &{b'}\!\!_{e_0}  \\
\begin{block}{c(ccc)}
  a' & 1 & 1 & 0  \\
  b' & 0 & 0 & 1  \\
\end{block}
\end{blockarray}.
$$

The matrices for the relabeled parent maps $g^k$ are 
$$
g^0=\begin{blockarray}{ccccc}
 &a'\!.a' & a'\!.b' & b'\!.a' & {a'}\!\!_{v_1} \\
\begin{block}{c(cccc)}
  a.a & 1 & 0 & 0 & 0  \\
  a.b & 0 & 1 & 0 & 0  \\
  b.a & 0 & 0 & 1 & 0  \\
\end{block}
\end{blockarray}
$$
$$
g^1=\begin{blockarray}{cccc}
 &{a'}\!\!_{e_0} &{a'}\!\!_{e_1} &{b'}\!\!_{e_0}  \\
\begin{block}{c(ccc)}
  a & 1 & 1 & 0  \\
  b & 0 & 0 & 1  \\
\end{block}
\end{blockarray}.
$$
The matrices for the section maps $s'^k$ are
$$
s'^0=\begin{blockarray}{cccc}
 &a'\!.a' & a'\!.b' & b'\!.a'  \\
\begin{block}{c(ccc)}
  a'\!.a' 		& 1 & 0 & 0   \\
  a'\!.b' 		& 0 & 1 & 0   \\
  b'\!.a' 		& 0 & 0 & 1   \\
{a'}\!\!_{v_1} & 1 & 1 & 0  \\
\end{block}
\end{blockarray}
$$

$$
s'^1=\begin{blockarray}{ccc}
 &a' & b'   \\
\begin{block}{c(cc)}
  {a'}\!\!_{e_0} & 1 & 0  \\
  {a'}\!\!_{e_1} & 0 & 0   \\
  {b'}\!\!_{e_0} & 0 & 1   \\
\end{block}
\end{blockarray}.
$$

The matrices for the section maps $s^k$ are
$$
s^0=\begin{blockarray}{cccc}
 &a\!.a & a\!.b & b\!.a  \\
\begin{block}{c(ccc)}
  a'\!.a' 		& 1 & 0 & 0   \\
  a'\!.b' 		& 0 & 1 & 0   \\
  b'\!.a' 		& 0 & 0 & 1   \\
{a'}\!\!_{v_1} & 1 & 1 & 0  \\
\end{block}
\end{blockarray}
$$

$$
s^1=\begin{blockarray}{ccc}
 &a & b   \\
\begin{block}{c(cc)}
  {a'}\!\!_{e_0} & 1 & 0  \\
  {a'}\!\!_{e_1} & 0 & 0   \\
  {b'}\!\!_{e_0} & 0 & 1   \\
\end{block}
\end{blockarray}.
$$

It is easy to check that $g^k\circ s^k=id$, as expected by Lemma \ref{l:gs=id}.
It is also easy to check that  $r^0 \circ s^0=W_V$ in Eq.~(\ref{e:Wv-fib}) and $r^1 \circ s^1=W_E$ in Eq.~(\ref{e:We-fib}), as expected by Lemma \ref{l:W=rs}.

\end{exam}



\section{\textbf{Asymptotic $C^*$-algebra $A$}}\label{s:A}
In this section we compute the $K$-theory of the asymptotic $C^*$-algebra $A$. To do so we will use the K\"unneh formula, and so we will first prove that both $S'$ and $S$ are UCT.
 
Recall that, by Section \ref{s:dynamics} and Section \ref{s:S}, the $C^*$-algebra $S$ is Morita equivalent to the transversal 
$S'=\lim\limits_{n\to\infty} (C_r^*(R_n),\iota_n)$.
\begin{thm}\label{t:An-UCT-amenable}
  For tilings of dimension $d=1,2$, the $C^*$-algebras $C_r^*(R_n)$ are type I. Hence $S'$ and $S$ are both  UCT and amenable.
\end{thm}
\begin{proof}
By Proposition \ref{p:compacts}(1),  $C_r^*(R_n|_{X_0})$ is isomorphic to a finite direct sum of the compacts.
Hence it is type $I$. By Proposition \ref{p:compacts}(2) $C_r^*(R_n|_{X_1-X_0})$ is isomorphic to a finite direct sum of $C_0( (0,1), K)$. Hence it is also type I.
Since the class of type I $C^*$-algebras is closed under extensions, $C_r^*(R_n|_{X_1})$ is type I by Eq.~(\ref{e:JBC}).
This proves that $C^*_r(R_n)$ is type I for tilings of dimension 1.

By Proposition \ref{p:compacts}(3) the ideal $C_r^*(R_n|_{X_2-X_1})$ is type $I$. Hence, by Eq.~(\ref{e:IAB}) $C_r^*(R_n)$ is type I.
Hence $C^*_r(R_n)$ is also type I for tilings of dimension 2.

Since type I is UCT (cf.~\cite[p.229]{Blackadar}), $C^*_r(R_n)$ is UCT. Since the class of UCT $C^*$-algebras is closed under direct limits (cf.~\cite[p.229]{Blackadar}) the stable transversal $C^*$-algebra $S'$ is UCT as well. Since type I $C^*$-algebras are amenable, and amenability is preserved under direct limits, $S'$ is amenable.
By the third line in \cite[p.229]{Blackadar}, UCT is preserved by Morita equivalence. Hence $S$ is UCT.
Amenability is also preserved by Morita equivalence. Hence $S$ is also amenable.
(Note: In \cite{Blackadar} the UCT class of $C^*$-algebras is called class N, cf.~\cite[Definition~22.3.4]{Blackadar}).
\end{proof}

\begin{cor}\label{c:CstarA-line}
  For tilings of dimension 1,
  \begin{eqnarray*}
  K_0(A)&=&\big(K_0(S)\otimes_{\Z} K_0(U)\big) \oplus \Z,\\
  K_1(A)&=&K_0(S) \oplus K_0(U),
  \end{eqnarray*}
  where $A$, $S$ and $U$ are the asymptotic, stable, and unstable $C^*$-algebras, respectively.
\end{cor}
\begin{proof}
By \cite[Theorem~3.1]{Putnam96SmaleSpaces} the $C^*$-algebras $A$ and $S \otimes_{max} U$ are strongly Morita equivalent.
The K\"unneth formula \cite[Theorem~23.1.3,~p.234]{Blackadar} holds here because:

1) $K_i(S)$, $K_i(U)$ $i=0,1$ are torsion free (Corollary \ref{c:torsionmovement}, Theorem \ref{t:Ktheory-tung}).

2) $S$ is UCT by the theorem.\\
The explicit formulas for the K\"unneth formula are
$$K_0(S \otimes U) = \big(K_0(S) \otimes_{\Z} K_0(U)\big) \oplus \big(K_1(S) \otimes_{\Z} K_1(U)\big),$$
$$K_1(S \otimes U) = \big(K_0(S) \otimes_{\Z} K_1(U)\big) \oplus \big(K_1(S) \otimes_{\Z} K_0(U)\big).$$
We do not need to worry about the kind of tensor product $S\otimes U$ since $S$ and $U$ are amenable.
Since $K_1(S)=\Z$ and $K_1(U)=\Z$ the corollary follows.
\end{proof}

For tilings of the plane we have, by Theorem \ref{t:Ktheory-tung},
that $K_1(S')$ has torsion only if $H_S^1(T)$ has torsion. 
Note that $K_0(S')$ cannot have torsion, because if an element has torsion in $K_0(S')$ then by injectivity of the inclusion in the short exact sequence in Theorem \ref{t:Ktheory-tung}, the element does not come from the ideal, as the ideal has no torsion, and thus the element must map to the quotient, with torsion. This is a contradiction since $H_S^0$ has no torsion by Corollary \ref{c:torsionmovement}.

By the K\"unneth formula, Corollary \ref{c:torsionmovement}, and by a similar proof as the one above, we have

\begin{cor}\label{c:CstarA-plane}
  For tilings of dimension 2, if $H_S^1(T)$ and $H_0^{ST}(T)$ are torsion free, then
  \begin{eqnarray*}
  K_0(A)&=&(K_0(S)\otimes_{\Z} K_0(U)) \oplus (K_1(S)\otimes_{\Z} K_1(U)),\\
  K_1(A)&=&(K_0(S)\otimes_{\Z} K_1(U)) \oplus (K_1(S)\otimes_{\Z} K_0(U)),
  \end{eqnarray*}
  where $A$, $S$ and $U$ are the asymptotic, stable, and unstable $C^*$-algebras, respectively.
\end{cor}

\section{\textbf{Examples}}\label{s:examples}
In this section, we calculate the stable and unstable $K$-theories for a number of tilings of the line and of the plane. 
We then use Corollary \ref{c:CstarA-line} and Corollary \ref{c:CstarA-plane} to obtain the asymptotic $K$-theories.
The simplification of the tensor products are done using Proposition \ref{p:tensorproduct}, Corollary \ref{c:tensorproduct} and Corollary \ref{c:tensorproduct2}.

Many of the following computations of direct limits could be done by hand. We choose however to use the general formulas stated in Appendix \ref{a.s:direct-limits-AG}, which rely heavily on the Smith normal form of integer matrices, in order to illustrate the use of them. 
Moreover, we have programmed these formulas in Mathematica. For instance, we have written functions in Mathematica that, in the absence of torsion, yield the isomorphisms $\ker A\cong\Z^r$, $\coker A\cong\Z^r$, $\frac{\ker A}{\im B}\cong\Z^r$ for integer matrices $A$, $B$. The Mathematica files for the examples found in this section can be downloaded at 
\begin{center}
\url{https://github.com/mariars/Tilings-Ktheory}
\end{center}
We should remark that in this section we use the notation for direct limit
$$\xymatrix{
&\lim\limits_{\to}(A,X):=&\!\!\!\!\!\!\!\!\!\!\!\!\!\!\!\!\lim\limits_{\to}X\ar[r]^{A}&X\ar[r]^{A}&X\ar[r]^{A}&,\\
}$$
where we write the matrix first, in order to emphasize it, in contrast to our previous notation where we wrote the group first.

\subsection{One dimensional tilings}
Let $T$ be a tiling of the line and let $e_1,\ldots,e_N\in T$ be the prototiles(=proto-edges).
For these tilings we can always put the homotopy that homotopes the leftmost edge in $\frac{1}{\lambda}\omega(e_i)$ to $e_i$.
Morever, we put the orientation of the vector $(1,0)\in\R^2$ to all the edges of the tiling.
\begin{exam}[Fibonacci tiling](cf.~\cite[Ex.~1,~p.30]{AP}).
Let $T$ be the Fibonacci tiling with proto-edges $a,b\in T$ and substitution rule $\omega(a):=ab$,\, $\omega(b):=a$.
The length of the interval $a$ is 1, and the length of $b$ is $1/\phi$, where $\phi$ is the golden ratio. The inflation factor $\lambda=\phi\approx 1.618$.
We illustrate this and the homotopy in the following figure:\\
\begin{tikzpicture}

\node at (10.7,0){$T_1=\frac1{\lambda}\omega(T)\lambda$};

\draw plot[mark=*,mark size=1] coordinates{(0.,0.1)};
\draw plot[mark=*,mark size=1] coordinates{(1.,0.1)};
\draw plot[mark=*,mark size=1] coordinates{(1.61803,0.1)};
\draw plot[mark=*,mark size=1] coordinates{(2.61803,0.1)};
\draw plot[mark=*,mark size=1] coordinates{(3.61803,0.1)};
\draw plot[mark=*,mark size=2] coordinates{(4.23607,0.1)};
\draw plot[mark=*,mark size=1] coordinates{(5.23607,0.1)};
\draw plot[mark=*,mark size=1] coordinates{(5.8541,0.1)};
\draw plot[mark=*,mark size=1] coordinates{(6.8541,0.1)};
\draw plot[mark=*,mark size=1] coordinates{(7.8541,0.1)};
\draw plot[mark=*,mark size=1] coordinates{(8.47214,0.1)};
\draw[blue,thick] (0.,0.1)--(1.,0.1);
\draw[red,thick] (1.,0.1)--(1.61803,0.1);
\draw[blue,thick] (1.61803,0.1)--(2.61803,0.1);
\draw[blue,thick] (2.61803,0.1)--(3.61803,0.1);
\draw[red,thick] (3.61803,0.1)--(4.23607,0.1);
\draw[blue,thick] (4.23607,0.1)--(5.23607,0.1);
\draw[red,thick] (5.23607,0.1)--(5.8541,0.1);
\draw[blue,thick] (5.8541,0.1)--(6.8541,0.1);
\draw[blue,thick] (6.8541,0.1)--(7.8541,0.1);
\draw[red,thick] (7.8541,0.1)--(8.47214,0.1);
\draw[blue,thick] (8.47214,0.1)--(9.47214,0.1);

\node at (0.5,0.4){$\frac{a}{\lambda}$};
\node at (1.30902,0.4){$\frac{b}{\lambda}$};
\node at (2.11803,0.4){$\frac{a}{\lambda}$};
\node at (3.11803,0.4){$\frac{a}{\lambda}$};
\node at (3.92705,0.4){$\frac{b}{\lambda}$};
\node at (4.73607,0.4){$\frac{a}{\lambda}$};
\node at (5.54508,0.4){$\frac{b}{\lambda}$};
\node at (6.3541,0.4){$\frac{a}{\lambda}$};
\node at (7.3541,0.4){$\frac{a}{\lambda}$};
\node at (8.16312,0.4){$\frac{b}{\lambda}$};
\node at (8.97214,0.4){$\frac{a}{\lambda}$};

\node at (10.7,-.6){$h_s$};
\draw[->,black,thin,dashed] (0.,0.1)--(0,-1.1);
\draw[->,black,thin,dashed] (1.,0.1)--(1.61803,-1.1);
\draw[->,black,thin,dashed] (1.61803,0.1)--(1.61803,-1.1);
\draw[->,black,thin,dashed] (2.61803,0.1)--(2.61803,-1.1);
\draw[->,black,thin,dashed] (3.61803,0.1)--(4.23607,-1.1);
\draw[->,black,thin,dashed] (4.23607,0.1)--(4.23607,-1.1);
\draw[->,black,thin,dashed] (5.23607,0.1)--(5.8541,-1.1);
\draw[->,black,thin,dashed] (5.8541,0.1)--(5.8541,-1.1);
\draw[->,black,thin,dashed] (6.8541,0.1)--(6.8541,-1.1);
\draw[->,black,thin,dashed] (7.8541,0.1)--(8.47214,-1.1);
\draw[->,black,thin,dashed] (8.47214,0.1)--(8.47214,-1.1);
\draw[->,black,thin,dashed] (9.47214,0.1)--(9.4721,-1.1);

\node at (10.7,-1.3){$T_0:=T$};
\draw plot[mark=*,mark size=1] coordinates{(0.,-1.2)};
\draw plot[mark=*,mark size=1] coordinates{(1.61803,-1.2)};
\draw plot[mark=*,mark size=1] coordinates{(2.61803,-1.2)};
\draw plot[mark=*,mark size=2] coordinates{(4.23607,-1.2)};
\draw plot[mark=*,mark size=1] coordinates{(5.8541,-1.2)};
\draw plot[mark=*,mark size=1] coordinates{(6.8541,-1.2)};
\draw plot[mark=*,mark size=1] coordinates{(8.47214,-1.2)};
\draw plot[mark=*,mark size=1] coordinates{(9.47214,-1.2)};
\draw[blue,thick] (0.,-1.2)--(1.61803,-1.2);
\draw[red,thick] (1.61803,-1.2)--(2.61803,-1.2);
\draw[blue,thick] (2.61803,-1.2)--(4.23607,-1.2);
\draw[blue,thick] (4.23607,-1.2)--(5.8541,-1.2);
\draw[red,thick] (5.8541,-1.2)--(6.8541,-1.2);
\draw[blue,thick] (6.8541,-1.2)--(8.47214,-1.2);
\draw[red,thick] (8.47214,-1.2)--(9.47214,-1.2);

\node at (0.809017,-1.5){$a$};
\node at (2.11803,-1.5){$b$};
\node at (3.42705,-1.5){$a$};
\node at (5.04508,-1.5){$a$};
\node at (6.3541,-1.5){$b$};
\node at (7.66312,-1.5){$a$};
\node at (8.97214,-1.5){$b$};

\end{tikzpicture}\\

\begin{itemize}
\item stable edges (2): $a$, $b$.
\item stable vertices (3): $a.a$, $a.b$, $b.a$
\end{itemize}
Note that $b.b$ never occurs.
Recall that the stable edges are always the proto-edges for 1-dimensional tilings.
Since $\omega^3(a)$ and $\omega^4(a)$ both contain precisely the stable vertices $a.a$, $a.b$, $b.a$ and $w(b)=a$, these are all the stable vertices in the tiling.
The prototiles of the shrunk tiling $T_1:=\frac{1}{\lambda}\omega(T)\lambda$ are $a':=\frac{a}{\lambda}$, $b':=\frac{b}{\lambda}$.

The substitution-homotopy map  $W_E:\Z^{sE}\to \Z^{sE}$ is given by
$W_E(a)=a'$, $W_E(b)=a'$. Thus its matrix is
\begin{equation}\label{e:We-fib}
W_E=\begin{blockarray}{ccc}
 &a  &b  \\
\begin{block}{c(cc)}
  a' & 1 & 1 \\
  b' & 0 & 0 \\
\end{block}
\end{blockarray}.
\end{equation}
The substitution-homotopy map  $W_V:\Z^{sV}\to \Z^{sV}$ is given by
$$W_V(a.a)=a'.b'+b'.a',\qquad  W_V(a.b)=a'.b'+b'.a',\qquad W_V(b.a)=a'.a'$$
 Thus its matrix is
\begin{equation}\label{e:Wv-fib}
W_V=\begin{blockarray}{cccc}
 &a.a  &a.b &b.a  \\
\begin{block}{c(ccc)}
  a'.a' & 0 & 0 & 1\\
  a'.b' & 1 & 1 & 0\\
  b'.a' & 1 & 1 & 0\\
\end{block}
\end{blockarray}.
\end{equation}
The exponential map $\delta^0:\Z^{sV}\to\Z^{sE}$ is given by
$$\delta^0(a.a)=a-a=0,\quad \delta^0(a.b)=a-b\qquad \delta^0(b.a)=b-a.$$
Thus its matrix is
$$
\delta^0=\begin{blockarray}{cccc}
 &a.a  &a.b &b.a  \\
\begin{block}{c(ccc)}
  a & 0 & 1 & -1\\
  b & 0 & -1 & 1\\
\end{block}
\end{blockarray}.
$$
By Proposition \ref{p:kerA-isom-Znminusr},
$$\lim_{\to} (W_V,\, \ker\delta^0)=\lim_{\to}(p\, W_V \, q,\,\Z^{n-r}),$$ 
where $p:\ker\delta^0\to \Z^{n-r}$ and $q:\Z^{n-r}\to\ker\delta^0$ are $\Z$-isomorphisms defined in the proposition.
Since 
$$W_V':=p\, W_V \,q= \left( \begin{array}{cc}
    0 & 1 \\ 
    1 & 1 \\ 
  \end{array}
\right)$$
is symmetric (i.e.~$W_V'$ is equal to its  transpose)  the zero stable cohomology group and zero stable-transpose homology group are the same:
$$H_S^0(T)=\lim_{\to}(W_V,\,\ker\delta^0)=\lim_{\to}(W_V',\,\Z^{n-r})=$$
$$\lim_{\to}((W_V')^t,\,\Z^{n-r})=\lim_{\to}((W_V)^t,\,(\ker\delta^0)^t)=H_0^{ST}(T).$$
Thus the zero $K$-groups for the stable and unstable $C^*$-algebras are the same:
$$K_0(S)=H_S^0(T)=H^{ST}_0(T)=K_0(U).$$
Since $W_V'$ has determinant 1, it is $\Z$-invertible and thus 
$$K_0(U)=K_0(S)=H^0_S(T)=\lim_{\to}(W_V',\,\Z^2)=\Z^2.$$
\end{exam}

\begin{exam}[Morse Tiling](cf.~\cite[p.33]{AP}).
Let $T$ be the Morse tiling with proto-edges $a,b\in T$ and substitution rule $\omega(a)=ab$, $\omega(b)=ba$. 
The length of the edges $a$, $b$ are 1, and  inflation factor is $\lambda=2$.
We follow the same procedure as for the Fibonacci tiling so we omit most of the details.
\begin{itemize}
\item stable edges (2): $a$, $b$
\item stable vertices (4): $a.a$, $a.b$, $b.a$, $b.b$ 
\end{itemize}
$$
W_V=\left(
\begin{array}{cccc}
 0 & 0 & 1 & 0 \\
 1 & 1 & 0 & 1 \\
 1 & 0 & 1 & 1 \\
 0 & 1 & 0 & 0 \\
\end{array}
\right),\qquad W_E=\left(
\begin{array}{cc}
 1 & 0 \\
 0 & 1 \\
\end{array}
\right),\qquad
\delta^0=\left(
\begin{array}{cccc}
 0 & 1 & -1 & 0 \\
 0 & -1 & 1 & 0 \\
\end{array}
\right).$$
By Proposition \ref{p:kerA-isom-Znminusr},
$$\lim_{\to} (W_V,\, \ker\delta^0)=\lim_{\to}(p\, W_V \, q,\,\Z^{n-r}),$$ 
where $p:\ker\delta^0\to \Z^{n-r}$ and $q:\Z^{n-r}\to\ker\delta^0$ are the $\Z$-isomorphisms defined in the proposition.
Since 
$$
W_V':=p\, W_V \, q=\left(
\begin{array}{ccc}
 0 & 1 & 0 \\
 1 & 1 & 1 \\
 0 & 1 & 0 \\
\end{array}
\right)
$$
is equal to its transpose, $K_0(S)=H_S^0(T)=H^{ST}_0(T)=K_0(U).$

Using Proposition \ref{p:imA-isom-Zr} we remove the zero eigenvalues of $W_V'$, i.e.
$$\lim_{\to} (W_V',\, \ker\delta^0)=\lim_{\to}(p'\, W_V' \, q',\,\Z^{n-r}),$$
where $p':W_V'\Z^n\to \Z^r$ and $q':\Z^r \to W_V'\Z^n$ are the maps defined in the proposition, and computing we get
$$
W_V'':=p'\,W_V'\,q'=\left(
\begin{array}{cc}
 1 & 1 \\
 2 & 0 \\
\end{array}
\right).
$$

Using Proposition \ref{p:extract-one-eigenvalues} we extract the eigenvalue -1 of $W_V''$, i.e.
$$\xymatrix@C=.7cm{0\ar[r]&\lim\limits_{\to} (W_V'',\, \ker q(W_V''))\ar@{^(->}[r]&\lim\limits_{\to} (W_V'',\, \Z^n)\ar@{>>}[r]&\lim\limits_{\to} (\lambda I_n,\, q(W_V'')\Z^n)\ar[r]&0},$$
where $p(x)=(x+1)(x-2)=(x+1)q(x)$ is the minimal polynomial of $W_V''$.  
By Proposition \ref{p:kerA-isom-Znminusr}, $\lim\limits_{\to}(W_V'',\ker q(W_V'')) = \lim\limits_{\to}( 2, \Z) = \Z[1/2]$, 
and by Proposition \ref{p:imA-isom-Zr} $\lim\limits_{\to}(W_V'', \im q(W_V'')) = \lim\limits_{\to}(-I, \im q(W_V'')) = \lim\limits_{\to}( -1, \Z) = \Z$. Since the short exact sequence
$$\xymatrix{0\ar[r]&\Z[\frac12]\ar@{^(->}[r]&\lim\limits_{\to} (W_V'',\, \Z^n)\ar@{>>}[r]&\Z\ar[r]&0}$$
splits, we get $$K_0(U)=H^{ST}_0(T)=H_S^0(T)=K_0(S)=\Z[\frac12]\oplus\Z.$$
\end{exam}

\begin{exam}[Pathologic]
Let $T$ be the Pathologic tiling with proto-edges $a,b\in T$ and substitution rule $\omega(a)=babbaaa$, $\omega(b)=abbbbb$.
The length of edge $a$ is $\frac{\sqrt{13}-1}2\approx 1.30$ and of $b$ is $1$, and the inflation factor is $\lambda=\frac{9+\sqrt{13}}2\approx 6.30$.
(To compute the lengths see \cite[Section~8,~p.26]{AP}).
We follow same procedure as for the Fibonacci tiling and Morse tiling so we omit most of the details.
\begin{itemize}
\item stable edges (2): $a$, $b$
\item stable vertices (4): $a.a$, $a.b$, $b.a$, $b.b$ 
\end{itemize}
$$W_V=\left(
\begin{array}{cccc}
 2 & 3 & 0 & 0 \\
 2 & 1 & 1 & 1 \\
 2 & 2 & 0 & 1 \\
 1 & 1 & 5 & 4 \\
\end{array}
\right),
\qquad
W_E=\left(
\begin{array}{cc}
 0 & 1 \\
 1 & 0 \\
\end{array}
\right),
\qquad
\delta^0=\left(
\begin{array}{cccc}
 0 & 1 & -1 & 0 \\
 0 & -1 & 1 & 0 \\
\end{array}
\right).
$$

By Proposition \ref{p:kerA-isom-Znminusr},
$$\lim_{\to} (W_V,\, \ker\delta^0)=\lim_{\to}(p\, W_V \, q,\,\Z^{n-r}),$$ 
where $p:\ker\delta^0\to \Z^{n-r}$ and $q:\Z^{n-r}\to\ker\delta^0$ are the $\Z$-isomorphisms defined in the proposition, and we get
$$
W_V':=p\, W_V \, q=\left(
\begin{array}{ccc}
 2 & 3 & 0 \\
 2 & 2 & 1 \\
 1 & 6 & 4 \\
\end{array}
\right).
$$
We now extract the eigenvalue -1 of $W_V'$ as follows.
By the proof of Proposition \ref{p:extract-one-eigenvalues}, (and Propositions \ref{p:kerA-isom-Znminusr}, \ref{p:imA-isom-Zr}) we get the following commutative diagram with exact rows
$$\xymatrix{
0\ar[r]&\Z^2\ar[d]^{W_V''}\ar@{^(->}[r]&\Z^3\ar[d]^{W_V'}\ar@{>>}[r]^{\,\,\,\,\,\,\,\,}&\Z\ar[d]^{-1}\ar[r]&0\\
0\ar[r]& \Z^2\ar@{^(->}[r]&\Z^3\ar@{>>}[r]^{\,\,\,\,\,\,\,\,}&\Z\ar[r]&0,
}$$
where
$$W_V'':=
\left(
\begin{array}{cc}
 7 & -1 \\
 3 & 2 \\
\end{array}
\right).
$$
We now get by the proof of Proposition \ref{p:extract-one-eigenvalues} that, 
$$\lim_{\to} (W_V',\, \Z^3)=\Z\oplus\lim_{\to} (W_V'',\, \Z^2).$$ 

Since
$$
W_V''':=\left(
\begin{array}{cc}
 1 & -1 \\
 1 & 0 \\
\end{array}
\right).\left(
\begin{array}{cc}
 7 & -1 \\
 3 & 2 \\
\end{array}
\right).\left(
\begin{array}{cc}
 1 & -1 \\
 1 & 0 \\
\end{array}
\right)^{-1}=\left(
\begin{array}{cc}
 3 & 1 \\
 1 & 6 \\
\end{array}
\right)
$$
we get
$$\lim_{\to} (W_V'',\, \Z^3)=\Z\oplus\lim_{\to} (W_V''',\, \Z^2).$$ 
Similarly, using Proposition \ref{p:extract-one-eigenvalues}, we extract the eigenvalue -1 of the transpose of $W_V'$ and get
$$\lim_{\to} ((W_V')^t,\, \Z^3)=\Z\oplus\lim_{\to} (
\left(
\begin{array}{cc}
 5 & 3 \\
 1 & 4 \\
\end{array}
\right)
,\, \Z^2).$$
Since 
$$
\left(
\begin{array}{cc}
 4 & -1 \\
 -3 & 1 \\
\end{array}
\right)^{-1}.\left(
\begin{array}{cc}
 5 & 3 \\
 1 & 4 \\
\end{array}
\right).\left(
\begin{array}{cc}
 4 & -1 \\
 -3 & 1 \\
\end{array}
\right)=\left(
\begin{array}{cc}
 3 & 1 \\
 1 & 6 \\
\end{array}
\right),$$
we get 
$$K_0(S)=H_S^0(T)=H^{ST}_0(T)=K_0(U)=\Z\oplus \lim_{\to}(\left(
\begin{array}{cc}
 3 & 1 \\
 1 & 6 \\
\end{array}
\right),\Z^2)
=\Z\oplus \lim_{\to} (W_V''',\Z^2)
.$$
We now show that $\lim\limits_{\to}(W_V''',\Z^2)$ cannot be written as a direct sum even though it has rank two!
First note that the group $\lim\limits_{\to}(W_V''',\Z^2)$ is of rank 2, as the matrix has determinant $17$.
Since the eigenvalues of $W_V'''$ are $\frac{9\pm \sqrt13}{2}$, hence two distinct irrational numbers,
the minimal polynomial for $W_V'''$ is the same as the characteristic polynomial for $W_V'''$, and is irreducible over $\Q$.
Let $\lambda:=\frac{9+\sqrt13}{2}$.
By  \cite[Prop.~4]{Dugas} the direct limit $\lim\limits_{\to}(W_V''',\Z^2)$ is quasi-isomorphic (as abelian groups) to 
$$L_{\lambda}:=R[\frac1{\lambda}]=\{\frac{q}{\lambda^n}\mid q\in R,\,n\in \Z\},$$
where $R$ is the ring of algebraic integers in the quadratic extension $\Q[\lambda]=\Q[\sqrt{13}]=\{q_1+q_2\sqrt{13}\mid q_1,q_2\in \Q\}$. That is,
$$R=\{\frac{m+n\sqrt{13}}{2}\mid m,n\in\Z\},$$ 
since $13$ is congruent to $1$ mod $4$. Note that $\lambda\in R$.
Since $13$ is a prime number, it is a square free integer.
Moreover, it is not a unit in $R$ since $\lambda^{-1}=\frac{9-\sqrt{13}}{34}\not\in R$.
Furthermore there is no prime number $p$ such that $\lambda \in p R$ because the equation
$$\frac{mp+n p\sqrt{13}}{2}=\frac{9+\sqrt{13}}{2}=\lambda$$
implies that $p$ must divide $9$ and $1$.
Similarly, there is no prime number $p$ such that $\lambda^2\in p R$ because the equation
$$\frac{mp+n p\sqrt{13}}{2}=\frac{47+9\sqrt{13}}{2}=\lambda^2$$ 
implies that $p$ must divide $47$ and $9$.
Hence, by  \cite[Prop.~9]{Dugas} $\lambda$ is strong. 
By \cite[Thm.~6]{Dugas} and the remark below the theorem, $L_{\lambda}$ is an E-ring whose additive group is strongly indecomposable (cf.~\cite[Thm.~14.3, p.~163]{Arnold}).
Hence by \cite[Thm.~6, p.~49]{Reid} $\lim\limits_{\to}(W_V''',\Z^2)$ is strongly indecomposable, since $\lim\limits_{\to}(W_V''',\Z^2)$ is quasi-isomorphic to the additive group of $L_{\lambda}$.
Recall that a ring R is said to be an E-ring if $R^+$, the additive group of $R$, is group isomorphic to $\mathrm{End}(R)$ via the map $r\mapsto (x\mapsto rx)$.
A group $G$ is said to be strongly indecomposable if for all integers $n\ne 0$ and all abelian subgroups $A,B$:
 $$(nG\subset A\oplus B\subset G) \imply (A=0 \text{ or } B=0).$$
A group $G$ is said to be indecomposable if for all abelian subgroups $A,B$:
$$(G=A\oplus B) \imply (A=0 \text{ or } B=0).$$
Thus $\lim\limits_{\to}(W_V''',\Z^2)$ is an indecomposable group, i.e.~it cannot be written as a direct sum of two nonzero abelian subgroups!
\end{exam}
In the following example we discuss briefly the computation of the stable cohomology and stable-transpose homology for several more tilings of the line.
\begin{exam}[more 1-dimensional tilings]\label{e:more1dimexamples}$ $\\
\textbf{Tiling OneFifth}:(cf.~\cite[Example~2.6]{Gahler}).
\begin{itemize}
\item Inflation factor: $\lambda=\frac{5+\sqrt{5}}2=3.618$
\item Proto-edges(=stable edges) (2): $a$,$b$,   $|a|=\frac{\sqrt{5}-1}2=0.618$, $|b|=1$
\item substitution: $\omega(a)=aba$, $\omega(b)=bbab$
\item Stable vertices (3): $a.b$, $b.a$, $b.b$
\item $K_0(U)=H^{ST}_0(T)=H_S^0(T)=K_0(S)=\lim\limits_{\to}(
\left(
\begin{array}{cc}
 3 & 1 \\
 1 & 2 \\
\end{array}
\right)
,\Z^2)=\Z[\frac15]^2$.
\end{itemize}
The substitution-homotopy matrix $W_V$ is calculated in a similar way as for the Fibonacci tiling. 
We then use Proposition \ref{p:kerA-isom-Znminusr} to compute the matrix $W_V':=\left(
\begin{array}{cc}
 3 & 1 \\
 1 & 2 \\
\end{array}
\right)
$. Since it is equal to its transpose, the stable 0-cohomology group equals the stable-transpose 0-homology group. By Proposition \ref{p:lim-union}, 
$$\lim\limits_{\to}(W_V',\Z^2)\subset \Z[\frac1{5}]^2.$$
We then check with a computer that the powers $5^k W_V'^{-k}$ are integer matrices for small values of $k\in\N$.
Diagonalizing the matrix $W_V'$ in $\C$, we compute explicitly $W_V'^{-k}$, and using an induction argument we show that the above inclusion is actually an equality.
We learn from this example that this direct limit is of the form $\Z[\frac1{\det A}]^2$.\\
\\
\noindent
\textbf{Tiling OneSixth}:(cf.~\cite[Example~1.21]{Gahler}).
\begin{itemize}
\item Inflation factor: $\lambda=3+\sqrt{3}\approx4.73$
\item Proto-edges(=stable edges) (2): $a$,$b$,   $|a|=\sqrt{3}\approx 1.73$, $|b|=1$
\item substitution: $\omega(a)=bbaaab$, $\omega(b)=bbab$
\item Stable vertices (4): $a.a$, $a.b$, $b.a$, $b.b$
\item $K_0(U)=H^{ST}_0(T)=H_S^0(T)=K_0(S)=\lim\limits_{\to}(
\left(
\begin{array}{cc}
 6 & -2 \\
 3 & 0 \\
\end{array}
\right),\Z^2)=\Z[\frac16]^2$
\item $H_S^0(T)=K_0(S)=\lim\limits_{\to}(
\left(
\begin{array}{cc}
 6 & -2 \\
 3 & 0 \\
\end{array}
\right),\Z^2)=\Z[\frac16]^2$
\item $K_0(U)=H^{ST}_0(T)=\lim\limits_{\to}(
\left(
\begin{array}{cc}
 6 & 3 \\
 -2 & 0 \\
\end{array}
\right),\Z^2)=\Z[\frac16]^2.$
\end{itemize}
The substitution-homotopy matrix $W_V$ is calculated in a similar way as for the Fibonacci tiling. 
We then use Proposition \ref{p:kerA-isom-Znminusr}  and  Proposition \ref{p:imA-isom-Zr} to compute the matrix $W_V':=\left(
\begin{array}{cc}
 6 & -2 \\
 3 & 0 \\
\end{array}
\right).$
The direct limit is computed the same way as we did for the tiling OneFifth and is also of the form $\Z[\frac1{\det A}]^2$.
Moreover, this example shows that the collared equivalence relation $R_c$  given in Definition \ref{d:collaredcells} is different from 
the equivalence relation $\sim_1$ defined in \cite[Section~4]{AP}, even though both yield the same \v{C}ech cohomology  (cf.~Section \ref{s:preliminaries}).
\\\\
\noindent
\textbf{Tiling nonReducible-4-Letter}:
\begin{itemize}
\item Inflation factor: $\lambda\approx2.508$
\item Proto-edges(=stable edges) (4): $a$, $b$, $c$, $d$   $|a|\approx 1.966$, $|b|\approx0.542$, $|c|\approx0.359$, $|d|=1$
\item substitution: $\omega(a)=aad$, $\omega(b)=cd$, $\omega(c)=cb$, $\omega(d)=ab$
\item Stable vertices (9): $a.a$,\, $a.b$,\, $a.d$,\, $b.a$,\, $b.c$,\, $c.b$,\, $c.d$,\, $d.a$,\, $d.c$
\item $K_0(U)=H^{ST}_0(T)=H_S^0(T)=K_0(S)=\Z^5$
\end{itemize}
We compute the substitution-homotopy matrix $W_V$ in a similar way as we did for the Fibonacci tiling. 
We then use Proposition \ref{p:kerA-isom-Znminusr}  and  Proposition \ref{p:imA-isom-Zr} to remove the eigenvalue 0. The result matrix is a $5\times 5$ matrix with determinant 1.
\\\\
\noindent
\textbf{Tiling PeriodDoubling}:
\begin{itemize}
\item Inflation factor: $\lambda=2$
\item Proto-edges(=stable edges) (2): $a$, $b$,    $|a|=1$, $|b|=1$
\item substitution: $\omega(a)=bb$, $\omega(b)=ba$
\item Stable vertices (3): $a.b$,\, $b.a$,\, $b.b$
\item $K_0(U)=H^{ST}_0(T)=H_S^0(T)=K_0(S)=\Z\oplus\Z[\frac12]$
\end{itemize}
We compute the substitution-homotopy matrix $W_V$ in a similar way as we did for the Fibonacci tiling. 
We then use Proposition \ref{p:kerA-isom-Znminusr}  to get the matrix $\lim\limits_{\to}(
\left(
\begin{array}{cc}
 1 & 1 \\
 2 & 0 \\
\end{array}
\right)
\,,\Z^2)$, which also occurs in the Morse tiling example, hence the same direct limit is obtained.
\\\\
\noindent
\textbf{Tiling Rauzy}:
\begin{itemize}
\item Inflation factor: $\lambda\approx 1.839$
\item Proto-edges(=stable edges) (3): $a$, $b$, $c$    $|a|\approx 1.839$, $|b|\approx 1.543$, $|c|=1$
\item substitution: $\omega(a)=ab$, $\omega(b)=ac$, $\omega(c)=a$
\item Stable vertices (5): $a.a$,\, $a.b$,\, $a.c$,\, $b.a$,\, $c.a$
\item $K_0(U)=H^{ST}_0(T)=H_S^0(T)=K_0(S)=\Z^3$
\end{itemize}
We compute the substitution-homotopy matrix $W_V$ in a similar way as we did for the Fibonacci tiling. 
We then use Proposition \ref{p:kerA-isom-Znminusr}  to get a $3\times3$ matrix with determinant 1.
\\\\
\noindent
\textbf{Tiling Rudin-Shapiro}:
\begin{itemize}
\item Inflation factor: $\lambda= 2$
\item Proto-edges(=stable edges) (4): $a$, $b$, $c$, $d$    $|a|=1$, $|b|=1$, $|c|=1$, $|d|=1$
\item substitution: $\omega(a)=ab$, $\omega(b)=ac$, $\omega(c)=db$, $\omega(d)=dc$
\item Stable vertices (8): $a.b$,\, $a.c$,\, $b.a$,\, $b.d$,\, $c.a$,\, $c.d$,\, $d.b$,\, $d.c$
\item $K_0(U)=H^{ST}_0(T)=H_S^0(T)=K_0(S)=\Z\oplus\Z[\frac12]^3$
\end{itemize}
We compute the substitution-homotopy matrix $W_V$ in a similar way as we did for the Fibonacci tiling. 
We then use Proposition \ref{p:kerA-isom-Znminusr} and Proposition \ref{p:imA-isom-Zr} to remove the eigenvalue 0.
We then extract the eigenvalue -1 with Proposition \ref{p:extract-one-eigenvalues} and get 
$$\lim\limits_{\to} (W_V,\ker\delta^0)=\Z\oplus \lim\limits_{\to} (W_V',\Z^3),$$
where
$$W_V':=\left(
\begin{array}{ccc}
 -2 & 1 & -1 \\
 -1 & 2 & 0 \\
 1 & 0 & 2 \\
\end{array}
\right),
\qquad
W_V'\,^2=\left(
\begin{array}{ccc}
 2 & 0 & 0 \\
 0 & 3 & 1 \\
 0 & 1 & 3 \\
\end{array}
\right).$$
By Proposition \ref{p:Asqr} and Proposition \ref{p:lim:direct-sum},
$$\lim\limits_{\to}(W_V',\Z^3)=\Z[\frac12]\oplus\lim\limits_{\to}(\left(
\begin{array}{cc}
 3 & 1 \\
 1 & 3 \\
\end{array}
\right)
,\,\Z^2).$$
Since 
$$\left(
\begin{array}{cc}
 1 & 1 \\
 0 & 1 \\
\end{array}
\right).\left(
\begin{array}{cc}
 3 & 1 \\
 1 & 3 \\
\end{array}
\right).\left(
\begin{array}{cc}
 1 & 1 \\
 0 & 1 \\
\end{array}
\right)^{-1}=\left(
\begin{array}{cc}
 4 & 0 \\
 1 & 2 \\
\end{array}
\right)$$
we get, by Proposition \ref{p:lim:Ginv}, that
$$
\lim\limits_{\to}(\left(
\begin{array}{cc}
 3 & 1 \\
 1 & 3 \\
\end{array}
\right)
,\,\Z^2)=
\lim\limits_{\to}(
\left(
\begin{array}{cc}
 4 & 0 \\
 1 & 2 \\
\end{array}
\right),\Z^2)=\lim\limits_{\to}(
A,\Z^2)
$$
where $A=\left(
\begin{array}{cc}
 4 & 0 \\
 1 & 2 \\
\end{array}
\right)$.
It is clear that $A^{-n} \Z^2$ is a subset of $\Z[\frac12]^2$ because every entry is of that form. The other inclusion is by the following argument:
  The second coordinate $A^{-n}.( 0, k)$ is obviously all the dyadic numbers for all integers $k$ and all $n$.
  The first coordinate $A^{-n}.(k,0)$ gives all the dyadics  in the first entry and some numbers in the second entry, which we know are dyadics, so we just subtract that and in this way we get all dyadic numbers in the first entry.
Hence by Proposition \ref{p:lim-union}, $\lim\limits_{\to}(A,\Z^2)=\Z[\frac12]^2$, and thus $H_S^0(T)=\lim\limits_{\to} (W_V,\ker\delta^0)=\Z\oplus\Z[\frac12]^3$.
\end{exam}
The $K$-theory groups of the above examples are summarized in Table \ref{table:K0groups-d1} and Table \ref{table:K1groups-d1}.
For computing the $K$-theory of the asymptotic  $C^*$-algebra $A$ we used the formula given in Proposition \ref{p:tensorproduct}.
We are the first ones to compute the stable and unstable $K$-theories of the ``Pathologic'' tiling. The unstable $K$-theories of the rest of the above examples are already well-known, and they agree with our computations.

\begin{table}[t]
\begin{center}\caption{$K_0$-groups for tilings of the line. 
  }\label{table:K0groups-d1}
  \begin{tabular}{| r | c | c| c| c| }
    \hline
      Tiling       &  $K_0(U)=H_0^{ST}(T) $               & $K_0(S)=H_S^0(T)$               & $K_0(A)$ \\ \hline
     Fibonnaci             & $\Z^2$                  & $\Z^2$                 &   $\Z^5$   \\
     Morse                 & $\Z\oplus\Z[\frac12]$   & $\Z\oplus\Z[\frac12]$  & $\Z^2\oplus\Z[\frac12]^3$\\
     Nonreducible 4 letter & $\Z^5$                  & $\Z^5$                 & $\Z^{26}$ \\
     Period doubling       & $\Z\oplus \Z[\frac12]$   & $\Z\oplus \Z[\frac12]$&  $\Z^2\oplus \Z[\frac12]^3$\\
     Rauzi (tribonacci)                 & $\Z^3$                  & $\Z^3$                 & $\Z^{10}$\\
     Rudin Shapiro         & $\Z\oplus \Z[\frac12]^3$ & $\Z\oplus \Z[\frac12]^3$& $\Z^2\oplus \Z[\frac12]^{15}$\\
     OneFifth              & $\Z[\frac15]^2$             & $\Z[\frac15]^2$    &    $\Z\oplus\Z[\frac15]^4$    \\  
     OneSixth              & $\Z[\frac16]^2$             & $\Z[\frac16]^2$    &    $\Z\oplus\Z[\frac16]^4$    \\  
     Pathologic              & $\Z\oplus\lim\limits_{\to} 
     {
       \left(\begin{array}{cc}
    3 & 1 \\ 
    1 & 6 \\ 
  \end{array}\right)
     }
      \subsetneqq\Z\oplus \Z[\frac1{17}]^2$             & $\Z\oplus \lim\limits_{\to}
     {
      \left(\begin{array}{cc}
    3 & 1 \\ 
    1 & 6 \\ 
  \end{array}\right)
     }
     $    &    $  $    \\  
    \hline
  \end{tabular}
    \footnotesize{Note: some authors prefer to write for example $\Z[1/4]$ instead of its reduced form $\Z[1/2]$.}
\end{center}
\end{table}

\begin{table}[t]
\caption{$K_1$-groups for tilings of the line.}\label{table:K1groups-d1}
\begin{center}
  \begin{tabular}{| r | c | c| c| c| }
    \hline
     Tiling        &  $K_1(U)=H_1^{ST}(T)$               & $K_1(S)=H_S^1(T)$              & $K_1(A)$ \\ \hline
     Fibonnaci             & $\Z$                    & $\Z$                  &  $\Z^4$   \\
     Morse                 & $\Z$                    & $\Z$                  & $\Z^2\oplus\Z[\frac12]^2$\\
     nonreducible 4 letter & $\Z$                    & $\Z$                  & $\Z^{10}$ \\
     Period doubling       & $\Z$                    & $\Z$                  &  $\Z^2\oplus \Z[\frac12]^2$\\
     Rauzi (tribonacci)                & $\Z$                    & $\Z$                  & $\Z^{6}$\\
     Rudin Shapiro         & $\Z$                    & $\Z$                  & $\Z^2\oplus \Z[\frac12]^{6}$\\
     OneFifth              & $\Z$             & $\Z$            &    $\Z[\frac15]^6$    \\
     OneSixth              & $\Z$             & $\Z$            &    $\Z[\frac16]^4$    \\
     Pathologic             & $\Z$             & $\Z$            &    $  $    \\
  \hline
  \end{tabular}
\end{center}
\end{table}

\subsection{Stable cells - Collared cells relationship}
Recall that for tilings of dimension 1, we can always use the homotopy that maps the leftmost edge of $\omega(e)$ to $e$ (the orientation of the edges is from left to right). In such case,  we can  write the substitution-homotopy matrices $W_V, W_E$ in terms of the collared-substitution matrices and the forgetful and inclusion maps defined as follows
\begin{center}
\begin{tabular}{ll}
 $F_E(ee'e''):=e'.e''$& $i_E(e'.e''):= ee'e''$\\
 $F_V(e.e'):=e'$ & $i_V(e'):=e.e',$
\end{tabular}
\end{center}
where $ee'e''$ is a collared edge, $e'$ is a stable edge, and $e.e'$ is a stable or collared vertex. Then it is easy to see that

\begin{thm}[Stable cells - Collared cells relationship] For any tiling $T\in \Omega$ of dimension 1 with the homotopy defined above, the following relations hold
\begin{itemize}
\item $W_E = F_V\circ \omega_V\circ i_V$
\item $W_V = F_E\circ \omega_E \circ i_E$
\item $\delta_0=F_V\circ\partial_1\circ i_E$
\item $F_V\circ\partial_1 = \delta_0 \circ F_E$
 \item $F_V \circ\omega_V = W_E \circ F_V$
 \item $F_E\circ\omega_E = W_V \circ F_E$.
\end{itemize}
\end{thm}
\noindent
The last relation in the above theorem gives rise to the following commutative diagram
\begin{equation}\label{e:cokercoker}
\xymatrix{\coker\partial_1^t\ar[r]^{\omega_E^t}&\coker\partial_1^t\ar[r]^{\omega_E^t}&\cdots\\
\coker\delta_0^t\ar[r]^{W_V^t}\ar[u]^{F_E^t}&\coker\delta_0^t\ar[r]^{W_V^t}\ar[u]^{F_E^t}&\cdots\,.}
\end{equation}
\noindent
All the above one dimensional tilings satisfied
\begin{center} 
\begin{tabular}{ccc}
$H_0^{ST}(T)$& $\quad K_0(U)$ & $\quad K_0(S)$\\
$\lim(\coker\delta_0^t,W_V^t)$ =& $\lim (\coker\partial_1^t, \omega_E^t)$ = &$\lim(\ker \delta_0, W_V).$
\end{tabular}
\end{center}
The first equality holds in general for any 1-dimensional tiling by Theorem \ref{t:unstableKtheory}. See also Corollary \ref{c:A-At}.
The second equality holds at least for the above examples, but we don't know if it is valid in general.
In all these 1-dimensional examples, except the tiling called OneSixth, the first equality was verified using diagram (\ref{e:cokercoker}) after removing some of the 0-eigenvalues; the second equality was verified empirically with the following 4 steps:
\begin{enumerate}
\item  Let $A_u: \Z^r \to \Z^r$ be the matrix for $\omega_E^t:\coker\partial_1^t \to \coker\partial_1^t$ with some of the zero eigenvalues removed.
\item Let $A_s: \Z^r \to \Z^r$ be the matrix for $W_V:\ker\delta_0 \to \ker\delta_0$ with some of the zero eigenvalues removed.
\item  Let the following be $\C$-diagonalization of the matrices $A_u$, $A_s$:
$$P_u D P_u^{-1}=A_u,\quad P_s D P_s^{-1}= A_s.$$
 \item If $P_uP_s^{-1}$ is a $\Z$-invertible integer matrix then
 $A_u$ and $A_s$ are $\Z$-similar.
\end{enumerate}
The matrix $P_uP_s^{-1}$ in Step 4 always yielded a $\Z$-invertible matrix for the above examples except for the tiling called OneSixth.

\subsection{Two dimensional tilings}[Block/Rectangular tilings.]
$ $
Let $T$ be a tiling of the plane, and let $f_1,\ldots,f_N\in T$ be the prototiles(=protofaces). 
If all the prototiles are rectangles or blocks (cf.~\cite[Section~8]{AP}) then one can always generalize the procedure that was used for the one-dimensional tilings.
Namely, for these ``block'' tilings one can always put the homotopy that homotopes the left-most bottom-most rectangle in $\frac{1}{\lambda}\omega(f_i)$  to $f_i$.
Moreover, we put the counterclockwise orientation on all the rectangles of the tiling, on its horizontal edges we put the orientation from left to right, and on its vertical edges we put the orientation from bottom to top.

\begin{figure}[h]
\centerline{
\includegraphics[scale=.3]{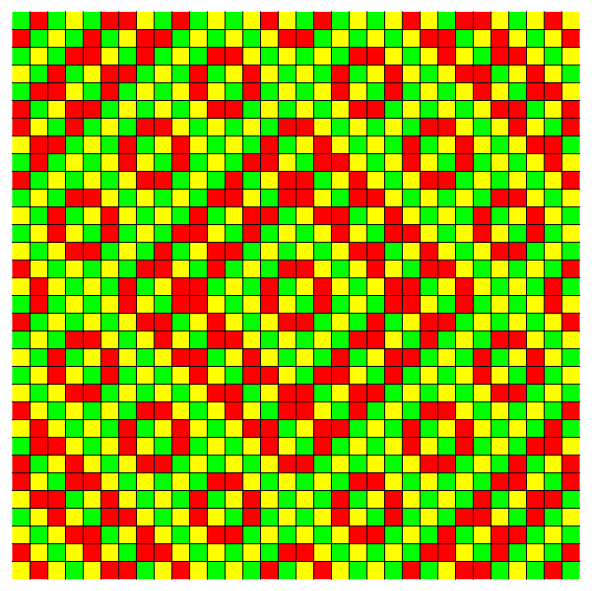}
}
 \caption{Tri-square tiling.\label{f:trisquare}}
\end{figure}
\begin{exam}[Tri-square tiling](cf.~\cite[Section~5.4,p.99]{Gustavo}).
Let $T$ be the tri-square tiling with protofaces $a,b,c\in T$ and substitution rule given as follows:
\begin{center}
\begin{tikzpicture}
\draw[fill=red] (.5,4) rectangle (1.5,5);
\draw[fill=green] (4.5,4) rectangle (5.5,5);
\draw[fill=yellow] (8.5,4) rectangle (9.5,5);
\node at (1,4.5){$a$};
\node at (5,4.5){$b$};
\node at (9,4.5){$c$};
\draw[->] (1,3.5) -- (1,2.5); 
\draw[->] (5,3.5) -- (5,2.5); 
\draw[->] (9,3.5) -- (9,2.5); 
\node at (1.3,3){$\omega$};
\node at (5.3,3){$\omega$};
\node at (9.3,3){$\omega$};

\draw[blue,fill=yellow] (0,0) rectangle (1,1);
\draw[blue,fill=yellow] (1,1) rectangle (2,2);
\draw[blue,fill=green] (1,0) rectangle (2,1);
\draw[blue,fill=green] (0,1) rectangle (1,2);
\draw (0,0) rectangle (2,2);
\node at (0.5,0.5){$c$};
\node at (1.5,1.5){$c$};
\node at (0.5,1.5){$b$};
\node at (1.5,0.5){$b$};
\draw[blue,fill=red] (4,0) rectangle (5,1);
\draw[blue,fill=red] (5,1) rectangle (6,2);
\draw[blue,fill=green] (4,1) rectangle (5,2);
\draw[blue,fill=green] (5,0) rectangle (6,1);
\draw (4,0) rectangle (6,2);
\node at (4.5,0.5){$a$};
\node at (5.5,1.5){$a$};
\node at (4.5,1.5){$b$};
\node at (5.5,0.5){$b$};
\draw[blue,fill=yellow] (8,0) rectangle (9,1);
\draw[blue,fill=yellow] (9,1) rectangle (10,2);
\draw[blue,fill=red] (8,1) rectangle (9,2);
\draw[blue,fill=red] (9,0) rectangle (10,1);
\draw (8,0) rectangle (10,2);
\node at (8.5,0.5){$c$};
\node at (9.5,1.5){$c$};
\node at (8.5,1.5){$a$};
\node at (9.5,0.5){$a$};
\end{tikzpicture} 
\end{center}
The length of the horizontal edge and of the vertical edge is 1.
The inflation factor is $\lambda=2$.
We illustrate the homotopy $h_s$, $0\le s\le 1$, on the vertices in the following figure:
\begin{center}
\begin{tikzpicture}
\draw[blue] (0,0) rectangle (1,1);
\draw[blue] (1,1) rectangle (2,2);
\draw (0,0) rectangle (2,2);
\draw[red,->] (0.1,1.1) -- (0.1,1.9);
\draw[red,->] (1.1,0.1) -- (1.9,0.1);
\draw[red,->] (2.1,1.1) -- (2.1,1.9);
\draw[red,->] (1.1,2.1) -- (1.9,2.1);
\draw[red,->] (1.1,1.1) -- (1.9,1.9);
\node[red] at (0.3,1.5){$h_s$};
\node[red] at (1.7,1.5){$h_s$};
\node[red] at (1.5,0.3){$h_s$};
\node[red] at (2.3,1.5){$h_s$};
\node[red] at (1.5,2.3){$h_s$};
\foreach \Point in { (1,1), (2,1), (1,2),(1,0),(0,1)}{
    \node at \Point[red] {\textbullet};
}    
\foreach \Point in {(0,0), (2,0), (2,2), (0,2)}{
    \node at \Point[black] {\textbullet};
}    
\end{tikzpicture}  
\end{center}
\begin{itemize}
\item stable faces (3):\,\, $a$, $b$, $c$
\item vertical stable edges (7):\,\,
$\vsE{a}{a},\,
\vsE{a}{b},\,
\vsE{a}{c},\,
\vsE{b}{a},\,
\vsE{b}{c},\,
\vsE{c}{a},\,
\vsE{c}{b}
$
\item horizontal stable edges (7):\,\,
$\hsE{a}{a},\,
\hsE{a}{b},\,
\hsE{a}{c},\,
\hsE{b}{a},\,
\hsE{b}{c},\,
\hsE{c}{a},\,
\hsE{c}{b}
$

\item stable vertices (21):\,\,
$\sVf{a}{a}{a}{a},\,
\sVf{a}{a}{a}{b},\,
\sVf{a}{a}{b}{c},\,
\sVf{a}{a}{c}{a},\,
\sVf{a}{a}{c}{b},\,
\sVf{a}{b}{a}{a},\,
\sVf{a}{b}{a}{b},\,
\sVf{a}{b}{c}{a},\,
\sVf{a}{c}{b}{a},\,
\sVf{a}{c}{b}{c},\,
\sVf{b}{a}{a}{c},\,
\sVf{b}{a}{b}{c},\,$\\
$
\sVf{b}{c}{a}{a},\,
\sVf{b}{c}{a}{c},\,
\sVf{b}{c}{b}{a},\,
\sVf{b}{c}{b}{c},\,
\sVf{c}{a}{a}{a},\,
\sVf{c}{a}{a}{b},\,
\sVf{c}{a}{c}{a},\,
\sVf{c}{b}{a}{a},\,
\sVf{c}{b}{c}{b}
$
\end{itemize}
We list the stable edges in the above order starting with the vertical stable edges. We denote them as $se_1,\ldots,se_{14}$.
The stable vertices are also listed in the above order and we denote them as $sv_1,\ldots, sv_{21}$.
The exponential map $\delta^0:\Z^{\sV}\to\Z^{\sE}$ is given by the matrix
\begin{center}
\includegraphics[scale=.6]{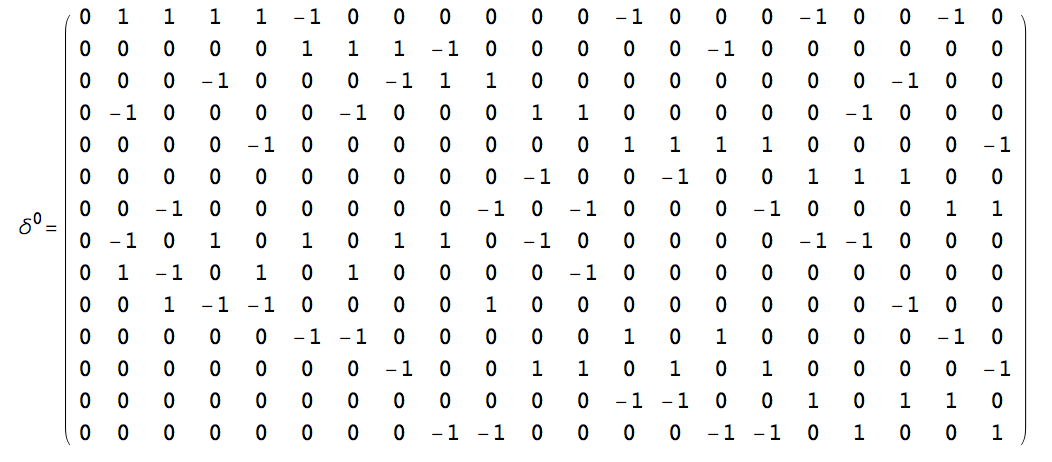}
\end{center}
For instance $\delta^0(sv_{3})=se_1-se_7-se_9+se_{10}$.
The index map $\delta^1:\Z^{\sE}\to\Z^{\sF}$ is given by the matrix
$$
\delta^1=\left(
\begin{array}{cccccccccccccc}
 0 & 1 & 1 & -1 & 0 & -1 & 0 & 0 & -1 & -1 & 1 & 0 & 1 & 0 \\
 0 & -1 & 0 & 1 & 1 & 0 & -1 & 0 & 1 & 0 & -1 & -1 & 0 & 1 \\
 0 & 0 & -1 & 0 & -1 & 1 & 1 & 0 & 0 & 1 & 0 & 1 & -1 & -1 \\
\end{array}
\right)
$$
For controlling computational errors, it is always good to check that $\delta^1\circ\delta^0=0$.
The substitution-homotopy map $W_V:\Z^{\sV}\to\Z^{\sV}$ is given by the matrix
\begin{center}
\includegraphics[scale=.6]{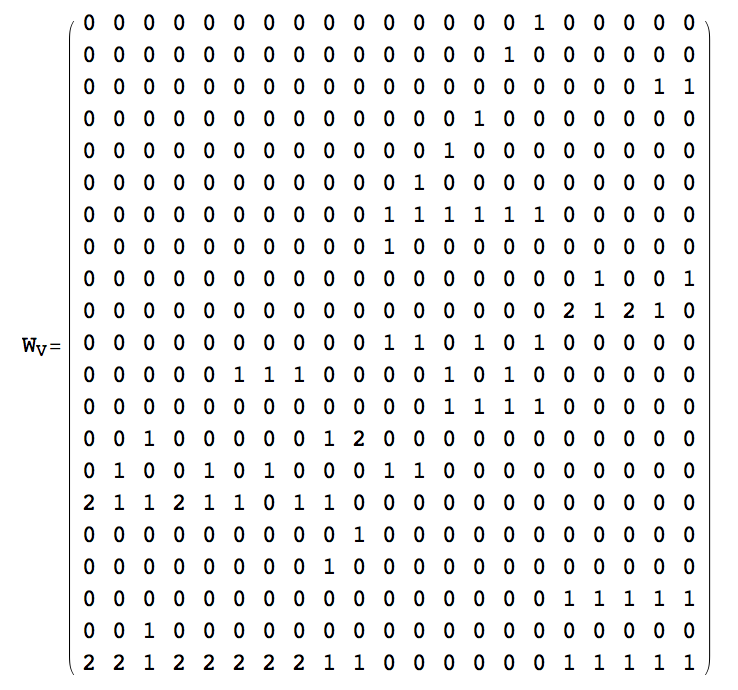}
\end{center}
For instance,
$W_V(sv_2)=sv_{15}+sv_{16}+2\,sv_{21}$, and it is
illustrated in Figure \ref{f:trisquare-Wv-sv2}.
\begin{figure}[t]
\centerline{
\includegraphics[scale=.4]{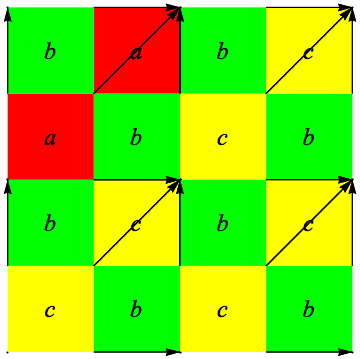}
}
 \caption{$W_V(sv_2)=sv_{15}+sv_{16}+2\,sv_{21}$\label{f:trisquare-Wv-sv2}}
\end{figure}
The substitution-homotopy map $W_E:\Z^{\sE}\to\Z^{\sE}$ is given by the matrix
\begin{center}
\includegraphics[scale=.6]{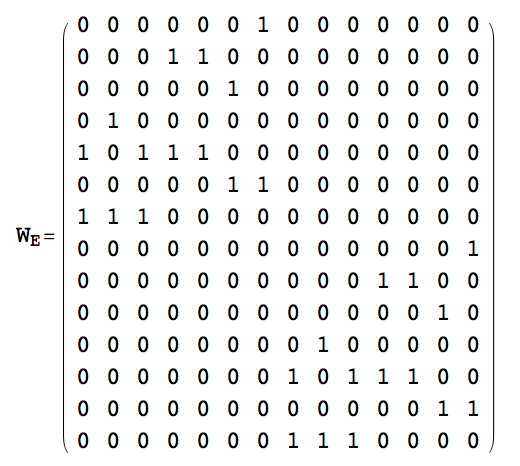}
\end{center}
For instance, $W_E(se_2)=se_4+se_7$, and it is illustrated in Figure \ref{f:trisquare-We-se2}.\\
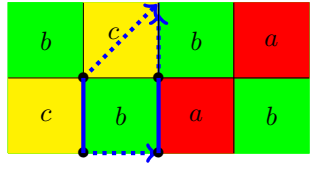
\begin{figure}[t]
\centerline{
\begin{tikzpicture}[scale=1]

\draw[green,fill] (0,0)--(4,0)--(4,2)--(0,2)--(0,0);
\draw[green] (0,0)--(4,0)--(4,2)--(0,2)--(0,0);
\draw[yellow,fill] (0,0)--(1,0)--(1,1)--(0,1)--(0,0);
\draw[yellow,fill] (1,1)--(2,1)--(2,2)--(1,2)--(1,1);
\draw[dotted,ultra thick,blue,->] (1,1)--(1.98,1.98);
\draw[dotted,ultra thick,blue,->] (1,0)--(1.98,0);
\draw[dotted,ultra thick,blue,->] (2,1)--(2,1.85);
\draw[thick,blue] (2,0)--(2,2);
\node at (.5,.5) {$c$};
\node at (1.4,1.6) {$c$};
\node at (1.5,.5) {$b$};
\node at (.5,1.5) {$b$};

\draw[red,fill] (2,0)--(3,0)--(3,1)--(2,1)--(2,0);
\draw[red,fill] (3,1)--(4,1)--(4,2)--(3,2)--(3,1);
\node at (2.5,.5) {$a$};
\node at (3.5,1.5) {$a$};
\node at (3.5,.5) {$b$};
\node at (2.5,1.5) {$b$};

\draw (0,1)--(4,1);
\draw (1,0)--(1,2);
\draw (2,0)--(2,2);
\draw (3,0)--(3,2);
\node at (1,1) {\textbullet};
\node at (1,0) {\textbullet};
\node at (2,1) {\textbullet};
\node at (2,0) {\textbullet};
\draw[ultra thick,blue] (1,0)--(1,1);
\draw[ultra thick,blue] (2,0)--(2,1);

\end{tikzpicture}
}
 \caption{$W_E(se_2)=se_4+se_7$\label{f:trisquare-We-se2}}
\end{figure}
\noindent
The substitution-homotopy map $W_F:\Z^{\sF}\to\Z^{\sF}$ is given by the matrix
$$W_F=\left(
\begin{array}{ccc}
 0 & 1 & 0 \\
 0 & 0 & 0 \\
 1 & 0 & 1 \\
\end{array}
\right)
$$
For controlling computational errors it is always good to check that the diagram (\ref{e:HSdiagram}) commutes, i.e.~that $W_E\circ \delta^0=\delta^0\circ W_V$ and $W_F\circ \delta^1=\delta^1\circ W_E$. The computation of the direct limits is done similarly as the ones done for the 1-dimensional tilings, so here we skip most of the steps. 
By Proposition \ref{p:kerA-isom-Znminusr} we replace  $\ker \delta^0$ with $\Z^{11}$.
By Proposition \ref{p:imA-isom-Zr} we remove the zero-eigenvalues of the matrix.
By Proposition \ref{p:extract-one-eigenvalues} we extract the $\pm1$-eigenvalues and get
$$H_S^0(T)=\lim_{\to}(W_V,\,\ker\delta^0)=\Z^4\oplus\lim_{\to}(\left(
\begin{array}{ccc}
 18 & -38 & 16 \\
 8 & -18 & 8 \\
 2 & -6 & 4 \\
\end{array}
\right),\,\Z^{3}).$$
Note that  
$$\left(
\begin{array}{ccc}
 -1 & 2 & -1 \\
 0 & 1 & 0 \\
 0 & 0 & -1 \\
\end{array}
\right)
.
\left(
\begin{array}{ccc}
 18 & -38 & 16 \\
 8 & -18 & 8 \\
 2 & -6 & 4 \\
\end{array}
\right)
.
\left(
\begin{array}{ccc}
 -1 & 2 & -1 \\
 0 & 1 & 0 \\
 0 & 0 & -1 \\
\end{array}
\right)^{-1}=
\left(
\begin{array}{ccc}
 4 & 0 & 0 \\
 -8 & -2 & 0 \\
 2 & 2 & 2 \\
\end{array}
\right)
$$
is a lower triangular matrix with eigenvalues $4,-2,2$ and, by Proposition \ref{p:lim:Ginv}, we get that
$$\lim_{\to}(\left(
\begin{array}{ccc}
 18 & -38 & 16 \\
 8 & -18 & 8 \\
 2 & -6 & 4 \\
\end{array}
\right),\,\Z^{3})=
\lim_{\to}(\left(
\begin{array}{ccc}
 4 & 0 & 0 \\
 -8 & -2 & 0 \\
 2 & 2 & 2 \\
\end{array}
\right),\,\Z^{3}).
$$
By the same argument as the one used for the lower triangular matrix in Example \ref{e:more1dimexamples}(Rudin-Shapiro) we conclude that
$$H_S^0(T)=\lim_{\to}(W_V,\,\ker\delta^0)=\Z^4\oplus\lim_{\to}(\left(
\begin{array}{ccc}
 4 & 0 & 0 \\
 -8 & -2 & 0 \\
 2 & 2 & 2 \\
\end{array}
\right),\,\Z^{3})=\Z^4\oplus \Z[\frac12]^3.$$
By Proposition \ref{p:kerOverIm-A-isom-Znminus}, $\lim\limits_{\to}(W_E,\,\frac{\ker\delta^1}{\im\delta^0})=\lim\limits_{\to}(W_E',\coker B)$ for some matrices $W_E',B$.
Then by Proposition \ref{p:cokerA-isom-Znminuss} we get rid of the coker and get
$$K_1(S)=H_S^1(T)=\lim_{\to}(W_E,\frac{\ker\delta^1}{\im\delta^0})=(
\left(
\begin{array}{cc}
 2 & 0 \\
 0 & 2 \\
\end{array}
\right)
,\Z^2)=\Z[\frac12]^2.$$
\end{exam}
We are the first ones to compute the stable and unstable $K$-theories of the Tri-square tiling.
The computation of $K(U)$ for this example and for the rest of the examples was done in Mathematica.
Since the computations are similar to the ones for $K(S)$, we do not explain them in this paper.

The Table tiling and its unstable $K$-theory was communicated by Franz G\"ahler.
However, we have refined its substitution rule in order to be able to apply our formulas.
\begin{figure}[b]
\centerline{
\includegraphics[scale=.5]{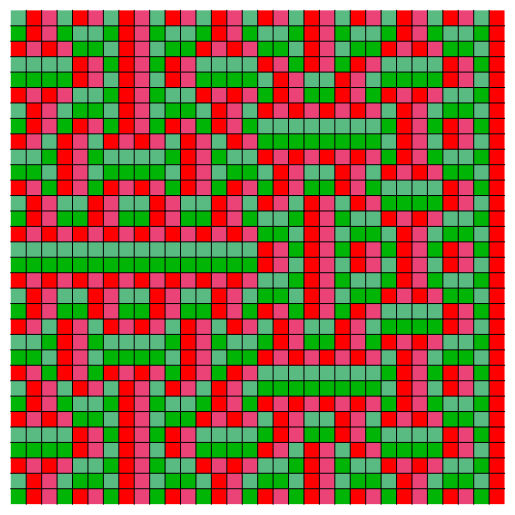}
}
 \caption{Table tiling.\label{f:table}}
\end{figure}

\begin{exam}[Table Tiling]
Let $T$ be the Table tiling with proto-faces $a,b,c,d\in T$ and substitution rule given below:
\begin{center}
\begin{tikzpicture}
\draw[fill=red] (1,4) rectangle (2,5);
\draw[fill=redish] (2.1,4) rectangle (3.1,5);
\node at (1.5,4.5){$a$};
\node at (2.6,4.5){$b$};
\draw[->] (1.5,3.5) -- (1.5,2.5); 
\draw[->] (2.6,3.5) -- (2.6,2.5); 
\draw[->] (9,3.5) -- (9,2.5); 
\node at (1.3,3){$\omega$};
\node at (2.8,3){$\omega$};
\draw[blue,fill=greenI] (0,0) rectangle (1,1);
\draw[blue,fill=red] (1,1) rectangle (2,2);
\draw[blue,fill=red] (1,0) rectangle (2,1);
\draw[blue,fill=greenish] (0,1) rectangle (1,2);
\draw (0,0) rectangle (2,2);
\node at (0.5,0.5){$c$};
\node at (1.5,1.5){$a$};
\node at (0.5,1.5){$d$};
\node at (1.5,0.5){$a$};
\draw[blue,fill=redish] (2.1,0) rectangle (3.1,1);
\draw[blue,fill=greenish] (3.1,1) rectangle (4.1,2);
\draw[blue,fill=redish] (2.1,1) rectangle (3.1,2);
\draw[blue,fill=greenI] (3.1,0) rectangle (4.1,1);
\draw (2.1,0) rectangle (4.1,2);
\node at (2.5,0.5){$b$};
\node at (3.5,1.5){$d$};
\node at (2.5,1.5){$b$};
\node at (3.5,0.5){$c$};
\draw[fill=greenI] (6,1) rectangle (7,2);
\draw[fill=greenish] (6,2.1) rectangle (7,3.1);
\node at (6.5,1.5){$c$};
\node at (6.5,2.5){$d$};
\draw[->] (7.5,1.5) -- (8.5,1.5); 
\draw[->] (7.6,2.5) -- (8.6,2.5); 
\node at (8,1.3){$\omega$};
\node at (8,2.7){$\omega$};
\draw[blue,fill=red] (9,0) rectangle (10,1);
\draw[blue,fill=greenI] (10,1) rectangle (11,2);
\draw[blue,fill=greenI] (9,1) rectangle (10,2);
\draw[blue,fill=redish] (10,0) rectangle (11,1);
\draw (9,0) rectangle (11,2);
\node at (9.5,0.5){$a$};
\node at (10.5,1.5){$c$};
\node at (9.5,1.5){$c$};
\node at (10.5,0.5){$b$};
\draw[blue,fill=greenish] (9,2.1) rectangle (10,3.1);
\draw[blue,fill=redish] (10,3.1) rectangle (11,4.1);
\draw[blue,fill=red] (9,3.1) rectangle (10,4.1);
\draw[blue,fill=greenish] (10,2.1) rectangle (11,3.1);
\draw (9,2.1) rectangle (11,4.1);
\node at (9.5,2.5){$d$};
\node at (10.5,3.5){$b$};
\node at (9.5,3.5){$a$};
\node at (10.5,2.5){$d$};
\end{tikzpicture} 
\end{center}
The length of each of the edges of the proto-faces is $1$, and the inflation factor is $\lambda=2$. The homotopy is the same as the one of the previous example.
\begin{itemize}
\item stable faces (4): $a,\,b,\,c,\,d$
\item vertical stable edges (10):
$\vsE{a}{b},\,
\vsE{b}{a},\,
\vsE{b}{c},\,
\vsE{b}{d},\,
\vsE{c}{a},\,
\vsE{c}{c},\,
\vsE{c}{d},\,
\vsE{d}{a},\,
\vsE{d}{c},\,
\vsE{d}{d}$
\item horizontal stable edges (10):
$\hsE{a}{a},\,
\hsE{a}{b},\,
\hsE{a}{c},\,
\hsE{b}{a},\,
\hsE{b}{b},\,
\hsE{b}{c},\,
\hsE{c}{d},\,
\hsE{d}{a},\,
\hsE{d}{b},\,
\hsE{d}{c}$
\item stable vertices (24):
$\sVf{a}{b}{a}{b},\,
\sVf{a}{b}{a}{c},\,
\sVf{a}{b}{b}{a},\,
\sVf{a}{b}{c}{b},\,
\sVf{a}{b}{c}{c},\,
\sVf{b}{a}{b}{a},\,
\sVf{b}{a}{c}{c},\,
\sVf{b}{c}{d}{b},\,
\sVf{b}{c}{d}{c},\,
\sVf{b}{d}{a}{c},\,
\sVf{b}{d}{b}{a},\,
\sVf{b}{d}{c}{b},\,$\\
$
\sVf{c}{a}{a}{d},\,
\sVf{c}{a}{c}{d},\,
\sVf{c}{c}{d}{d},\,
\sVf{c}{d}{a}{d},\,
\sVf{c}{d}{c}{d},\,
\sVf{d}{a}{a}{c},\,
\sVf{d}{a}{b}{a},\,
\sVf{d}{a}{c}{b},\,
\sVf{d}{c}{d}{b},\,
\sVf{d}{c}{d}{c},\,
\sVf{d}{d}{a}{b},\,
\sVf{d}{d}{b}{a}$
\end{itemize}
We list the stable edges in the above order starting with the vertical stable edges. We denote them as $se_1,\ldots,se_{20}$.
The stable vertices are also listed in the above order and we denote them as $sv_1,\ldots sv_{24}$.
The exponential map $\delta^0:\Z^{\sV}\to\Z^{\sE}$ is given by the matrix 
\begin{center}
\includegraphics[scale=.6]{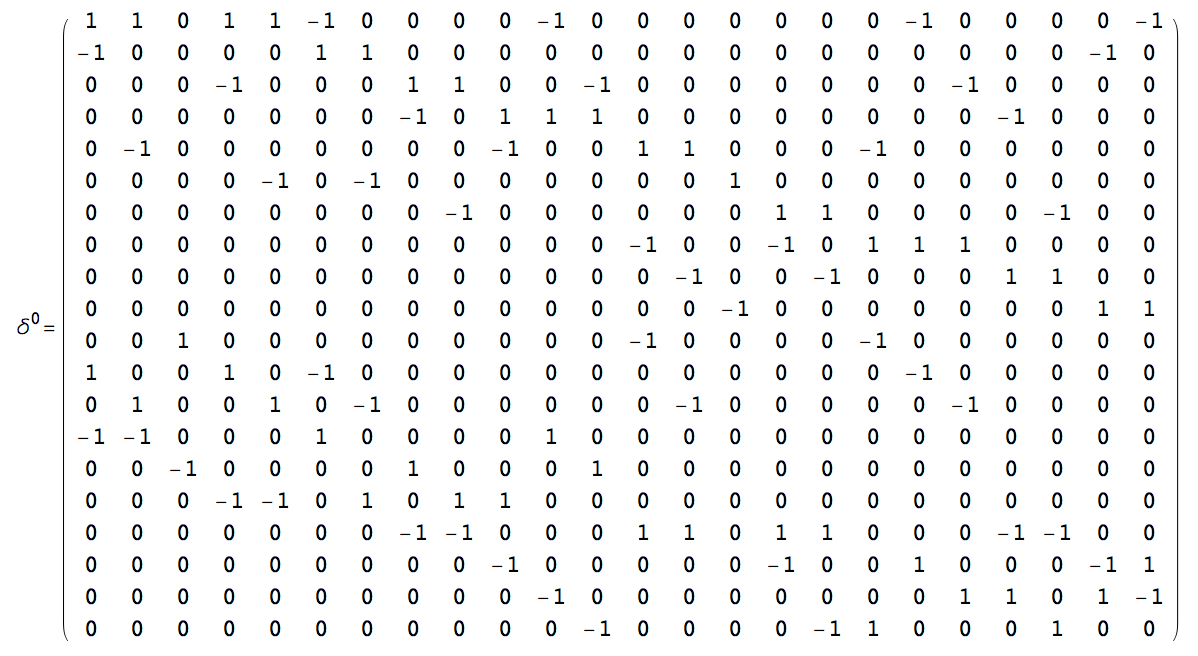}
\end{center}
The index map $\delta^1:\Z^{\sE}\to\Z^{sF}$ is given by the matrix
\begin{center}
\includegraphics[scale=.6]{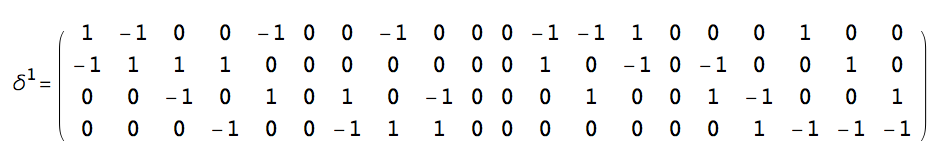}
\end{center}
The substitution-homotopy map $W_V:\Z^{\sV}\to\Z^{sV}$ is given by the matrix
\begin{center}
\includegraphics[scale=.6]{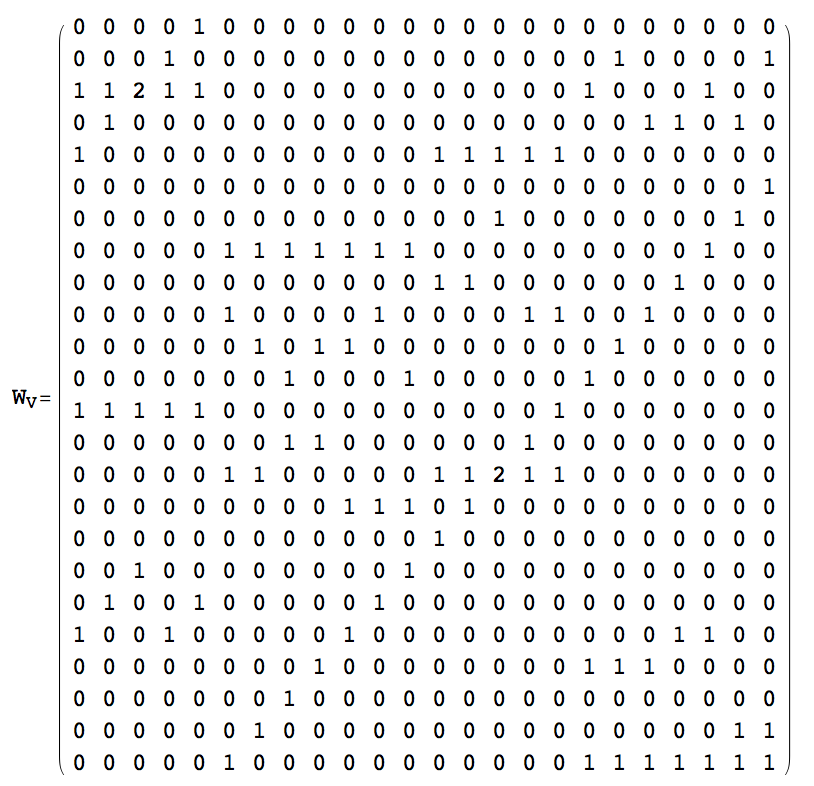}
\end{center}
The substitution-homotopy map $W_E:\Z^{\sE}\to\Z^{\sE}$ is given by the matrix
\begin{center}
\includegraphics[scale=.6]{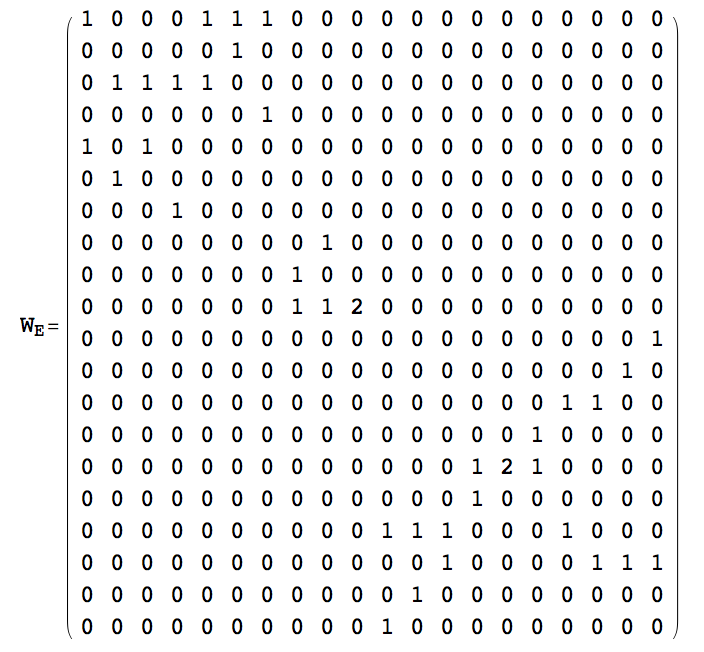}
\end{center}
The substitution-homotopy map $W_F:\Z^{\sF}\to\Z^{\sF}$ is given by the matrix
$$W_F=\left(
\begin{array}{cccc}
 0 & 0 & 1 & 0 \\
 0 & 1 & 0 & 0 \\
 1 & 0 & 0 & 0 \\
 0 & 0 & 0 & 1 \\
\end{array}
\right).
$$
The computation of the direct limits is done similar to the one from the previous example.
By Proposition \ref{p:kerA-isom-Znminusr} we replace  $\ker \delta^0$ with $\Z^{9}$.
By Proposition \ref{p:imA-isom-Zr} we remove the zero-eigenvalue of the matrix.
By Proposition \ref{p:extract-one-eigenvalues} we extract the $\pm1$-eigenvalues and get
$$H_S^0(T)=\lim_{\to}(W_V,\,\ker\delta^0)=\Z^3\oplus\lim_{\to}(
\left(
\begin{array}{ccccc}
 0 & 2 & 8 & 18 & -8 \\
 2 & 0 & -6 & -10 & 8 \\
 0 & 0 & -10 & -16 & 16 \\
 0 & 0 & 8 & 14 & -10 \\
 0 & 0 & 2 & 4 & 0 \\
\end{array}
\right),\,\Z^{5}).$$
Note that
\begin{center}
\includegraphics[scale=.6]{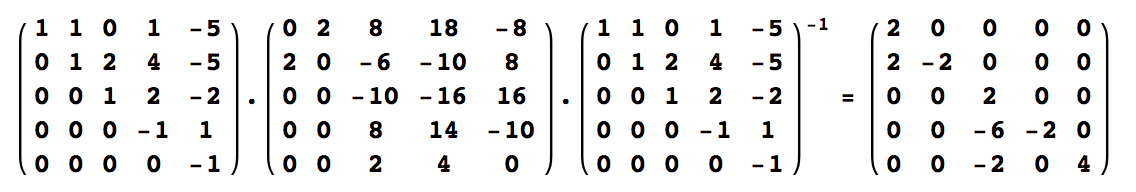}
\end{center}
is a lower triangular matrix with eigenvalues 4, -2, -2, 2, 2, and by Proposition \ref{p:lim:Ginv} we get that
$$\lim_{\to}(   
\left(
\begin{array}{ccccc}
 0 & 2 & 8 & 18 & -8 \\
 2 & 0 & -6 & -10 & 8 \\
 0 & 0 & -10 & -16 & 16 \\
 0 & 0 & 8 & 14 & -10 \\
 0 & 0 & 2 & 4 & 0 \\
\end{array}
\right)
,\,\Z^{5})=
\lim_{\to}(
\left(
\begin{array}{ccccc}
 2 & 0 & 0 & 0 & 0 \\
 2 & -2 & 0 & 0 & 0 \\
 0 & 0 & 2 & 0 & 0 \\
 0 & 0 & -6 & -2 & 0 \\
 0 & 0 & -2 & 0 & 4 \\
\end{array}
\right)
,\,\Z^{5}).
$$
By the same argument as the one used for the lower triangular matrix in Example \ref{e:more1dimexamples}(Rudin-Shapiro) we conclude that
$$H_S^0(T)=\lim_{\to}(W_V,\,\ker\delta^0)=\Z^3\oplus\lim_{\to}(
\left(
\begin{array}{ccccc}
 2 & 0 & 0 & 0 & 0 \\
 2 & -2 & 0 & 0 & 0 \\
 0 & 0 & 2 & 0 & 0 \\
 0 & 0 & -6 & -2 & 0 \\
 0 & 0 & -2 & 0 & 4 \\
\end{array}
\right)
,\,\Z^{5})=\Z^3\oplus \Z[\frac12]^5.$$
By Proposition \ref{p:kerOverIm-A-isom-Znminus}, $\lim\limits_{\to}(W_E,\,\frac{\ker\delta^1}{\im\delta^0})=\lim\limits_{\to}(W_E',\coker B)$ for some matrices $W_E',B$.
Then by Proposition \ref{p:cokerA-isom-Znminuss} we get rid of the coker and get
$$K_1(S)=H_S^1(T)=\lim_{\to}(W_E,\frac{\ker\delta^1}{\im\delta^0})=\lim_{\to}(
\left(
\begin{array}{ccc}
 1 & -1 & 0 \\
 0 & 2 & 0 \\
 0 & 0 & 2 \\
\end{array}
\right),\Z_2\oplus\Z^2)=\Z_2\oplus\Z[\frac12]^2.$$
\end{exam}

\subsection{Two dimensional tilings}[more general tilings.]
%
\begin{figure}[b]
\centerline{
\includegraphics[scale=.5]{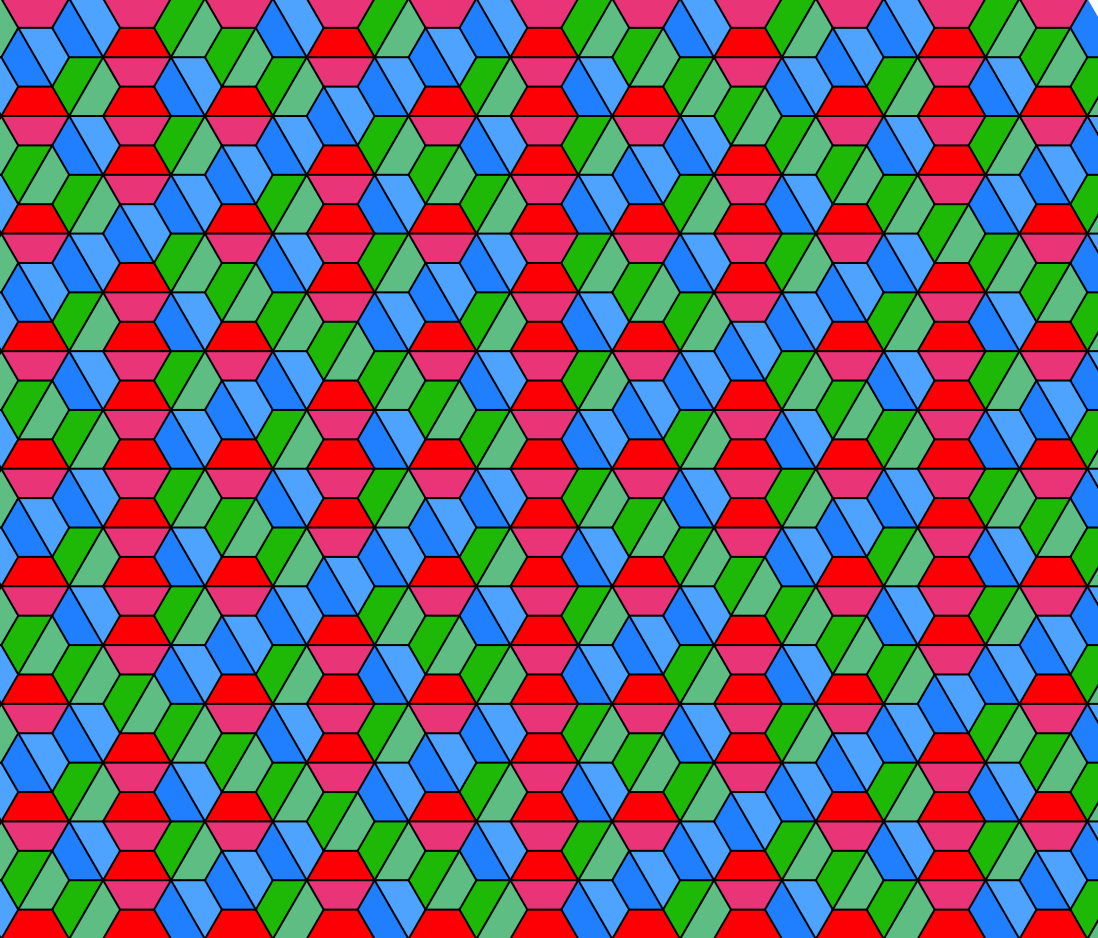}
}
 \caption{Half-hex tiling.\label{f:halfhex}}
\end{figure}
\begin{exam}[Half-hex tiling]
Let $T$ be the Half-hex tiling with proto-faces $a,b,c,d,e,f\in T$ and substitution rule
\begin{center}
\includegraphics[scale=.4]{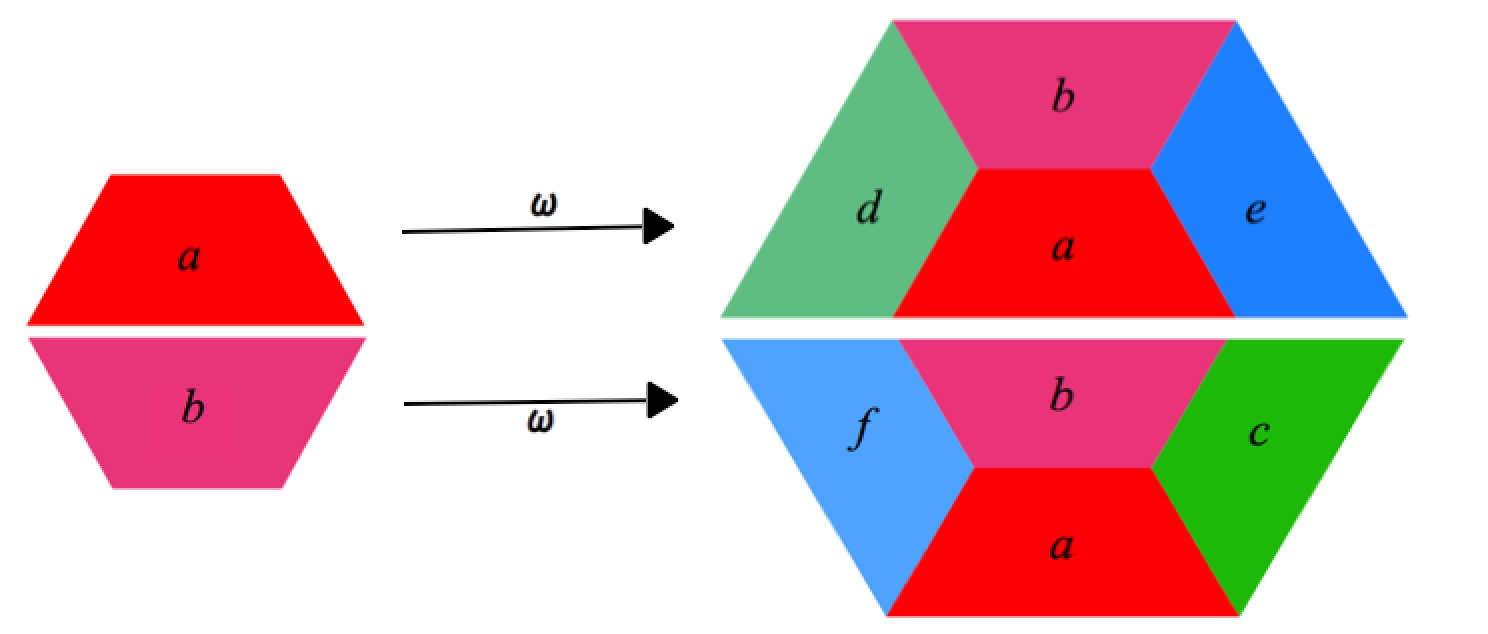}
\end{center}
\begin{center}
\includegraphics[scale=.35]{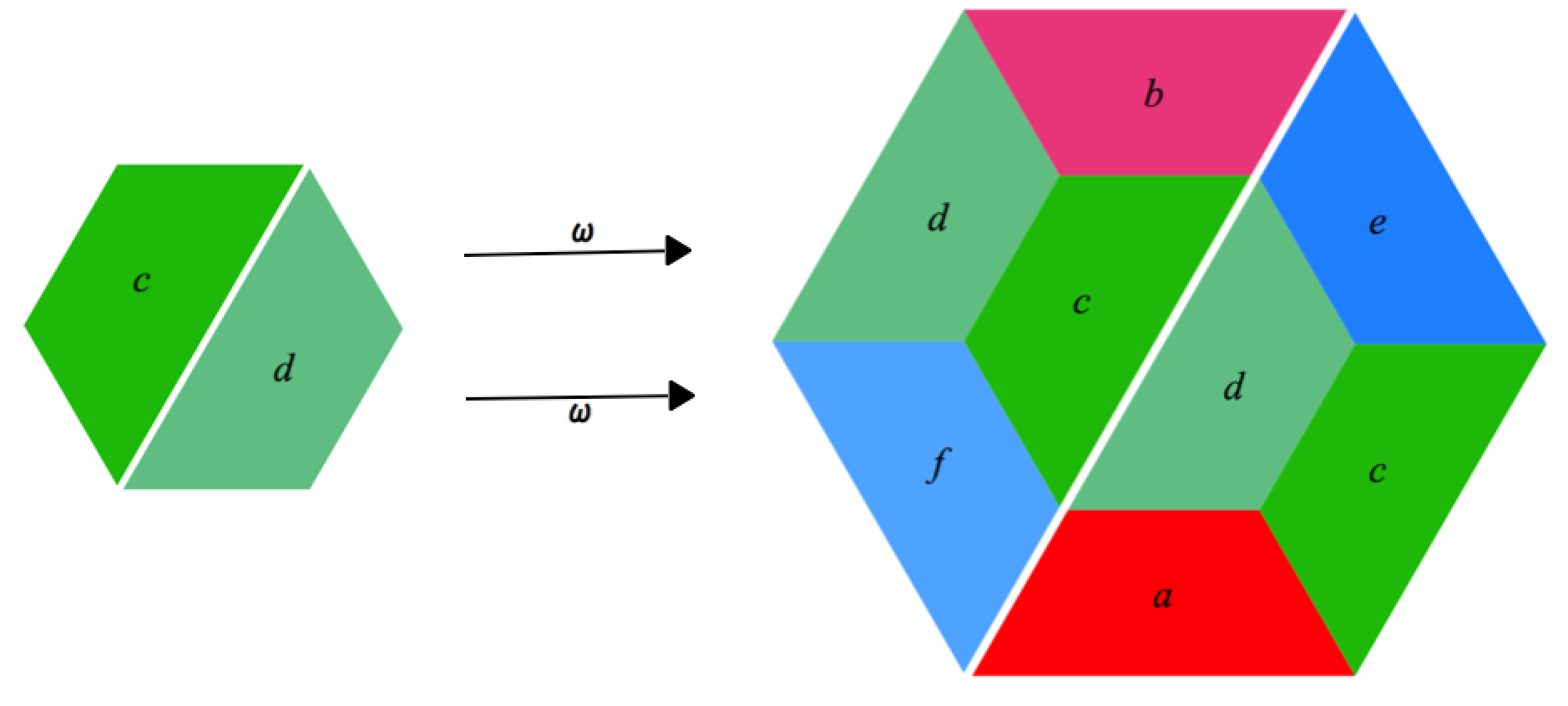}
\end{center}
\begin{center}
\includegraphics[scale=.35]{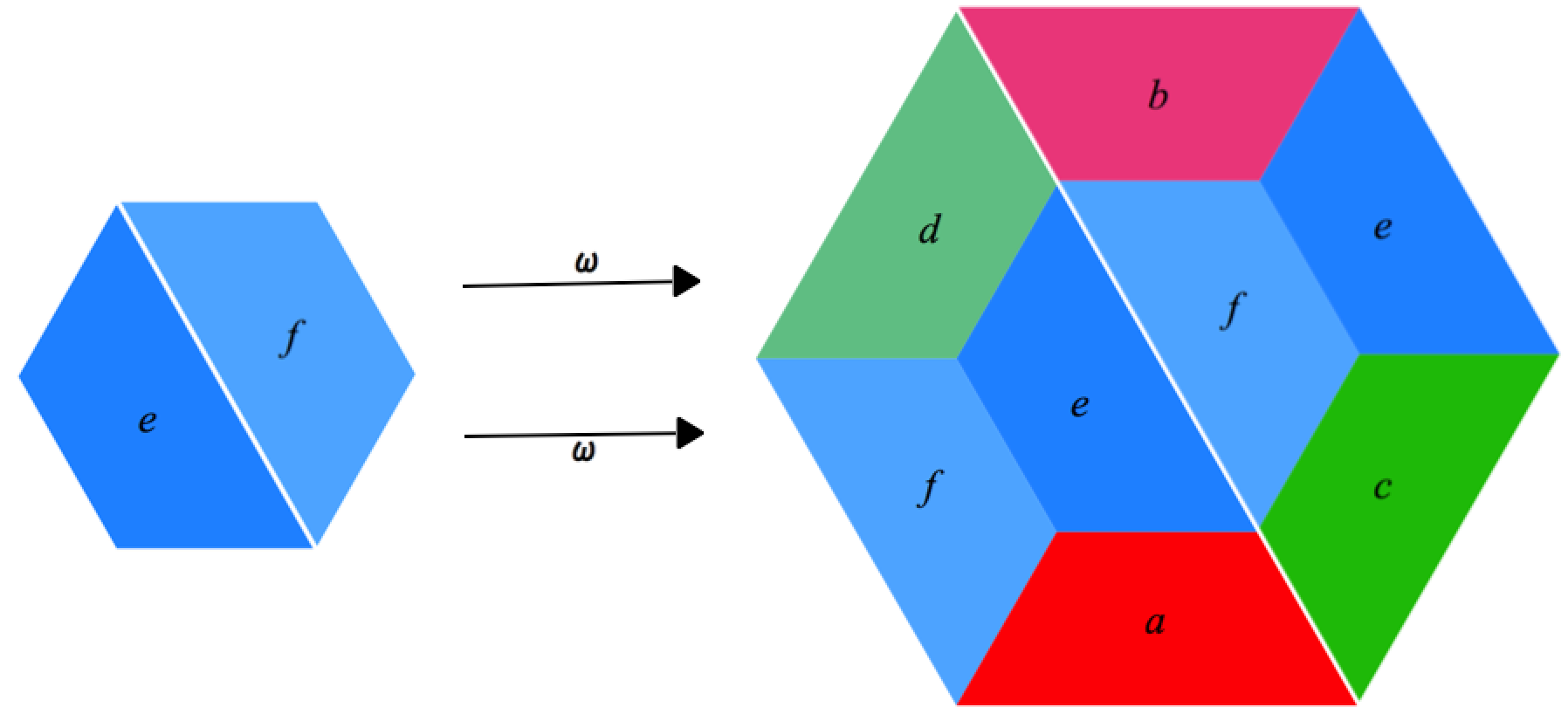}
\end{center}
The inflation factor is $\lambda=2$. There are 24 stable edges and 20 stable vertices.\\\\
{\centering
\includegraphics[scale=.35]{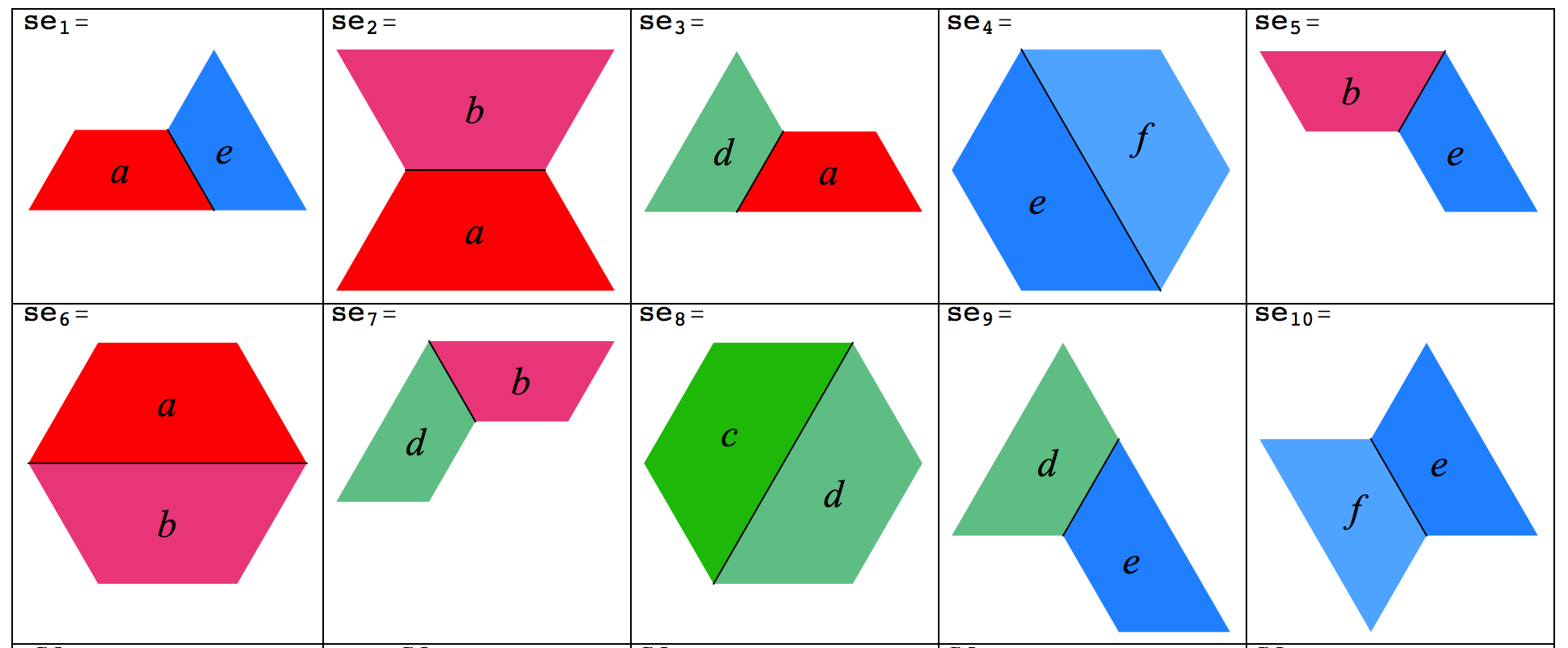}\\
\includegraphics[scale=.35]{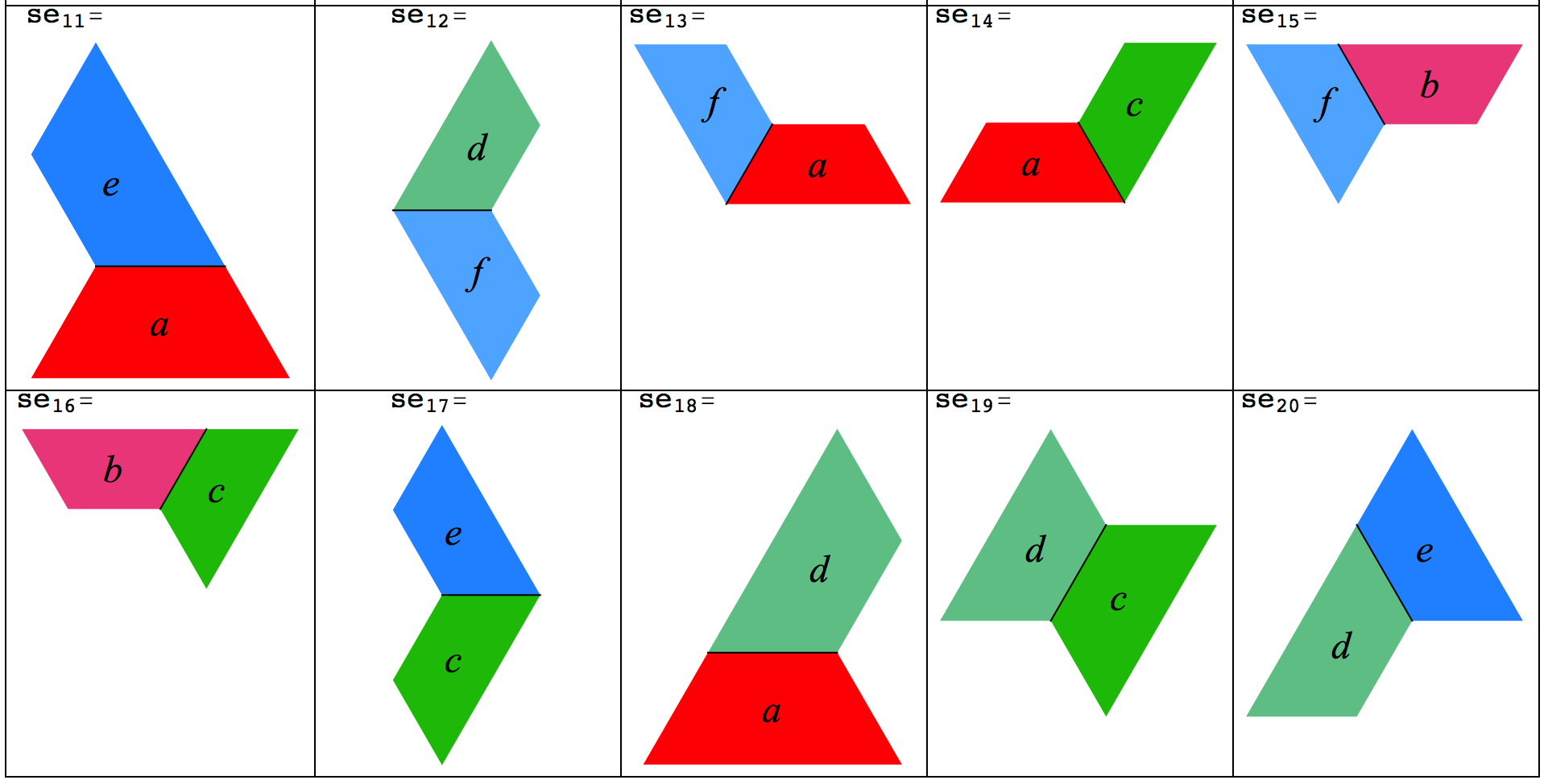}\\
\includegraphics[scale=.35]{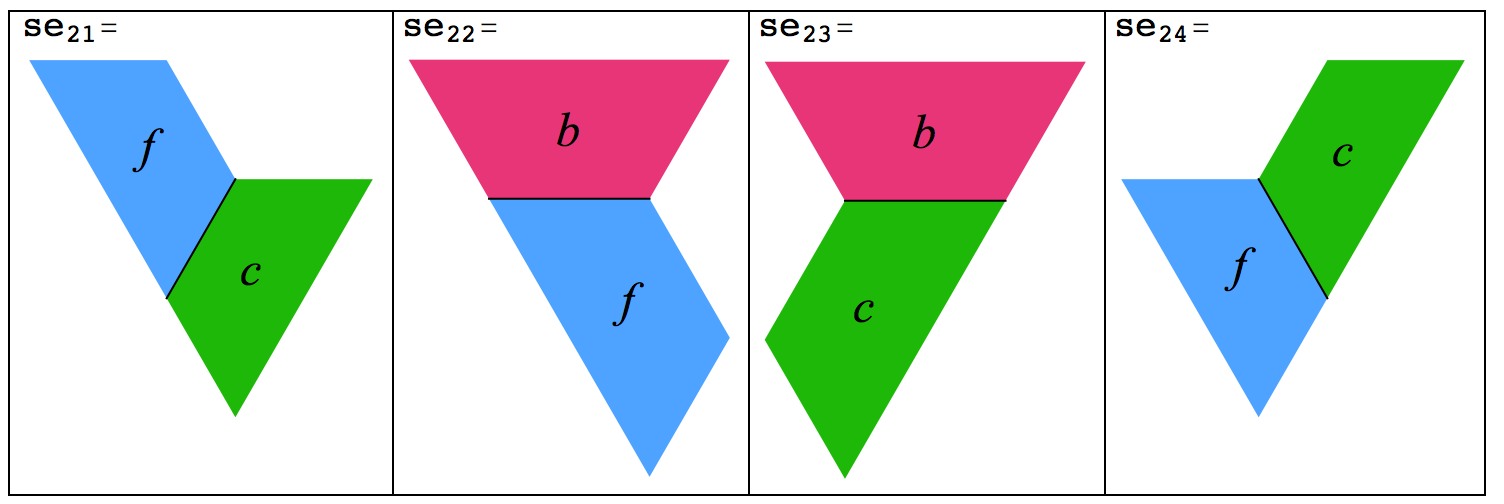}\\
\includegraphics[scale=.35]{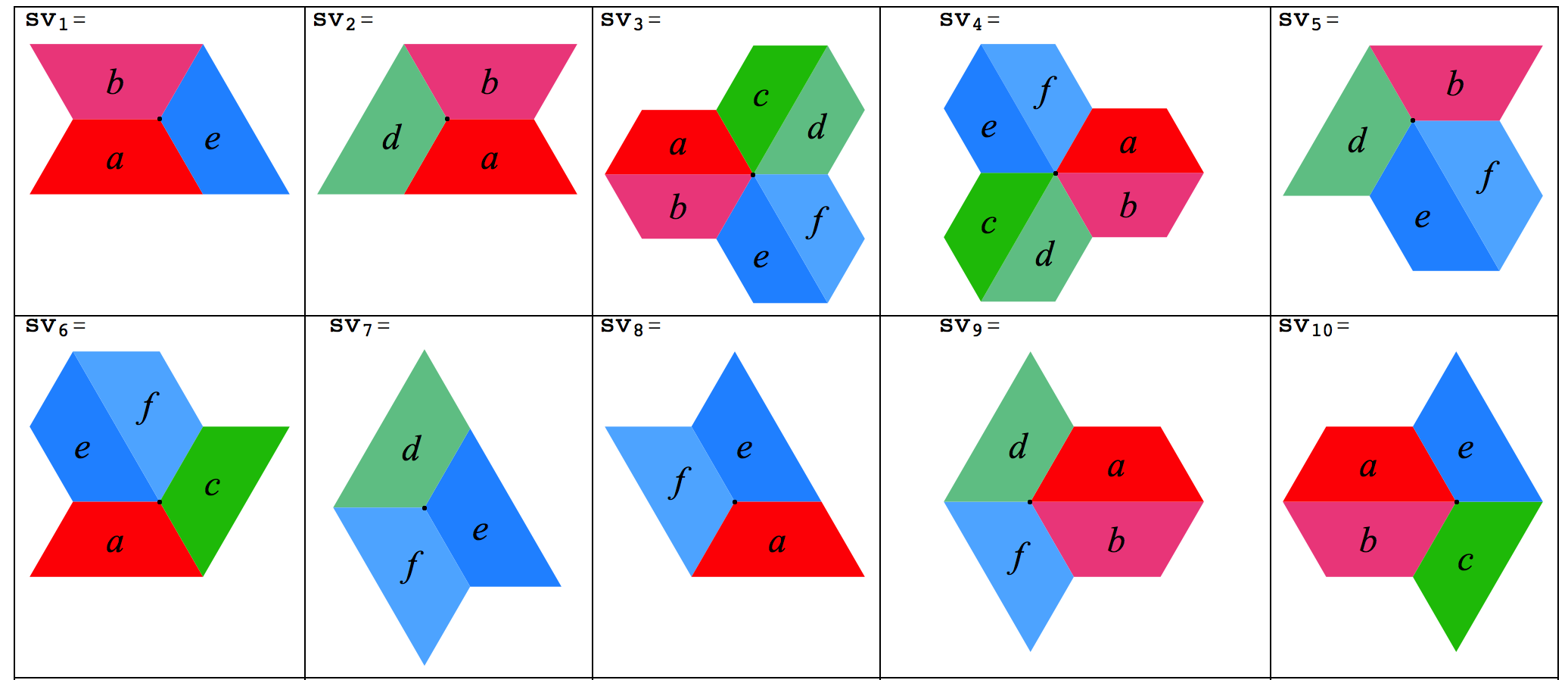}\\
\includegraphics[scale=.36]{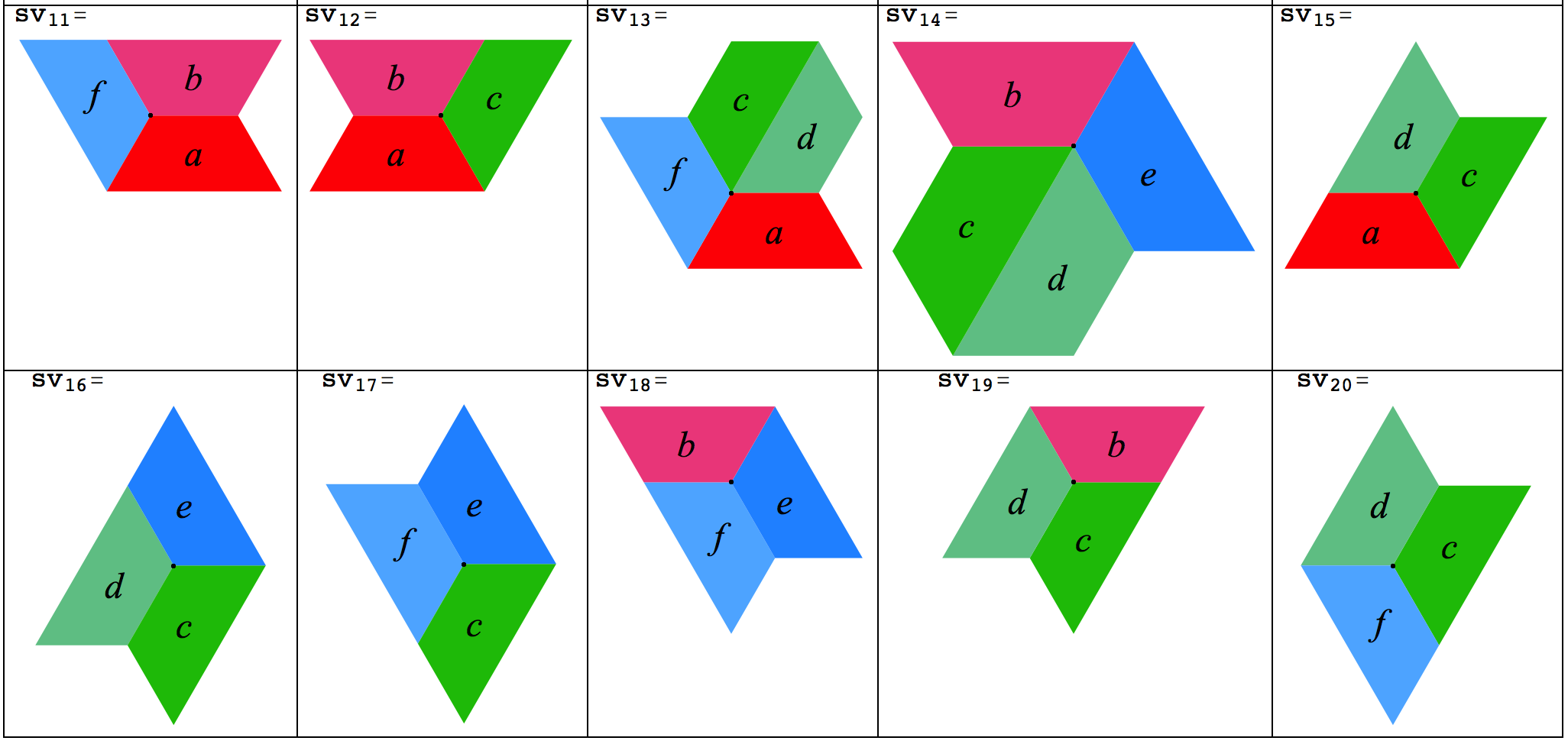}\\
}
$ $\\\\
The exponential map $\delta^0:\Z^{\sV}\to\Z^{\sE}$ is given by the matrix
\begin{center}
\includegraphics[scale=.6]{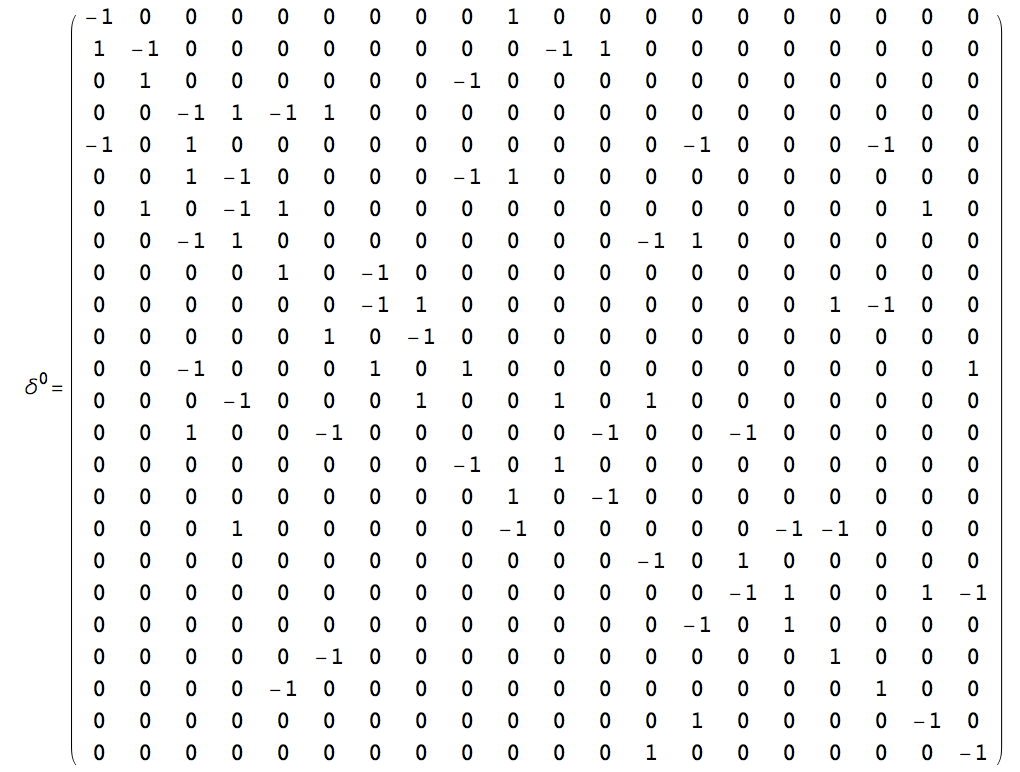}
\end{center}
The index map $\delta^1:\Z^{\sE}\to\Z^{sF}$ is given by the matrix
\begin{center}
\includegraphics[scale=.6]{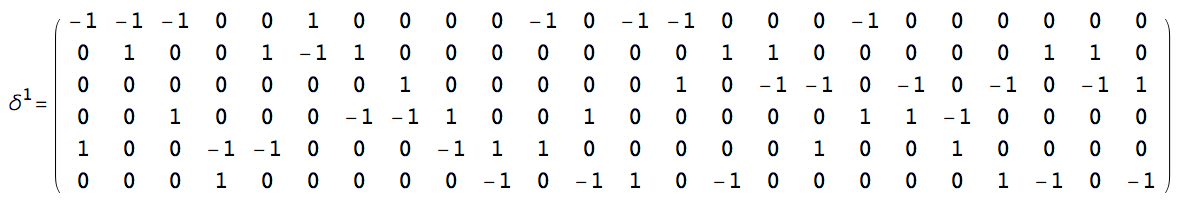}
\end{center}
We illustrate the homotopy $h_s$, $0\le s\le 1$, on the vertices in the following figure:
\begin{center}
\includegraphics[scale=.3]{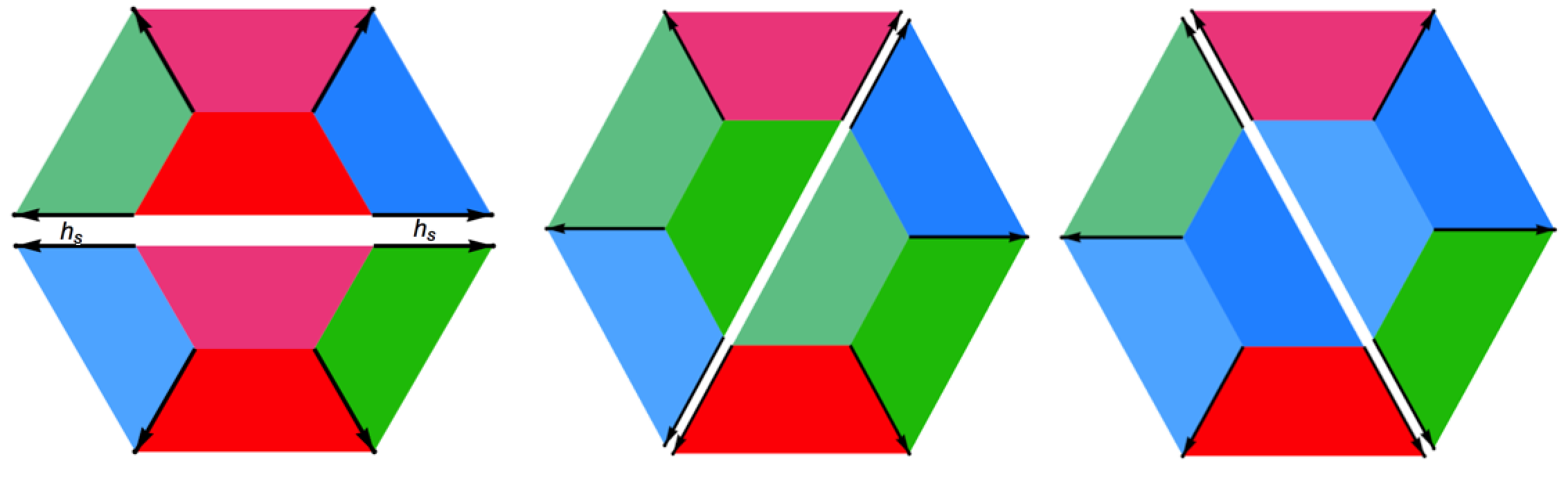}
\end{center}
The substitution-homotopy map $W_V:\Z^{\sV}\to\Z^{sV}$ is given by the matrix
\begin{center}
\includegraphics[scale=.61]{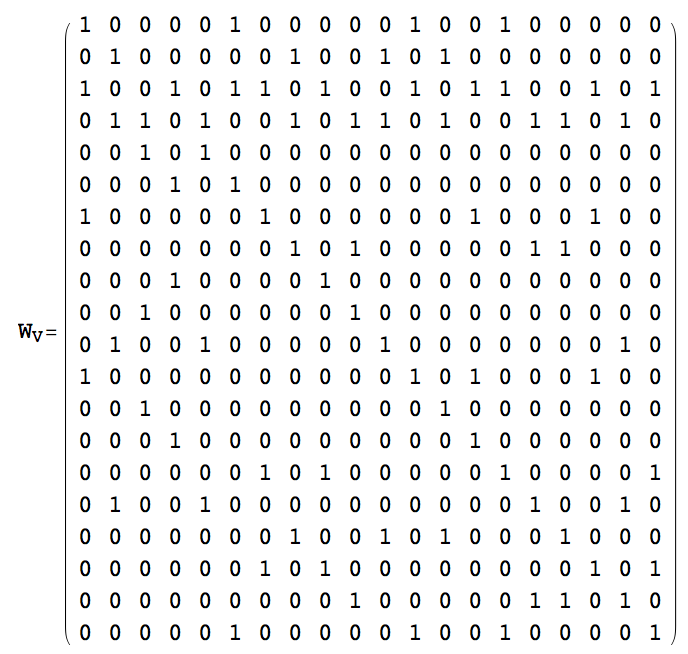}
\end{center}
The substitution-homotopy map $W_E:\Z^{\sE}\to\Z^{\sE}$ is given by the matrix
\begin{center}
\includegraphics[scale=.61]{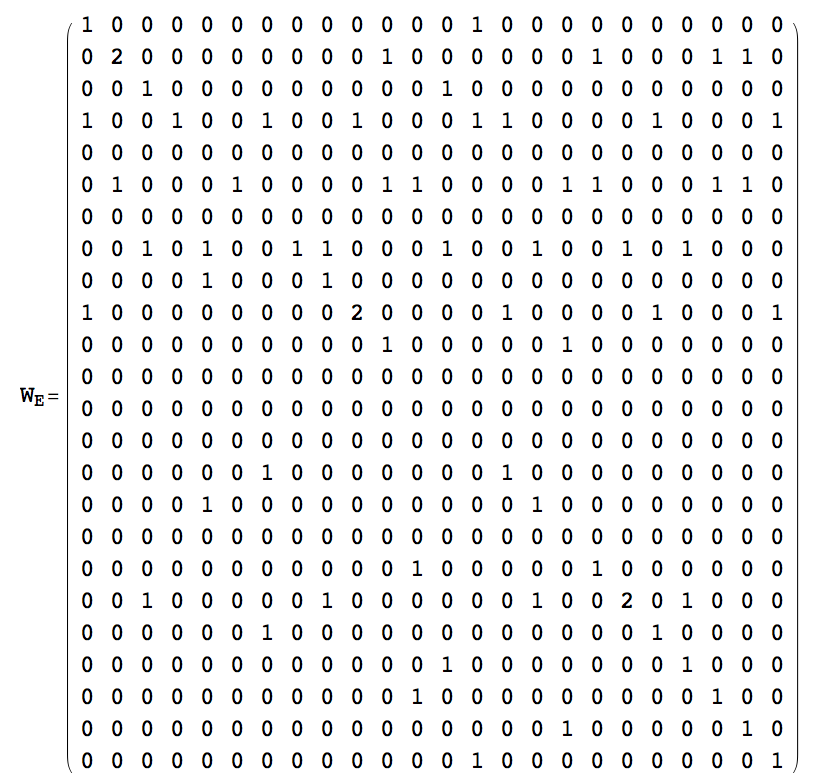}
\end{center}
The substitution-homotopy map $W_F:\Z^{\sF}\to\Z^{\sF}$ is given by the matrix
\begin{center}
\includegraphics[scale=.65]{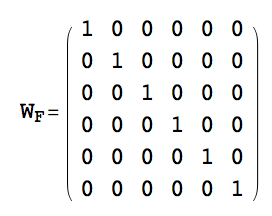}
\end{center}
By Proposition \ref{p:kerA-isom-Znminusr}, 
$$H_S^0(T)=\dlim(W_V,\ker \delta^0)=\dlim(
\left(
\begin{array}{ccc}
 2 & 1 & 1 \\
 1 & 2 & 1 \\
 1 & 1 & 2 \\
\end{array}
\right)
,\Z^{3}).$$
By Proposition \ref{p:extract-one-eigenvalues} we extract the $1$-eigenvalues and get
$$H_S^0(T)=\dlim(
\left(
\begin{array}{ccc}
 2 & 1 & 1 \\
 1 & 2 & 1 \\
 1 & 1 & 2 \\
\end{array}
\right)
,\Z^{3})=\Z^2\oplus\dlim(4,\Z)=\Z^2\oplus\Z[\frac12].$$
By Proposition \ref{p:kerOverIm-A-isom-Znminus}, $\lim\limits_{\to}(W_E,\,\frac{\ker\delta^1}{\im\delta^0})=\lim\limits_{\to}(W_E',\coker B)$ for some matrices $W_E',B$.
Then by Proposition \ref{p:cokerA-isom-Znminuss} we get rid of the coker and get
$$K_1(S)=H_S^1(T)=\dlim(W_E,\frac{\ker\delta^1}{\im\delta^0})=\dlim(
\left(
\begin{array}{ccc}
 0 & -2 & -2 \\
 0 & 2 & 0 \\
 0 & 0 & 2 \\
\end{array}
\right),\Z_2\oplus\Z^2)
=\Z[\frac12]^2.$$
\end{exam}

\begin{figure}[b]
\centerline{
\includegraphics[scale=.4]{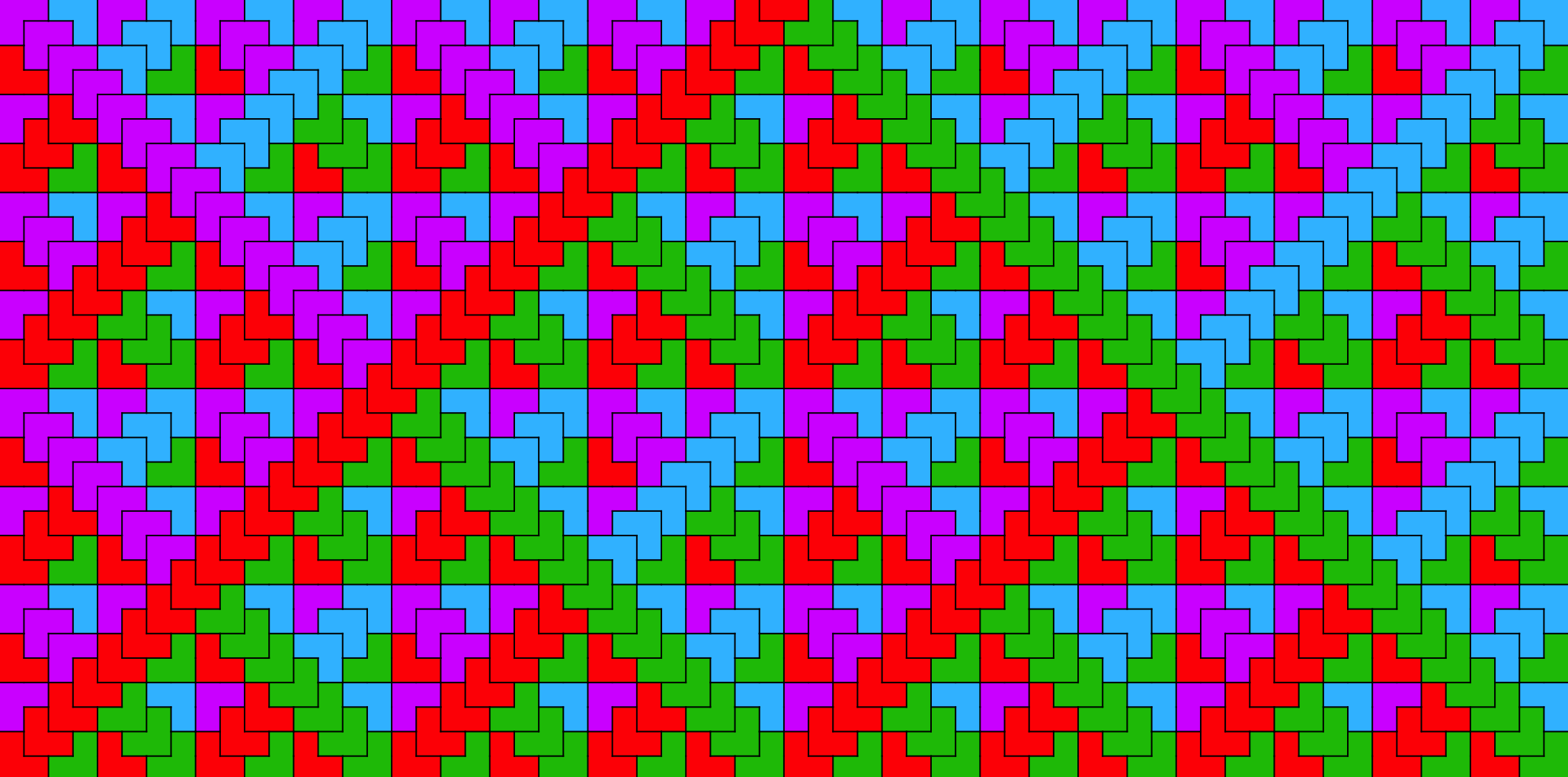}
}
 \caption{Chair tiling.\label{f:chair}}
\end{figure}

\begin{exam}[Chair tiling]
Let $T$ be the Chair tiling with proto-faces $a,b,c,d\in T$ and substitution rule given by:
\begin{center}
\includegraphics[scale=.3]{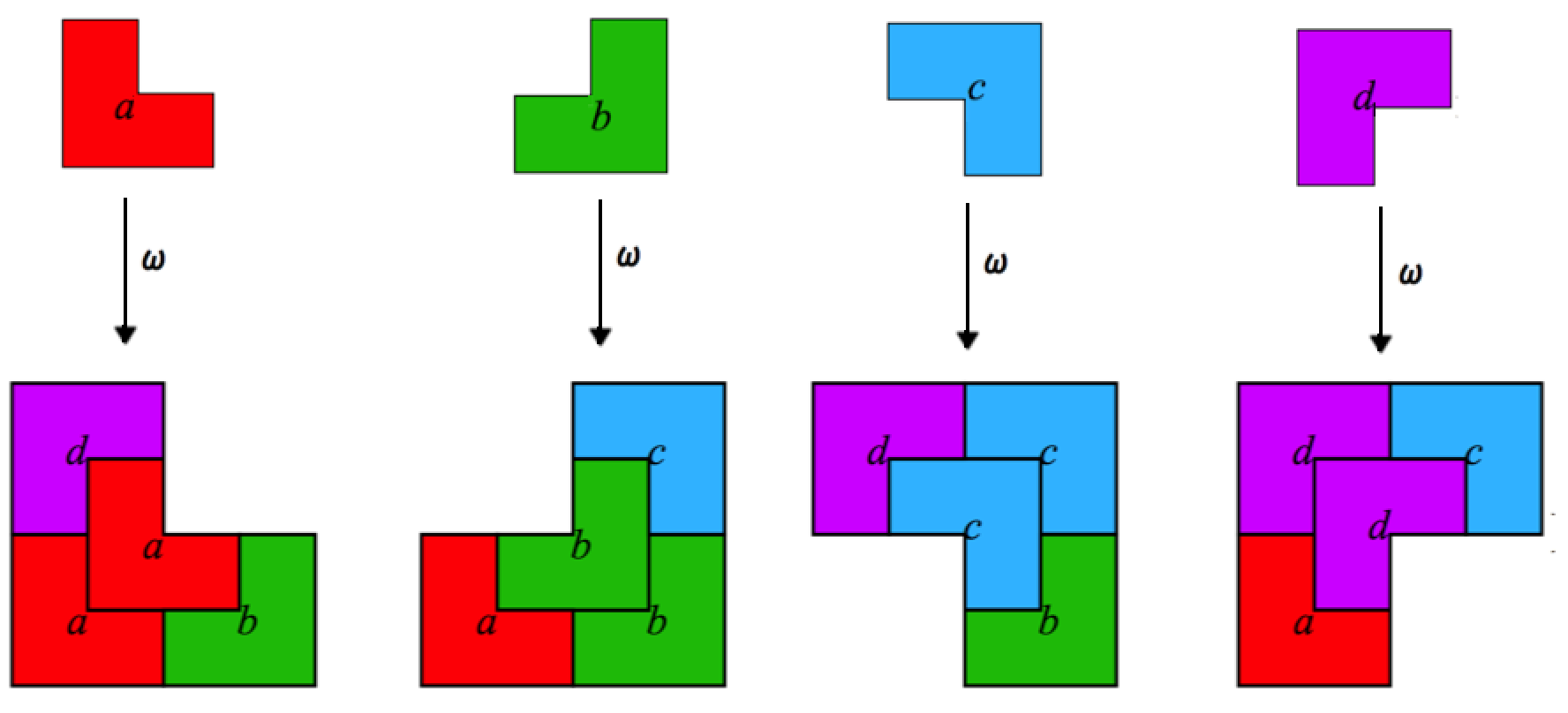}
\end{center}
The inflation factor is $\lambda=2$. There are $44$ stable edges and 47 stable vertices\\\\
{\centering
\includegraphics[scale=.3]{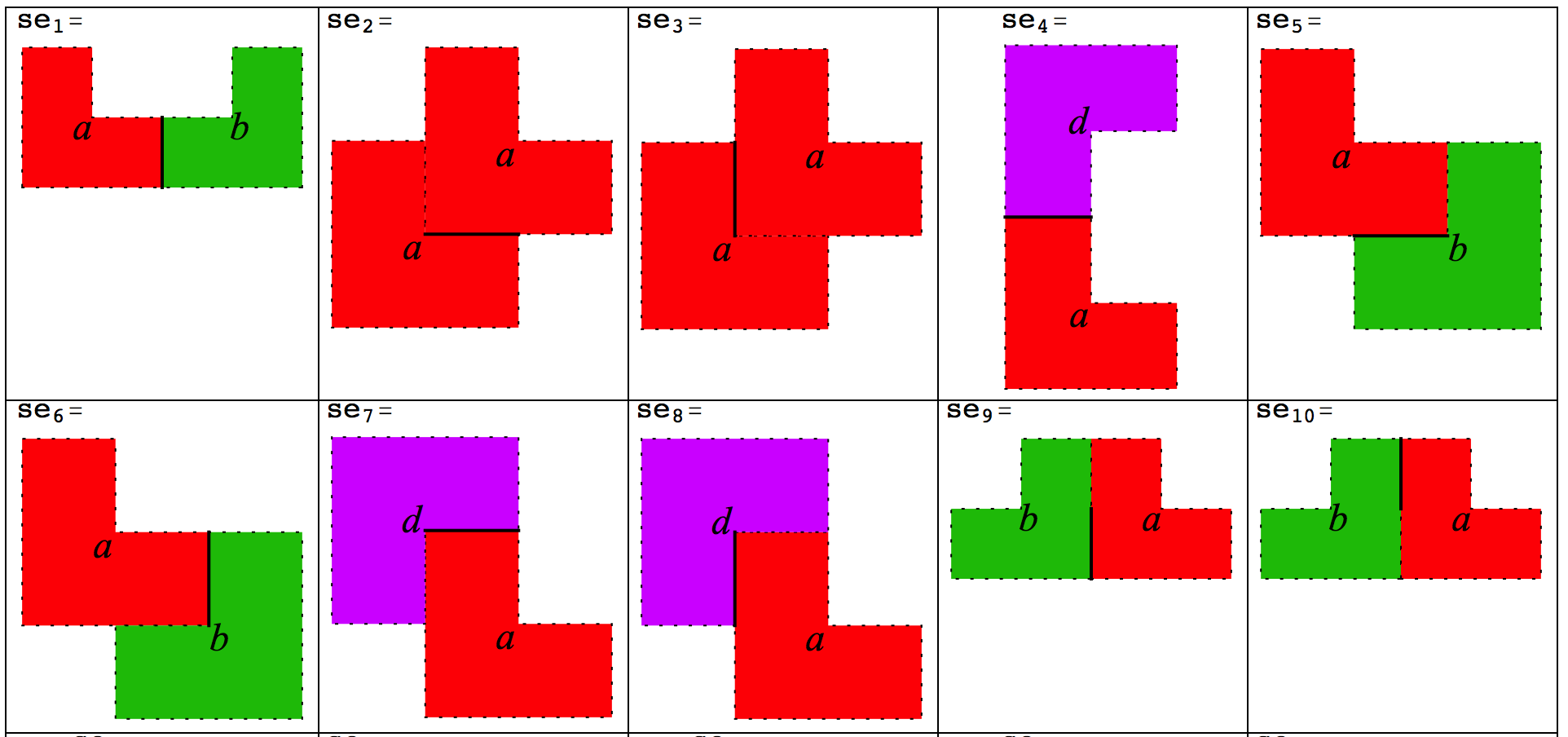}\\
\includegraphics[scale=.3]{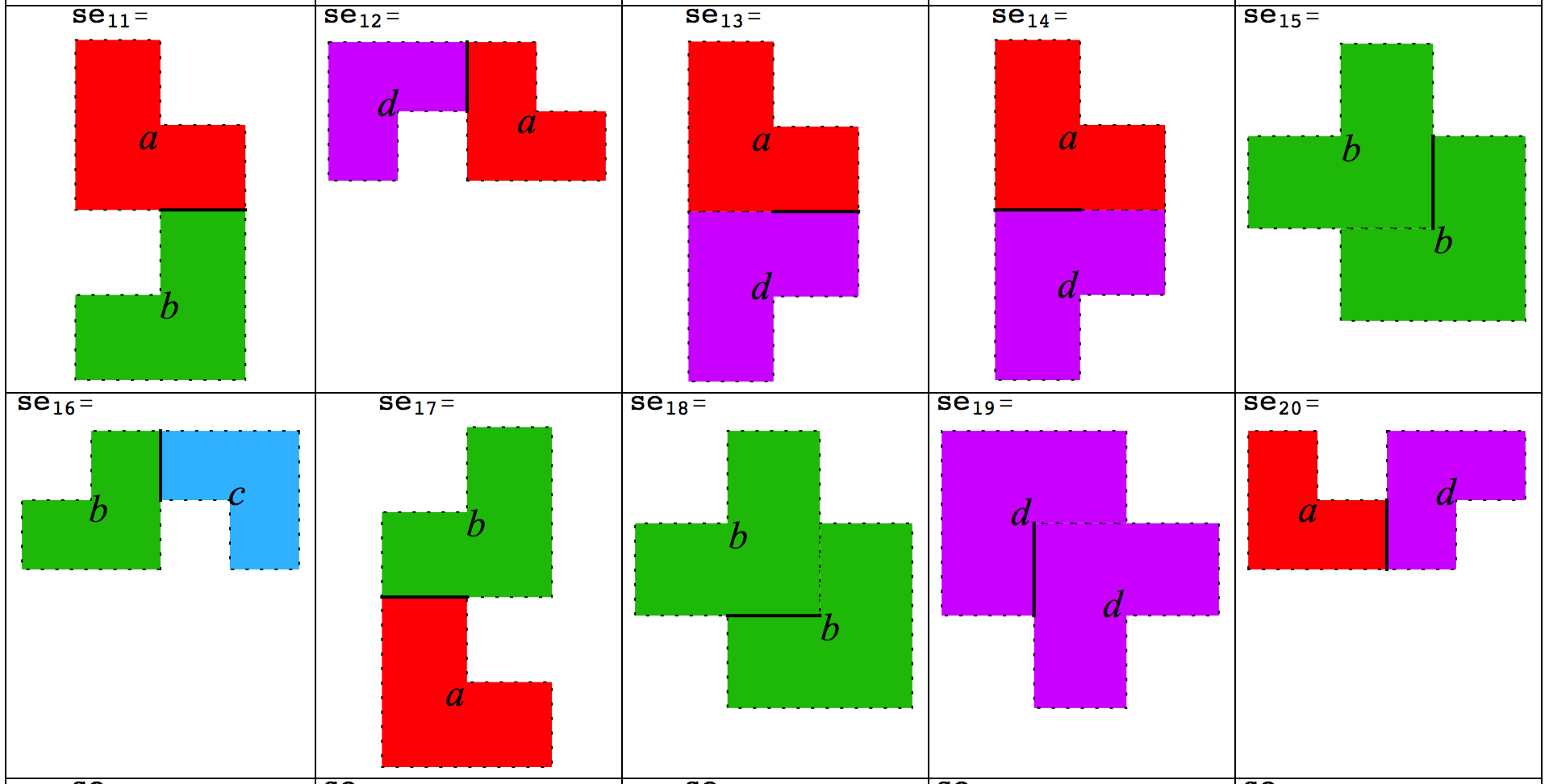}\\
\includegraphics[scale=.3]{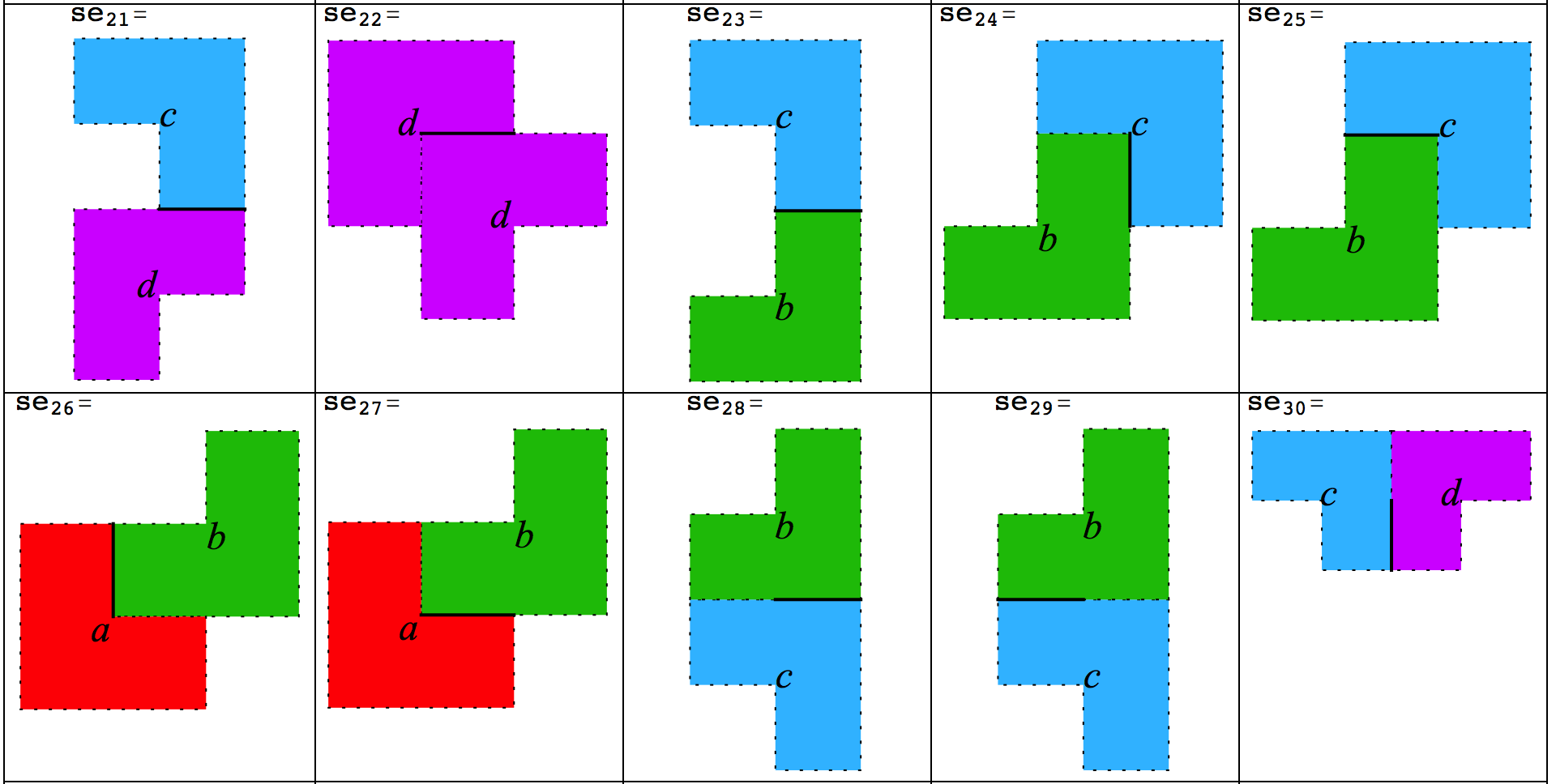}\\
\includegraphics[scale=.34]{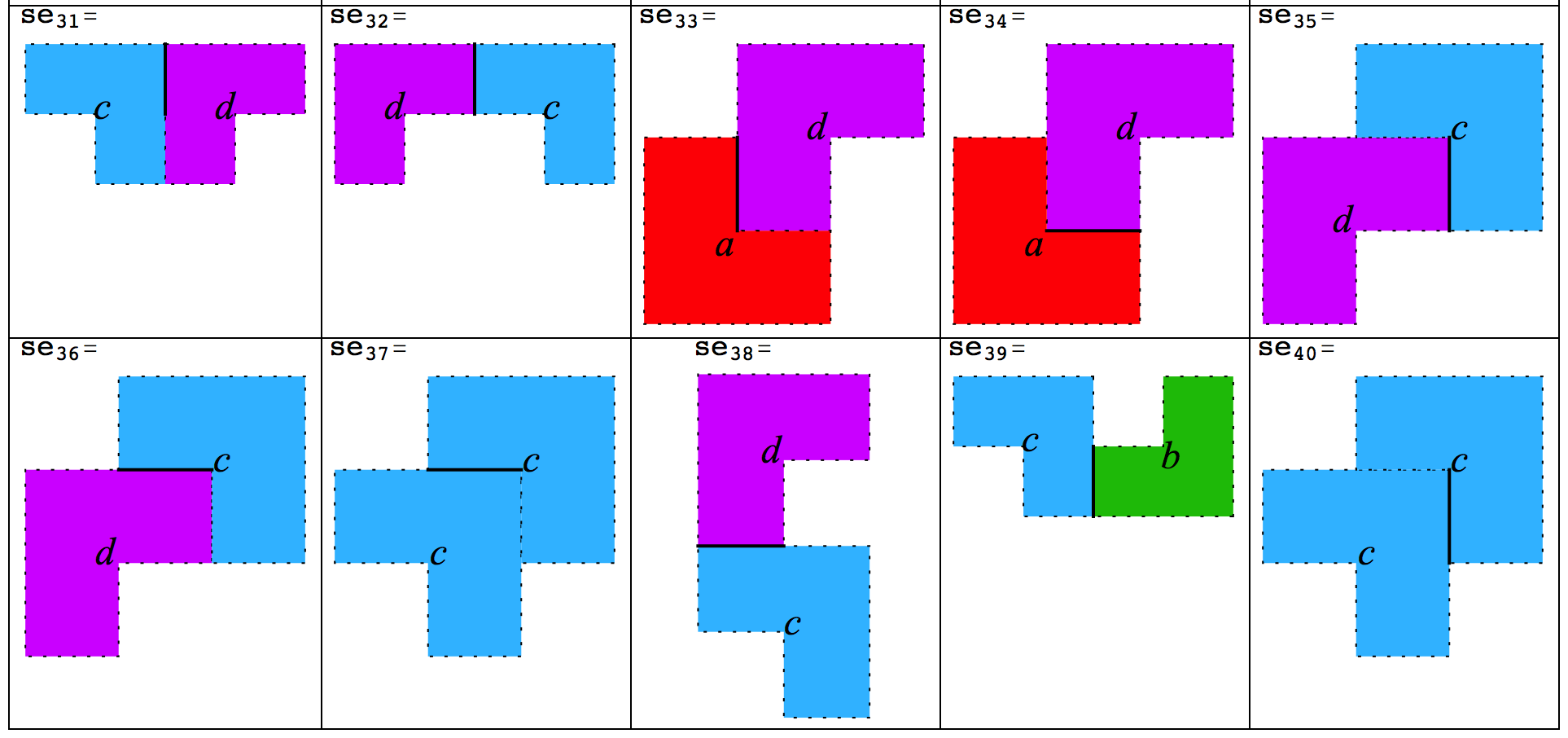}\\
\includegraphics[scale=.34]{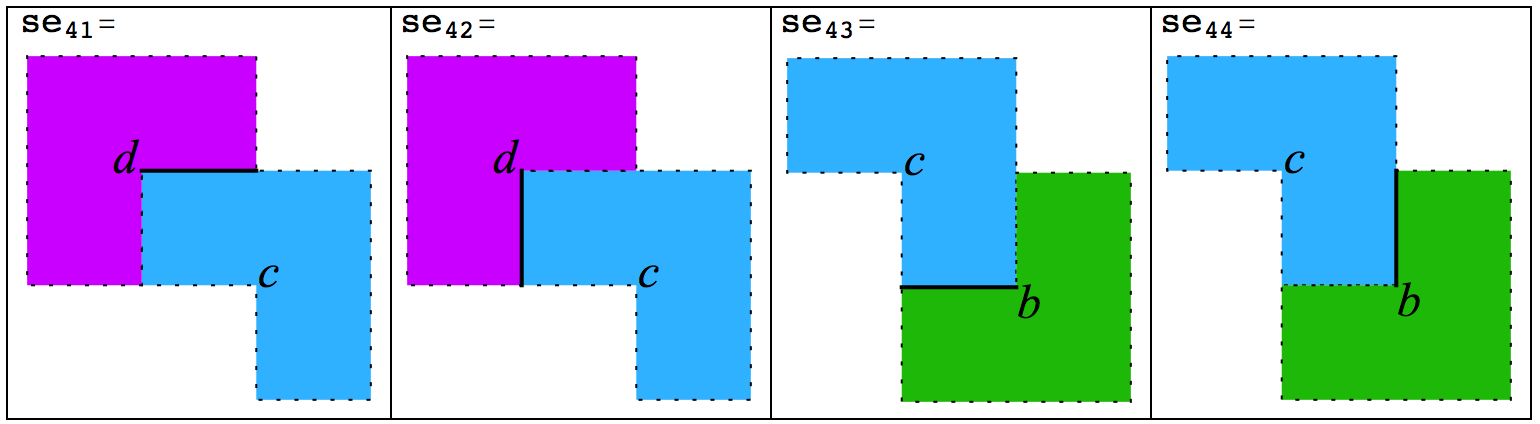}\\
\includegraphics[scale=.34]{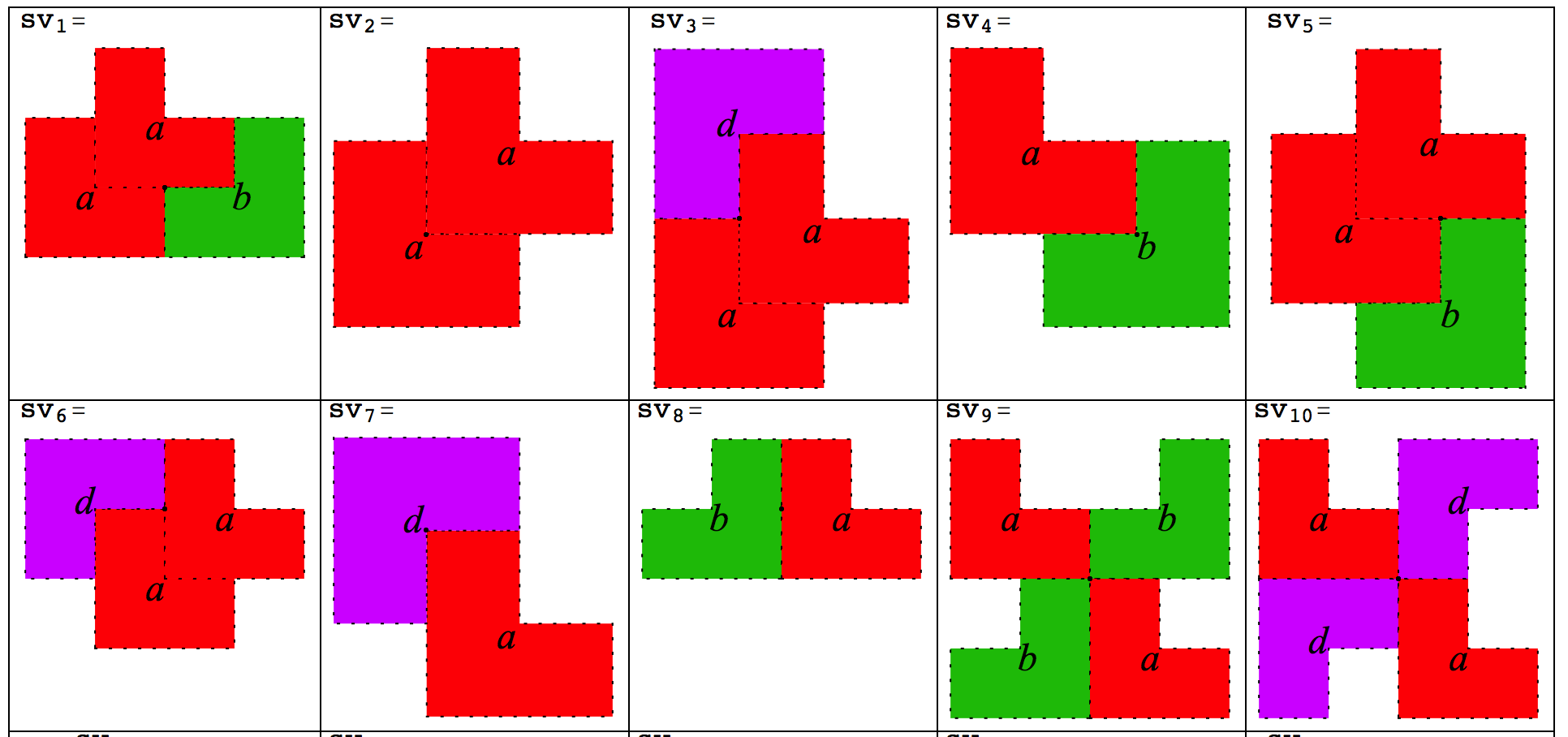}\\
\includegraphics[scale=.34]{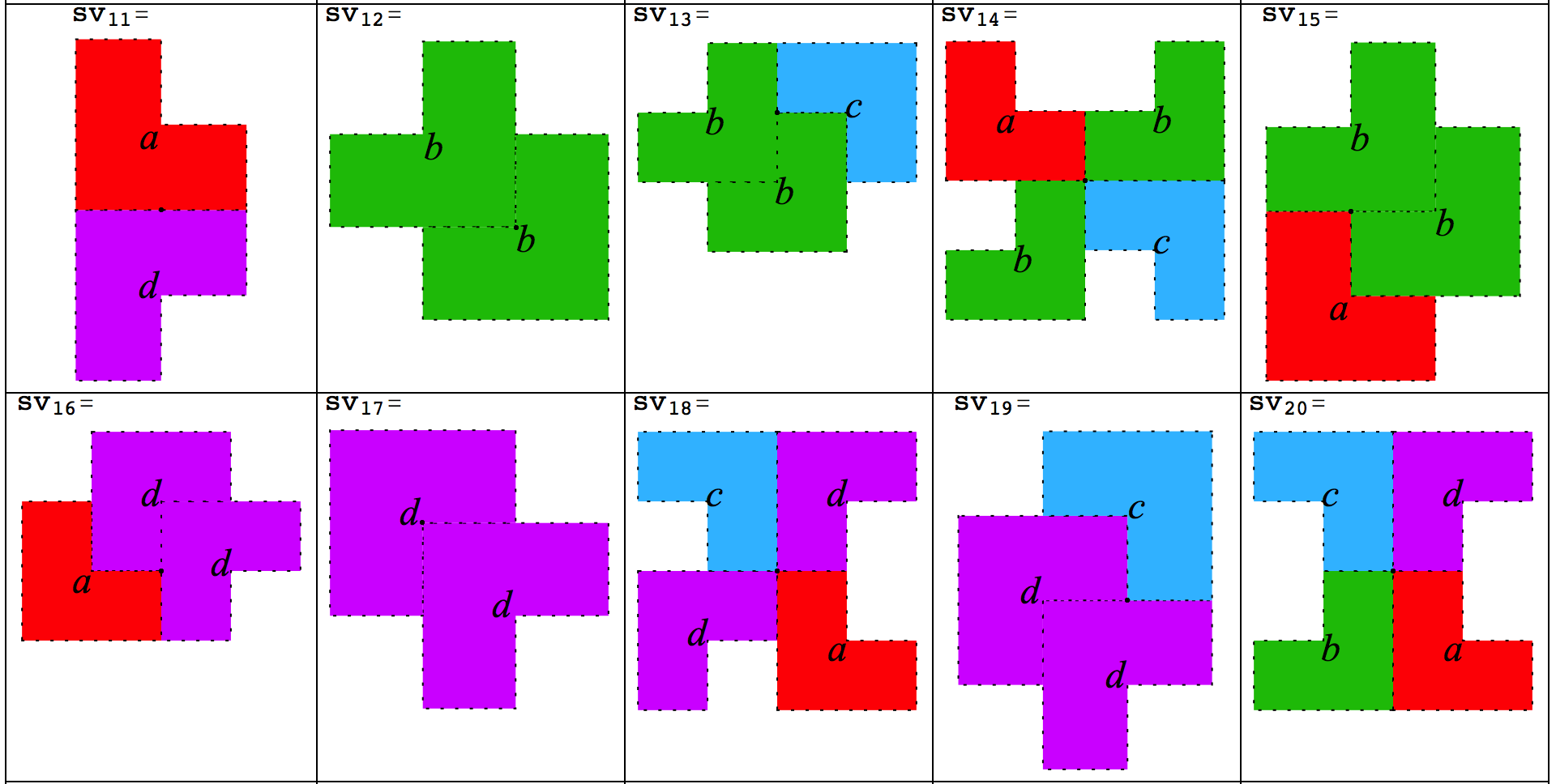}\\
\includegraphics[scale=.32]{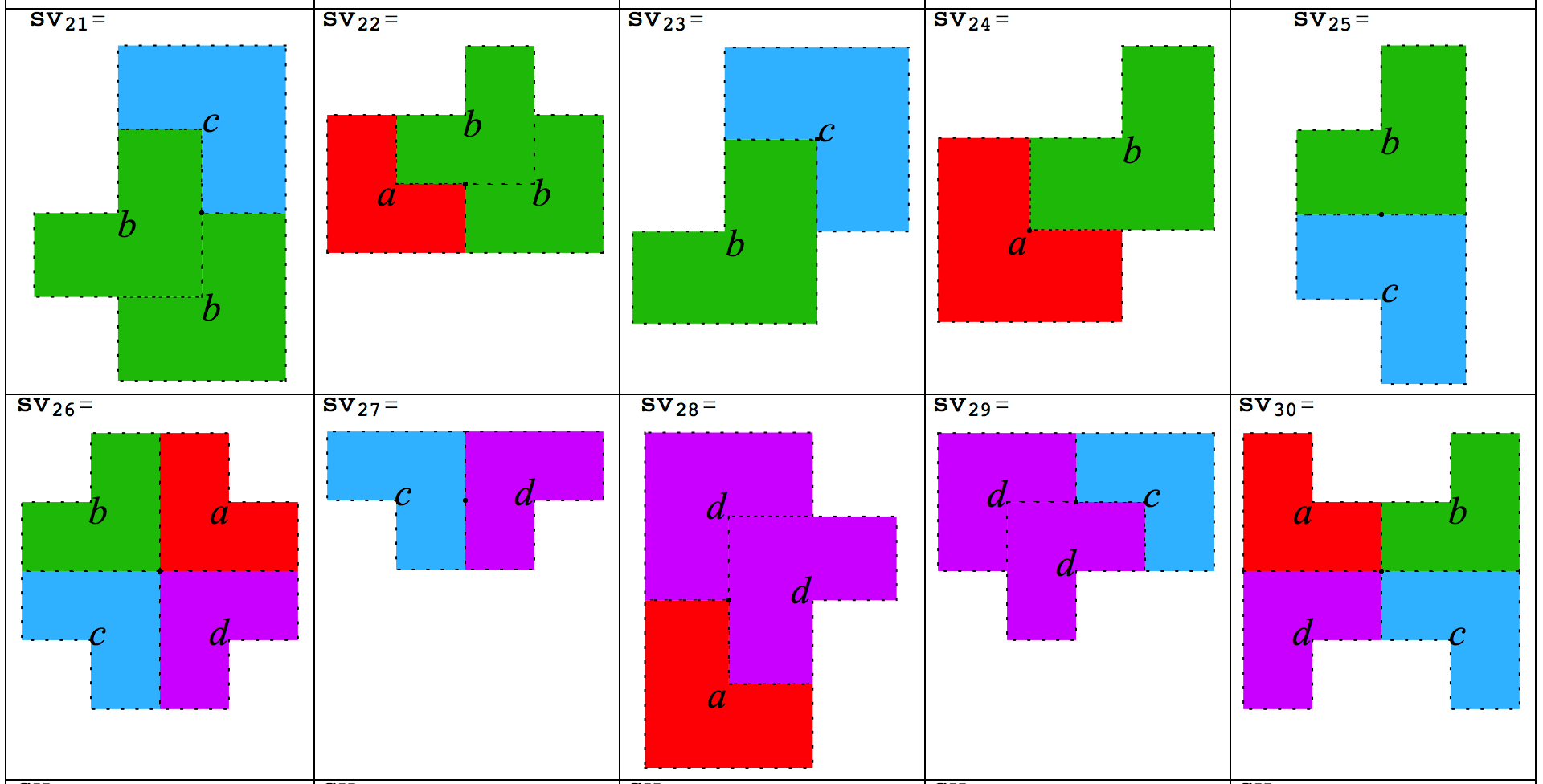}\\
\includegraphics[scale=.32]{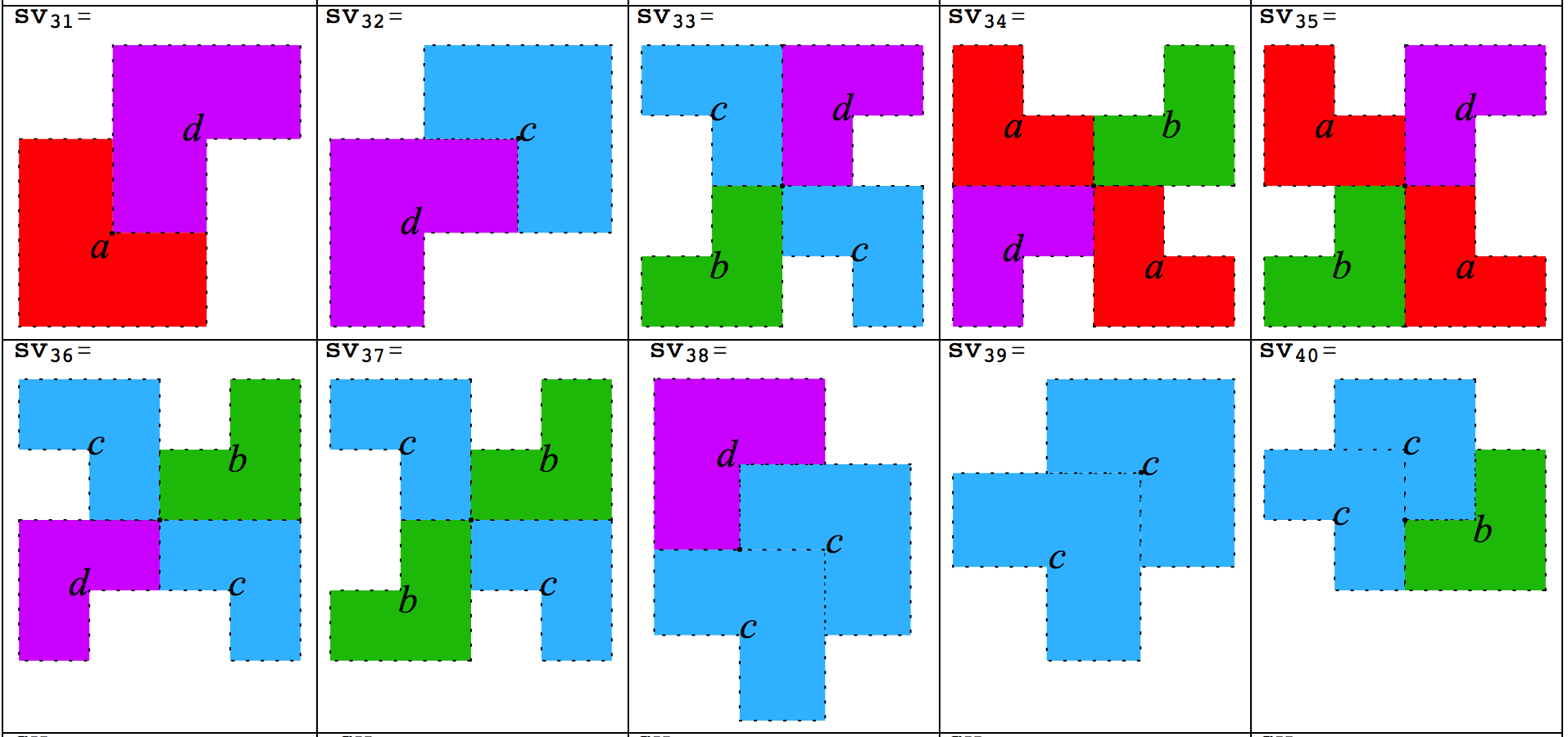}\\
\includegraphics[scale=.32]{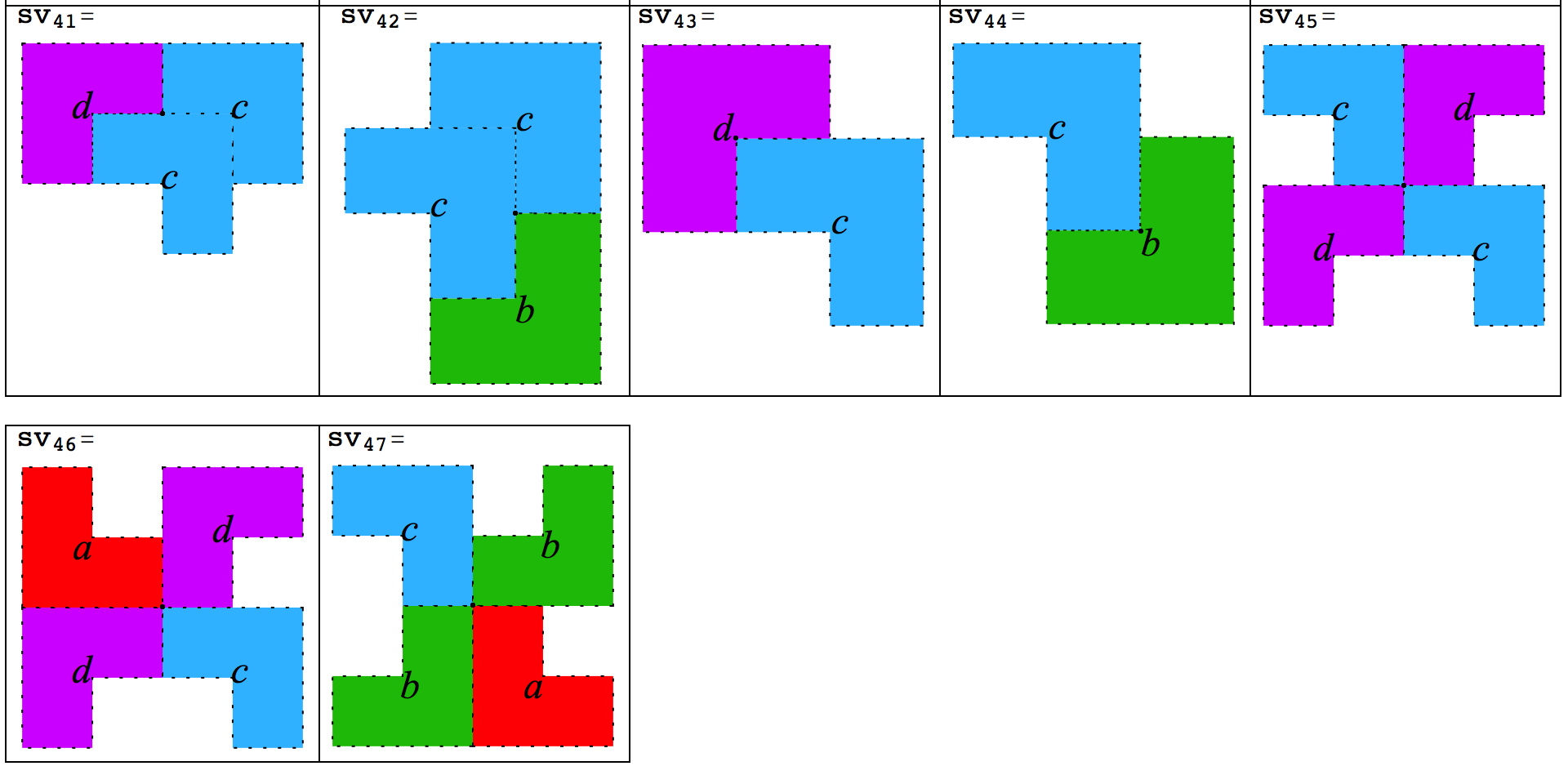}\\
}
$ $\\
We remark that we will write big matrices in the format: 
$$\{\{row_1\},\{row_2\},\ldots\}.$$
The exponential map $\delta^0:\Z^{\sV}\to\Z^{\sE}$ is given by the matrix:\\
{\fontsize{4}{0}\selectfont
$ $\\
\{\{1, 0, 0, 0, 0, 0, 0, 0, -1, 0, 0, 0, 0, -1, 0, 0, 0, 0, 0, 0, 0, 1, 
  0, 0, 0, 0, 0, 0, 0, -1, 0, 0, 0, -1, 0, 0, 0, 0, 0, 0, 0, 0, 0, 0, 
  0, 0, 0\}, \{1, -1, 0, 0, 1, 0, 0, 0, 0, 0, 0, 0, 0, 0, 0, 0, 0, 0, 0,
   0, 0, 0, 0, 0, 0, 0, 0, 0, 0, 0, 0, 0, 0, 0, 0, 0, 0, 0, 0, 0, 0, 
  0, 0, 0, 0, 0, 0\}, \{0, -1, 1, 0, 0, 1, 0, 0, 0, 0, 0, 0, 0, 0, 0, 0,
   0, 0, 0, 0, 0, 0, 0, 0, 0, 0, 0, 0, 0, 0, 0, 0, 0, 0, 0, 0, 0, 0, 
  0, 0, 0, 0, 0, 0, 0, 0, 0\}, \{0, 0, 1, 0, 0, 0, 0, 0, 0, -1, 0, 0, 0,
   0, 0, 0, 0, -1, 0, -1, 0, 0, 0, 0, 0, 0, 0, 1, 0, 0, 0, 0, 0, 
  0, -1, 0, 0, 0, 0, 0, 0, 0, 0, 0, 0, 0, 0\}, \{-1, 0, 0, 1, 0, 0, 0, 
  0, 0, 0, 0, 0, 0, 0, 0, 0, 0, 0, 0, 0, 0, 0, 0, 0, 0, 0, 0, 0, 0, 0,
   0, 0, 0, 0, 0, 0, 0, 0, 0, 0, 0, 0, 0, 0, 0, 0, 0\}, \{0, 0, 0, -1, 
  1, 0, 0, 0, 0, 0, 0, 0, 0, 0, 0, 0, 0, 0, 0, 0, 0, 0, 0, 0, 0, 0, 0,
   0, 0, 0, 0, 0, 0, 0, 0, 0, 0, 0, 0, 0, 0, 0, 0, 0, 0, 0, 0\}, \{0, 0,
   0, 0, 0, 1, -1, 0, 0, 0, 0, 0, 0, 0, 0, 0, 0, 0, 0, 0, 0, 0, 0, 0, 
  0, 0, 0, 0, 0, 0, 0, 0, 0, 0, 0, 0, 0, 0, 0, 0, 0, 0, 0, 0, 0, 0, 
  0\}, \{0, 0, -1, 0, 0, 0, 1, 0, 0, 0, 0, 0, 0, 0, 0, 0, 0, 0, 0, 0, 0,
   0, 0, 0, 0, 0, 0, 0, 0, 0, 0, 0, 0, 0, 0, 0, 0, 0, 0, 0, 0, 0, 0, 
  0, 0, 0, 0\}, \{0, 0, 0, 0, 0, 0, 0, 1, 0, 0, 0, 0, 0, 0, 0, 0, 0, 0, 
  0, 0, 0, 0, 0, 0, 0, -1, 0, 0, 0, 0, 0, 0, 0, 0, 0, 0, 0, 0, 0, 0, 
  0, 0, 0, 0, 0, 0, 0\}, \{0, 0, 0, 0, 0, 0, 0, -1, 1, 0, 0, 0, 0, 0, 0,
   0, 0, 0, 0, 1, 0, 0, 0, 0, 0, 0, 0, 0, 0, 0, 0, 0, 0, 0, 1, 0, 0, 
  0, 0, 0, 0, 0, 0, 0, 0, 0, 1\}, \{0, 0, 0, 0, -1, 0, 0, 0, 1, 0, 0, 0,
   0, 1, 0, 0, 0, 0, 0, 0, 0, 0, 0, 0, 0, 0, 0, 0, 0, 0, 0, 0, 0, 0, 
  1, 0, 0, 0, 0, 0, 0, 0, 0, 0, 0, 0, 0\}, \{0, 0, 0, 0, 0, -1, 0, 0, 0,
   1, 0, 0, 0, 0, 0, 0, 0, 1, 0, 0, 0, 0, 0, 0, 0, 0, 0, 0, 0, 0, 0, 
  0, 0, 1, 0, 0, 0, 0, 0, 0, 0, 0, 0, 0, 0, 0, 0\}, \{0, 0, 0, 0, 0, 0, 
  0, 0, 0, 1, -1, 0, 0, 0, 0, 0, 0, 0, 0, 0, 0, 0, 0, 0, 0, 0, 0, 0, 
  0, 1, 0, 0, 0, 1, 0, 0, 0, 0, 0, 0, 0, 0, 0, 0, 0, 1, 0\}, \{0, 0, 0, 
  0, 0, 0, 0, 0, 0, 0, 1, 0, 0, 0, 0, 0, 0, 0, 0, 0, 0, 0, 0, 0, 
  0, -1, 0, 0, 0, 0, 0, 0, 0, 0, 0, 0, 0, 0, 0, 0, 0, 0, 0, 0, 0, 0, 
  0\}, \{0, 0, 0, 0, 0, 0, 0, 0, 0, 0, 0, -1, 1, 0, 0, 0, 0, 0, 0, 0, 1,
   0, 0, 0, 0, 0, 0, 0, 0, 0, 0, 0, 0, 0, 0, 0, 0, 0, 0, 0, 0, 0, 0, 
  0, 0, 0, 0\}, \{0, 0, 0, 0, 0, 0, 0, 0, 0, 0, 0, 0, -1, 1, 0, 0, 0, 0,
   0, 0, 0, 0, 0, 0, 0, 0, 0, 0, 0, 0, 0, 0, 1, 0, 0, 0, 1, 0, 0, 0, 
  0, 0, 0, 0, 0, 0, 0\}, \{0, 0, 0, 0, 0, 0, 0, 0, -1, 0, 0, 0, 0, 0, 1,
   0, 0, 0, 0, 0, 0, 0, 0, 0, 0, 0, 0, 0, 0, 0, 0, 0, 0, -1, 0, 0, 0, 
  0, 0, 0, 0, 0, 0, 0, 0, 0, -1\}, \{0, 0, 0, 0, 0, 0, 0, 0, 0, 0, 0, 1,
   0, 0, -1, 0, 0, 0, 0, 0, 0, -1, 0, 0, 0, 0, 0, 0, 0, 0, 0, 0, 0, 0,
   0, 0, 0, 0, 0, 0, 0, 0, 0, 0, 0, 0, 0\}, \{0, 0, 0, 0, 0, 0, 0, 0, 0,
   0, 0, 0, 0, 0, 0, -1, 1, 0, 0, 0, 0, 0, 0, 0, 0, 0, 0, -1, 0, 0, 0,
   0, 0, 0, 0, 0, 0, 0, 0, 0, 0, 0, 0, 0, 0, 0, 0\}, \{0, 0, 0, 0, 0, 0,
   0, 0, 0, -1, 0, 0, 0, 0, 0, 1, 0, 0, 0, 0, 0, 0, 0, 0, 0, 0, 0, 0, 
  0, 0, 0, 0, 0, 0, -1, 0, 0, 0, 0, 0, 0, 0, 0, 0, 0, -1, 0\}, \{0, 0, 
  0, 0, 0, 0, 0, 0, 0, 0, 0, 0, 0, 0, 0, 0, 0, 1, -1, 0, 0, 0, 0, 0, 
  0, 0, 0, 0, 0, 0, 0, 0, 0, 0, 0, 1, 0, 0, 0, 0, 0, 0, 0, 0, 1, 0, 
  0\}, \{0, 0, 0, 0, 0, 0, 0, 0, 0, 0, 0, 0, 0, 0, 0, 0, -1, 0, 1, 0, 0,
   0, 0, 0, 0, 0, 0, 0, 1, 0, 0, 0, 0, 0, 0, 0, 0, 0, 0, 0, 0, 0, 0, 
  0, 0, 0, 0\}, \{0, 0, 0, 0, 0, 0, 0, 0, 0, 0, 0, 0, 0, 0, 0, 0, 0, 0, 
  0, 1, -1, 0, 0, 0, 0, 0, 0, 0, 0, 0, 0, 0, 1, 0, 0, 0, 1, 0, 0, 0, 
  0, -1, 0, 0, 0, 0, 1\}, \{0, 0, 0, 0, 0, 0, 0, 0, 0, 0, 0, 0, 0, 0, 0,
   0, 0, 0, 0, 0, -1, 0, 1, 0, 0, 0, 0, 0, 0, 0, 0, 0, 0, 0, 0, 0, 0, 
  0, 0, 0, 0, 0, 0, 0, 0, 0, 0\}, \{0, 0, 0, 0, 0, 0, 0, 0, 0, 0, 0, 
  0, -1, 0, 0, 0, 0, 0, 0, 0, 0, 0, 1, 0, 0, 0, 0, 0, 0, 0, 0, 0, 0, 
  0, 0, 0, 0, 0, 0, 0, 0, 0, 0, 0, 0, 0, 0\}, \{0, 0, 0, 0, 0, 0, 0, 0, 
  0, 0, 0, 0, 0, 0, 1, 0, 0, 0, 0, 0, 0, 0, 0, -1, 0, 0, 0, 0, 0, 0, 
  0, 0, 0, 0, 0, 0, 0, 0, 0, 0, 0, 0, 0, 0, 0, 0, 0\}, \{0, 0, 0, 0, 0, 
  0, 0, 0, 0, 0, 0, 0, 0, 0, 0, 0, 0, 0, 0, 0, 0, 1, 0, -1, 0, 0, 0, 
  0, 0, 0, 0, 0, 0, 0, 0, 0, 0, 0, 0, 0, 0, 0, 0, 0, 0, 0, 0\}, \{0, 0, 
  0, 0, 0, 0, 0, 0, 0, 0, 0, 0, 0, 0, 0, 0, 0, 0, 0, 0, 0, 0, 0, 
  0, -1, 1, 0, 0, 0, 0, 0, 0, 0, 0, 0, 0, 0, 0, 0, 0, 0, 0, 0, 0, 0, 
  0, 0\}, \{0, 0, 0, 0, 0, 0, 0, 0, 0, 0, 0, 0, 0, -1, 0, 0, 0, 0, 0, 0,
   0, 0, 0, 0, 1, 0, 0, 0, 0, -1, 0, 0, 0, 0, 0, -1, -1, 0, 0, 0, 0, 
  0, 0, 0, 0, 0, 0\}, \{0, 0, 0, 0, 0, 0, 0, 0, 0, 0, 0, 0, 0, 0, 0, 0, 
  0, -1, 0, -1, 0, 0, 0, 0, 0, 0, 1, 0, 0, 0, 0, 0, -1, 0, 0, 0, 0, 0,
   0, 0, 0, 0, 0, 0, -1, 0, 0\}, \{0, 0, 0, 0, 0, 0, 0, 0, 0, 0, 0, 0, 
  0, 0, 0, 0, 0, 0, 0, 0, 0, 0, 0, 0, 0, 1, -1, 0, 0, 0, 0, 0, 0, 0, 
  0, 0, 0, 0, 0, 0, 0, 0, 0, 0, 0, 0, 0\}, \{0, 0, 0, 0, 0, 0, 0, 0, 0, 
  0, 0, 0, 0, 0, 0, 0, 0, 0, 0, 0, 0, 0, 0, 0, 0, 0, 0, 0, -1, 1, 0, 
  0, 0, 0, 0, 1, 0, 0, 0, 0, -1, 0, 0, 0, 1, 1, 0\}, \{0, 0, 0, 0, 0, 0,
   0, 0, 0, 0, 0, 0, 0, 0, 0, 0, 0, 0, 0, 0, 0, 0, 0, 0, 0, 0, 0, 1, 
  0, 0, -1, 0, 0, 0, 0, 0, 0, 0, 0, 0, 0, 0, 0, 0, 0, 0, 0\}, \{0, 0, 0,
   0, 0, 0, 0, 0, 0, 0, 0, 0, 0, 0, 0, 1, 0, 0, 0, 0, 0, 0, 0, 0, 0, 
  0, 0, 0, 0, 0, -1, 0, 0, 0, 0, 0, 0, 0, 0, 0, 0, 0, 0, 0, 0, 0, 
  0\}, \{0, 0, 0, 0, 0, 0, 0, 0, 0, 0, 0, 0, 0, 0, 0, 0, 0, 0, -1, 0, 0,
   0, 0, 0, 0, 0, 0, 0, 0, 0, 0, 1, 0, 0, 0, 0, 0, 0, 0, 0, 0, 0, 0, 
  0, 0, 0, 0\}, \{0, 0, 0, 0, 0, 0, 0, 0, 0, 0, 0, 0, 0, 0, 0, 0, 0, 0, 
  0, 0, 0, 0, 0, 0, 0, 0, 0, 0, -1, 0, 0, 1, 0, 0, 0, 0, 0, 0, 0, 0, 
  0, 0, 0, 0, 0, 0, 0\}, \{0, 0, 0, 0, 0, 0, 0, 0, 0, 0, 0, 0, 0, 0, 0, 
  0, 0, 0, 0, 0, 0, 0, 0, 0, 0, 0, 0, 0, 0, 0, 0, 0, 0, 0, 0, 0, 
  0, -1, 1, 0, -1, 0, 0, 0, 0, 0, 0\}, \{0, 0, 0, 0, 0, 0, 0, 0, 0, 0, 
  0, 0, 0, 0, 0, 0, 0, 0, 0, 0, 0, 0, 0, 0, 0, 0, 0, 0, 0, 0, 0, 
  0, -1, 0, 0, 0, 0, 1, 0, 0, 0, 0, 0, 0, -1, -1, 0\}, \{0, 0, 0, 0, 0, 
  0, 0, 0, 0, 0, 0, 0, 0, 0, 0, 0, 0, 0, 0, 0, 0, 0, 0, 0, 0, 0, 0, 0,
   0, 0, 0, 0, 0, 0, 0, -1, -1, 0, 0, 1, 0, 0, 0, 0, 0, 0, -1\}, \{0, 0,
   0, 0, 0, 0, 0, 0, 0, 0, 0, 0, 0, 0, 0, 0, 0, 0, 0, 0, 0, 0, 0, 0, 
  0, 0, 0, 0, 0, 0, 0, 0, 0, 0, 0, 0, 0, 0, 1, -1, 0, -1, 0, 0, 0, 0, 
  0\}, \{0, 0, 0, 0, 0, 0, 0, 0, 0, 0, 0, 0, 0, 0, 0, 0, 0, 0, 0, 0, 0, 
  0, 0, 0, 0, 0, 0, 0, 0, 0, 0, 0, 0, 0, 0, 0, 0, 0, 0, 0, 1, 0, -1, 
  0, 0, 0, 0\}, \{0, 0, 0, 0, 0, 0, 0, 0, 0, 0, 0, 0, 0, 0, 0, 0, 0, 0, 
  0, 0, 0, 0, 0, 0, 0, 0, 0, 0, 0, 0, 0, 0, 0, 0, 0, 0, 0, -1, 0, 0, 
  0, 0, 1, 0, 0, 0, 0\}, \{0, 0, 0, 0, 0, 0, 0, 0, 0, 0, 0, 0, 0, 0, 0, 
  0, 0, 0, 0, 0, 0, 0, 0, 0, 0, 0, 0, 0, 0, 0, 0, 0, 0, 0, 0, 0, 0, 0,
   0, -1, 0, 0, 0, 1, 0, 0, 0\}, \{0, 0, 0, 0, 0, 0, 0, 0, 0, 0, 0, 0, 
  0, 0, 0, 0, 0, 0, 0, 0, 0, 0, 0, 0, 0, 0, 0, 0, 0, 0, 0, 0, 0, 0, 0,
   0, 0, 0, 0, 0, 0, 1, 0, -1, 0, 0, 0\}\}
       
}   
The index map $\delta^1:\Z^{\sE}\to\Z^{sF}$ is given by the matrix:\\
{\fontsize{4}{0}\selectfont
$ $\\
\{\{1, 0, 0, -1, 1, 1, -1, -1, -1, -1, 1, -1, 1, 1, 0, 0, -1, 0, 0, 1, 
  0, 0, 0, 0, 0, 1, -1, 0, 0, 0, 0, 0, 1, -1, 0, 0, 0, 0, 0, 0, 0, 0, 
  0, 0\}, \{-1, 0, 0, 0, -1, -1, 0, 0, 1, 1, -1, 0, 0, 0, 0, 1, 1, 0, 0,
   0, 0, 0, -1, 1, -1, -1, 1, 1, 1, 0, 0, 0, 0, 0, 0, 0, 0, 0, -1, 0, 
  0, 0, -1, -1\}, \{0, 0, 0, 0, 0, 0, 0, 0, 0, 0, 0, 0, 0, 0, 0, -1, 0, 
  0, 0, 0, 1, 0, 1, -1, 1, 0, 0, -1, -1, 1, 1, -1, 0, 0, -1, 1, 0, -1,
   1, 0, -1, -1, 1, 1\}, \{0, 0, 0, 1, 0, 0, 1, 1, 0, 0, 0, 1, -1, -1, 
  0, 0, 0, 0, 0, -1, -1, 0, 0, 0, 0, 0, 0, 0, 0, -1, -1, 1, -1, 1, 
  1, -1, 0, 1, 0, 0, 1, 1, 0, 0\}\}
       
}   

We illustrate the homotopy $h_s$, $0\le s\le 1$, on the vertices in the following figure:
\begin{center}
\includegraphics[scale=.3]{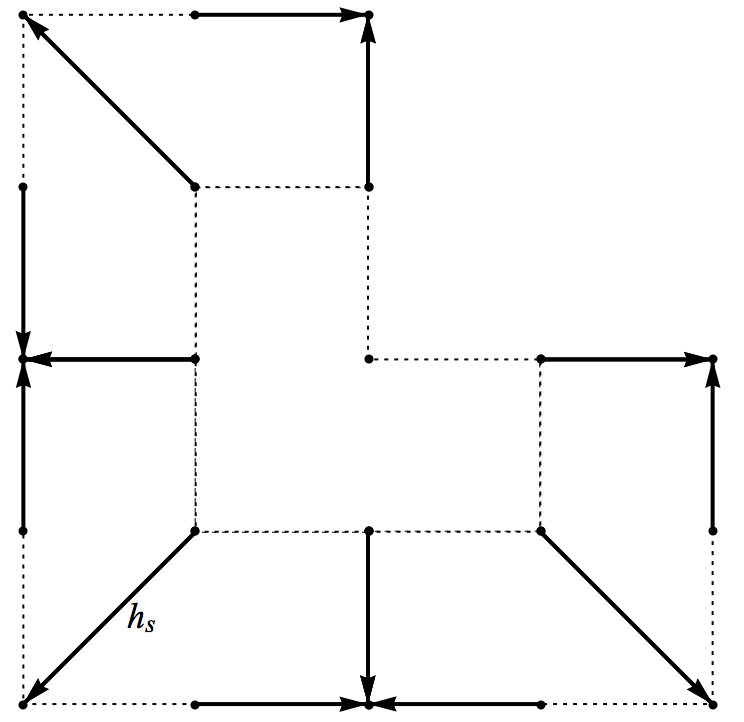}
\end{center}
The substitution-homotopy map $W_V:\Z^{\sV}\to\Z^{sV}$ is given by the matrix:\\
{\fontsize{4}{0}\selectfont
$ $\\
\{\{1, 0, 0, 0, 1, 0, 0, 0, 0, 0, 1, 0, 0, 0, 0, 0, 0, 0, 0, 0, 0, 0, 0,
   0, 0, 0, 0, 0, 0, 0, 0, 0, 0, 0, 0, 0, 0, 0, 0, 0, 0, 0, 0, 0, 0, 
  0, 0\}, \{0, 2, 0, 0, 0, 0, 0, 0, 0, 0, 0, 0, 0, 0, 0, 0, 0, 0, 0, 0, 
  0, 0, 0, 1, 0, 1, 0, 0, 0, 0, 1, 0, 0, 0, 0, 0, 0, 0, 0, 0, 0, 0, 0,
   0, 0, 0, 0\}, \{0, 0, 1, 0, 0, 1, 0, 1, 0, 0, 0, 0, 0, 0, 0, 0, 0, 0,
   0, 0, 0, 0, 0, 0, 0, 0, 0, 0, 0, 0, 0, 0, 0, 0, 0, 0, 0, 0, 0, 0, 
  0, 0, 0, 0, 0, 0, 0\}, \{0, 0, 0, 1, 0, 0, 0, 0, 1, 1, 0, 0, 0, 1, 0, 
  0, 0, 0, 0, 0, 0, 0, 0, 0, 0, 0, 0, 0, 0, 1, 0, 0, 0, 1, 1, 0, 0, 0,
   0, 0, 0, 0, 0, 0, 0, 1, 0\}, \{1, 0, 0, 0, 1, 0, 0, 0, 0, 0, 0, 0, 0,
   0, 0, 1, 0, 0, 0, 0, 0, 1, 0, 0, 0, 0, 0, 0, 0, 0, 0, 0, 0, 0, 0, 
  0, 0, 0, 0, 0, 0, 0, 0, 0, 0, 0, 0\}, \{0, 0, 1, 0, 0, 1, 0, 0, 0, 0, 
  0, 0, 0, 0, 1, 0, 0, 0, 0, 0, 0, 0, 0, 0, 0, 0, 0, 1, 0, 0, 0, 0, 0,
   0, 0, 0, 0, 0, 0, 0, 0, 0, 0, 0, 0, 0, 0\}, \{0, 0, 0, 0, 0, 0, 1, 0,
   1, 1, 0, 0, 0, 0, 0, 0, 0, 1, 0, 1, 0, 0, 0, 0, 0, 0, 0, 0, 0, 0, 
  0, 0, 0, 1, 1, 0, 0, 0, 0, 0, 0, 0, 0, 0, 0, 0, 1\}, \{1, 0, 0, 0, 0, 
  0, 0, 1, 0, 0, 0, 0, 0, 0, 0, 1, 0, 0, 0, 0, 0, 1, 0, 0, 0, 0, 1, 0,
   0, 0, 0, 0, 0, 0, 0, 0, 0, 0, 0, 1, 0, 0, 0, 0, 0, 0, 0\}, \{1, 0, 0,
   0, 0, 0, 0, 0, 0, 0, 0, 0, 0, 0, 0, 0, 0, 0, 0, 0, 0, 1, 0, 0, 0, 
  0, 0, 0, 0, 0, 0, 0, 0, 0, 0, 0, 0, 0, 0, 0, 0, 0, 0, 0, 0, 0, 
  0\}, \{0, 0, 1, 0, 0, 0, 0, 0, 0, 0, 0, 0, 0, 0, 0, 0, 0, 0, 0, 0, 0, 
  0, 0, 0, 0, 0, 0, 1, 0, 0, 0, 0, 0, 0, 0, 0, 0, 0, 0, 0, 0, 0, 0, 0,
   0, 0, 0\}, \{0, 0, 1, 0, 0, 0, 0, 0, 0, 0, 1, 0, 0, 0, 1, 0, 0, 0, 0,
   0, 0, 0, 0, 0, 1, 0, 0, 1, 0, 0, 0, 0, 0, 0, 0, 0, 0, 1, 0, 0, 0, 
  0, 0, 0, 0, 0, 0\}, \{0, 0, 0, 1, 0, 0, 0, 0, 0, 0, 0, 2, 0, 0, 0, 0, 
  0, 0, 0, 0, 0, 0, 0, 0, 0, 1, 0, 0, 0, 0, 0, 0, 0, 0, 0, 0, 0, 0, 0,
   0, 0, 0, 0, 1, 0, 0, 0\}, \{0, 0, 0, 0, 1, 0, 0, 0, 0, 0, 0, 0, 1, 0,
   0, 0, 0, 0, 0, 0, 1, 0, 0, 0, 0, 0, 0, 0, 0, 0, 0, 0, 0, 0, 0, 0, 
  0, 0, 0, 0, 0, 1, 0, 0, 0, 0, 0\}, \{0, 0, 0, 0, 1, 0, 0, 0, 0, 0, 0, 
  0, 0, 0, 0, 0, 0, 0, 0, 0, 0, 0, 0, 0, 0, 0, 0, 0, 0, 0, 0, 0, 0, 0,
   0, 0, 0, 0, 0, 0, 0, 0, 0, 0, 0, 0, 0\}, \{1, 0, 0, 0, 0, 0, 0, 0, 0,
   0, 0, 0, 0, 0, 1, 0, 0, 0, 0, 0, 0, 1, 0, 0, 0, 0, 0, 0, 0, 0, 0, 
  0, 0, 0, 0, 0, 0, 0, 0, 1, 0, 0, 0, 0, 0, 0, 0\}, \{0, 0, 1, 0, 0, 0, 
  0, 0, 0, 0, 0, 0, 0, 0, 0, 1, 0, 0, 0, 0, 0, 0, 0, 0, 0, 0, 0, 1, 0,
   0, 0, 0, 0, 0, 0, 0, 0, 1, 0, 0, 0, 0, 0, 0, 0, 0, 0\}, \{0, 0, 0, 0,
   0, 0, 1, 0, 0, 0, 0, 0, 0, 0, 0, 0, 2, 0, 0, 0, 0, 0, 0, 0, 0, 1, 
  0, 0, 0, 0, 0, 0, 0, 0, 0, 0, 0, 0, 0, 0, 0, 0, 1, 0, 0, 0, 0\}, \{0, 
  0, 0, 0, 0, 1, 0, 0, 0, 0, 0, 0, 0, 0, 0, 0, 0, 0, 0, 0, 0, 0, 0, 0,
   0, 0, 0, 0, 0, 0, 0, 0, 0, 0, 0, 0, 0, 0, 0, 0, 0, 0, 0, 0, 0, 0, 
  0\}, \{0, 0, 0, 0, 0, 1, 0, 0, 0, 0, 0, 0, 0, 0, 0, 0, 0, 0, 1, 0, 0, 
  0, 0, 0, 0, 0, 0, 0, 1, 0, 0, 0, 0, 0, 0, 0, 0, 0, 0, 0, 1, 0, 0, 0,
   0, 0, 0\}, \{0, 0, 0, 0, 0, 0, 0, 1, 0, 0, 0, 0, 0, 0, 0, 0, 0, 0, 0,
   0, 0, 0, 0, 0, 0, 0, 1, 0, 0, 0, 0, 0, 0, 0, 0, 0, 0, 0, 0, 0, 0, 
  0, 0, 0, 0, 0, 0\}, \{0, 0, 0, 0, 0, 0, 0, 1, 0, 0, 0, 0, 1, 0, 0, 0, 
  0, 0, 0, 0, 1, 0, 0, 0, 0, 0, 0, 0, 0, 0, 0, 0, 0, 0, 0, 0, 0, 0, 0,
   0, 0, 0, 0, 0, 0, 0, 0\}, \{0, 0, 0, 0, 0, 0, 0, 0, 0, 0, 0, 0, 0, 0,
   1, 0, 0, 0, 0, 0, 0, 1, 0, 0, 1, 0, 0, 0, 0, 0, 0, 0, 0, 0, 0, 0, 
  0, 0, 0, 0, 0, 0, 0, 0, 0, 0, 0\}, \{0, 0, 0, 0, 0, 0, 0, 0, 1, 0, 0, 
  0, 0, 1, 0, 0, 0, 0, 0, 1, 0, 0, 1, 0, 0, 0, 0, 0, 0, 0, 0, 0, 1, 0,
   1, 0, 1, 0, 0, 0, 0, 0, 0, 0, 0, 0, 1\}, \{0, 0, 0, 0, 0, 0, 0, 0, 1,
   0, 0, 0, 0, 1, 0, 0, 0, 0, 0, 0, 0, 0, 0, 1, 0, 0, 0, 0, 0, 1, 0, 
  0, 0, 1, 0, 1, 1, 0, 0, 0, 0, 0, 0, 0, 0, 0, 1\}, \{0, 0, 0, 0, 1, 0, 
  0, 0, 0, 0, 1, 0, 0, 0, 0, 0, 0, 0, 1, 0, 1, 0, 0, 0, 1, 0, 0, 0, 0,
   0, 0, 0, 0, 0, 0, 0, 0, 0, 0, 0, 0, 1, 0, 0, 0, 0, 0\}, \{0, 0, 0, 0,
   0, 0, 0, 0, 1, 1, 0, 0, 0, 1, 0, 0, 0, 1, 0, 1, 0, 0, 0, 0, 0, 1, 
  0, 0, 0, 1, 0, 0, 1, 1, 1, 1, 1, 0, 0, 0, 0, 0, 0, 0, 1, 1, 1\}, \{0, 
  0, 0, 0, 0, 1, 0, 1, 0, 0, 0, 0, 1, 0, 0, 0, 0, 0, 0, 0, 0, 0, 0, 0,
   0, 0, 1, 0, 1, 0, 0, 0, 0, 0, 0, 0, 0, 0, 0, 0, 1, 0, 0, 0, 0, 0, 
  0\}, \{0, 0, 0, 0, 0, 0, 0, 0, 0, 0, 0, 0, 0, 0, 0, 1, 0, 0, 0, 0, 0, 
  0, 0, 0, 0, 0, 1, 1, 0, 0, 0, 0, 0, 0, 0, 0, 0, 0, 0, 0, 0, 0, 0, 0,
   0, 0, 0\}, \{0, 0, 0, 0, 0, 0, 0, 0, 0, 0, 1, 0, 0, 0, 0, 0, 0, 0, 1,
   0, 0, 0, 0, 0, 0, 0, 0, 0, 1, 0, 0, 0, 0, 0, 0, 0, 0, 0, 0, 0, 0, 
  0, 0, 0, 0, 0, 0\}, \{0, 0, 0, 0, 0, 0, 0, 0, 0, 0, 1, 0, 0, 0, 0, 0, 
  0, 0, 0, 0, 0, 0, 0, 0, 1, 0, 0, 0, 0, 0, 0, 0, 0, 0, 0, 0, 0, 0, 0,
   0, 0, 0, 0, 0, 0, 0, 0\}, \{0, 0, 0, 0, 0, 0, 0, 0, 0, 1, 0, 0, 0, 0,
   0, 0, 0, 1, 0, 1, 0, 0, 0, 0, 0, 0, 0, 0, 0, 0, 1, 0, 1, 0, 1, 0, 
  0, 0, 0, 0, 0, 0, 0, 0, 1, 1, 0\}, \{0, 0, 0, 0, 0, 0, 0, 0, 0, 1, 0, 
  0, 0, 0, 0, 0, 0, 1, 0, 0, 0, 0, 0, 0, 0, 0, 0, 0, 0, 1, 0, 1, 0, 1,
   0, 1, 0, 0, 0, 0, 0, 0, 0, 0, 1, 1, 0\}, \{0, 0, 0, 0, 0, 0, 0, 0, 0,
   0, 0, 0, 1, 0, 0, 0, 0, 0, 0, 0, 0, 0, 0, 0, 0, 0, 0, 0, 0, 0, 0, 
  0, 0, 0, 0, 0, 0, 0, 0, 0, 0, 0, 0, 0, 0, 0, 0\}, \{0, 0, 0, 0, 0, 0, 
  0, 0, 0, 0, 0, 0, 0, 0, 1, 0, 0, 0, 0, 0, 0, 0, 0, 0, 0, 0, 0, 0, 0,
   0, 0, 0, 0, 0, 0, 0, 0, 0, 0, 0, 0, 0, 0, 0, 0, 0, 0\}, \{0, 0, 0, 0,
   0, 0, 0, 0, 0, 0, 0, 0, 0, 0, 0, 1, 0, 0, 0, 0, 0, 0, 0, 0, 0, 0, 
  0, 0, 0, 0, 0, 0, 0, 0, 0, 0, 0, 0, 0, 0, 0, 0, 0, 0, 0, 0, 0\}, \{0, 
  0, 0, 0, 0, 0, 0, 0, 0, 0, 0, 0, 0, 0, 0, 0, 0, 0, 1, 0, 0, 0, 0, 0,
   0, 0, 0, 0, 0, 0, 0, 0, 0, 0, 0, 0, 0, 0, 0, 0, 0, 0, 0, 0, 0, 0, 
  0\}, \{0, 0, 0, 0, 0, 0, 0, 0, 0, 0, 0, 0, 0, 0, 0, 0, 0, 0, 0, 0, 1, 
  0, 0, 0, 0, 0, 0, 0, 0, 0, 0, 0, 0, 0, 0, 0, 0, 0, 0, 0, 0, 1, 0, 0,
   0, 0, 0\}, \{0, 0, 0, 0, 0, 0, 0, 0, 0, 0, 0, 0, 1, 0, 0, 0, 0, 0, 0,
   0, 0, 0, 0, 0, 0, 0, 0, 0, 1, 0, 0, 0, 0, 0, 0, 0, 0, 1, 0, 0, 1, 
  0, 0, 0, 0, 0, 0\}, \{0, 0, 0, 0, 0, 0, 0, 0, 0, 0, 0, 0, 0, 0, 0, 0, 
  0, 0, 0, 0, 0, 0, 1, 0, 0, 1, 0, 0, 0, 0, 0, 1, 0, 0, 0, 0, 0, 0, 2,
   0, 0, 0, 0, 0, 0, 0, 0\}, \{0, 0, 0, 0, 0, 0, 0, 0, 0, 0, 0, 0, 0, 0,
   0, 0, 0, 0, 1, 0, 1, 0, 0, 0, 0, 0, 0, 0, 0, 0, 0, 0, 0, 0, 0, 0, 
  0, 0, 0, 1, 0, 1, 0, 0, 0, 0, 0\}, \{0, 0, 0, 0, 0, 0, 0, 0, 0, 0, 0, 
  0, 0, 0, 0, 0, 0, 0, 0, 0, 0, 0, 0, 0, 1, 0, 0, 0, 0, 0, 0, 0, 0, 0,
   0, 0, 0, 1, 0, 0, 1, 0, 0, 0, 0, 0, 0\}, \{0, 0, 0, 0, 0, 0, 0, 0, 0,
   0, 0, 0, 0, 0, 0, 0, 0, 0, 0, 0, 0, 0, 0, 0, 0, 0, 1, 0, 0, 0, 0, 
  0, 0, 0, 0, 0, 0, 0, 0, 1, 0, 1, 0, 0, 0, 0, 0\}, \{0, 0, 0, 0, 0, 0, 
  0, 0, 0, 0, 0, 0, 0, 1, 0, 0, 0, 0, 0, 0, 0, 0, 0, 0, 0, 0, 0, 0, 0,
   1, 0, 0, 1, 0, 0, 1, 1, 0, 0, 0, 0, 0, 1, 0, 1, 1, 0\}, \{0, 0, 0, 0,
   0, 0, 0, 0, 0, 0, 0, 0, 0, 0, 0, 0, 0, 1, 0, 1, 0, 0, 0, 0, 0, 0, 
  0, 0, 0, 0, 0, 0, 1, 0, 0, 1, 1, 0, 0, 0, 0, 0, 0, 1, 1, 0, 1\}, \{0, 
  0, 0, 0, 0, 0, 0, 0, 0, 0, 0, 0, 0, 0, 0, 0, 0, 0, 0, 0, 0, 0, 0, 0,
   0, 0, 0, 0, 1, 0, 0, 0, 0, 0, 0, 0, 0, 0, 0, 0, 1, 0, 0, 0, 0, 0, 
  0\}, \{0, 0, 0, 0, 0, 0, 0, 0, 0, 0, 0, 0, 0, 0, 0, 0, 0, 0, 0, 0, 0, 
  0, 0, 0, 0, 0, 0, 0, 0, 0, 0, 0, 0, 0, 0, 0, 0, 1, 0, 0, 0, 0, 0, 0,
   0, 0, 0\}, \{0, 0, 0, 0, 0, 0, 0, 0, 0, 0, 0, 0, 0, 0, 0, 0, 0, 0, 0,
   0, 0, 0, 0, 0, 0, 0, 0, 0, 0, 0, 0, 0, 0, 0, 0, 0, 0, 0, 0, 1, 0, 
  0, 0, 0, 0, 0, 0\}\}
       
}
The substitution-homotopy map $W_E:\Z^{\sE}\to\Z^{\sE}$ is given by the matrix:\\
{\fontsize{4}{0}\selectfont
$ $\\
\{\{0, 0, 0, 0, 0, 0, 0, 0, 0, 0, 0, 0, 0, 0, 0, 0, 0, 0, 0, 0, 0, 0, 0,
   0, 0, 0, 0, 0, 0, 0, 0, 0, 0, 0, 0, 0, 0, 0, 0, 0, 0, 0, 0, 0\}, \{0,
   2, 0, 0, 0, 0, 0, 0, 0, 0, 0, 0, 0, 1, 0, 0, 0, 0, 0, 0, 0, 0, 0, 
  0, 0, 0, 1, 0, 0, 0, 0, 0, 0, 1, 0, 0, 0, 0, 0, 0, 0, 0, 0, 0\}, \{0, 
  0, 2, 0, 0, 0, 0, 0, 1, 0, 0, 0, 0, 0, 0, 0, 0, 0, 0, 0, 0, 0, 0, 0,
   0, 1, 0, 0, 0, 0, 0, 0, 1, 0, 0, 0, 0, 0, 0, 0, 0, 0, 0, 0\}, \{0, 0,
   0, 0, 0, 0, 0, 0, 0, 0, 0, 0, 0, 0, 0, 0, 0, 0, 0, 0, 0, 0, 0, 0, 
  0, 0, 0, 0, 0, 0, 0, 0, 0, 0, 0, 0, 0, 0, 0, 0, 0, 0, 0, 0\}, \{0, 0, 
  0, 0, 1, 0, 0, 0, 0, 0, 1, 0, 1, 0, 0, 0, 0, 0, 0, 0, 0, 0, 0, 0, 0,
   0, 0, 0, 0, 0, 0, 0, 0, 0, 0, 0, 0, 0, 0, 0, 0, 0, 0, 0\}, \{1, 0, 0,
   0, 0, 1, 0, 0, 0, 0, 0, 0, 0, 0, 0, 0, 0, 0, 0, 1, 0, 0, 0, 0, 0, 
  0, 0, 0, 0, 0, 0, 0, 0, 0, 0, 0, 0, 0, 0, 0, 0, 0, 0, 0\}, \{0, 0, 0, 
  1, 0, 0, 1, 0, 0, 0, 0, 0, 0, 0, 0, 0, 1, 0, 0, 0, 0, 0, 0, 0, 0, 0,
   0, 0, 0, 0, 0, 0, 0, 0, 0, 0, 0, 0, 0, 0, 0, 0, 0, 0\}, \{0, 0, 0, 0,
   0, 0, 0, 1, 0, 1, 0, 1, 0, 0, 0, 0, 0, 0, 0, 0, 0, 0, 0, 0, 0, 0, 
  0, 0, 0, 0, 0, 0, 0, 0, 0, 0, 0, 0, 0, 0, 0, 0, 0, 0\}, \{1, 0, 0, 0, 
  0, 0, 0, 0, 1, 0, 0, 0, 0, 0, 0, 0, 0, 0, 0, 1, 0, 0, 0, 0, 0, 0, 0,
   0, 0, 1, 0, 0, 0, 0, 0, 0, 0, 0, 1, 0, 0, 0, 0, 0\}, \{0, 0, 0, 0, 0,
   0, 0, 0, 0, 0, 0, 0, 0, 0, 0, 0, 0, 0, 0, 0, 0, 0, 0, 0, 0, 0, 0, 
  0, 0, 0, 0, 0, 0, 0, 0, 0, 0, 0, 0, 0, 0, 0, 0, 0\}, \{0, 0, 0, 0, 0, 
  0, 0, 0, 0, 0, 0, 0, 0, 0, 0, 0, 0, 0, 0, 0, 0, 0, 0, 0, 0, 0, 0, 0,
   0, 0, 0, 0, 0, 0, 0, 0, 0, 0, 0, 0, 0, 0, 0, 0\}, \{0, 0, 0, 0, 0, 0,
   0, 0, 0, 0, 0, 0, 0, 0, 0, 0, 0, 0, 0, 0, 0, 0, 0, 0, 0, 0, 0, 0, 
  0, 0, 0, 0, 0, 0, 0, 0, 0, 0, 0, 0, 0, 0, 0, 0\}, \{0, 0, 0, 0, 0, 0, 
  0, 0, 0, 0, 0, 0, 0, 0, 0, 0, 0, 0, 0, 0, 0, 0, 0, 0, 0, 0, 0, 0, 0,
   0, 0, 0, 0, 0, 0, 0, 0, 0, 0, 0, 0, 0, 0, 0\}, \{0, 0, 0, 1, 0, 0, 0,
   0, 0, 0, 0, 0, 0, 1, 0, 0, 1, 0, 0, 0, 0, 0, 0, 0, 0, 0, 0, 0, 1, 
  0, 0, 0, 0, 0, 0, 0, 0, 1, 0, 0, 0, 0, 0, 0\}, \{0, 0, 0, 0, 0, 1, 0, 
  0, 1, 0, 0, 0, 0, 0, 2, 0, 0, 0, 0, 0, 0, 0, 0, 0, 0, 0, 0, 0, 0, 0,
   0, 0, 0, 0, 0, 0, 0, 0, 0, 0, 0, 0, 0, 1\}, \{0, 0, 0, 0, 0, 0, 0, 0,
   0, 0, 0, 0, 0, 0, 0, 0, 0, 0, 0, 0, 0, 0, 0, 0, 0, 0, 0, 0, 0, 0, 
  0, 0, 0, 0, 0, 0, 0, 0, 0, 0, 0, 0, 0, 0\}, \{0, 0, 0, 0, 0, 0, 0, 0, 
  0, 0, 0, 0, 0, 0, 0, 0, 0, 0, 0, 0, 0, 0, 0, 0, 0, 0, 0, 0, 0, 0, 0,
   0, 0, 0, 0, 0, 0, 0, 0, 0, 0, 0, 0, 0\}, \{0, 0, 0, 0, 1, 0, 0, 0, 0,
   0, 0, 0, 0, 0, 0, 0, 0, 2, 0, 0, 0, 0, 0, 0, 0, 0, 0, 1, 0, 0, 0, 
  0, 0, 0, 0, 0, 0, 0, 0, 0, 0, 0, 1, 0\}, \{0, 0, 0, 0, 0, 0, 0, 1, 0, 
  0, 0, 0, 0, 0, 0, 0, 0, 0, 2, 0, 0, 0, 0, 0, 0, 0, 0, 0, 0, 0, 1, 0,
   0, 0, 0, 0, 0, 0, 0, 0, 0, 1, 0, 0\}, \{0, 0, 0, 0, 0, 0, 0, 0, 0, 0,
   0, 0, 0, 0, 0, 0, 0, 0, 0, 0, 0, 0, 0, 0, 0, 0, 0, 0, 0, 0, 0, 0, 
  0, 0, 0, 0, 0, 0, 0, 0, 0, 0, 0, 0\}, \{0, 0, 0, 0, 0, 0, 0, 0, 0, 0, 
  0, 0, 0, 0, 0, 0, 0, 0, 0, 0, 0, 0, 0, 0, 0, 0, 0, 0, 0, 0, 0, 0, 0,
   0, 0, 0, 0, 0, 0, 0, 0, 0, 0, 0\}, \{0, 0, 0, 0, 0, 0, 1, 0, 0, 0, 0,
   0, 0, 1, 0, 0, 0, 0, 0, 0, 0, 2, 0, 0, 0, 0, 0, 0, 0, 0, 0, 0, 0, 
  0, 0, 0, 0, 0, 0, 0, 1, 0, 0, 0\}, \{0, 0, 0, 0, 0, 0, 0, 0, 0, 0, 0, 
  0, 0, 0, 0, 0, 0, 0, 0, 0, 0, 0, 0, 0, 0, 0, 0, 0, 0, 0, 0, 0, 0, 0,
   0, 0, 0, 0, 0, 0, 0, 0, 0, 0\}, \{0, 0, 0, 0, 0, 0, 0, 0, 0, 1, 0, 0,
   0, 0, 0, 1, 0, 0, 0, 0, 0, 0, 0, 1, 0, 0, 0, 0, 0, 0, 0, 0, 0, 0, 
  0, 0, 0, 0, 0, 0, 0, 0, 0, 0\}, \{0, 0, 0, 0, 0, 0, 0, 0, 0, 0, 1, 0, 
  0, 0, 0, 0, 0, 0, 0, 0, 0, 0, 1, 0, 1, 0, 0, 0, 0, 0, 0, 0, 0, 0, 0,
   0, 0, 0, 0, 0, 0, 0, 0, 0\}, \{1, 0, 0, 0, 0, 0, 0, 0, 0, 0, 0, 0, 0,
   0, 0, 0, 0, 0, 0, 0, 0, 0, 0, 0, 0, 1, 0, 0, 0, 0, 0, 0, 0, 0, 0, 
  0, 0, 0, 1, 0, 0, 0, 0, 0\}, \{0, 0, 0, 0, 0, 0, 0, 0, 0, 0, 0, 0, 0, 
  0, 0, 0, 1, 0, 0, 0, 0, 0, 0, 0, 0, 0, 1, 0, 1, 0, 0, 0, 0, 0, 0, 0,
   0, 0, 0, 0, 0, 0, 0, 0\}, \{0, 0, 0, 0, 0, 0, 0, 0, 0, 0, 1, 0, 1, 0,
   0, 0, 0, 0, 0, 0, 1, 0, 1, 0, 0, 0, 0, 1, 0, 0, 0, 0, 0, 0, 0, 0, 
  0, 0, 0, 0, 0, 0, 0, 0\}, \{0, 0, 0, 0, 0, 0, 0, 0, 0, 0, 0, 0, 0, 0, 
  0, 0, 0, 0, 0, 0, 0, 0, 0, 0, 0, 0, 0, 0, 0, 0, 0, 0, 0, 0, 0, 0, 0,
   0, 0, 0, 0, 0, 0, 0\}, \{0, 0, 0, 0, 0, 0, 0, 0, 0, 0, 0, 0, 0, 0, 0,
   0, 0, 0, 0, 0, 0, 0, 0, 0, 0, 0, 0, 0, 0, 0, 0, 0, 0, 0, 0, 0, 0, 
  0, 0, 0, 0, 0, 0, 0\}, \{0, 0, 0, 0, 0, 0, 0, 0, 0, 1, 0, 1, 0, 0, 0, 
  1, 0, 0, 0, 0, 0, 0, 0, 0, 0, 0, 0, 0, 0, 0, 1, 1, 0, 0, 0, 0, 0, 0,
   0, 0, 0, 0, 0, 0\}, \{0, 0, 0, 0, 0, 0, 0, 0, 0, 0, 0, 0, 0, 0, 0, 0,
   0, 0, 0, 0, 0, 0, 0, 0, 0, 0, 0, 0, 0, 0, 0, 0, 0, 0, 0, 0, 0, 0, 
  0, 0, 0, 0, 0, 0\}, \{0, 0, 0, 0, 0, 0, 0, 0, 0, 0, 0, 0, 0, 0, 0, 0, 
  0, 0, 0, 1, 0, 0, 0, 0, 0, 0, 0, 0, 0, 1, 0, 0, 1, 0, 0, 0, 0, 0, 0,
   0, 0, 0, 0, 0\}, \{0, 0, 0, 1, 0, 0, 0, 0, 0, 0, 0, 0, 0, 0, 0, 0, 0,
   0, 0, 0, 0, 0, 0, 0, 0, 0, 0, 0, 0, 0, 0, 0, 0, 1, 0, 0, 0, 1, 0, 
  0, 0, 0, 0, 0\}, \{0, 0, 0, 0, 0, 0, 0, 0, 0, 0, 0, 1, 0, 0, 0, 0, 0, 
  0, 0, 0, 0, 0, 0, 0, 0, 0, 0, 0, 0, 0, 0, 1, 0, 0, 1, 0, 0, 0, 0, 0,
   0, 0, 0, 0\}, \{0, 0, 0, 0, 0, 0, 0, 0, 0, 0, 0, 0, 1, 0, 0, 0, 0, 0,
   0, 0, 1, 0, 0, 0, 0, 0, 0, 0, 0, 0, 0, 0, 0, 0, 0, 1, 0, 0, 0, 0, 
  0, 0, 0, 0\}, \{0, 0, 0, 0, 0, 0, 0, 0, 0, 0, 0, 0, 0, 0, 0, 0, 0, 0, 
  0, 0, 0, 0, 0, 0, 1, 0, 0, 1, 0, 0, 0, 0, 0, 0, 0, 1, 2, 0, 0, 0, 0,
   0, 0, 0\}, \{0, 0, 0, 0, 0, 0, 0, 0, 0, 0, 0, 0, 0, 0, 0, 0, 0, 0, 0,
   0, 0, 0, 0, 0, 0, 0, 0, 0, 0, 0, 0, 0, 0, 0, 0, 0, 0, 0, 0, 0, 0, 
  0, 0, 0\}, \{0, 0, 0, 0, 0, 0, 0, 0, 0, 0, 0, 0, 0, 0, 0, 0, 0, 0, 0, 
  0, 0, 0, 0, 0, 0, 0, 0, 0, 0, 0, 0, 0, 0, 0, 0, 0, 0, 0, 0, 0, 0, 0,
   0, 0\}, \{0, 0, 0, 0, 0, 0, 0, 0, 0, 0, 0, 0, 0, 0, 0, 0, 0, 0, 0, 0,
   0, 0, 0, 1, 0, 0, 0, 0, 0, 0, 1, 0, 0, 0, 1, 0, 0, 0, 0, 2, 0, 0, 
  0, 0\}, \{0, 0, 0, 0, 0, 0, 0, 0, 0, 0, 0, 0, 0, 0, 0, 0, 0, 0, 0, 0, 
  0, 0, 0, 0, 0, 0, 0, 0, 1, 0, 0, 0, 0, 0, 0, 0, 0, 1, 0, 0, 1, 0, 0,
   0\}, \{0, 0, 0, 0, 0, 0, 0, 0, 0, 0, 0, 0, 0, 0, 0, 1, 0, 0, 0, 0, 0,
   0, 0, 0, 0, 0, 0, 0, 0, 0, 0, 1, 0, 0, 0, 0, 0, 0, 0, 0, 0, 1, 0, 
  0\}, \{0, 0, 0, 0, 0, 0, 0, 0, 0, 0, 0, 0, 0, 0, 0, 0, 0, 0, 0, 0, 1, 
  0, 1, 0, 0, 0, 0, 0, 0, 0, 0, 0, 0, 0, 0, 0, 0, 0, 0, 0, 0, 0, 1, 
  0\}, \{0, 0, 0, 0, 0, 0, 0, 0, 0, 0, 0, 0, 0, 0, 0, 0, 0, 0, 0, 0, 0, 
  0, 0, 0, 0, 0, 0, 0, 0, 1, 0, 0, 0, 0, 0, 0, 0, 0, 1, 0, 0, 0, 0, 
  1\}\}
       
}
The substitution-homotopy map $W_F:\Z^{\sF}\to\Z^{\sF}$ is given by the matrix
\begin{center}
\includegraphics[scale=.6]{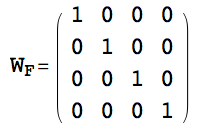}
\end{center}
By Proposition \ref{p:kerA-isom-Znminusr}, 
$$H_S^0(T)=\dlim(W_V,\ker \delta^0)=\dlim(
\left(
\begin{array}{cccccccc}
 0 & 0 & -1 & 1 & 0 & 1 & 1 & -1 \\
 0 & 4 & 0 & 0 & 0 & 0 & 0 & 0 \\
 0 & 1 & 1 & -1 & 1 & -1 & -1 & 1 \\
 0 & 2 & 0 & 0 & 0 & 0 & 0 & 0 \\
 0 & 1 & 0 & 0 & 2 & 0 & 0 & 0 \\
 0 & 0 & -1 & 1 & 1 & 1 & 1 & -1 \\
 0 & 0 & 0 & 0 & 1 & 0 & 0 & 0 \\
 0 & 0 & 0 & 0 & 1 & 0 & 0 & 0 \\
\end{array}
\right)
,\Z^{8}).$$
By Proposition \ref{p:imA-isom-Zr} we remove the zero-eigenvalues and get
$$H_S^0(T)=\dlim(W_V,\ker \delta^0)=\dlim(
\left(
\begin{array}{ccc}
 2 & -1 & 1 \\
 0 & 5 & -3 \\
 0 & 1 & 1 \\
\end{array}
\right)
,\Z^{3}).$$
Since
$$ 
\left(
\begin{array}{ccc}
 1 & 0 & -1 \\
 0 & -1 & 1 \\
 0 & 0 & 1 \\
\end{array}
\right)^{-1}.\left(
\begin{array}{ccc}
 2 & -1 & 1 \\
 0 & 5 & -3 \\
 0 & 1 & 1 \\
\end{array}
\right).\left(
\begin{array}{ccc}
 1 & 0 & -1 \\
 0 & -1 & 1 \\
 0 & 0 & 1 \\
\end{array}
\right)=\left(
\begin{array}{ccc}
 2 & 0 & 0 \\
 0 & 4 & 0 \\
 0 & -1 & 2 \\
\end{array}
\right),$$
$$H_S^0(T)=\dlim(W_V,\ker \delta^0)=\dlim(
\left(
\begin{array}{ccc}
 2 & 0 & 0 \\
 0 & 4 & 0 \\
 0 & -1 & 2 \\
\end{array}
\right)
,\Z^{3})=\Z[\frac12]^3.$$
By Proposition \ref{p:kerOverIm-A-isom-Znminus}, $\lim\limits_{\to}(W_E,\,\frac{\ker\delta^1}{\im\delta^0})=\lim\limits_{\to}(W_E',\coker B)$ for some matrices $W_E',B$.
Then by Proposition \ref{p:cokerA-isom-Znminuss} we get rid of the coker and get
$$K_1(S)=H_S^1(T)=\dlim(W_E,\frac{\ker\delta^1}{\im\delta^0})=\dlim(
\left(
\begin{array}{cccc}
 2 & 0 & 0 & 0 \\
 0 & 2 & 0 & 0 \\
 0 & 0 & 2 & 0 \\
 0 & 0 & 0 & 2 \\
\end{array}
\right)
,\Z_2^2\oplus\Z^2)=\Z[\frac12]^2$$ 
because $2\mod 2$ is zero.
\end{exam}
\begin{figure}[t]
\centerline{
\includegraphics[scale=.4]{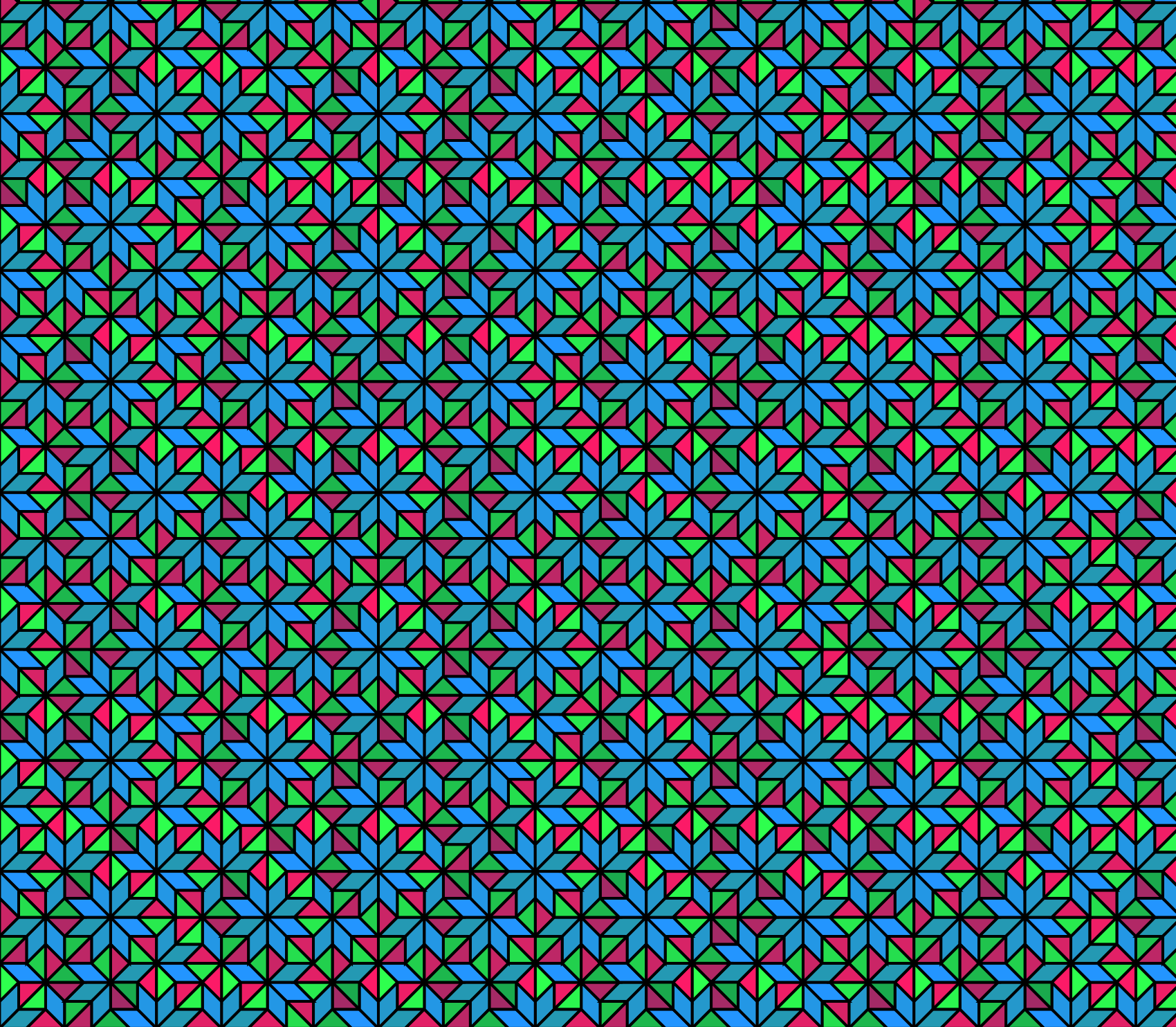}
}
 \caption{Octagonal tiling.\label{f:octagonal}}
\end{figure}

\begin{exam}[Octagonal tiling](cf.~\cite{KelPut}).
Let $T$ be the Octagonal tiling with proto-faces $a,b,c,d,e,f,g,h,i,j,k,l,m,n,o,p,q,r,s,t,\in T$ and substitution rule given by:
\begin{center}
\includegraphics[scale=.4]{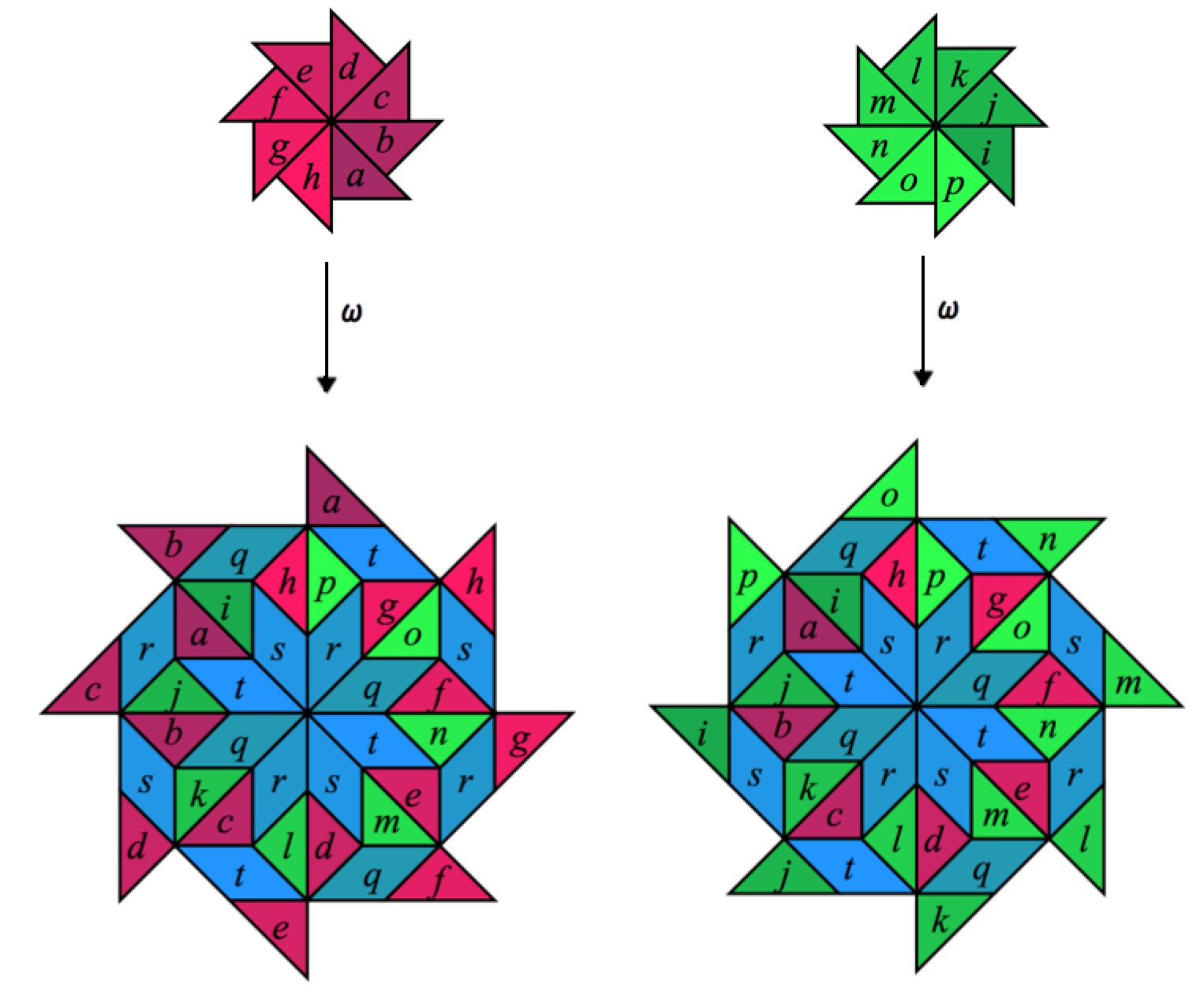}
\end{center}
\begin{center}
\includegraphics[scale=.4]{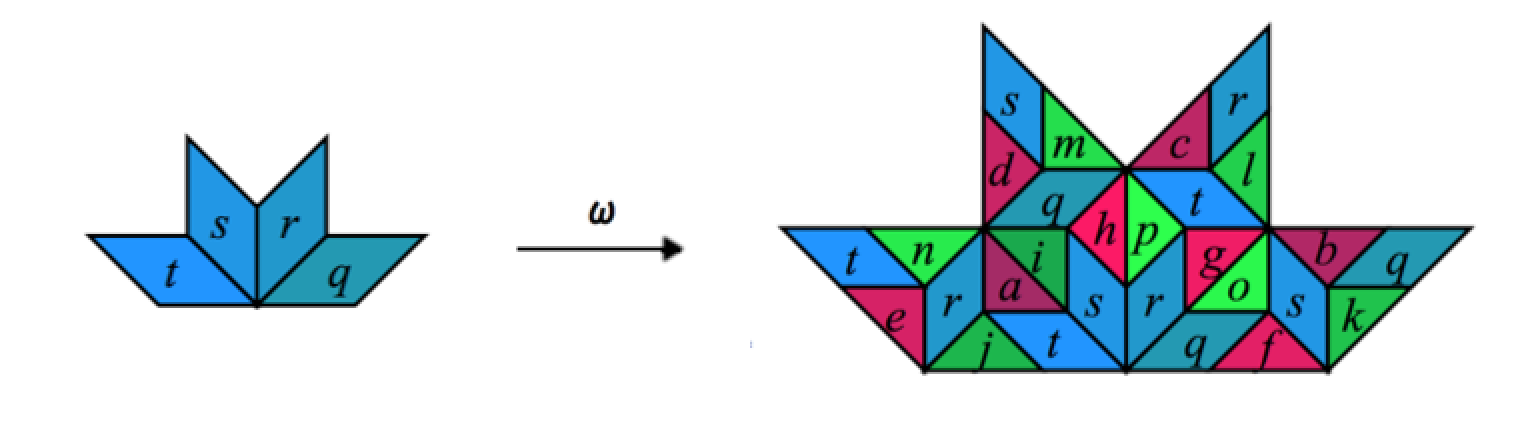}
\end{center}
The tiles $a,i$ form a unit square, and the angles in the tile $q$ are $45^\circ,135^\circ$.
The inflation factor is $\lambda=1+\sqrt{2}$.
There are 56 stable edges and 49 stable vertices.\\\\
{\centering
\includegraphics[scale=.32]{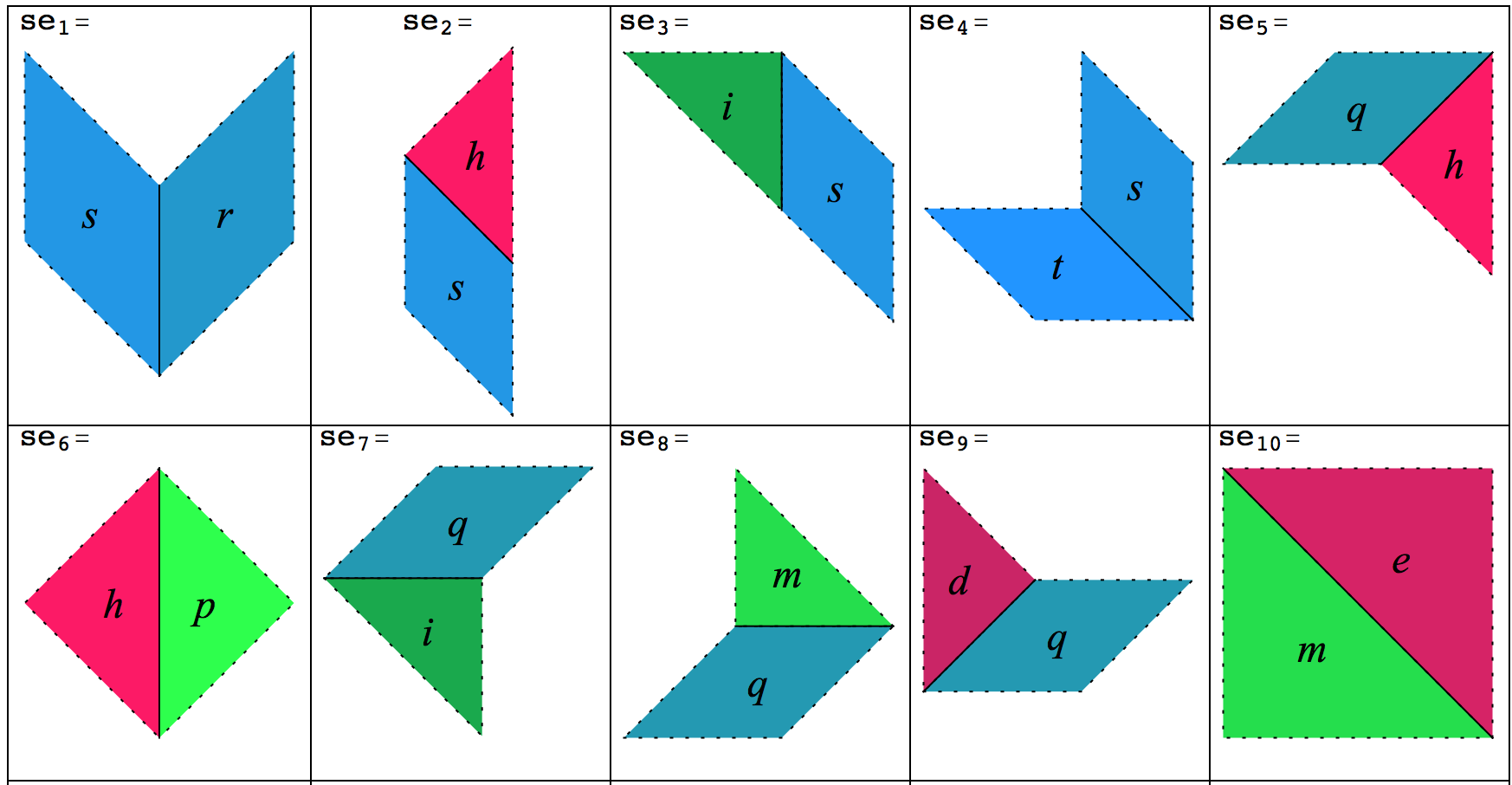}\\
\includegraphics[scale=.32]{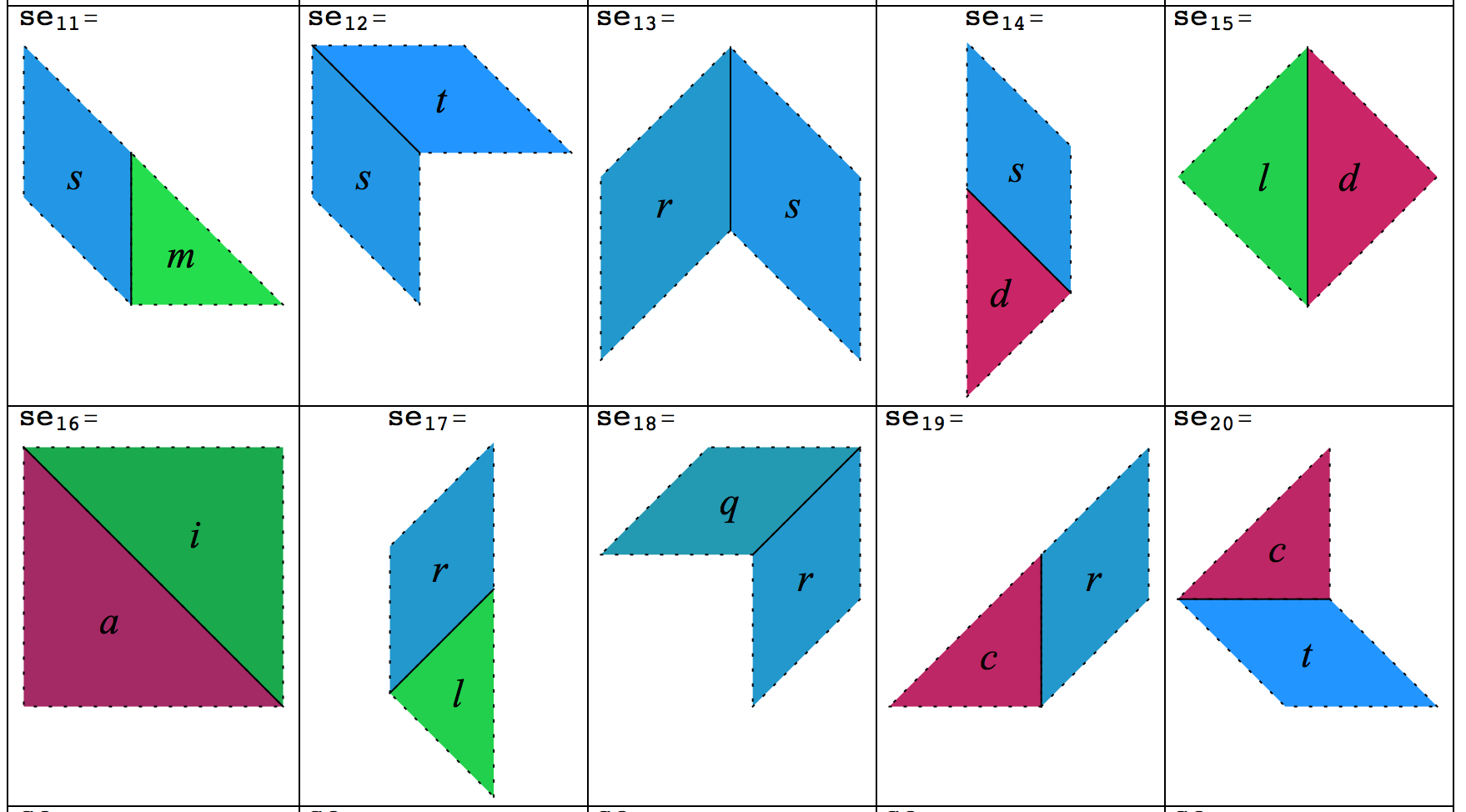}\\
\includegraphics[scale=.32]{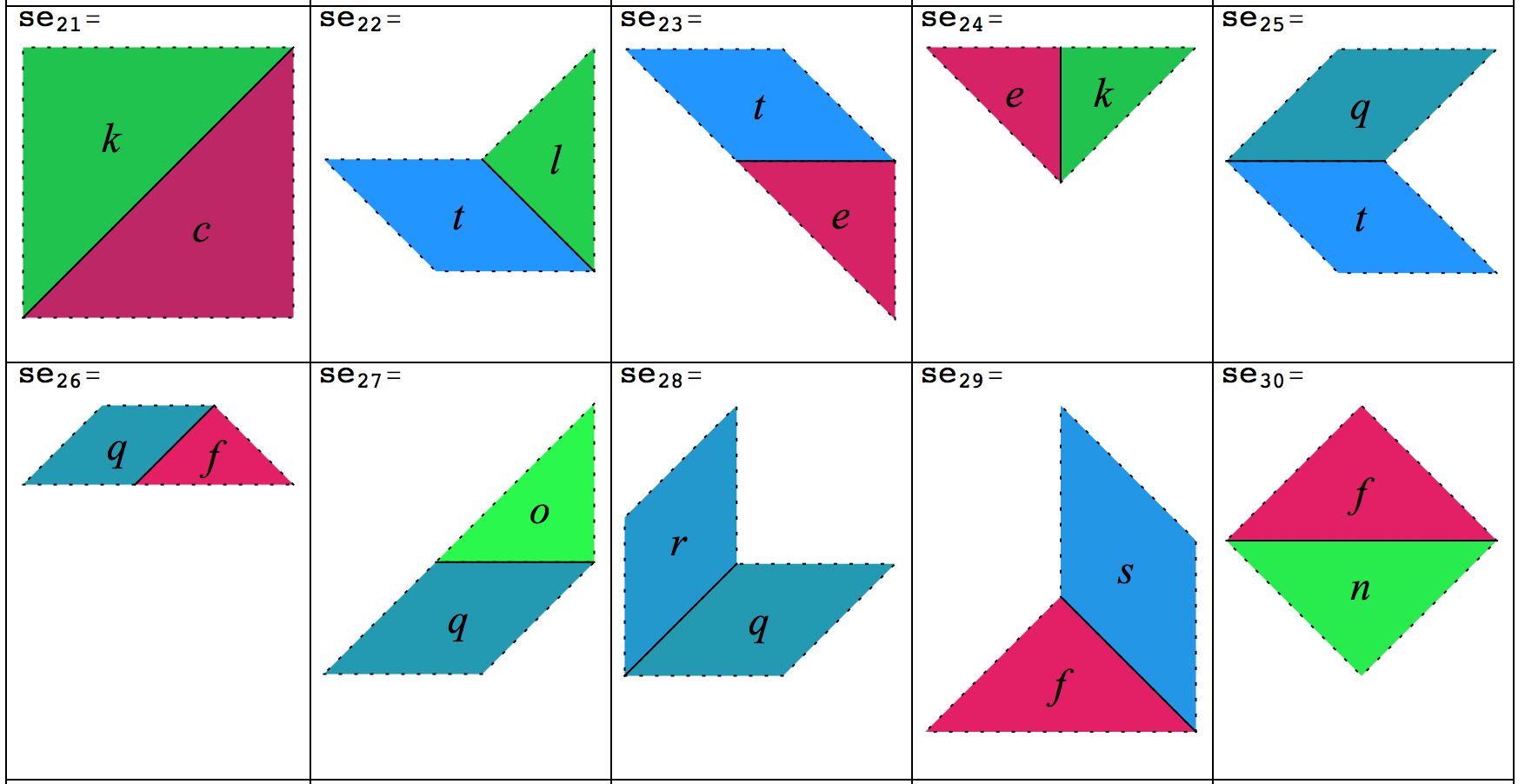}\\
\includegraphics[scale=.3]{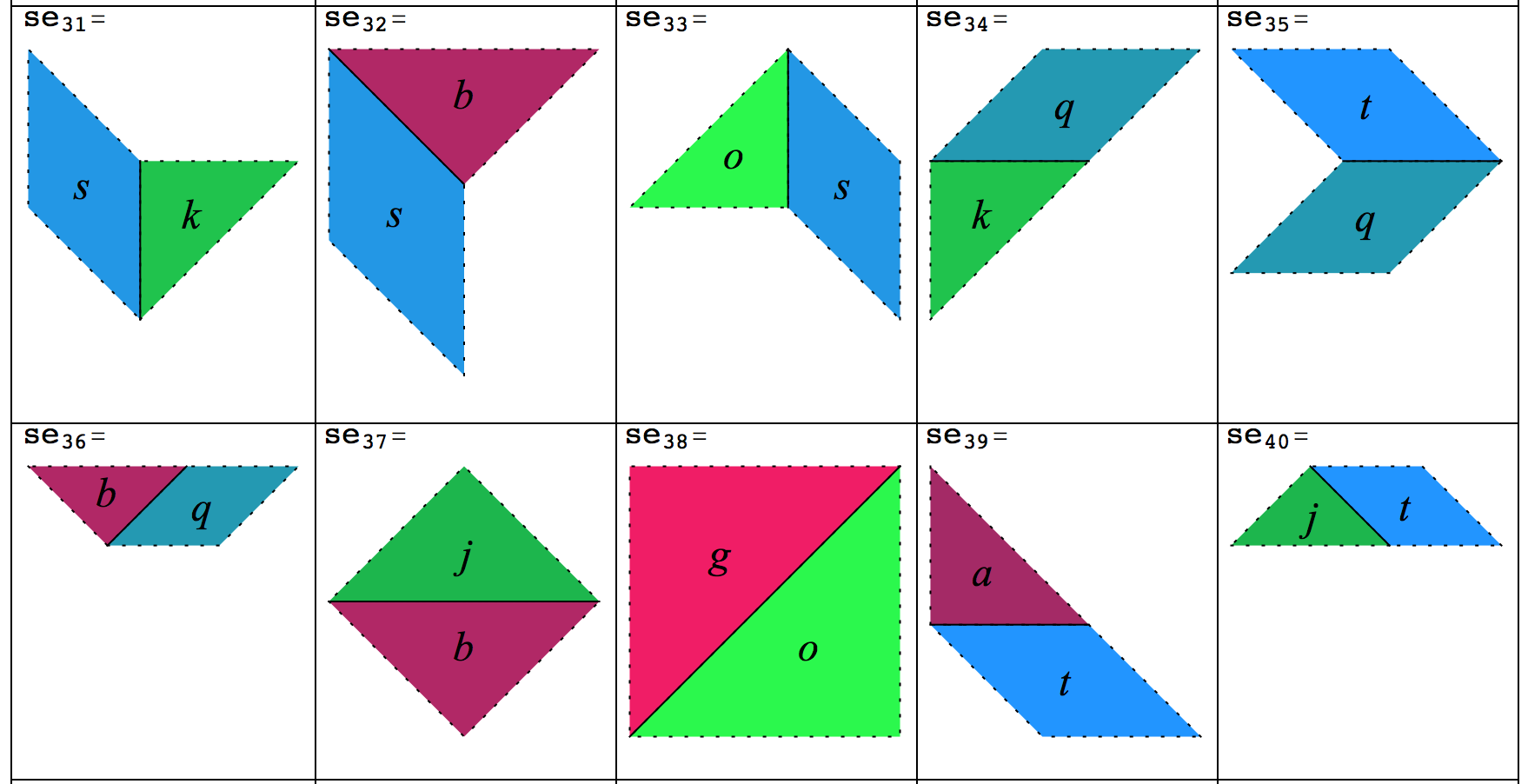}\\
\includegraphics[scale=.3]{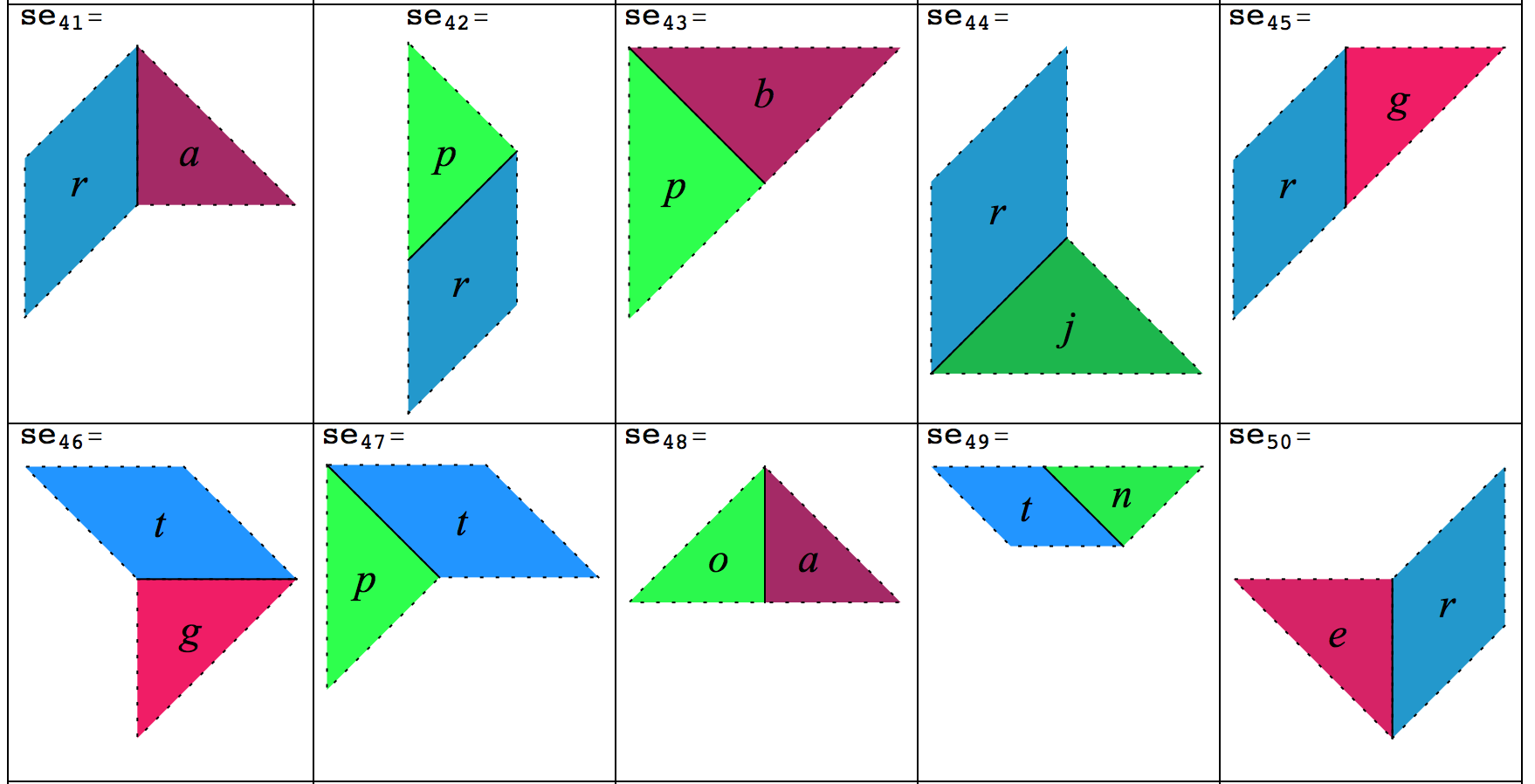}\\
\includegraphics[scale=.3]{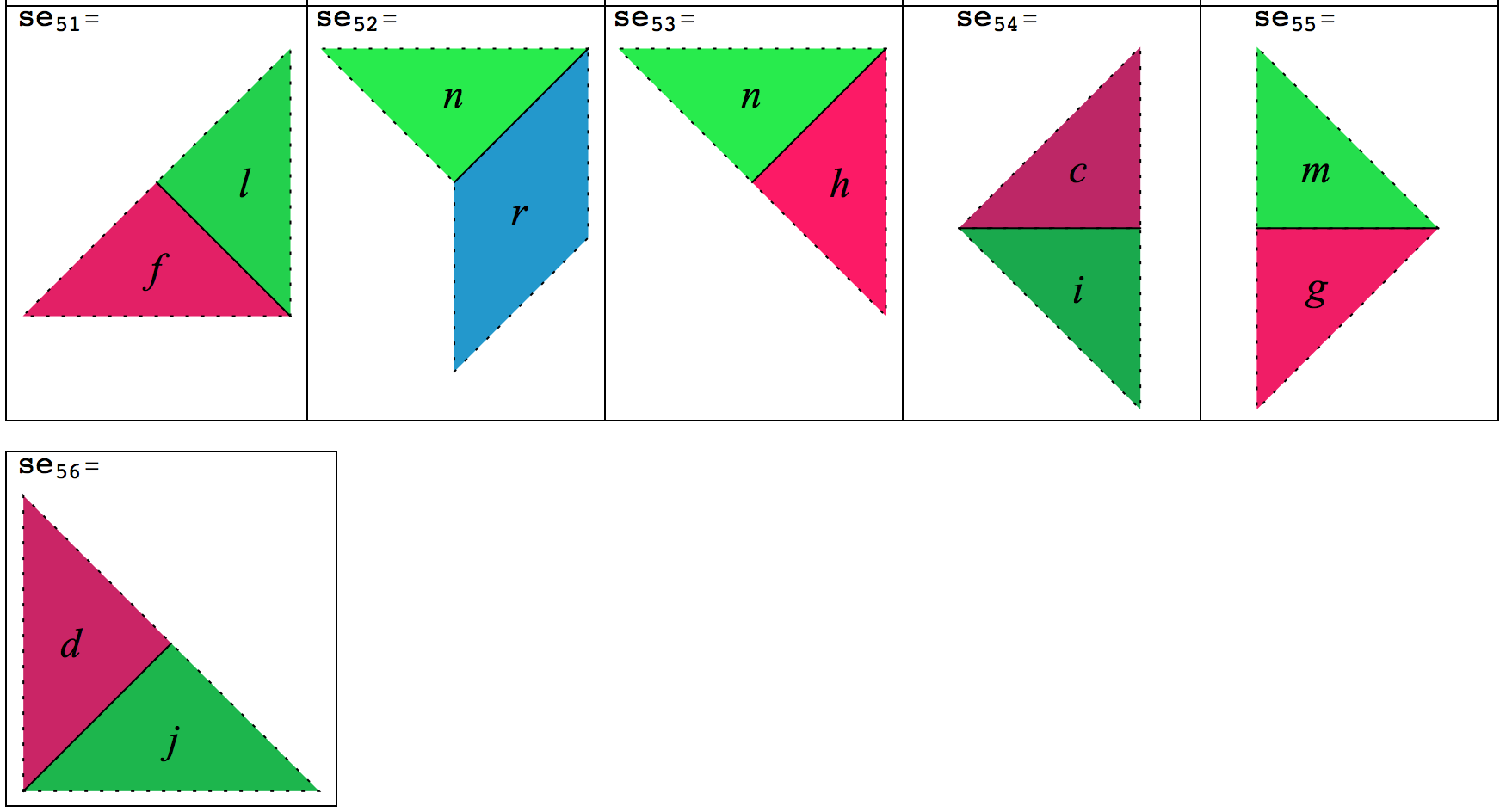}\\
\includegraphics[scale=.3]{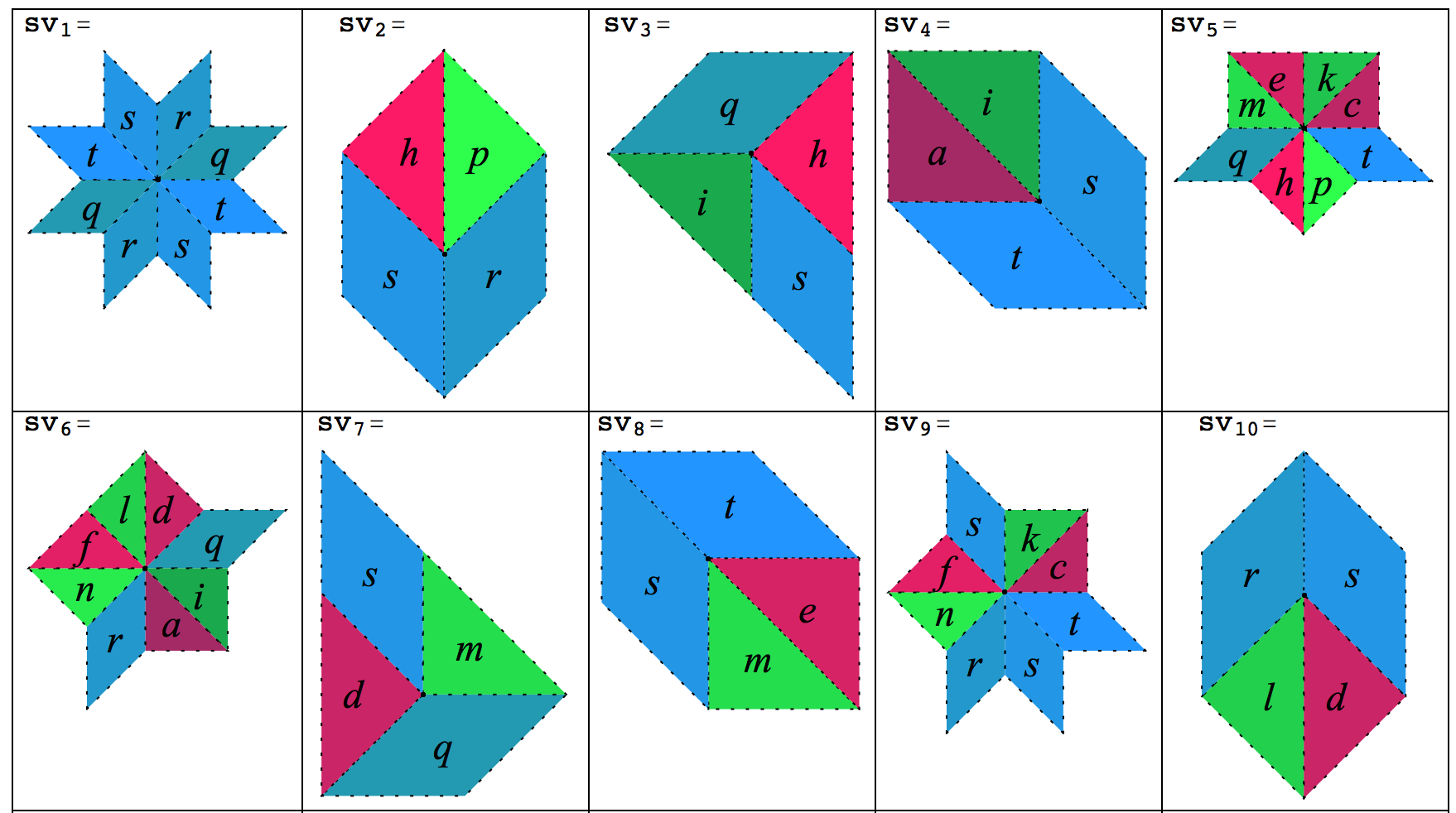}\\
\includegraphics[scale=.305]{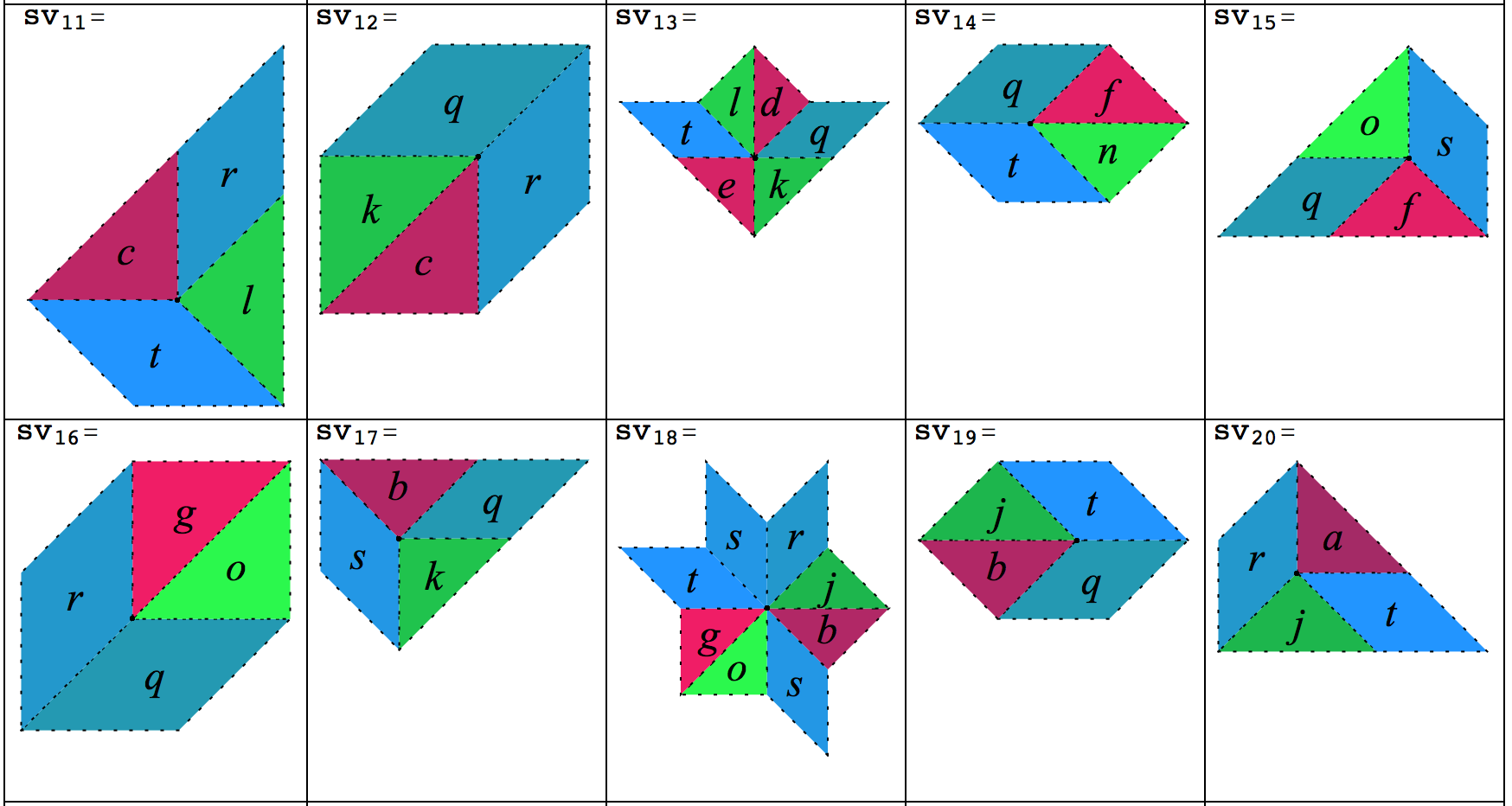}\\
\includegraphics[scale=.305]{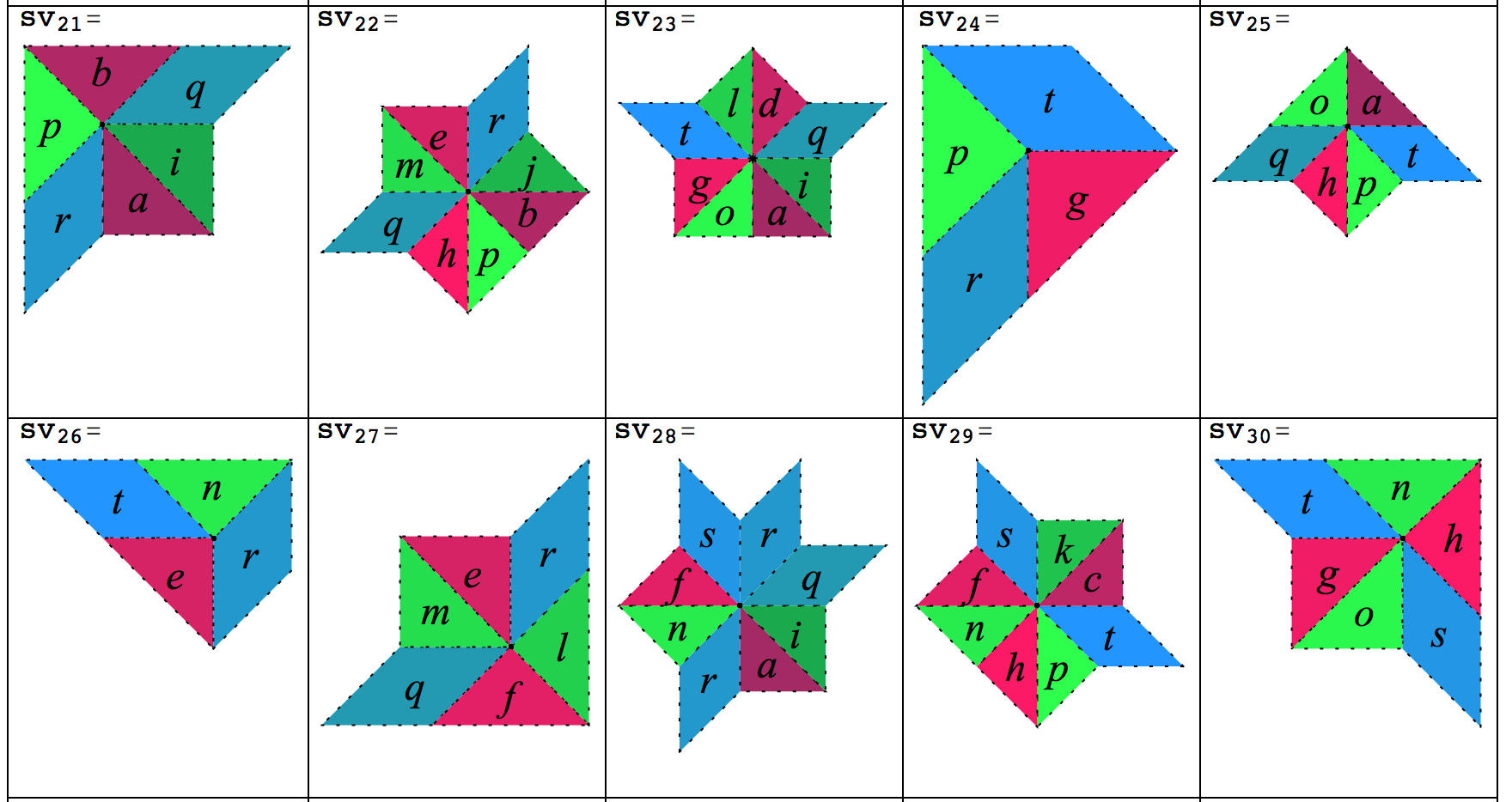}\\
\includegraphics[scale=.305]{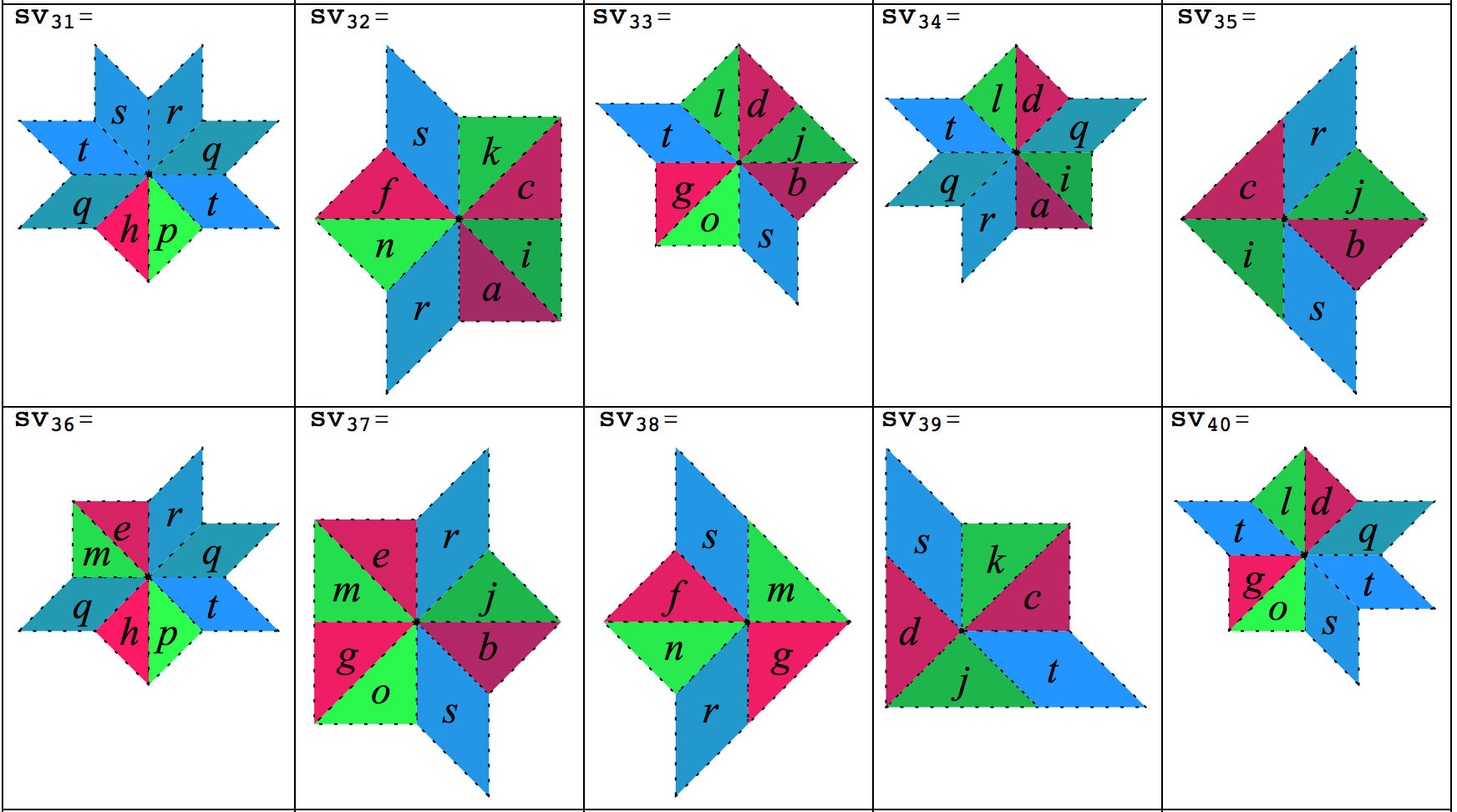}\\
\includegraphics[scale=.305]{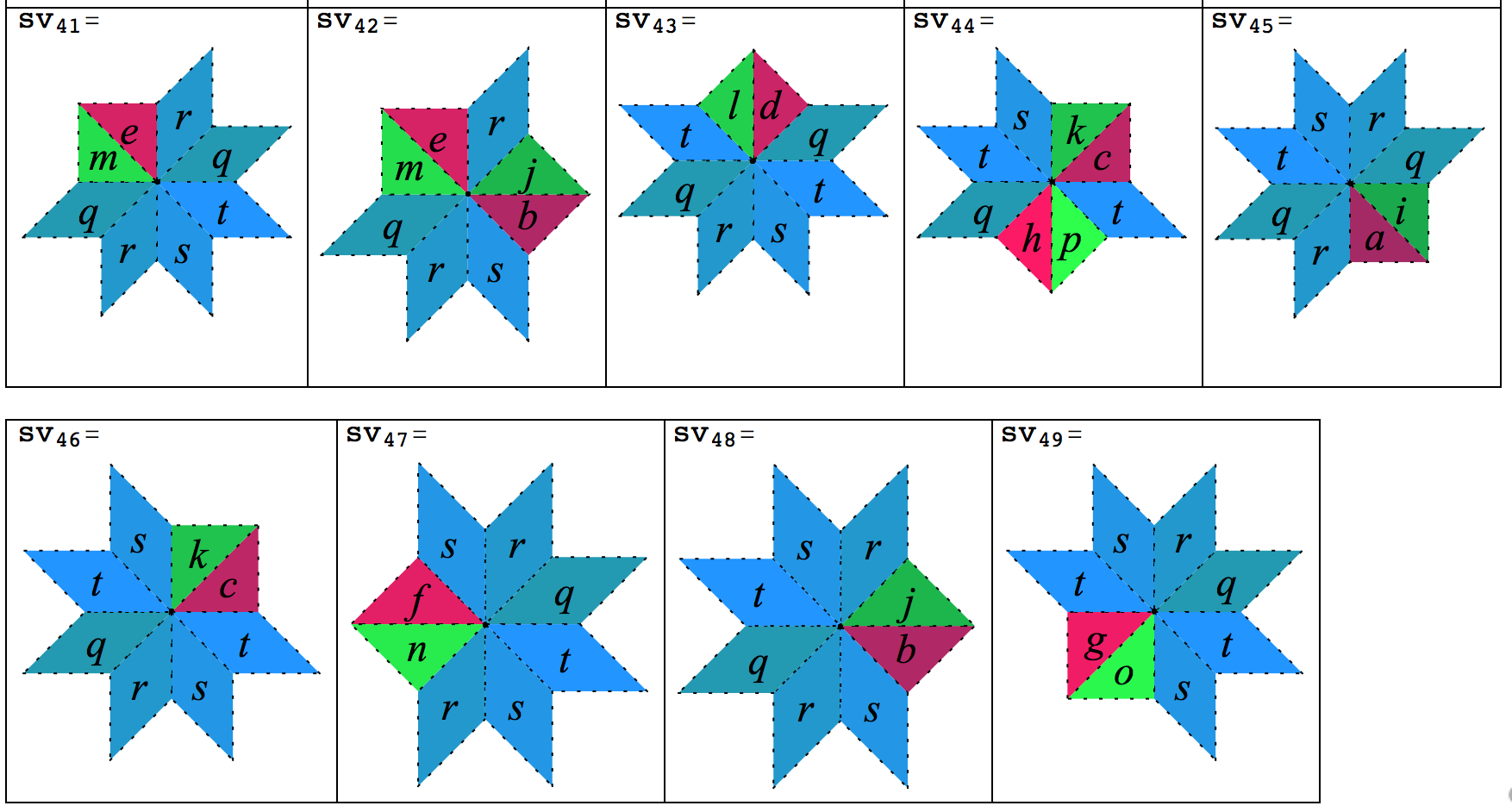}\\
}
$ $\\\\
The exponential map $\delta^0:\Z^{\sV}\to\Z^{\sE}$ is given by the matrix:\\
{\fontsize{4}{0}\selectfont
$ $\\
\{\{-1, 1, 0, 0, 0, 0, 0, 0, 0, 0, 0, 0, 0, 0, 0, 0, 0, -1, 0, 0, 0, 0, 
  0, 0, 0, 0, 0, -1, 0, 0, -1, 0, 0, 0, 0, 0, 0, 0, 0, 0, 0, 0, 0, 
  0, -1, 0, -1, -1, -1\}, \{0, 1, -1, 0, 0, 0, 0, 0, 0, 0, 0, 0, 0, 0, 
  0, 0, 0, 0, 0, 0, 0, 0, 0, 0, 0, 0, 0, 0, 0, -1, 0, 0, 0, 0, 0, 0, 
  0, 0, 0, 0, 0, 0, 0, 0, 0, 0, 0, 0, 0\}, \{0, 0, 1, -1, 0, 0, 0, 0, 0,
   0, 0, 0, 0, 0, 0, 0, 0, 0, 0, 0, 0, 0, 0, 0, 0, 0, 0, 0, 0, 0, 0, 
  0, 0, 0, 1, 0, 0, 0, 0, 0, 0, 0, 0, 0, 0, 0, 0, 0, 0\}, \{1, 0, 0, -1,
   0, 0, 0, 0, 0, 0, 0, 0, 0, 0, 0, 0, 0, 1, 0, 0, 0, 0, 0, 0, 0, 0, 
  0, 0, 0, 0, 1, 0, 0, 0, 0, 0, 0, 0, 0, 0, 0, 0, 0, 1, 1, 1, 0, 1, 
  1\}, \{0, 0, -1, 0, 1, 0, 0, 0, 0, 0, 0, 0, 0, 0, 0, 0, 0, 0, 0, 0, 0,
   1, 0, 0, 1, 0, 0, 0, 0, 0, 1, 0, 0, 0, 0, 1, 0, 0, 0, 0, 0, 0, 0, 
  1, 0, 0, 0, 0, 0\}, \{0, -1, 0, 0, 1, 0, 0, 0, 0, 0, 0, 0, 0, 0, 0, 0,
   0, 0, 0, 0, 0, 1, 0, 0, 1, 0, 0, 0, 1, 0, 1, 0, 0, 0, 0, 1, 0, 0, 
  0, 0, 0, 0, 0, 1, 0, 0, 0, 0, 0\}, \{0, 0, 1, 0, 0, -1, 0, 0, 0, 0, 0,
   0, 0, 0, 0, 0, 0, 0, 0, 0, -1, 0, -1, 0, 0, 0, 0, -1, 0, 0, 0, 0, 
  0, -1, 0, 0, 0, 0, 0, 0, 0, 0, 0, 0, -1, 0, 0, 0, 0\}, \{0, 0, 0, 0, 
  1, 0, -1, 0, 0, 0, 0, 0, 0, 0, 0, 0, 0, 0, 0, 0, 0, 1, 0, 0, 0, 0, 
  1, 0, 0, 0, 0, 0, 0, 0, 0, 1, 0, 0, 0, 0, 1, 1, 0, 0, 0, 0, 0, 0, 
  0\}, \{0, 0, 0, 0, 0, -1, 1, 0, 0, 0, 0, 0, -1, 0, 0, 0, 0, 0, 0, 0, 
  0, 0, -1, 0, 0, 0, 0, 0, 0, 0, 0, 0, 0, -1, 0, 0, 0, 0, 0, -1, 0, 
  0, -1, 0, 0, 0, 0, 0, 0\}, \{0, 0, 0, 0, 1, 0, 0, -1, 0, 0, 0, 0, 0, 
  0, 0, 0, 0, 0, 0, 0, 0, 1, 0, 0, 0, 0, 1, 0, 0, 0, 0, 0, 0, 0, 0, 1,
   1, 0, 0, 0, 1, 1, 0, 0, 0, 0, 0, 0, 0\}, \{0, 0, 0, 0, 0, 0, -1, 1, 
  0, 0, 0, 0, 0, 0, 0, 0, 0, 0, 0, 0, 0, 0, 0, 0, 0, 0, 0, 0, 0, 0, 0,
   0, 0, 0, 0, 0, 0, -1, 0, 0, 0, 0, 0, 0, 0, 0, 0, 0, 0\}, \{-1, 0, 0, 
  0, 0, 0, 0, 1, -1, 0, 0, 0, 0, 0, 0, 0, 0, 0, 0, 0, 0, 0, 0, 0, 0, 
  0, 0, 0, 0, 0, 0, 0, 0, 0, 0, 0, 0, 0, 0, -1, -1, 0, -1, 0, 
  0, -1, -1, 0, -1\}, \{1, 0, 0, 0, 0, 0, 0, 0, 1, -1, 0, 0, 0, 0, 0, 0,
   0, 0, 0, 0, 0, 0, 0, 0, 0, 0, 0, 0, 0, 0, 0, 0, 0, 0, 0, 0, 0, 0, 
  0, 0, 1, 1, 1, 0, 0, 1, 1, 1, 0\}, \{0, 0, 0, 0, 0, 0, 1, 0, 0, -1, 0,
   0, 0, 0, 0, 0, 0, 0, 0, 0, 0, 0, 0, 0, 0, 0, 0, 0, 0, 0, 0, 0, 0, 
  0, 0, 0, 0, 0, 1, 0, 0, 0, 0, 0, 0, 0, 0, 0, 0\}, \{0, 0, 0, 0, 0, -1,
   0, 0, 0, 1, 0, 0, -1, 0, 0, 0, 0, 0, 0, 0, 0, 0, -1, 0, 0, 0, 0, 0,
   0, 0, 0, 0, -1, -1, 0, 0, 0, 0, 0, -1, 0, 0, -1, 0, 0, 0, 0, 0, 
  0\}, \{0, 0, 0, 1, 0, -1, 0, 0, 0, 0, 0, 0, 0, 0, 0, 0, 0, 0, 0, 
  0, -1, 0, -1, 0, 0, 0, 0, -1, 0, 0, 0, -1, 0, -1, 0, 0, 0, 0, 0, 0, 
  0, 0, 0, 0, -1, 0, 0, 0, 0\}, \{0, 0, 0, 0, 0, 0, 0, 0, 0, 1, -1, 0, 
  0, 0, 0, 0, 0, 0, 0, 0, 0, 0, 0, 0, 0, 0, -1, 0, 0, 0, 0, 0, 0, 0, 
  0, 0, 0, 0, 0, 0, 0, 0, 0, 0, 0, 0, 0, 0, 0\}, \{1, 0, 0, 0, 0, 0, 0, 
  0, 0, 0, 0, -1, 0, 0, 0, 0, 0, 0, 0, 0, 0, 0, 0, 0, 0, 0, 0, 0, 0, 
  0, 0, 0, 0, 1, 0, 0, 0, 0, 0, 0, 1, 1, 1, 0, 1, 1, 0, 1, 0\}, \{0, 0, 
  0, 0, 0, 0, 0, 0, 0, 0, -1, 1, 0, 0, 0, 0, 0, 0, 0, 0, 0, 0, 0, 0, 
  0, 0, 0, 0, 0, 0, 0, 0, 0, 0, -1, 0, 0, 0, 0, 0, 0, 0, 0, 0, 0, 0, 
  0, 0, 0\}, \{0, 0, 0, 0, -1, 0, 0, 0, -1, 0, 1, 0, 0, 0, 0, 0, 0, 0, 
  0, 0, 0, 0, 0, 0, 0, 0, 0, 0, -1, 0, 0, 0, 0, 0, 0, 0, 0, 0, -1, 0, 
  0, 0, 0, -1, 0, -1, 0, 0, 0\}, \{0, 0, 0, 0, -1, 0, 0, 0, -1, 0, 0, 1,
   0, 0, 0, 0, 0, 0, 0, 0, 0, 0, 0, 0, 0, 0, 0, 0, -1, 0, 0, -1, 0, 0,
   0, 0, 0, 0, -1, 0, 0, 0, 0, -1, 0, -1, 0, 0, 0\}, \{0, 0, 0, 0, 0, 0,
   0, 0, 0, 0, -1, 0, 1, 0, 0, 0, 0, 0, 0, 0, 0, 0, 1, 0, 0, 0, 0, 0, 
  0, 0, 0, 0, 1, 1, 0, 0, 0, 0, 0, 1, 0, 0, 1, 0, 0, 0, 0, 0, 0\}, \{0, 
  0, 0, 0, 0, 0, 0, -1, 0, 0, 0, 0, 1, 0, 0, 0, 0, 0, 0, 0, 0, 0, 0, 
  0, 0, 1, 0, 0, 0, 0, 0, 0, 0, 0, 0, 0, 0, 0, 0, 0, 0, 0, 0, 0, 0, 0,
   0, 0, 0\}, \{0, 0, 0, 0, -1, 0, 0, 0, 0, 0, 0, 0, 1, 0, 0, 0, 0, 0, 
  0, 0, 0, 0, 0, 0, 0, 0, 0, 0, 0, 0, 0, 0, 0, 0, 0, 0, 0, 0, 0, 0, 0,
   0, 0, 0, 0, 0, 0, 0, 0\}, \{-1, 0, 0, 0, 0, 0, 0, 0, 0, 0, 0, 0, 0, 
  1, 0, 0, 0, 0, 0, 0, 0, 0, 0, 0, 0, 0, 0, 0, 0, 0, -1, 0, 0, 0, 
  0, -1, 0, 0, 0, -1, -1, 0, -1, 0, 0, 0, -1, 0, -1\}, \{0, 0, 0, 0, 0, 
  0, 0, 0, 0, 0, 0, 0, 0, -1, 1, 0, 0, 0, 0, 0, 0, 0, 0, 0, 0, 0, 1, 
  0, 0, 0, 0, 0, 0, 0, 0, 0, 0, 0, 0, 0, 0, 0, 0, 0, 0, 0, 0, 0, 
  0\}, \{0, 0, 0, 0, 0, 0, 0, 0, 0, 0, 0, 0, 0, 0, 1, -1, 0, 0, 0, 0, 0,
   0, 0, 0, 1, 0, 0, 0, 0, 0, 0, 0, 0, 0, 0, 0, 0, 0, 0, 0, 0, 0, 0, 
  0, 0, 0, 0, 0, 0\}, \{-1, 0, 0, 0, 0, 0, 0, 0, 0, 0, 0, 0, 0, 0, 0, 1,
   0, 0, 0, 0, 0, 0, 0, 0, 0, 0, 0, -1, 0, 0, -1, 0, 0, 0, 0, -1, 0, 
  0, 0, 0, -1, 0, 0, 0, -1, 0, -1, 0, -1\}, \{0, 0, 0, 0, 0, 0, 0, 0, 1,
   0, 0, 0, 0, 0, -1, 0, 0, 0, 0, 0, 0, 0, 0, 0, 0, 0, 0, 1, 1, 0, 0, 
  1, 0, 0, 0, 0, 0, 1, 0, 0, 0, 0, 0, 0, 0, 0, 1, 0, 0\}, \{0, 0, 0, 0, 
  0, 1, 0, 0, 1, 0, 0, 0, 0, -1, 0, 0, 0, 0, 0, 0, 0, 0, 0, 0, 0, 0, 
  0, 1, 1, 0, 0, 1, 0, 0, 0, 0, 0, 1, 0, 0, 0, 0, 0, 0, 0, 0, 1, 0, 
  0\}, \{0, 0, 0, 0, 0, 0, 0, 0, -1, 0, 0, 0, 0, 0, 0, 0, 1, 0, 0, 0, 0,
   0, 0, 0, 0, 0, 0, 0, -1, 0, 0, -1, 0, 0, 0, 0, 0, 0, -1, 0, 0, 0, 
  0, -1, 0, -1, 0, 0, 0\}, \{0, 0, 0, 0, 0, 0, 0, 0, 0, 0, 0, 0, 0, 0, 
  0, 0, 1, -1, 0, 0, 0, 0, 0, 0, 0, 0, 0, 0, 0, 0, 0, 0, -1, 0, -1, 
  0, -1, 0, 0, 0, 0, -1, 0, 0, 0, 0, 0, -1, 0\}, \{0, 0, 0, 0, 0, 0, 0, 
  0, 0, 0, 0, 0, 0, 0, -1, 0, 0, 1, 0, 0, 0, 0, 0, 0, 0, 0, 0, 0, 0, 
  1, 0, 0, 1, 0, 0, 0, 1, 0, 0, 1, 0, 0, 0, 0, 0, 0, 0, 0, 1\}, \{0, 0, 
  0, 0, 0, 0, 0, 0, 0, 0, 0, 1, -1, 0, 0, 0, -1, 0, 0, 0, 0, 0, 0, 0, 
  0, 0, 0, 0, 0, 0, 0, 0, 0, 0, 0, 0, 0, 0, 0, 0, 0, 0, 0, 0, 0, 0, 0,
   0, 0\}, \{1, 0, 0, 0, 0, 0, 0, 0, 0, 0, 0, 0, 0, 0, 0, 0, 0, 0, -1, 
  0, 0, 0, 0, 0, 0, 0, 0, 0, 0, 0, 1, 0, 0, 1, 0, 0, 0, 0, 0, 0, 0, 0,
   1, 1, 1, 1, 0, 1, 0\}, \{0, 0, 0, 0, 0, 0, 0, 0, 0, 0, 0, 0, 0, 0, 0,
   0, -1, 0, 1, 0, -1, 0, 0, 0, 0, 0, 0, 0, 0, 0, 0, 0, 0, 0, 0, 0, 0,
   0, 0, 0, 0, 0, 0, 0, 0, 0, 0, 0, 0\}, \{0, 0, 0, 0, 0, 0, 0, 0, 0, 0,
   0, 0, 0, 0, 0, 0, 0, -1, 1, 0, 0, -1, 0, 0, 0, 0, 0, 0, 0, 0, 0, 
  0, -1, 0, -1, 0, -1, 0, 0, 0, 0, -1, 0, 0, 0, 0, 0, -1, 0\}, \{0, 0, 
  0, 0, 0, 0, 0, 0, 0, 0, 0, 0, 0, 0, 0, -1, 0, 1, 0, 0, 0, 0, 1, 0, 
  0, 0, 0, 0, 0, 1, 0, 0, 1, 0, 0, 0, 1, 0, 0, 1, 0, 0, 0, 0, 0, 0, 0,
   0, 1\}, \{0, 0, 0, 1, 0, 0, 0, 0, 0, 0, 0, 0, 0, 0, 0, 0, 0, 0, 
  0, -1, 0, 0, 0, 0, -1, 0, 0, 0, 0, 0, 0, 0, 0, 0, 0, 0, 0, 0, 0, 0, 
  0, 0, 0, 0, 0, 0, 0, 0, 0\}, \{0, 0, 0, 0, 0, 0, 0, 0, 0, 0, 0, 0, 0, 
  0, 0, 0, 0, 0, 1, -1, 0, 0, 0, 0, 0, 0, 0, 0, 0, 0, 0, 0, 0, 0, 0, 
  0, 0, 0, -1, 0, 0, 0, 0, 0, 0, 0, 0, 0, 0\}, \{0, 0, 0, 0, 0, 1, 0, 0,
   0, 0, 0, 0, 0, 0, 0, 0, 0, 0, 0, -1, 1, 0, 0, 0, 0, 0, 0, 1, 0, 0, 
  0, 1, 0, 1, 0, 0, 0, 0, 0, 0, 0, 0, 0, 0, 1, 0, 0, 0, 0\}, \{0, -1, 0,
   0, 0, 0, 0, 0, 0, 0, 0, 0, 0, 0, 0, 0, 0, 0, 0, 0, 1, 0, 0, 1, 0, 
  0, 0, 0, 0, 0, 0, 0, 0, 0, 0, 0, 0, 0, 0, 0, 0, 0, 0, 0, 0, 0, 0, 0,
   0\}, \{0, 0, 0, 0, 0, 0, 0, 0, 0, 0, 0, 0, 0, 0, 0, 0, 0, 0, 0, 0, 
  1, -1, 0, 0, 0, 0, 0, 0, 0, 0, 0, 0, 0, 0, 0, 0, 0, 0, 0, 0, 0, 0, 
  0, 0, 0, 0, 0, 0, 0\}, \{0, 0, 0, 0, 0, 0, 0, 0, 0, 0, 0, 0, 0, 0, 0, 
  0, 0, -1, 0, 1, 0, -1, 0, 0, 0, 0, 0, 0, 0, 0, 0, 0, 0, 0, -1, 
  0, -1, 0, 0, 0, 0, -1, 0, 0, 0, 0, 0, -1, 0\}, \{0, 0, 0, 0, 0, 0, 0, 
  0, 0, 0, 0, 0, 0, 0, 0, -1, 0, 0, 0, 0, 0, 0, 0, 1, 0, 0, 0, 0, 0, 
  0, 0, 0, 0, 0, 0, 0, 0, 1, 0, 0, 0, 0, 0, 0, 0, 0, 0, 0, 0\}, \{0, 0, 
  0, 0, 0, 0, 0, 0, 0, 0, 0, 0, 0, 0, 0, 0, 0, 1, 0, 0, 0, 0, 1, -1, 
  0, 0, 0, 0, 0, 1, 0, 0, 1, 0, 0, 0, 0, 0, 0, 1, 0, 0, 0, 0, 0, 0, 0,
   0, 1\}, \{0, 0, 0, 0, -1, 0, 0, 0, 0, 0, 0, 0, 0, 0, 0, 0, 0, 0, 0, 
  0, 0, 0, 0, 1, -1, 0, 0, 0, -1, 0, -1, 0, 0, 0, 0, -1, 0, 0, 0, 0, 
  0, 0, 0, -1, 0, 0, 0, 0, 0\}, \{0, 0, 0, 0, 0, 0, 0, 0, 0, 0, 0, 0, 0,
   0, 0, 0, 0, 0, 0, 0, 0, 0, 1, 0, -1, 0, 0, 0, 0, 0, 0, 0, 0, 0, 0, 
  0, 0, 0, 0, 0, 0, 0, 0, 0, 0, 0, 0, 0, 0\}, \{0, 0, 0, 0, 0, 0, 0, 0, 
  0, 0, 0, 0, 0, -1, 0, 0, 0, 0, 0, 0, 0, 0, 0, 0, 0, 1, 0, 0, 0, 1, 
  0, 0, 0, 0, 0, 0, 0, 0, 0, 0, 0, 0, 0, 0, 0, 0, 0, 0, 0\}, \{0, 0, 0, 
  0, 0, 0, 0, 0, 0, 0, 0, 0, 0, 0, 0, 0, 0, 0, 0, 0, 0, -1, 0, 0, 0, 
  1, -1, 0, 0, 0, 0, 0, 0, 0, 0, -1, -1, 0, 0, 0, -1, -1, 0, 0, 0, 0, 
  0, 0, 0\}, \{0, 0, 0, 0, 0, 1, 0, 0, 0, 0, 0, 0, 0, 0, 0, 0, 0, 0, 0, 
  0, 0, 0, 0, 0, 0, 0, -1, 0, 0, 0, 0, 0, 0, 0, 0, 0, 0, 0, 0, 0, 0, 
  0, 0, 0, 0, 0, 0, 0, 0\}, \{0, 0, 0, 0, 0, 1, 0, 0, 1, 0, 0, 0, 0, 0, 
  0, 0, 0, 0, 0, 0, 0, 0, 0, 0, 0, -1, 0, 1, 0, 0, 0, 1, 0, 0, 0, 0, 
  0, 1, 0, 0, 0, 0, 0, 0, 0, 0, 1, 0, 0\}, \{0, 0, 0, 0, 0, 0, 0, 0, 0, 
  0, 0, 0, 0, 0, 0, 0, 0, 0, 0, 0, 0, 0, 0, 0, 0, 0, 0, 0, 1, -1, 0, 
  0, 0, 0, 0, 0, 0, 0, 0, 0, 0, 0, 0, 0, 0, 0, 0, 0, 0\}, \{0, 0, 0, 0, 
  0, 0, 0, 0, 0, 0, 0, 0, 0, 0, 0, 0, 0, 0, 0, 0, 0, 0, 0, 0, 0, 0, 0,
   0, 0, 0, 0, -1, 0, 0, 1, 0, 0, 0, 0, 0, 0, 0, 0, 0, 0, 0, 0, 0, 
  0\}, \{0, 0, 0, 0, 0, 0, 0, 0, 0, 0, 0, 0, 0, 0, 0, 0, 0, 0, 0, 0, 0, 
  0, 0, 0, 0, 0, 0, 0, 0, 0, 0, 0, 0, 0, 0, 0, 1, -1, 0, 0, 0, 0, 0, 
  0, 0, 0, 0, 0, 0\}, \{0, 0, 0, 0, 0, 0, 0, 0, 0, 0, 0, 0, 0, 0, 0, 0, 
  0, 0, 0, 0, 0, 0, 0, 0, 0, 0, 0, 0, 0, 0, 0, 0, -1, 0, 0, 0, 0, 0, 
  1, 0, 0, 0, 0, 0, 0, 0, 0, 0, 0\}\}
       
} 
The index map $\delta^1:\Z^{\sE}\to\Z^{sF}$ is given by the matrix:\\
{\fontsize{4}{0}\selectfont
$ $\\
\{\{0, 0, 0, 0, 0, 0, 0, 0, 0, 0, 0, 0, 0, 0, 0, -1, 0, 0, 0, 0, 0, 0, 
  0, 0, 0, 0, 0, 0, 0, 0, 0, 0, 0, 0, 0, 0, 0, 0, 1, 0, -1, 0, 0, 0, 
  0, 0, 0, -1, 0, 0, 0, 0, 0, 0, 0, 0\}, \{0, 0, 0, 0, 0, 0, 0, 0, 0, 0,
   0, 0, 0, 0, 0, 0, 0, 0, 0, 0, 0, 0, 0, 0, 0, 0, 0, 0, 0, 0, 0, 1, 
  0, 0, 0, 1, -1, 0, 0, 0, 0, 0, 1, 0, 0, 0, 0, 0, 0, 0, 0, 0, 0, 0, 
  0, 0\}, \{0, 0, 0, 0, 0, 0, 0, 0, 0, 0, 0, 0, 0, 0, 0, 0, 0, 0, 1, 
  1, -1, 0, 0, 0, 0, 0, 0, 0, 0, 0, 0, 0, 0, 0, 0, 0, 0, 0, 0, 0, 0, 
  0, 0, 0, 0, 0, 0, 0, 0, 0, 0, 0, 0, 1, 0, 0\}, \{0, 0, 0, 0, 0, 0, 0, 
  0, 1, 0, 0, 0, 0, -1, -1, 0, 0, 0, 0, 0, 0, 0, 0, 0, 0, 0, 0, 0, 0, 
  0, 0, 0, 0, 0, 0, 0, 0, 0, 0, 0, 0, 0, 0, 0, 0, 0, 0, 0, 0, 0, 0, 0,
   0, 0, 0, 1\}, \{0, 0, 0, 0, 0, 0, 0, 0, 0, 1, 0, 0, 0, 0, 0, 0, 0, 0,
   0, 0, 0, 0, -1, 1, 0, 0, 0, 0, 0, 0, 0, 0, 0, 0, 0, 0, 0, 0, 0, 0, 
  0, 0, 0, 0, 0, 0, 0, 0, 0, 1, 0, 0, 0, 0, 0, 0\}, \{0, 0, 0, 0, 0, 0, 
  0, 0, 0, 0, 0, 0, 0, 0, 0, 0, 0, 0, 0, 0, 0, 0, 0, 0, 0, -1, 0, 
  0, -1, 1, 0, 0, 0, 0, 0, 0, 0, 0, 0, 0, 0, 0, 0, 0, 0, 0, 0, 0, 0, 
  0, -1, 0, 0, 0, 0, 0\}, \{0, 0, 0, 0, 0, 0, 0, 0, 0, 0, 0, 0, 0, 0, 0,
   0, 0, 0, 0, 0, 0, 0, 0, 0, 0, 0, 0, 0, 0, 0, 0, 0, 0, 0, 0, 0, 0, 
  1, 0, 0, 0, 0, 0, 0, -1, -1, 0, 0, 0, 0, 0, 0, 0, 0, -1, 0\}, \{0, 1, 
  0, 0, -1, 1, 0, 0, 0, 0, 0, 0, 0, 0, 0, 0, 0, 0, 0, 0, 0, 0, 0, 0, 
  0, 0, 0, 0, 0, 0, 0, 0, 0, 0, 0, 0, 0, 0, 0, 0, 0, 0, 0, 0, 0, 0, 0,
   0, 0, 0, 0, 0, -1, 0, 0, 0\}, \{0, 0, 1, 0, 0, 0, -1, 0, 0, 0, 0, 0, 
  0, 0, 0, 1, 0, 0, 0, 0, 0, 0, 0, 0, 0, 0, 0, 0, 0, 0, 0, 0, 0, 0, 0,
   0, 0, 0, 0, 0, 0, 0, 0, 0, 0, 0, 0, 0, 0, 0, 0, 0, 0, -1, 0, 
  0\}, \{0, 0, 0, 0, 0, 0, 0, 0, 0, 0, 0, 0, 0, 0, 0, 0, 0, 0, 0, 0, 0, 
  0, 0, 0, 0, 0, 0, 0, 0, 0, 0, 0, 0, 0, 0, 0, 1, 0, 0, -1, 0, 0, 
  0, -1, 0, 0, 0, 0, 0, 0, 0, 0, 0, 0, 0, -1\}, \{0, 0, 0, 0, 0, 0, 0, 
  0, 0, 0, 0, 0, 0, 0, 0, 0, 0, 0, 0, 0, 1, 0, 0, -1, 0, 0, 0, 0, 0, 
  0, -1, 0, 0, -1, 0, 0, 0, 0, 0, 0, 0, 0, 0, 0, 0, 0, 0, 0, 0, 0, 0, 
  0, 0, 0, 0, 0\}, \{0, 0, 0, 0, 0, 0, 0, 0, 0, 0, 0, 0, 0, 0, 1, 0, -1,
   0, 0, 0, 0, 1, 0, 0, 0, 0, 0, 0, 0, 0, 0, 0, 0, 0, 0, 0, 0, 0, 0, 
  0, 0, 0, 0, 0, 0, 0, 0, 0, 0, 0, 1, 0, 0, 0, 0, 0\}, \{0, 0, 0, 0, 0, 
  0, 0, 1, 0, -1, -1, 0, 0, 0, 0, 0, 0, 0, 0, 0, 0, 0, 0, 0, 0, 0, 0, 
  0, 0, 0, 0, 0, 0, 0, 0, 0, 0, 0, 0, 0, 0, 0, 0, 0, 0, 0, 0, 0, 0, 0,
   0, 0, 0, 0, 1, 0\}, \{0, 0, 0, 0, 0, 0, 0, 0, 0, 0, 0, 0, 0, 0, 0, 0,
   0, 0, 0, 0, 0, 0, 0, 0, 0, 0, 0, 0, 0, -1, 0, 0, 0, 0, 0, 0, 0, 0, 
  0, 0, 0, 0, 0, 0, 0, 0, 0, 0, 1, 0, 0, 1, 1, 0, 0, 0\}, \{0, 0, 0, 0, 
  0, 0, 0, 0, 0, 0, 0, 0, 0, 0, 0, 0, 0, 0, 0, 0, 0, 0, 0, 0, 0, 0, 1,
   0, 0, 0, 0, 0, 1, 0, 0, 0, 0, -1, 0, 0, 0, 0, 0, 0, 0, 0, 0, 1, 0, 
  0, 0, 0, 0, 0, 0, 0\}, \{0, 0, 0, 0, 0, -1, 0, 0, 0, 0, 0, 0, 0, 0, 0,
   0, 0, 0, 0, 0, 0, 0, 0, 0, 0, 0, 0, 0, 0, 0, 0, 0, 0, 0, 0, 0, 0, 
  0, 0, 0, 0, 1, -1, 0, 0, 0, -1, 0, 0, 0, 0, 0, 0, 0, 0, 0\}, \{0, 0, 
  0, 0, 1, 0, 1, -1, -1, 0, 0, 0, 0, 0, 0, 0, 0, 1, 0, 0, 0, 0, 0, 0, 
  1, 1, -1, -1, 0, 0, 0, 0, 0, 1, -1, -1, 0, 0, 0, 0, 0, 0, 0, 0, 0, 
  0, 0, 0, 0, 0, 0, 0, 0, 0, 0, 0\}, \{-1, 0, 0, 0, 0, 0, 0, 0, 0, 0, 0,
   0, 1, 0, 0, 0, 1, -1, -1, 0, 0, 0, 0, 0, 0, 0, 0, 1, 0, 0, 0, 0, 0,
   0, 0, 0, 0, 0, 0, 0, 1, -1, 0, 1, 1, 0, 0, 0, 0, -1, 0, -1, 0, 0, 
  0, 0\}, \{1, -1, -1, 1, 0, 0, 0, 0, 0, 0, 1, -1, -1, 1, 0, 0, 0, 0, 0,
   0, 0, 0, 0, 0, 0, 0, 0, 0, 1, 0, 1, -1, -1, 0, 0, 0, 0, 0, 0, 0, 0,
   0, 0, 0, 0, 0, 0, 0, 0, 0, 0, 0, 0, 0, 0, 0\}, \{0, 0, 0, -1, 0, 0, 
  0, 0, 0, 0, 0, 1, 0, 0, 0, 0, 0, 0, 0, -1, 0, -1, 1, 0, -1, 0, 0, 0,
   0, 0, 0, 0, 0, 0, 1, 0, 0, 0, -1, 1, 0, 0, 0, 0, 0, 1, 1, 0, -1, 0,
   0, 0, 0, 0, 0, 0\}\}
       
}
We illustrate the homotopy $h_s$, $0\le s\le 1$, on the vertices in the following figure:
\begin{center}
\includegraphics[scale=.36]{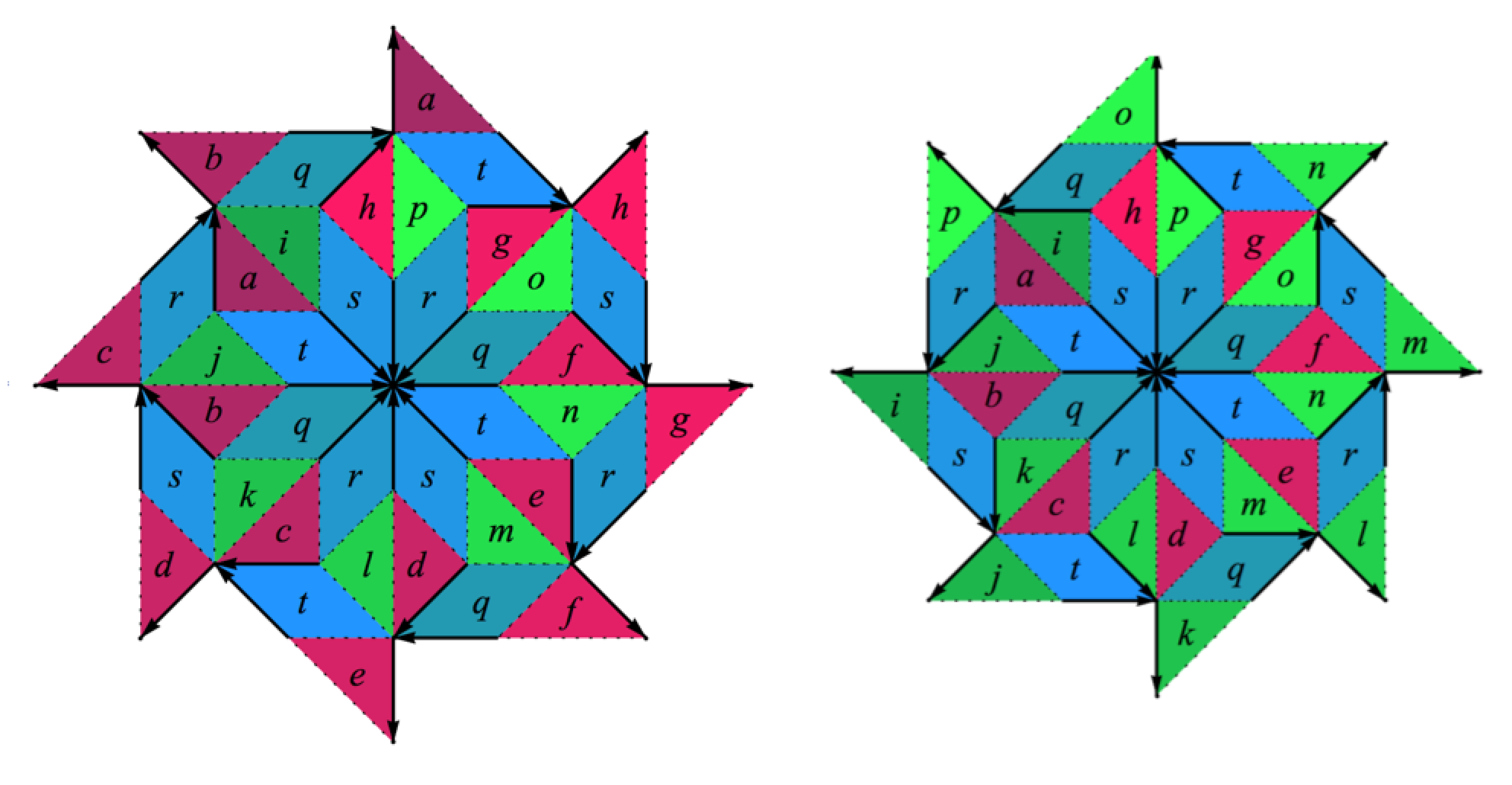}
\end{center}
\begin{center}
\includegraphics[scale=.25]{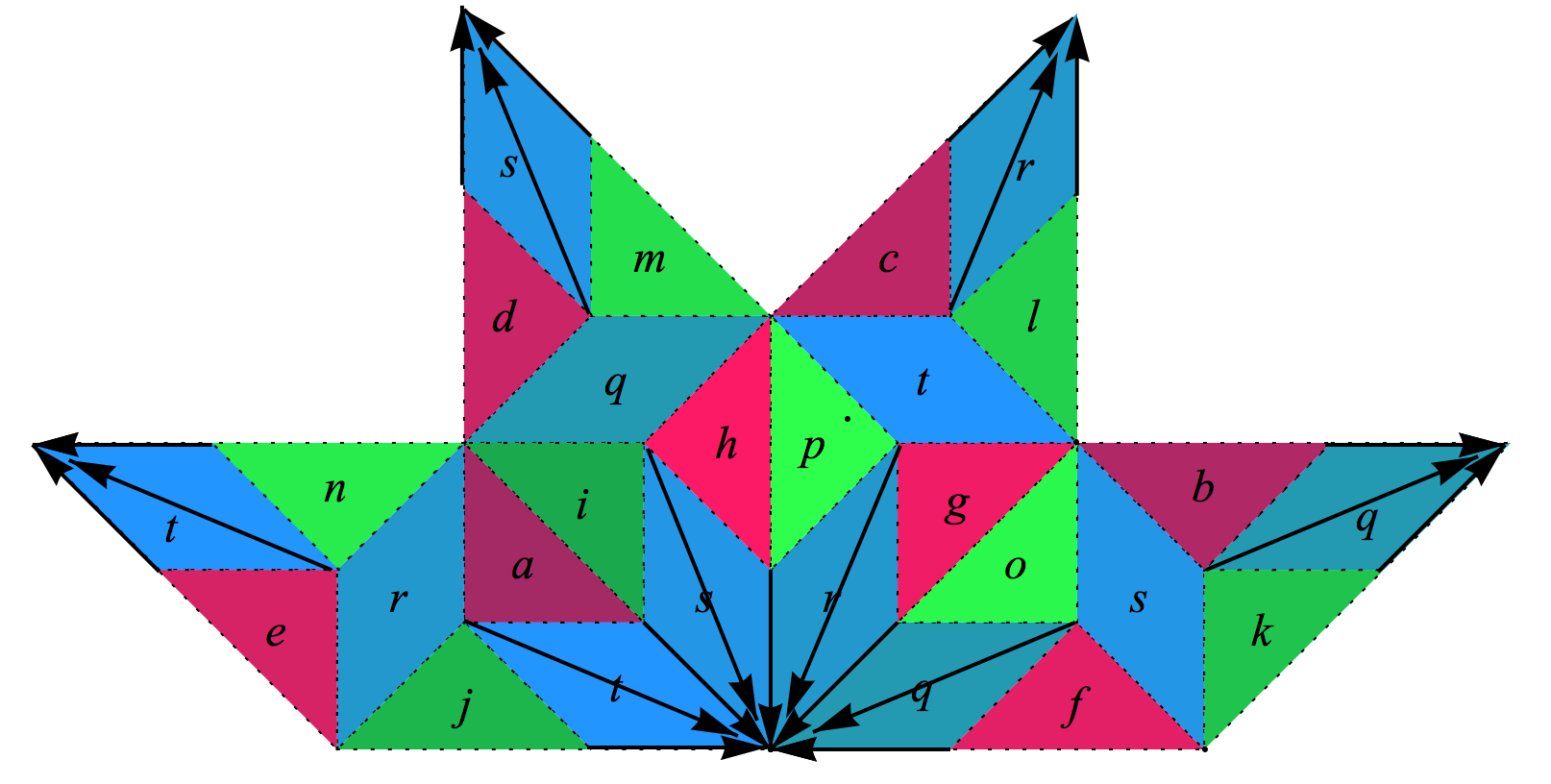}
\end{center}
The substitution-homotopy map $W_V:\Z^{\sV}\to\Z^{sV}$ is given by the matrix:
{\fontsize{4}{0}\selectfont
$ $\\
\{\{1, 0, 0, 0, 1, 1, 0, 0, 1, 0, 0, 0, 0, 0, 0, 0, 0, 1, 0, 0, 0, 1, 1,
   0, 0, 0, 0, 1, 1, 0, 1, 1, 1, 1, 0, 1, 1, 0, 0, 1, 1, 1, 1, 1, 1, 
  1, 1, 1, 1\}, \{1, 0, 0, 0, 1, 1, 1, 0, 1, 0, 1, 0, 1, 0, 0, 0, 0, 1, 
  0, 0, 0, 1, 1, 0, 0, 0, 1, 1, 1, 0, 1, 1, 1, 1, 1, 1, 1, 1, 1, 1, 1,
   1, 1, 1, 1, 1, 1, 1, 1\}, \{1, 0, 0, 0, 0, 0, 1, 0, 1, 0, 1, 0, 1, 0,
   0, 0, 0, 1, 0, 0, 0, 0, 0, 0, 0, 1, 1, 1, 1, 0, 1, 1, 0, 0, 0, 0, 
  0, 1, 1, 0, 0, 0, 0, 1, 1, 1, 1, 1, 1\}, \{1, 0, 0, 0, 1, 1, 1, 0, 1, 
  0, 0, 0, 1, 0, 0, 0, 0, 1, 0, 0, 0, 1, 1, 0, 0, 1, 1, 1, 1, 1, 1, 1,
   1, 1, 0, 1, 1, 1, 1, 1, 1, 1, 1, 1, 1, 1, 1, 1, 1\}, \{0, 1, 0, 0, 0,
   0, 0, 0, 0, 0, 0, 0, 0, 0, 0, 0, 0, 0, 0, 0, 0, 0, 0, 0, 0, 0, 0, 
  0, 0, 0, 0, 0, 0, 0, 0, 0, 0, 0, 0, 0, 0, 0, 0, 0, 0, 0, 0, 0, 
  0\}, \{0, 0, 0, 1, 0, 0, 0, 0, 0, 0, 0, 0, 0, 0, 0, 0, 0, 0, 0, 0, 0, 
  0, 0, 0, 0, 0, 0, 0, 0, 0, 0, 0, 0, 0, 0, 0, 0, 0, 0, 0, 0, 0, 0, 0,
   0, 0, 0, 0, 0\}, \{1, 0, 1, 0, 0, 0, 0, 0, 1, 0, 0, 0, 0, 0, 0, 0, 0,
   1, 0, 1, 1, 0, 0, 1, 1, 0, 0, 0, 0, 1, 0, 0, 1, 0, 1, 0, 1, 0, 0, 
  1, 1, 1, 1, 0, 0, 1, 1, 1, 1\}, \{1, 0, 1, 0, 1, 1, 0, 0, 1, 0, 0, 0, 
  0, 0, 0, 0, 0, 1, 0, 1, 1, 1, 1, 0, 1, 0, 0, 1, 1, 1, 1, 1, 1, 1, 1,
   1, 1, 0, 1, 1, 1, 1, 1, 1, 1, 1, 1, 1, 1\}, \{0, 0, 1, 0, 0, 0, 0, 0,
   0, 0, 0, 0, 0, 0, 0, 0, 0, 0, 0, 0, 0, 0, 0, 0, 0, 0, 0, 0, 0, 0, 
  0, 0, 0, 0, 0, 0, 0, 0, 0, 0, 0, 0, 0, 0, 0, 0, 0, 0, 0\}, \{1, 0, 1, 
  0, 1, 1, 0, 0, 1, 0, 0, 0, 0, 0, 0, 0, 0, 1, 0, 0, 1, 1, 1, 1, 1, 0,
   0, 1, 1, 1, 1, 1, 1, 1, 1, 1, 1, 1, 0, 1, 1, 1, 1, 1, 1, 1, 1, 1, 
  1\}, \{1, 0, 1, 0, 0, 1, 0, 0, 1, 0, 0, 0, 0, 0, 1, 0, 0, 0, 0, 0, 1, 
  0, 0, 1, 1, 0, 0, 1, 0, 1, 0, 1, 0, 1, 0, 0, 0, 1, 0, 0, 1, 1, 1, 0,
   1, 1, 1, 1, 0\}, \{1, 0, 0, 0, 1, 1, 0, 0, 1, 0, 0, 0, 0, 0, 1, 0, 0,
   1, 0, 0, 1, 1, 1, 1, 1, 0, 1, 1, 1, 1, 1, 1, 1, 1, 0, 1, 1, 1, 0, 
  1, 1, 1, 1, 1, 1, 1, 1, 1, 1\}, \{0, 1, 0, 0, 0, 0, 0, 0, 0, 0, 0, 0, 
  0, 0, 0, 0, 0, 0, 0, 0, 0, 0, 0, 0, 0, 0, 0, 0, 0, 0, 0, 0, 0, 0, 0,
   0, 0, 0, 0, 0, 0, 0, 0, 0, 0, 0, 0, 0, 0\}, \{1, 0, 0, 0, 1, 1, 0, 0,
   1, 0, 0, 0, 1, 0, 0, 0, 1, 1, 0, 1, 1, 1, 1, 0, 1, 0, 0, 1, 1, 0, 
  1, 1, 1, 1, 1, 1, 1, 0, 1, 1, 1, 1, 1, 1, 1, 1, 1, 1, 1\}, \{1, 0, 0, 
  0, 0, 1, 0, 0, 0, 0, 1, 0, 1, 0, 0, 0, 1, 0, 0, 1, 1, 0, 1, 0, 0, 0,
   0, 1, 0, 0, 1, 0, 0, 1, 1, 1, 0, 0, 1, 1, 1, 0, 1, 0, 1, 0, 1, 0, 
  1\}, \{1, 0, 0, 0, 1, 1, 0, 0, 1, 0, 1, 0, 1, 0, 0, 0, 1, 1, 0, 0, 1, 
  1, 1, 0, 0, 0, 1, 1, 1, 0, 1, 1, 1, 1, 1, 1, 1, 0, 1, 1, 1, 1, 1, 1,
   1, 1, 1, 1, 1\}, \{1, 0, 0, 0, 1, 0, 0, 0, 0, 0, 0, 0, 0, 0, 1, 0, 0,
   0, 0, 0, 0, 1, 0, 1, 1, 1, 1, 0, 0, 1, 1, 0, 0, 1, 0, 1, 0, 1, 0, 
  0, 1, 1, 1, 1, 1, 1, 0, 1, 0\}, \{0, 0, 0, 0, 0, 0, 1, 0, 0, 0, 0, 0, 
  0, 0, 0, 0, 0, 0, 0, 0, 0, 0, 0, 0, 0, 0, 0, 0, 0, 0, 0, 0, 0, 0, 0,
   0, 0, 0, 0, 0, 0, 0, 0, 0, 0, 0, 0, 0, 0\}, \{1, 0, 0, 0, 1, 1, 0, 0,
   1, 0, 0, 0, 1, 0, 1, 0, 0, 1, 0, 0, 0, 1, 1, 0, 1, 1, 1, 1, 1, 1, 
  1, 1, 1, 1, 0, 1, 1, 1, 0, 1, 1, 1, 1, 1, 1, 1, 1, 1, 1\}, \{1, 0, 0, 
  0, 0, 0, 1, 0, 0, 0, 0, 0, 1, 0, 1, 0, 0, 1, 0, 0, 0, 0, 1, 0, 0, 1,
   1, 0, 0, 1, 1, 0, 1, 1, 0, 0, 0, 1, 0, 1, 0, 0, 1, 1, 1, 1, 0, 1, 
  1\}, \{0, 0, 0, 0, 0, 0, 0, 1, 0, 0, 0, 0, 0, 0, 0, 0, 0, 0, 0, 0, 0, 
  0, 0, 0, 0, 0, 0, 0, 0, 0, 0, 0, 0, 0, 0, 0, 0, 0, 0, 0, 0, 0, 0, 0,
   0, 0, 0, 0, 0\}, \{0, 0, 0, 0, 0, 0, 0, 1, 0, 0, 0, 0, 0, 0, 0, 0, 0,
   0, 0, 0, 0, 0, 0, 0, 0, 0, 0, 0, 0, 0, 0, 0, 0, 0, 0, 0, 0, 0, 0, 
  0, 0, 0, 0, 0, 0, 0, 0, 0, 0\}, \{0, 0, 0, 0, 0, 0, 0, 0, 0, 1, 0, 0, 
  0, 0, 0, 0, 0, 0, 0, 0, 0, 0, 0, 0, 0, 0, 0, 0, 0, 0, 0, 0, 0, 0, 0,
   0, 0, 0, 0, 0, 0, 0, 0, 0, 0, 0, 0, 0, 0\}, \{1, 0, 0, 0, 0, 0, 1, 0,
   0, 0, 1, 0, 1, 0, 0, 0, 1, 1, 0, 0, 0, 1, 0, 0, 0, 0, 1, 1, 0, 0, 
  1, 0, 0, 0, 1, 1, 1, 0, 1, 0, 1, 1, 0, 0, 1, 0, 1, 1, 1\}, \{0, 0, 0, 
  0, 0, 0, 0, 0, 0, 1, 0, 0, 0, 0, 0, 0, 0, 0, 0, 0, 0, 0, 0, 0, 0, 0,
   0, 0, 0, 0, 0, 0, 0, 0, 0, 0, 0, 0, 0, 0, 0, 0, 0, 0, 0, 0, 0, 0, 
  0\}, \{1, 0, 1, 0, 1, 0, 0, 0, 1, 0, 0, 0, 0, 0, 0, 0, 1, 0, 0, 1, 1, 
  0, 0, 0, 1, 0, 0, 0, 1, 0, 1, 0, 0, 0, 1, 1, 0, 0, 1, 1, 1, 0, 1, 1,
   0, 1, 1, 0, 1\}, \{0, 0, 0, 1, 0, 0, 0, 0, 0, 0, 0, 0, 0, 0, 0, 0, 0,
   0, 0, 0, 0, 0, 0, 0, 0, 0, 0, 0, 0, 0, 0, 0, 0, 0, 0, 0, 0, 0, 0, 
  0, 0, 0, 0, 0, 0, 0, 0, 0, 0\}, \{0, 0, 0, 0, 0, 0, 0, 0, 0, 0, 1, 0, 
  0, 0, 0, 0, 0, 0, 0, 0, 0, 0, 0, 0, 0, 0, 0, 0, 0, 0, 0, 0, 0, 0, 0,
   0, 0, 0, 0, 0, 0, 0, 0, 0, 0, 0, 0, 0, 0\}, \{0, 0, 0, 0, 0, 0, 0, 0,
   0, 0, 0, 1, 0, 0, 0, 0, 0, 0, 0, 0, 0, 0, 0, 0, 0, 0, 0, 0, 0, 0, 
  0, 0, 0, 0, 0, 0, 0, 0, 0, 0, 0, 0, 0, 0, 0, 0, 0, 0, 0\}, \{0, 0, 0, 
  0, 0, 0, 0, 0, 0, 0, 0, 1, 0, 0, 0, 0, 0, 0, 0, 0, 0, 0, 0, 0, 0, 0,
   0, 0, 0, 0, 0, 0, 0, 0, 0, 0, 0, 0, 0, 0, 0, 0, 0, 0, 0, 0, 0, 0, 
  0\}, \{0, 0, 0, 0, 0, 0, 0, 0, 0, 0, 0, 0, 1, 0, 0, 0, 0, 0, 0, 0, 0, 
  0, 0, 0, 0, 0, 0, 0, 0, 0, 0, 0, 0, 0, 0, 0, 0, 0, 0, 0, 0, 0, 0, 0,
   0, 0, 0, 0, 0\}, \{0, 0, 0, 0, 0, 0, 0, 0, 0, 0, 0, 0, 0, 1, 0, 0, 0,
   0, 0, 0, 0, 0, 0, 0, 0, 0, 0, 0, 0, 0, 0, 0, 0, 0, 0, 0, 0, 0, 0, 
  0, 0, 0, 0, 0, 0, 0, 0, 0, 0\}, \{0, 0, 0, 0, 0, 0, 0, 0, 0, 0, 0, 0, 
  0, 0, 0, 1, 0, 0, 0, 0, 0, 0, 0, 0, 0, 0, 0, 0, 0, 0, 0, 0, 0, 0, 0,
   0, 0, 0, 0, 0, 0, 0, 0, 0, 0, 0, 0, 0, 0\}, \{0, 0, 0, 0, 0, 0, 0, 0,
   0, 0, 0, 0, 0, 0, 1, 0, 0, 0, 0, 0, 0, 0, 0, 0, 0, 0, 0, 0, 0, 0, 
  0, 0, 0, 0, 0, 0, 0, 0, 0, 0, 0, 0, 0, 0, 0, 0, 0, 0, 0\}, \{0, 0, 0, 
  0, 0, 0, 0, 0, 0, 0, 0, 0, 0, 1, 0, 0, 0, 0, 0, 0, 0, 0, 0, 0, 0, 0,
   0, 0, 0, 0, 0, 0, 0, 0, 0, 0, 0, 0, 0, 0, 0, 0, 0, 0, 0, 0, 0, 0, 
  0\}, \{0, 0, 0, 0, 0, 0, 0, 0, 0, 0, 0, 0, 0, 0, 0, 0, 1, 0, 0, 0, 0, 
  0, 0, 0, 0, 0, 0, 0, 0, 0, 0, 0, 0, 0, 0, 0, 0, 0, 0, 0, 0, 0, 0, 0,
   0, 0, 0, 0, 0\}, \{0, 0, 0, 0, 0, 0, 0, 0, 0, 0, 0, 0, 0, 0, 0, 0, 0,
   0, 1, 0, 0, 0, 0, 0, 0, 0, 0, 0, 0, 0, 0, 0, 0, 0, 0, 0, 0, 0, 0, 
  0, 0, 0, 0, 0, 0, 0, 0, 0, 0\}, \{0, 0, 0, 0, 0, 0, 0, 0, 0, 0, 0, 0, 
  0, 0, 0, 0, 0, 0, 1, 0, 0, 0, 0, 0, 0, 0, 0, 0, 0, 0, 0, 0, 0, 0, 0,
   0, 0, 0, 0, 0, 0, 0, 0, 0, 0, 0, 0, 0, 0\}, \{0, 0, 0, 0, 0, 0, 0, 0,
   0, 0, 0, 0, 0, 0, 0, 1, 0, 0, 0, 0, 0, 0, 0, 0, 0, 0, 0, 0, 0, 0, 
  0, 0, 0, 0, 0, 0, 0, 0, 0, 0, 0, 0, 0, 0, 0, 0, 0, 0, 0\}, \{0, 0, 0, 
  0, 0, 0, 0, 0, 0, 0, 0, 0, 0, 0, 0, 0, 0, 0, 0, 1, 0, 0, 0, 0, 0, 0,
   0, 0, 0, 0, 0, 0, 0, 0, 0, 0, 0, 0, 0, 0, 0, 0, 0, 0, 0, 0, 0, 0, 
  0\}, \{0, 0, 0, 0, 0, 0, 0, 0, 0, 0, 0, 0, 0, 0, 0, 0, 0, 0, 0, 0, 1, 
  0, 0, 0, 0, 0, 0, 0, 0, 0, 0, 0, 0, 0, 0, 0, 0, 0, 0, 0, 0, 0, 0, 0,
   0, 0, 0, 0, 0\}, \{0, 0, 0, 0, 0, 0, 0, 0, 0, 0, 0, 0, 0, 0, 0, 0, 0,
   0, 0, 0, 0, 0, 0, 1, 0, 0, 0, 0, 0, 0, 0, 0, 0, 0, 0, 0, 0, 0, 0, 
  0, 0, 0, 0, 0, 0, 0, 0, 0, 0\}, \{0, 0, 0, 0, 0, 0, 0, 0, 0, 0, 0, 0, 
  0, 0, 0, 0, 0, 0, 0, 0, 0, 0, 0, 0, 1, 0, 0, 0, 0, 0, 0, 0, 0, 0, 0,
   0, 0, 0, 0, 0, 0, 0, 0, 0, 0, 0, 0, 0, 0\}, \{0, 0, 0, 0, 0, 0, 0, 0,
   0, 0, 0, 0, 0, 0, 0, 0, 0, 0, 0, 0, 0, 0, 0, 0, 0, 1, 0, 0, 0, 0, 
  0, 0, 0, 0, 0, 0, 0, 0, 0, 0, 0, 0, 0, 0, 0, 0, 0, 0, 0\}, \{0, 0, 0, 
  0, 0, 0, 0, 0, 0, 0, 0, 0, 0, 0, 0, 0, 0, 0, 0, 0, 0, 0, 0, 0, 0, 0,
   1, 0, 0, 0, 0, 0, 0, 0, 0, 0, 0, 0, 0, 0, 0, 0, 0, 0, 0, 0, 0, 0, 
  0\}, \{0, 0, 0, 0, 0, 0, 0, 0, 0, 0, 0, 0, 0, 0, 0, 0, 0, 0, 0, 0, 0, 
  0, 0, 0, 0, 0, 0, 0, 0, 1, 0, 0, 0, 0, 0, 0, 0, 0, 0, 0, 0, 0, 0, 0,
   0, 0, 0, 0, 0\}, \{0, 0, 0, 0, 0, 0, 0, 0, 0, 0, 0, 0, 0, 0, 0, 0, 0,
   0, 0, 0, 0, 0, 0, 0, 0, 0, 0, 0, 0, 0, 0, 0, 0, 0, 1, 0, 0, 0, 0, 
  0, 0, 0, 0, 0, 0, 0, 0, 0, 0\}, \{0, 0, 0, 0, 0, 0, 0, 0, 0, 0, 0, 0, 
  0, 0, 0, 0, 0, 0, 0, 0, 0, 0, 0, 0, 0, 0, 0, 0, 0, 0, 0, 0, 0, 0, 0,
   0, 0, 1, 0, 0, 0, 0, 0, 0, 0, 0, 0, 0, 0\}, \{0, 0, 0, 0, 0, 0, 0, 0,
   0, 0, 0, 0, 0, 0, 0, 0, 0, 0, 0, 0, 0, 0, 0, 0, 0, 0, 0, 0, 0, 0, 
  0, 0, 0, 0, 0, 0, 0, 0, 1, 0, 0, 0, 0, 0, 0, 0, 0, 0, 0\}\}
       
}
The substitution-homotopy map $W_E:\Z^{\sE}\to\Z^{\sE}$ is given by the matrix:\\
{\fontsize{4}{0}\selectfont
$ $\\
\{\{0, 0, 0, 0, 0, 0, 0, 0, 0, 0, 0, 0, 0, 0, 0, 0, 0, 0, 0, 0, 0, 0, 0,
   0, 0, 0, 0, 0, 0, 0, 0, 0, 0, 0, 0, 0, 0, 0, 0, 0, 0, 0, 0, 0, 0, 
  0, 0, 0, 0, 0, 0, 0, 0, 0, 0, 0\}, \{0, 0, 0, 0, 0, 0, 0, 0, 0, 0, 0, 
  0, 0, 0, 0, 0, 0, 0, -1, 0, 0, 1, 0, -1, 0, 0, 0, 0, 0, 0, 0, 0, 0, 
  0, 0, 0, 0, 0, 0, 0, 0, 0, 0, 0, 0, 0, 0, 0, 0, -1, 1, 0, 0, 0, 0, 
  0\}, \{0, 0, 0, 0, 0, 0, 0, 0, 0, 0, 0, 0, 0, 0, 0, 0, 0, 0, 0, 0, 
  0, -1, 0, 1, 0, 0, 0, 0, 0, 0, 0, 0, 0, 0, 0, 0, 0, 0, 0, 0, 0, 0, 
  0, 0, 0, 0, 0, 0, -1, 1, -1, 0, 0, 0, 0, 0\}, \{0, 0, 0, 0, 0, 0, 0, 
  0, 0, 0, 0, 0, 0, 0, 0, 0, 0, 0, 0, 0, 0, 0, 0, 0, 0, 0, 0, 0, 0, 0,
   0, 0, 0, 0, 0, 0, 0, 0, 0, 0, 0, 0, 0, 0, 0, 0, 0, 0, 0, 0, 0, 0, 
  0, 0, 0, 0\}, \{1, 0, 0, 0, 0, 0, 0, 0, 0, 0, 1, 0, 0, 0, 0, 0, 1, 0, 
  0, 0, 0, 0, 0, 0, 0, 0, 0, 0, 0, 0, 1, 0, 0, 0, 0, 0, 0, 0, 0, 0, 0,
   0, 0, 0, 0, 0, 0, 0, 0, 0, 0, 0, 0, 0, 0, 0\}, \{1, 0, 0, 0, 0, 0, 0,
   0, 0, 0, 1, 0, 0, 0, 1, 0, 0, 0, 1, 0, 0, 0, 0, 1, 0, 0, 0, 0, 0, 
  0, 1, 0, 0, 0, 0, 0, 0, 0, 0, 0, 0, 0, 0, 0, 0, 0, 0, 0, 0, 1, 0, 0,
   0, 0, 0, 0\}, \{0, 0, 0, 1, 0, 0, 0, 0, 0, 0, 0, 0, 0, 1, 0, 0, 0, 0,
   0, 0, 0, 0, 1, 0, 0, 0, 0, 0, 1, 0, 0, 0, 0, 0, 0, 0, 0, 0, 0, 0, 
  0, 0, 0, 0, 0, 0, 0, 0, 0, 0, 0, 0, 0, 0, 0, 0\}, \{0, 1, 0, 0, 0, 0, 
  0, 0, 0, 0, 0, 1, 0, 0, 0, 0, 0, 0, 0, 0, 0, 0, 0, 0, 0, 0, 0, 0, 0,
   0, 0, 1, 0, 0, 0, 0, 0, 0, 1, 0, 0, 0, 0, 0, 0, 0, 0, 0, 0, 0, 0, 
  0, 0, 0, 0, 0\}, \{0, 0, 1, 0, 0, 0, 0, 0, 0, 0, 0, 0, 1, 0, 0, 0, 0, 
  0, 0, 0, 0, 0, 0, 0, 0, 0, 0, 0, 0, 0, 0, 0, 1, 0, 0, 0, 0, 0, 0, 0,
   0, 1, 0, 0, 0, 0, 0, 0, 0, 0, 0, 0, 0, 0, 0, 0\}, \{0, 1, 0, 0, 0, 0,
   0, 0, 0, 0, 0, 1, 0, 0, 0, 1, 0, 0, 0, 0, 0, 0, 0, 0, 0, 0, 0, 0, 
  0, 0, 0, 1, 0, 0, 0, 0, 0, 0, 0, 1, 0, 0, 1, 0, 0, 0, 1, 0, 0, 0, 0,
   0, 0, 0, 0, 0\}, \{0, 0, 0, 0, 0, 0, 0, 0, 0, 0, 0, 0, 0, 0, 0, 0, 0,
   0, 0, 0, 0, 0, 0, 0, 0, 0, 0, 0, 0, 0, 0, 0, 0, 0, 0, 0, 0, 0, 
  0, -1, 1, 0, -1, 0, 0, 0, -1, 1, 0, 0, 0, 0, 0, 0, 0, 0\}, \{0, 0, 0, 
  0, 0, 0, 0, 0, 0, 0, 0, 0, 0, 0, 0, 0, 0, 0, 0, 0, 0, 0, 0, 0, 0, 0,
   0, 0, 0, 0, 0, 0, 0, 0, 0, 0, 0, 0, 0, 0, 0, 0, 0, 0, 0, 0, 0, 0, 
  0, 0, 0, 0, 0, 0, 0, 0\}, \{0, 0, 0, 0, 0, 0, 0, 0, 0, 0, 0, 0, 0, 0, 
  0, 0, 0, 0, 0, 0, 0, 0, 0, 0, 0, 0, 0, 0, 0, 0, 0, 0, 0, 0, 0, 0, 0,
   0, 0, 0, 0, 0, 0, 0, 0, 0, 0, 0, 0, 0, 0, 0, 0, 0, 0, 0\}, \{0, 0, 0,
   0, 0, 0, 0, 0, 0, 0, 0, 0, 0, 0, 0, 0, 0, 0, 0, 0, 0, 0, 0, 0, 0, 
  0, 0, 0, 0, 0, 0, 0, 0, 0, 0, 0, 0, 0, 0, 0, -1, 0, 1, 0, -1, 0, 
  1, -1, 0, 0, 0, 0, 0, 0, 0, 0\}, \{0, 0, 1, 0, 0, 1, 0, 0, 0, 0, 0, 0,
   1, 0, 0, 0, 0, 0, 0, 0, 0, 0, 0, 0, 0, 0, 0, 0, 0, 0, 0, 0, 1, 0, 
  0, 0, 0, 0, 0, 0, 1, 0, 0, 0, 1, 0, 0, 1, 0, 0, 0, 0, 0, 0, 0, 
  0\}, \{0, 0, 0, 1, 0, 0, 0, 0, 0, 1, 0, 0, 0, 1, 0, 0, 0, 0, 0, 0, 0, 
  1, 0, 0, 0, 0, 0, 0, 1, 0, 0, 0, 0, 0, 0, 0, 0, 0, 0, 0, 0, 0, 0, 0,
   0, 0, 0, 0, 1, 0, 1, 0, 0, 0, 0, 0\}, \{0, 0, 1, 0, 1, 0, 0, 0, 0, 0,
   0, 0, 0, 0, 0, 0, 0, 0, 0, 0, 0, 0, 0, 0, 0, 0, 0, 0, 0, 0, 0, 0, 
  1, 0, 0, 0, 0, 0, 0, 0, 0, 0, 0, 0, 0, 0, 0, 1, 0, 0, 0, 0, 1, 0, 0,
   0\}, \{0, 0, 0, 0, 0, 0, 0, 0, 0, 0, 0, 0, 0, 0, 0, 0, 0, 0, 0, 0, 0,
   0, 0, 0, 0, 0, 0, 0, 0, 0, 0, 0, 0, 0, 0, 0, 0, 0, 0, 0, 0, 0, 0, 
  0, 0, 0, 0, 0, 0, 0, 0, 0, 0, 0, 0, 0\}, \{0, 0, 0, 0, 1, 0, 0, 0, 0, 
  0, 0, 0, 0, 0, 0, 0, 0, 0, 0, 0, 0, 0, 0, 0, 0, 1, 0, 0, 0, 0, 0, 0,
   1, 0, 0, 0, 0, 0, 0, 0, 0, 0, 0, 0, 0, 0, 0, 1, 0, 0, 0, 0, 1, 0, 
  0, 0\}, \{0, 0, 0, 0, 0, 0, 0, 0, 0, 0, 0, 0, 0, 0, 0, 0, 0, 1, 0, 0, 
  0, 0, 0, 0, 0, 0, 1, 0, 0, 0, 0, 0, 0, 0, 0, 0, 0, 0, 0, 0, 0, 1, 0,
   0, 0, 0, 0, 0, 0, 0, 0, 1, 0, 0, 0, 0\}, \{0, 0, 0, 0, 1, 0, 0, 0, 0,
   0, 0, 0, 0, 0, 0, 0, 0, 1, 0, 0, 0, 0, 0, 0, 0, 1, 0, 0, 0, 0, 0, 
  0, 0, 0, 0, 0, 0, 1, 0, 0, 0, 1, 0, 0, 0, 0, 0, 0, 0, 0, 0, 1, 1, 0,
   0, 0\}, \{0, 1, 0, 0, 0, 0, 0, 0, 0, 0, 0, 0, -1, 0, 0, 0, 0, 0, 0, 
  0, 0, 0, 0, 0, 0, 0, 0, 0, 0, 0, 0, 0, 0, 0, 0, 0, 0, 0, 0, 0, -1, 
  0, 0, 0, -1, 0, 0, 0, 0, 0, 0, 0, 0, 0, 0, 0\}, \{0, 1, 0, 0, 0, 0, 1,
   0, 0, 0, 0, 0, 0, 0, 0, 0, 0, 0, 0, 0, 0, 0, 0, 0, 0, 0, 0, 0, 0, 
  0, 0, 1, 0, 0, 0, 0, 0, 0, 0, 0, 0, 0, 1, 0, 0, 0, 0, 0, 0, 0, 0, 0,
   0, 1, 0, 0\}, \{0, 0, 0, 0, 0, 0, 0, 0, 0, 0, 0, 0, 0, 0, 0, 0, 0, 0,
   0, 0, 0, 0, 0, 0, 0, 0, 0, 0, 0, 0, 0, 0, 0, 0, 0, 0, 0, 0, 0, 0, 
  0, 0, 0, 0, 0, 0, 0, 0, 0, 0, 0, 0, 0, 0, 0, 0\}, \{0, 0, 0, 0, 0, 0, 
  0, 0, 0, 0, 0, 0, 0, 0, 0, 0, 0, 0, 0, 0, 0, 0, 0, 0, 0, 0, 0, 0, 0,
   0, 0, 0, 0, 0, 0, 0, 0, 0, 0, 0, 0, 0, 0, 0, 0, 0, 0, 0, 0, 0, 0, 
  0, 0, 0, 0, 0\}, \{0, 0, 0, 0, 0, 0, 0, 0, 0, 0, 0, 0, 0, 0, 0, 0, 0, 
  0, 0, 1, 0, 0, 0, 0, 0, 0, 0, 0, 0, 0, 0, 0, 0, 0, 0, 0, 0, 0, 1, 0,
   0, 0, 0, 1, 0, 0, 0, 0, 0, 0, 0, 0, 0, 1, 0, 1\}, \{0, 0, 0, 0, 0, 0,
   0, 0, 0, 0, 0, 0, 0, 0, 0, 0, 1, 0, 0, 1, 0, 0, 0, 0, 0, 0, 0, 0, 
  0, 0, 0, 0, 0, 0, 0, 0, 0, 0, 0, 0, 0, 0, 0, 1, 0, 0, 0, 0, 0, 0, 0,
   0, 0, 1, 0, 1\}, \{0, 0, 0, 0, 0, 0, 0, 0, 0, 0, 0, 0, 0, 0, 0, 0, 0,
   0, 0, 0, 0, 0, 0, 0, 0, 0, 0, 0, 0, 0, 0, 0, 0, 0, 0, 0, 0, 0, 0, 
  0, 0, 0, 0, 0, 0, 0, 0, 0, 0, 0, 0, 0, 0, 0, 0, 0\}, \{0, 0, 0, 0, 0, 
  0, 1, 0, 0, 0, 0, 0, 0, 0, 0, 0, 0, 0, 0, 0, 0, 0, 0, 0, 1, 0, 0, 0,
   0, 0, 0, 0, 0, 1, 0, 0, 0, 0, 0, 1, 0, 0, 0, 0, 0, 0, 0, 0, 0, 0, 
  0, 0, 0, 0, 0, 0\}, \{0, 0, 0, 0, 0, 0, 1, 0, 0, 0, 0, 0, 0, 0, 0, 0, 
  0, 0, 0, 1, 0, 0, 0, 0, 1, 0, 0, 0, 0, 0, 0, 0, 0, 1, 0, 0, 1, 0, 1,
   0, 0, 0, 0, 0, 0, 0, 0, 0, 0, 0, 0, 0, 0, 1, 0, 0\}, \{0, 0, 0, 0, 1,
   0, 0, 0, 0, 0, 0, 0, 0, 0, 0, 0, 0, 1, 0, 0, 0, 0, 0, 0, 0, 1, 0, 
  0, 0, 0, 0, 0, 0, 0, 0, 0, 0, 0, 0, 0, 0, 0, 0, 0, 1, 0, 0, 0, 0, 0,
   0, 0, 0, 0, 0, 0\}, \{0, 0, 0, 0, 0, 0, 0, 1, 0, 0, 0, 0, 0, 0, 0, 0,
   0, 0, 0, 0, 0, 0, 0, 0, 0, 0, 1, 0, 0, 0, 0, 0, 0, 0, 1, 0, 0, 0, 
  0, 0, 0, 0, 0, 0, 0, 0, 0, 0, 1, 0, 0, 0, 0, 0, 0, 0\}, \{0, 0, 0, 0, 
  0, 0, 0, 0, 1, 0, 0, 0, 0, 0, 0, 0, 0, 0, 1, 0, 0, 0, 0, 0, 0, 0, 0,
   1, 0, 0, 0, 0, 0, 0, 0, 1, 0, 0, 0, 0, 0, 0, 0, 0, 0, 0, 0, 0, 0, 
  0, 0, 0, 0, 0, 0, 0\}, \{0, 0, 0, 0, 0, 0, 0, 0, 0, 0, 0, 0, 0, 0, 0, 
  0, 0, 0, 0, 0, 0, 0, 0, 0, 0, 0, 0, 0, 0, 0, 0, 0, 0, 0, 0, 0, 0, 0,
   0, 0, 0, 1, 0, 0, 0, 1, 0, 0, 0, 0, 0, 1, 1, 0, 1, 0\}, \{0, 0, 0, 0,
   0, 0, 0, 0, 0, 0, 0, 0, 0, 0, 0, 0, 0, 0, 0, 0, 0, 0, 0, 0, 0, 0, 
  0, 0, 0, 0, 0, 0, 0, 0, 0, 0, 0, 0, 0, 0, 0, 0, 0, 0, 0, 0, 0, 0, 0,
   0, 0, 0, 0, 0, 0, 0\}, \{0, 0, 0, 0, 0, 0, 0, 0, 0, 0, 0, 0, 0, 0, 0,
   0, 0, 0, 0, 0, 0, 0, 1, 0, 0, 0, 0, 0, 0, 0, 0, 0, 0, 0, 0, 0, 0, 
  0, 0, 0, 0, 0, 0, 0, 0, 1, 0, 0, 0, 0, 0, 1, 1, 0, 1, 0\}, \{0, 0, 0, 
  0, 0, 0, 0, 1, 0, 0, 0, 0, 0, 0, 0, 0, 0, 0, 0, 0, 0, 0, 1, 0, 0, 0,
   1, 0, 0, 1, 0, 0, 0, 0, 1, 0, 0, 0, 0, 0, 0, 0, 0, 0, 0, 1, 0, 0, 
  0, 0, 0, 0, 0, 0, 1, 0\}, \{0, 0, 0, 0, 0, 0, 0, 0, 1, 0, 0, 0, 0, 0, 
  0, 0, 1, 0, 0, 0, 1, 0, 0, 0, 0, 0, 0, 1, 0, 0, 0, 0, 0, 0, 0, 1, 0,
   0, 0, 0, 0, 0, 0, 1, 0, 0, 0, 0, 0, 0, 0, 0, 0, 0, 0, 1\}, \{0, 0, 0,
   0, 0, 0, 0, 1, 0, 0, 0, 0, 0, 1, 0, 0, 0, 0, 0, 0, 0, 0, 0, 0, 0, 
  0, 0, 0, 1, 0, 0, 0, 0, 0, 0, 0, 0, 0, 0, 0, 0, 0, 0, 0, 0, 0, 0, 0,
   0, 0, 1, 0, 0, 0, 1, 0\}, \{0, 0, 0, 0, 0, 0, 0, 1, 0, 0, 0, 0, 0, 0,
   0, 0, 0, 0, 0, 0, 0, 0, 0, 0, 0, 0, 1, 0, 1, 0, 0, 0, 0, 0, 0, 0, 
  0, 0, 0, 0, 0, 0, 0, 0, 0, 0, 0, 0, 0, 0, 1, 0, 0, 0, 1, 0\}, \{0, 0, 
  0, -1, 0, 0, 0, 0, 0, 0, 1, 0, 0, 0, 0, 0, 0, 0, 0, 0, 0, -1, 0, 0, 
  0, 0, 0, 0, 0, 0, 0, 0, 0, 0, 0, 0, 0, 0, 0, 0, 0, 0, 0, 0, 0, 0, 0,
   0, -1, 0, 0, 0, 0, 0, 0, 0\}, \{0, 0, 0, 0, 0, 0, 0, 0, 1, 0, 1, 0, 
  0, 0, 0, 0, 0, 0, 0, 0, 0, 0, 0, 1, 0, 0, 0, 0, 0, 0, 1, 0, 0, 0, 0,
   0, 0, 0, 0, 0, 0, 0, 0, 0, 0, 0, 0, 0, 0, 0, 0, 0, 0, 0, 0, 1\}, \{0,
   0, 0, 0, 0, 0, 0, 0, 0, 0, 0, 0, 0, 0, 0, 0, 0, 0, 0, 0, 0, 0, 0, 
  0, 0, 0, 0, 0, 0, 0, 0, 0, 0, 0, 0, 0, 0, 0, 0, 0, 0, 0, 0, 0, 0, 0,
   0, 0, 0, 0, 0, 0, 0, 0, 0, 0\}, \{0, 0, 0, 0, 0, 0, 0, 0, 0, 0, 0, 0,
   0, 0, 0, 0, 0, 0, 0, 0, 0, 0, 1, 0, 0, 1, 0, 0, 0, 0, 0, 0, 0, 0, 
  1, 0, 0, 0, 0, 0, 0, 0, 0, 0, 0, 1, 0, 0, 0, 0, 0, 0, 0, 0, 0, 
  0\}, \{0, 0, 0, 0, 0, 0, 0, 0, 1, 0, 0, 0, 0, 0, 0, 0, 0, 0, 0, 0, 0, 
  0, 0, 1, 0, 0, 0, 0, 0, 0, 1, 0, 0, 0, 0, 1, 0, 0, 0, 0, 0, 0, 0, 0,
   0, 0, 0, 0, 0, 0, 0, 0, 0, 0, 0, 1\}, \{0, 0, 0, 0, 0, 0, 0, 0, 0, 0,
   0, 0, 0, 0, 0, 0, 1, 0, 0, 0, 0, 0, 0, 0, 0, 0, 0, 1, 0, 0, 0, 0, 
  0, 1, 0, 0, 0, 0, 0, 0, 0, 0, 0, 1, 0, 0, 0, 0, 0, 0, 0, 0, 0, 0, 0,
   0\}, \{-1, 0, 0, 0, 0, 0, 0, 0, 0, 0, 0, 0, 0, 1, 0, 0, 0, 0, -1, 0, 
  0, 0, 0, 0, 0, 0, 0, 0, 0, 0, 0, 0, 0, 0, 0, 0, 0, 0, 0, 0, 0, 0, 0,
   0, 0, 0, 0, 0, 0, -1, 0, 0, 0, 0, 0, 0\}, \{0, 0, 0, 0, 0, 0, 0, 0, 
  0, 0, 0, 0, 0, 0, 0, 0, 0, 0, 0, 0, 0, 0, 0, 0, 0, 0, 0, 0, 0, 0, 0,
   0, 0, 0, 0, 0, 0, 0, 0, 0, 0, 0, 0, 0, 0, 0, 0, 0, 0, 0, 0, 0, 0, 
  0, 0, 0\}, \{0, 0, 0, 0, 0, 0, 1, 0, 0, 0, 0, 0, 0, 0, 0, 0, 0, 0, 0, 
  0, 0, 0, 0, 0, 0, 0, 0, 0, 0, 0, 0, 1, 0, 1, 0, 0, 0, 0, 0, 0, 0, 0,
   1, 0, 0, 0, 0, 0, 0, 0, 0, 0, 0, 1, 0, 0\}, \{0, 0, 1, 0, 0, 0, 0, 0,
   0, 0, 0, -1, 0, 0, 0, 0, 0, 0, 0, 0, 0, 0, 0, 0, 0, 0, 0, 0, 0, 0, 
  0, 0, 0, 0, 0, 0, 0, 0, 0, -1, 0, 0, 0, 0, 0, 0, -1, 0, 0, 0, 0, 0, 
  0, 0, 0, 0\}, \{0, 0, 0, 0, 0, 0, 0, 0, 0, 0, 0, 0, 0, 0, 0, 0, 0, 0, 
  0, 0, 0, 0, 0, 0, 0, 0, 0, 0, 0, 0, 0, 0, 0, 0, 0, 0, 0, 0, 0, 0, 0,
   0, 0, 0, 0, 0, 0, 0, 0, 0, 0, 0, 0, 0, 0, 0\}, \{0, 0, 0, 0, 0, 0, 0,
   0, 0, 0, 0, 0, 0, 0, 0, 0, 0, 0, 0, 1, 0, 0, 0, 0, 1, 0, 0, 0, 0, 
  0, 0, 0, 0, 0, 0, 1, 0, 0, 1, 0, 0, 0, 0, 0, 0, 0, 0, 0, 0, 0, 0, 0,
   0, 0, 0, 0\}, \{0, 0, 0, 0, 0, 0, 0, 0, 0, 0, 0, 0, 0, 0, 0, 0, 0, 0,
   0, 0, 0, 0, 0, 0, 0, 0, 0, 0, 0, 0, 0, 0, 0, 0, 0, 0, 0, 0, 0, 0, 
  0, 0, 0, 0, 0, 0, 0, 0, 0, 0, 0, 0, 0, 0, 0, 0\}, \{0, 0, 0, 0, 0, 0, 
  0, 0, 0, 0, 0, 0, 0, 0, 0, 0, 0, 0, 0, 0, 0, 0, 0, 0, 0, 0, 0, 0, 0,
   0, 0, 0, 0, 0, 0, 0, 0, 0, 0, 0, 0, 0, 0, 0, 0, 0, 0, 0, 0, 0, 0, 
  0, 0, 0, 0, 0\}, \{0, 0, 0, 0, 0, 0, 0, 0, 0, 0, 0, 0, 0, 0, 0, 0, 0, 
  0, 0, 0, 0, 0, 0, 0, 0, 0, 0, 0, 0, 0, 0, 0, 0, 0, 0, 0, 0, 0, 0, 0,
   0, 0, 0, 0, 0, 0, 0, 0, 0, 0, 0, 0, 0, 0, 0, 0\}, \{0, 0, 0, 0, 0, 0,
   0, 0, 0, 0, 0, 0, 0, 0, 0, 0, 0, 0, 0, 0, 0, 0, 0, 0, 0, 0, 0, 0, 
  0, 0, 0, 0, 0, 0, 0, 0, 0, 0, 0, 0, 0, 0, 0, 0, 0, 0, 0, 0, 0, 0, 0,
   0, 0, 0, 0, 0\}\}
       
}
The substitution-homotopy map $W_F:\Z^{\sF}\to\Z^{\sF}$ is given by the matrix
\begin{center}
\includegraphics[scale=.6]{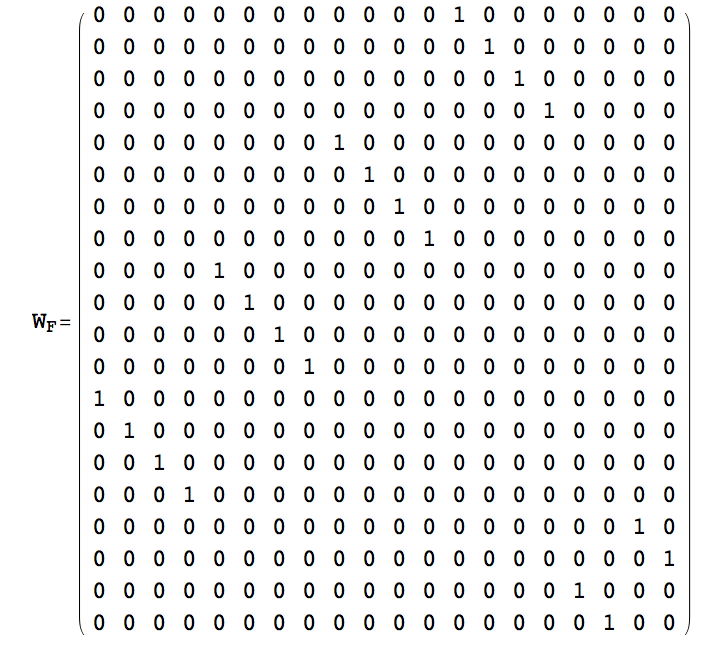}
\end{center}
By Proposition \ref{p:kerA-isom-Znminusr} we replace  $\ker \delta^0$ with $\Z^{17}$.
By Proposition \ref{p:imA-isom-Zr} we remove the zero-eigenvalues of the corresponding matrix and get
\begin{eqnarray*}
H_S^0(T)&=&\dlim(W_V,\,\ker\delta^0)\\
&=&\Z^3\oplus\dlim(
\begin{tiny}
\left(
\begin{array}{ccccccccc}
 8 & -19 & -63 & 12 & -14 & -7 & 12 & -1 & 7 \\
 10 & -24 & -74 & 15 & -15 & -6 & 8 & -1 & 4 \\
 -2 & 2 & 5 & -1 & 1 & 0 & 0 & 0 & 0 \\
 3 & -15 & -48 & 10 & -10 & -5 & 7 & -1 & 3 \\
 0 & 6 & 21 & -4 & 5 & 3 & -5 & 1 & -2 \\
 -2 & 1 & 7 & -2 & -1 & -1 & 5 & -1 & 1 \\
 0 & 1 & 7 & -1 & 2 & 2 & -4 & 1 & -2 \\
 1 & 1 & 7 & -1 & 2 & 2 & -4 & 1 & -2 \\
 0 & -1 & -4 & 0 & -3 & -3 & 8 & -2 & 3 \\
\end{array}
\right)
\end{tiny},
\Z^9)\\
&=&\Z^9
\end{eqnarray*}
because the determinant of the $9\times 9$ matrix is -1. 
By Proposition \ref{p:kerOverIm-A-isom-Znminus}, $\lim\limits_{\to}(W_E,\,\frac{\ker\delta^1}{\im\delta^0})=\lim\limits_{\to}(W_E',\coker B)$ for some matrices $W_E',B$.
Then by Proposition \ref{p:cokerA-isom-Znminuss} we get rid of the coker and get
$$K_1(S)=H_S^1(T)=\lim_{\to}(W_E,\frac{\ker\delta^1}{\im\delta^0})=\lim_{\to}(
\left(
\begin{array}{ccccc}
 0 & 1 & 0 & -1 & 0 \\
 3 & 1 & 0 & 0 & 1 \\
 2 & 2 & -1 & -2 & 0 \\
 0 & -1 & 1 & 2 & 1 \\
 1 & -1 & 1 & 3 & 1 \\
\end{array}
\right),\Z^5)=\Z^5$$
because the determinant of the $5\times5$ matrix is -1.
\end{exam}

The $K$-theory groups of the above examples are summarized in Table~\ref{table:K0groups-d2} and Table \ref{table:K1groups-d2}. 
We are the first ones to compute the stable and unstable $K$-theories of the tri-square tiling. 
The unstable $K$-theories of the rest of the above examples are already well-known, and they agree with our computations.

\begin{table}[t]
\begin{center}\caption{$K_0$-groups for tilings of the plane.  }\label{table:K0groups-d2}
  \begin{tabular}{| r | c | c| c| c| }
    \hline
      Tiling       &  $K_0(U)=\Z\oplus H_0^{ST}(T) $               & $K_0(S)/\Z=H_S^0(T)$               & $K_0(A)$ \\ \hline
     Octagonal             & $\Z^{10}$               & $\Z^{9}$              & $\Z^{200}$\\
     Chair                  & $\Z\oplus\Z[\frac12]^3$         &  $\Z[\frac12]^3$       & \\
     Tri-square     &  $\Z^5\oplus\Z[\frac12]^3$        &  $\Z^4\oplus\Z[\frac12]^3$       & \\
     Half-hex              & $\Z^3\oplus\Z[\frac12]$         & $\Z^2\oplus\Z[\frac12]$      &  \\
     Table & $\Z^4\oplus\Z[\frac12]^5\oplus \Z_2$ & $\Z^3\oplus\Z[\frac12]^5 $   & --\\
    \hline
  \end{tabular}

\end{center}
\end{table}
\begin{table}[t]
\caption{$K_1$-groups for tilings of the plane.}\label{table:K1groups-d2}
\begin{center}
  \begin{tabular}{| r | c | c| c| c| }
    \hline
     Tiling        &  $K_1(U)=H_1^{ST}(T)$               & $K_1(S)=H^1_{S}(T)$              & $K_1(A)$ \\ \hline
     Octagonal             & $\Z^{5}$                & $\Z^{5}$              & $\Z^{200}$\\
     Chair                  & $\Z[\frac12]^2$         &  $\Z[\frac12]^2$      & \\
     Tri-square      &  $\Z[\frac12]^2$        &  $\Z[\frac12]^2$      & \\
     Half-hex              & $\Z[\frac12]^2$         &  $\Z[\frac12]^2$      &  \\
    Table                 & $\Z[\frac12]^2$         &  $\Z[\frac12]^2\oplus\Z_2$ & --\\
    \hline
  \end{tabular}
\end{center}
\end{table}



\appendix

\subsection*{\underline{Acknowledgments}}
The authors would like to thank Jean Renault, Erik Christensen, Rohit Dilip Holkar, Uffe Haagerup, and Toke N{\o}rg{\aa}rd-Solano for the fruitful discussions and valuable suggestions. 
The first author was partially supported by CNPq. 
The second author was
supported by the Villum Foundation under the project ``Local and global structures
of groups and their algebras'' at University of Southern Denmark, and by CNPq grant at Universidade Federal de Santa Catarina.

\section{\textbf{Direct limits of abelian groups}}\label{a.s:direct-limits-AG}
In this section we give some tools to compute direct limits of integer matrices.
\begin{defn}[Direct limit notation]
Suppose that $A:X\to X$ is an endomorphism of some abelian group $X$ (i.e.~a $\Z$-module). Then we introduce the notation for the direct limit
$$\xymatrix{
&\lim\limits_{\to}(A,X):=&\!\!\!\!\!\!\!\!\!\!\!\!\!\!\!\!\lim\limits_{\to}X\ar[r]^{A}&X\ar[r]^{A}&X\ar[r]^{A}&.
}$$
Note that we write the homomorphism $A$ first, which we do in order to emphasize it. 
\end{defn}

\begin{pro}\label{p:lim-union}
Let $A\in M_n(\Z)$ be an $n\times n$ integer matrix such that $\det A\ne 0$. Then
$$\xymatrix{
&\bigcup_{k=0}^{\infty} A^{-k}\Z^n\,=\,&\!\!\!\!\!\!\!\!\!\!\!\!\!\!\!\!\lim\limits_{\to}\,\Z^n\ar[r]^{A}&\Z^n\ar[r]^{A}&\Z^n\ar[r]^{A}&.
}$$
Note that 
$$\bigcup_{k=0}^{\infty} A^{-k}\Z^n\subset \Z[\frac{1}{\det A}]^n.$$
\end{pro}
\begin{proof}
Since $A\Z^n\subset \Z^n$ we get that $\Z^n=A^{-1}(A\Z^n)\subset A^{-1}\Z^n$. Thus the following diagram is well defined
$$\xymatrix{
&\Z^n\ar[r]^{A}\ar[d]^{I}_{\cong}&\Z^n\ar[r]^{A}\ar[d]^{A^{-1}}_{\cong}&\Z^n\ar[r]^{A}\ar[d]^{A^{-2}}_{\cong}&\Z^n\ar[r]^{A}\ar[d]^{A^{-3}}_{\cong}&\\
&\Z^n\ar@{^(->}[r]^{}&A^{-1}\Z^n\ar@{^(->}[r]^{}&A^{-2}\Z^n\ar@{^(->}[r]^{}&A^{-3}\Z^n\ar@{^(->}[r]^{A}&.
}$$
Moreover, since the diagram commutes, and the vertical maps are isomorphisms, and the maps on the bottom row are the inclusion maps the statement of the proposition holds.
\end{proof}

\begin{cor}
Let $A\in M_n(\Z)$ be an integer matrix such that $\det A=\pm 1$. Then
$$\lim\limits_{\to}(A,\Z^n)=\Z^n.$$
\end{cor}
\begin{proof}
If $\det A=\pm 1$ then $A^{-1}\Z^n=\Z^n$. The statement of the corollary then follows by the proposition.
\end{proof}

\begin{pro}\label{p:lim:Ginv}
Let $A,G\in M_n(\Z)$ be integer matrices, and suppose that $\det G=\pm 1$. Then
$$\lim\limits_{\to}(A,\Z^n)=\lim\limits_{\to}(GAG^{-1} , \Z^n).$$
\end{pro}
\begin{proof}
Since the diagram
$$\xymatrix{
&\Z^n\ar[r]^{A}\ar[d]^{G}_{\cong}&\Z^n\ar[r]^{A}\ar[d]^{G}_{\cong}&\Z^n\ar[r]^{A}\ar[d]^{G}_{\cong}&\Z^n\ar[r]^{A}\ar[d]^{G}_{\cong}&\\
&\Z^n\ar[r]_{G A G^{-1}}&\Z^n\ar[r]_{G A G^{-1}}&\Z^n\ar[r]_{G A G^{-1}}&\Z^n\ar[r]_{G A G^{-1}}&
}$$
commutes, and its vertical maps are isomorphisms, the statement of the proposition holds.
\end{proof}

\begin{pro}\label{p:limA:Zn-isom-limA:AZn}
Let $A\in M_n(\Z)$ be an integer matrix. Then
$$\lim\limits_{\to}(A,\Z^n)=\lim\limits_{\to}(A , A\Z^n).$$
\end{pro}
\begin{proof}
Since the diagram
$$\xymatrix{
&\Z^n\ar[r]^{A}\ar[d]_{A}&\Z^n\ar[r]^{A}\ar[d]_{A}&\Z^n\ar[r]^{A}\ar[d]_{A}&\Z^n\ar[r]^{A}\ar[d]_{A}&\\
&A\Z^n\ar[r]_{A}\ar@{^(->}[ru]^{\iota}&A\Z^n\ar[r]_{A}\ar@{^(->}[ru]^{\iota}&A\Z^n\ar[r]_{A}\ar@{^(->}[ru]^{\iota}&A\Z^n\ar[r]_{A}\ar@{^(->}[ru]^{\iota}&
}$$
commutes and we can go up and down in the diagram, the statement of the proposition holds.
\end{proof}

\begin{pro}\label{p:Asqr}
Let $A\in M_n(\Z)$ be an integer matrix. Then
$$\lim\limits_{\to}(A,\Z^n)=\lim\limits_{\to}(A^2 , \Z^n).$$
\end{pro}
\begin{proof}
Since the diagram
$$\xymatrix{
&\Z^n\ar[r]^{A}\ar[d]_{I}&\Z^n\ar[r]^{A}\ar[d]_{A}&\Z^n\ar[r]^{A}\ar[d]_{I}&\Z^n\ar[r]^{A}\ar[d]_{A}&\\
&\Z^n\ar[r]_{A^2}\ar[ru]^{A}&\Z^n\ar[r]_{I}\ar[ru]^{I}&\Z^n\ar[r]_{A^2}\ar[ru]^{A}&\Z^n\ar[r]_{I}\ar[ru]^{I}&
}$$
commutes, and we can go up and down in the diagram, the statement of the proposition holds.
\end{proof}

\begin{pro}\label{p:pAq}
Let $X,Y$ be $\Z$-modules.
Let $A:X\to X$ be an endomorphism, and let $X$ and $Y$ be $\Z$-isomorphic with isomorphisms $p:X\to Y$, $q:Y\to X$. Then
$$\lim\limits_{\to}(A,X)=\lim\limits_{\to}(pAq,Y).$$ 
\end{pro}
\begin{proof}
This is because the following diagram commutes and we can go up and down the diagram.
$$\xymatrix{
&X\ar[r]^{A}\ar@<-.5ex>[d]_{p}&X\ar[r]^{A}\ar@<-.5ex>[d]_{p}&X\ar[r]^{A}\ar@<-.5ex>[d]_{p}&X\ar[r]^{A}\ar@<-.5ex>[d]_{p}&\\
&Y\ar[r]_{p A q}\ar@<-.5ex>[u]_{q}&Y\ar[r]_{p A q}\ar@<-.5ex>[u]_{q}&Y\ar[r]_{p A q}\ar@<-.5ex>[u]_{q}&Y\ar[r]_{p A q}\ar@<-.5ex>[u]_{q}&
}$$
\end{proof}

\begin{pro}[direct sum]\label{p:lim:direct-sum}
Let $A\in M_m(\Z)$, $B\in M_n(\Z)$ be integer matrices. Then
 $$\lim\limits_{\to}(  
 \left(\begin{array}{cc}
    A & 0 \\ 
    0 & B \\ 
  \end{array}\right)
 ,\Z^{m+n})=\lim\limits_{\to}(A,\Z^m)\oplus \lim\limits_{\to}(B,\Z^n).$$ 
\end{pro}
\begin{proof}
Recall that 
$$\lim\limits_{\to}(A,\Z_i^m)=\frac{\bigoplus_{i=1}^\infty \Z_i^m}{\Span\{x_i e_i- (Ax_i)e_{i+1}\mid i\in\N\}}=\{[z e_i]\mid i\in\N,z\in\Z^m\},$$
where $e_i=(0,\ldots,0,1,0,\ldots)$ has 1 at the $i$-th position, and $\Z_i:=\Z$.
For instance if 
$$x=x_1 e_{i_1}+x_2 e_{i_2} + x_3 e_{i_3} \in \bigoplus_{i\in\N}\Z_i^m$$ with $i_1<i_2<i_3$ then
$$x\sim (A^{i_3-i_i}x_1+A^{i_3-i_2}x_2+x_3)e_{i_3}.$$
We claim that
$$\frac{\bigoplus_{i\in\N} (\Z_{i,a}^m\oplus\Z_{i,b}^n)}{\langle (x,y)e_i-(Ax,By)e_{i+1}\rangle}=
\frac{(\bigoplus_{i\in\N} \Z_{i,a}^m)\oplus(\bigoplus_{i\in\N}\Z_{i,b}^n)}{\langle x e^{(a)}_{i}-(Ax)e^{(a)}_{i+1},y e^{(b)}_{j}-(By)e^{(b)}_{j+1}\rangle}
$$
$$
=\left(\frac{\bigoplus_{i\in\N} \Z_{i,a}^m}{xe^{(a)}_i-(Ax)e^{(a)}_{i+1}}\right)\oplus
\left(\frac{\bigoplus_{i\in\N} \Z_{j,b}^n}{ye^{(b)}_i-(By)e^{(b)}_{j+1}}\right).
$$
The second equality is clear since the relations do not intertwine. The proof of the first equality is as follows.
Note that the left numerator of the equality can be identified with the right numerator.
Hence it suffices to show that the denominators are equal.
Proof that the left denominator is contained in the right denominator:
Let $(x,y)e_i-(Ax,By)e_{i+1}$ be given. 
This element is in the right denominator because it is the sum of two  generators of the right denominator with $i=j$, namely
$$(x e_i^{(a)}-(Ax)e_{i+1}^{(a)})+(y e_i^{(b)}-(By)e_{i+1}^{(b)})= (x,y)e_i-(Ax,By)e_{i+1}.$$
Proof that the left denominator contains the right denominator:
Let $x e^{(a)}_{i}-(Ax)e^{(a)}_{i+1}$ be given. This element is in the left denominator since 
$$(x,0)e_i-(Ax,0)e_{i+1}=x e^{(a)}_{i}-(Ax)e^{(a)}_{i+1}.$$
By the same argument as before, we see that the element $y e^{(b)}_{j}-(By)e^{(b)}$ is in the left denominator.
\end{proof}
The next proposition tells us how to remove the zero-eigenvalues of a matrix.

\begin{pro}[$\im A \cong \Z^r$]\label{p:imA-isom-Zr}
Let $A\in M_{m\times n}(\Z)$ be an $m \times n$ integer matrix and let 
$D = P A Q$ be its Smith normal form. In particular, all the nonzero entries in the diagonal of 
$$D=\mathrm{diag}(d_1,d_2,...,d_r,0,...0)$$
 are given first. Moreover, $d_i$ divides $d_{i+1}$, $i=1,\ldots,r-1$.

Define  $\pi:\Z^n \to \Z^r$ to be the projection onto the first $r$ coordinates i.e.~$\pi$ is the first $r$ rows of the identity matrix $I_n$.

Let  $q:\Z^r \to A\Z^n$ be defined by
$$q:=P^{-1} D \pi^t,$$
where $\pi^t$ is the transpose of $\pi$, and let $p:A\Z^n\to \Z^r$ be defined by
$$p:=\pi D^\dagger P$$
where the pseudo-inverse $D^\dagger$ is the diagonal matrix 
 $$D^\dagger=\mathrm{diag}(\frac{1}{d_1}, \frac{1}{d_2},...,\frac{1}{d_r},0,...,0).$$
Then $A\Z^n$ and $\Z^r$ are $\Z$-isomorphic with isomorphisms $p,q$. In particular, (for computer testing)
$$pq=I_r, \qquad qpA=A,$$
where $I_r$ is the $r \times r$ identity matrix. 
\end{pro}

\begin{proof}
Observe that $p$ is a rational matrix, not necessarily an integer matrix.
We have
$$q\Z^r=P^{-1} D \pi^t\Z^r=P^{-1} (D \pi^t\pi)\Z^n=P^{-1} D\Z^n=A(Q\Z^n)=A\Z^n.$$
Thus $q$ is surjective. 
The map $pq$ is the identity on $\Z^r$ because 
$$pq=\pi D^\dagger P P^{-1} D\pi^t =\pi (D^\dagger D \pi^t)=\pi\pi^t=I_r.$$
We now show that the map $qp$ is the identity on $A\Z^n$.
Let $x\in A\Z^n$. Since $q$ is surjective $x=q(y)$ for some $y\in \Z^r$.
Then
$$qp(x)=qpq(y)=q(y)=x.$$
Thus $A\Z^n$ and $\Z^r$ are $\Z$-isomorphic with isomorphisms $p$, $q$.
\end{proof}

\begin{pro}[$\ker A\cong \Z^{n-r}$]\label{p:kerA-isom-Znminusr}
Let $A\in M_{m\times n}(\Z)$ be an $m \times n$ integer matrix and let 
$D = P A Q$ be its Smith normal form. In particular, all the nonzero entries in the diagonal of 
$$D=\mathrm{diag}(d_1,d_2,...,d_r,0,...0)$$
 are given first. Moreover, $d_i$ divides $d_{i+1}$, $i=1,\ldots,r-1$.
If $r=n$ then $\ker A=\{0\}$, so assume $r<n$.
Define  $\pi:\Z^n \to \Z^{n-r}$ to be the projection onto the last $n-r$ coordinates i.e.~$\pi$ is the last $n-r$ rows of the identity square matrix $I_n$.
Let $q:\Z^{n-r}\to \ker A$ be defined by
$$q:=Q\circ\pi^{t},$$
where $\pi^t$ is the transpose of $\pi$, and let $p:\ker A\to \Z^{n-r}$ be defined by
$$p:=\pi\circ Q^{-1}.$$
Then $\ker A$ and $\Z^{n-r}$ are $\Z$-isomorphic, with isomorphisms $p,q$. In particular, (for computer testing)
$$pq=I_{n-r},\qquad Aqp =Aq=0.$$
\end{pro}
\begin{proof}
\begin{eqnarray*}
\ker A&=&\{x\in \Z^n\mid Ax=0\}\\
&=&\{x\in \Z^n\mid P^{-1}DQ^{-1}x=0\}\\
&=&\{x\in\Z^n\mid Dy=0, y=Q^{-1}x\}\\
&=&\{Qy\mid \lambda_1y_1 e_1+\cdots +\lambda_r y_r e_r=0\}\\
&=&\{Qy\mid y_1=0,\ldots, y_r=0\}\\
&=&\Span(Q e_{r+1},\ldots ,Q e_n)\\
&=&Q\pi^t\Z^{n-r}\\
&=&q\Z^{n-r}.
\end{eqnarray*}
Thus $q$ is surjective.
The map $pq$ is the identity on $\Z^{n-r}$ because
$$pq=\pi Q^{-1}Q \pi^t= \pi\pi^{t}=I_{n-r}.$$
We now show that the map $qp$ is the identity on $\ker A$.
Let $x\in \ker A$. Since $q$ is surjective, $x=q(y)$ for some $y\in \Z^{n-r}$.
Then 
$$qp(x)=qpq(y)=q(y)=x.$$ 
Thus $\ker A$ and $\Z^{n-r}$ are $\Z$-isomorphic, with isomorphisms $p$, $q$. 
\end{proof}

\begin{pro}[$\coker A\cong \frac{\Z}{d_{s+1}}\oplus\cdots\oplus\frac{\Z}{d_{m}}$\,\,(for $m\le n$)]\label{p:cokerA-isom-Znminuss}
Let $A\in M_{m\times n}(\Z)$ be an $m \times n$ integer matrix such that $m\le n$, and let 
$D = P A Q$ be its Smith normal form. In particular, all the nonzero entries in the diagonal of 
$$D=\mathrm{diag}(\pm 1,\ldots,\pm 1,d_{s+1},d_{s+2},...,d_r,0,...0)$$
 are given first. Moreover, $d_i$ divides $d_{i+1}$, $i=1,\ldots,r-1$, and there are $s\ge0$ ones.
If $s=m$ then $\coker A=\{0\}$. So assume $s< m$.
Define  $\pi:\Z^m \to \Z^{m-s}$ to be the projection onto the last $m-s$ coordinates i.e.~$\pi$ is the last $m-s$ rows of the identity square matrix $I_m$.
Let $q:\Z^{m-s}\to \Z^m$ be defined by
$$q:=P^{-1}\circ{\pi}^{t},$$
where $\pi^t$ is the transpose of $\pi$, 
and let $p:\Z^m\to \Z^{m-s}$ be defined by
$$p:=\pi\circ P.$$
Then $\coker A$ and $\frac{\Z}{d_{s+1}}\oplus\cdots\oplus\frac{\Z}{d_{m}}$ are $\Z$-isomorphic, with isomorphisms 
$\bar q:\frac{\Z}{d_{s+1}}\oplus\cdots\oplus\frac{\Z}{d_{m}}\to \coker A$ and 
$\bar p:\coker A\to \frac{\Z}{d_{s+1}}\oplus\cdots\oplus\frac{\Z}{d_{m}}$ induced by $q$, $p$, respectively.
In particular, (for computer testing)
$$pq=I_{m-s},\qquad qpA =0.$$
\end{pro}

\begin{proof}
Consider the following diagram whose rows are short exact sequences of abelian groups.
$$
\xymatrix{
0\ar[r] & \im A\ar@{^(->}[r] 							 & \Z^m\ar@{->>}[r]\ar[d]^{P}_{\cong} & \frac{\Z^m}{\im A}\ar[r]\ar[d]_{\bar P}^{\cong} & 0\\
0\ar[r] & \im D\ar@{^(->}[r] 							 & \Z^m\ar@{->>}[r]^{g}\ar[d]^{\pi} & \frac{\Z^m}{\im D}\ar[r]\ar[d]^{\bar \pi} & 0\\
0\ar[r] & \Z d_{s+1}\oplus\cdots\oplus\Z d_m\ar@{^(->}[r] & \Z^{m-s}\ar@{->>}[r]^{ h\qquad\quad} & \frac{\Z}{\Z d_{s+1}}\oplus\cdots\oplus\frac{\Z}{\Z d_m}\ar[r] & 0
}$$
Since $h$ is a quotient map and $\pi$ is a  projection, they are both surjective and hence $h\pi$ is surjective.
Since
$$\ker(h\pi)=\pi^{-1}(\ker h)=\pi^{-1}( \Z d_{s+1}\oplus\cdots\oplus\Z d_m)=\Z^s\oplus \Z d_{s+1}\oplus\cdots\oplus\Z d_m=\im D$$
we get by the first isomorphism theorem of groups that
$$\frac{\Z^m}{\im D}=\frac{\Z^m}{\ker h\pi}\cong\frac{\Z}{\Z d_{s+1}}\oplus\cdots\oplus\frac{\Z}{\Z d_m},$$
where the isomorphism is induced by $\pi$ and it is denoted by $\bar \pi$.
Now since $\pi\pi^t = I_{m-s}$ and since
$$
\ker(g\pi^t) = (\pi^t)^{-1}(\ker g) = (\pi^t)^{-1}(\im D) = (\pi^t)^{-1}(\ker h\pi) = \ker(h\pi\pi^t) = \ker h
$$
we get that $\pi^t$ induces an isomorphism $\overline{\pi^t}: \frac{\Z}{\Z d_{s+1}}\oplus\cdots\oplus\frac{\Z}{\Z d_m} \xrightarrow{\sim} \frac{\Z^m}{\im D}$ and that $\overline{\pi^t} = \bar\pi^{-1}$.

Similarly, since $gP$ is surjective and
$$\ker gP=P^{-1}(\ker g)=P^{-1}(\im D)= P^{-1}D \Z^n = AQ\Z^n=A\Z^n=\im A,$$
we get that $\bar P$ exists and that it is an isomorphism with inverse $\overline{P^{-1}}$.
Thus $q:=P^{-1}\pi^t$ and $p:=\pi P$ induce the isomorphisms $\bar q$, $\bar p$ :
$$\bar q=\overline{P^{-1}}\circ\overline{\pi^t}=\overline{P^{-1}\pi^t},$$
$$\bar p=\bar \pi\circ\bar P=\overline{\pi P}.$$
\end{proof}

\begin{pro}[$\frac{\ker\delta^1}{\im\delta^0}\cong \coker\pi Q^{-1}\delta^0$]\label{p:kerOverIm-A-isom-Znminus}
Let 
\begin{equation*}
\xymatrix{0\ar[r]&\Z^{V}\ar[r]^{\delta^0}&\Z^{E}\ar[r]^{{\delta^1}}&\Z^{F}\ar[r]&0}
\end{equation*}
be a cochain complex (i.e.~$\delta^1\circ\delta^0=0$) for some $V,E,F\in\N$.
Let 
$D = P \delta^1 Q$ be the Smith normal form of $\delta^1$. 
In particular, all the nonzero entries in the diagonal of 
$$D=\mathrm{diag}(d_1,...,d_r,0,...0)$$
 are given first. Moreover, $d_i$ divides $d_{i+1}$, $i=1,\ldots,r-1$.
Define  $\pi:\Z^E \to \Z^{E-r}$ to be the projection onto the last $E-r$ coordinates, i.e.~$\pi$ is the last $E-r$ rows of the identity square matrix $I_E$.
Let $q:\Z^{E-r}\to \Z^E$ be defined by
$$q:=Q\circ\pi^{t},$$
where $\pi^t$ is the transpose of $\pi$, and let $p:\Z^E\to \Z^{E-r}$ be defined by
$$p:=\pi\circ Q^{-1}.$$
Then $\frac{\ker\delta^1}{\im\delta^0}$ and $\coker\pi Q^{-1}\delta^0$ are isomorphic, with isomorphisms 
$ q':\coker\pi Q^{-1}\delta^0\to \frac{\ker\delta^1}{\im\delta^0}$,\,\,
$ p':\frac{\ker\delta^1}{\im\delta^0}\to \coker\pi Q^{-1}\delta^0$
induced  by $q$, $p$, respectively. In particular, (for computer testing)
$$pq=I_{E-r},\qquad \delta^1qp =0.$$
\end{pro}

\begin{proof}
From the commutative diagram of cochain complexes (not a short exact sequence)
\begin{equation*}
\xymatrix{
0\ar[r]&\Z^{V}\ar[d]^{id}_{\cong}\ar[r]^{\delta^0}&\Z^{E}\ar[d]^{Q^{-1}}_{\cong}\ar[r]^{\delta^1}&\ar[d]^{P}_{\cong}\Z^{F}\ar[r]&0\\
0\ar[r]&\Z^{V}\ar[r]^{Q^{-1}\delta^0}\ar[d]^{id}_{\cong}&\Z^{E}\ar[r]^{D}\ar@<2pt>[d]^{\pi}&\Z^{F}\ar[r]\ar[d]^{0}&0\\
0\ar[r]&\Z^{V}\ar[r]^{\pi Q^{-1}\delta^0}&\Z^{E-r}\ar@<2pt>[u]^{\pi^t}\ar[r]&0\ar[r]&0
}
\end{equation*}
we get the isomorphisms 
$$\xymatrix{H^1(Q^{-1}):\frac{\ker\delta^1}{\im\delta^0}\ar[r]^{\quad\sim}& \frac{\ker D}{\im Q^{-1}\delta^0}},$$
$$\xymatrix{H^1(\pi):\frac{\ker D}{\im Q^{-1}\delta^0}\ar[r]^{\sim\qquad} & \frac{\Z^{E-r}}{\pi Q^{-1}\delta^0 \Z^V}\,=\,\coker\pi Q^{-1}\delta^0},$$
where the inverse of the last isomorphism is described next.
Note that $\im Q^{-1}\delta^0\subset \ker D=0^r\oplus\Z^{E-r}$ because
$DQ^{-1}\delta^0\Z^V=P\delta^1\delta^0\Z^V=0$.
Hence
$$\pi^t\pi Q^{-1}\delta^0=Q^{-1}\delta^0,$$
because $\pi^t\pi=0^r\oplus I_{E-r}$. So $\pi^t$ induces the map $H^1(\pi^t)$ and, since $\pi\pi^t=I_{E-r}$, we get that $H^{1}(\pi)^{-1}=H^1(\pi^t)$.
Thus $q:=Q\pi^t$ and $p:=\pi Q^{-1}$ induce the isomorphisms $ q'$, $ p'$ :
$$ q'=H^1(Q)\circ H^1(\pi^t)=H^1(q)$$
$$p'=H^1(\pi)\circ H^1(Q^{-1})=H^1(p).$$
\end{proof}

\begin{pro}
Let $A\in M_n(\Z)$ be an $n\times n$ integer matrix. Suppose that 0 is an eigenvalue of $A$ and let 
$p(x)=x^r q(x)$ be the characteristic polynomial of $A$. 
Then 
$$\lim_{\to} (A^r, \Z^n)=\lim_{\to} (A^r,\, \ker q(A)).$$
\end{pro}
\begin{proof}
Since $q(A)A^r=A^rq(A)=p(A)=0$, we get $A^r\Z^n\subset \ker q(A)$.
Moreover, $A \ker q(A)\subset \ker q(A)$ because if $x\in \ker q(A)$ then $q(A)x=0$ and thus $q(A)Ax=A(q(A)x)=0$. Hence the following diagram is well-defined
$$\xymatrix{
\Z^n\ar[r]^{A^r}\ar[d]_{A^r}&\Z^n\ar[r]^{A^r}\ar[d]_{A^r}&\Z^n\ar[r]^{A^r}\ar[d]_{A^r}&\\
\ker q(A)\ar[r]^{A^r}\ar@{^(->}[ru]^{\iota}&\ker q(A)\ar[r]^{A^r}\ar@{^(->}[ru]^{\iota}&\ker q(A)\ar[r]^{A^r}\ar@{^(->}[ru]^{\iota}&
}$$
and since it commutes $\lim\limits_{\to} (A^r, \Z^n)=\lim\limits_{\to} (A^r,\, \ker q(A)).$
\end{proof}

\begin{pro}[extract the 1-eigenvalues]\label{p:extract-one-eigenvalues}
Let $A\in M_n(\Z)$ be an $n\times n$ integer matrix and let $p(x)$ be its minimal polynomial. Suppose that $\lambda\in\Z$ is an eigenvalue of $A$ and let $q(x)$ be defined from the equation
$p(x)=(x-\lambda)q(x)$. Let $q(x)$ have integer coefficients.
Then 
$$\xymatrix{0\ar[r]&\lim\limits_{\to} (A,\, \ker q(A))\ar@{^(->}[r]&\lim\limits_{\to} (A,\, \Z^n)\ar@{>>}[r]&\lim\limits_{\to} (\lambda I_n,\, q(A)\Z^n)\ar[r]&0}$$
is a short exact sequence, where $I_n\in M_n(\Z)$ is the identity matrix.
\end{pro}
\begin{proof}
Recall that an eigenvalue of $A$ is a root of the minimal polynomial $p(x)$ of $A$, and thus by the assumption it makes sense to write $p(x)=(x-\lambda)q(x)$.
By definition of minimal polynomial $0=p(A)=(A-\lambda I)q(A)$, i.e.~$Aq(A)=\lambda q(A)$.
Moreover, $A \ker q(A)\subset \ker q(A)$ because if $x\in \ker q(A)$ then $q(A)x=0$ and thus $q(A)Ax=A(q(A)x)=0$. Hence the following diagram is well-defined
$$\xymatrix{
0\ar[r]& \ker q(A)\ar[d]^{A}\ar[d]^{A}\ar@{^(->}[r]&\Z^n\ar[d]^{A}\ar@{>>}[r]^{q(A)\,\,\,\,\,\,\,\,}&q(A)\Z^n\ar[d]^{\lambda I}\ar[r]&0\\
0\ar[r]& \ker q(A)\ar@{^(->}[r]&\Z^n\ar@{>>}[r]^{q(A)\,\,\,\,\,\,\,\,}&q(A)\Z^n\ar[r]&0.
}$$
Moreover, the rows are clearly short exact sequences and the diagram is commutative.
Recall that the direct limit of a directed system of short exact sequences of abelian groups is a short exact sequence because the direct limit is an exact functor in the category of abelian groups. Hence the statement of the proposition holds.
\end{proof}

\begin{pro}\label{p:tensorproduct}
Let $m,n\in\N$. Then
$$\Z[\frac1m]\otimes_{\Z} \Z[\frac1n]=\Z[\frac1{mn}].$$
\end{pro}
\begin{proof}
We use the formula ``localization of a module'' (see \cite[Lemma~2.4,~p.66]{Eisenbud}): 
$$S^{-1}M=M\otimes_{\Z} S^{-1}\Z$$
with $S=\{n^i\mid i\in \N_0\}$ and $M=\Z[\frac1m]$.
Notice that $1\in S$ and $n^i n^j\in S$.
Recall that
$$S^{-1}M =\{[(\tilde m, s)]_{\sim} \mid \tilde m\in M, s\in S\}$$
where
$$(\tilde m, s)\sim (\tilde n , t)\iff \exists u\in S: u(s\tilde n - t\tilde m)=0$$
$$\frac{\tilde m}{s}:=[(\tilde m, s)]_{\sim}$$
and $S^{-1}M$ is a module with operations
$$\frac{\tilde m}{s}+\frac{\tilde n}t:=\frac{t\tilde m + s\tilde n}{st},\qquad a\cdot \frac ms:=\frac{am}s.$$
Thus,
\begin{eqnarray*}
S^{-1}M&=&\{\frac{\tilde m}s\mid \tilde m\in \Z[\frac1m],s\in S\}\\
&=&\{\frac{\frac{a}{m^i}}{n^j}\mid a\in \Z, i,j\in\N_0\}\\
&=&\{\frac{a}{m^i n^j}\mid a\in \Z, i,j\in\N_0\}\\
&=&\{\frac{a}{(m n)^i}\mid a\in \Z, i\in\N_0\}\\
&=&\Z[\frac1{mn}]
\end{eqnarray*}
\end{proof}

\begin{cor}\label{c:tensorproduct}
Let $m,n,a,b,\in\N$. Then
$$\Z[\frac1m]^a\otimes_{\Z}\Z[\frac1n]^b=\Z[\frac1{mn}]^{ab}.$$
\end{cor}

\begin{cor}\label{c:tensorproduct2}
Let $k,\ell,m,n,a,b,c,d\in\N$. Then 
$$
\left(\Z[\frac 1k]^a\oplus \Z[\frac1{\ell}]^b\right)\otimes
\left(\Z[\frac 1m]^c\oplus \Z[\frac1{n}]^d\right) =
\Z[\frac1{km}]^{ac}\oplus \Z[\frac1{kn}]^{ad}\oplus
\Z[\frac1{\ell m}]^{bc}\oplus \Z[\frac1{\ell n}]^{bd}
$$

\end{cor}

\subsection{Short exact sequences}

\begin{lem}
Let 
$$\xymatrix{0\ar[r]&I\ar@{^{(}->}[r]^{\iota} &A\ar@{>>}[r]^{r}&B\ar[r]&0}$$
be a short exact sequence of abelian groups. If $Ext^1(B,\Z)=0$ then the transpose
$$\xymatrix{0\ar[r]&B^t\ar@{^{(}->}[r]^{r^t} &A^t\ar@{>>}[r]^{\iota^t}&I^t\ar[r]&0}$$
is also a short exact sequence.
\end{lem}
\begin{proof}
From the given short exact sequence we get the long exact sequence
$$\xymatrix@C=1.3cm{0\ar[r]&\Hom(B,\Z)\ar@{^{(}->}[r]^{\Hom(r,\Z)} &\Hom(A,\Z)\ar[r]^{\Hom(\iota,\Z)}&\Hom(I,\Z)\ar[r]&}$$
$$\xymatrix@C=1.3cm{&\Ext^1(B,\Z)\ar[r]^{\Ext^1(r,\Z)} &\Ext^1(A,\Z)\ar[r]^{\Ext^1(\iota,\Z)}&\Ext^1(I,\Z)\ar[r]&}$$
$$\xymatrix@C=1.3cm{&\Ext^2(B,\Z)\ar[r]^{\Ext^2(r,\Z)} &\Ext^2(A,\Z)\ar[r]^{\Ext^2(\iota,\Z)}&\Ext^2(I,\Z)\ar[r]&\cdots}.$$
Hence if $\Ext^1(B,\Z)=0$ we get a short exact sequence.
\end{proof}
\begin{cor}\label{c:ses-Z-transponse}
Let 
$$\xymatrix{0\ar[r]&I\ar@{^{(}->}[r]^{\iota} &A\ar@{>>}[r]^{r}&\Z^n\ar[r]&0}$$
be a short exact sequence of abelian groups. Then the transpose
$$\xymatrix{0\ar[r]&\Z^n\ar@{^{(}->}[r]^{r^t} &A^t\ar@{>>}[r]^{\iota^t}&I^t\ar[r]&0}$$
is also a short exact sequence.
\end{cor}
\begin{proof}
We start by constructing a projective resolution
$$\xymatrix{
\ar[r]&0\ar[r]&0\ar[r] &\Z^n_{_{_{\!\!\!\!\!\!0}}}\ar[r]^{id}&\Z^n_{_{_{\!\!\!\!\!\!-1}}}\ar[r]&0.}$$
Applying $\Hom(-,\Z)$ to this projective resolution of abelian groups (without the -1 term) we get
$$\xymatrix{
&0\ar[l]&0\ar[l] &0_{_{_{\!\!\!\!1}}}\ar[l]^{}& \Hom_{_{_{\!\!\!\!\!\!\!\!\!\!0\,\,\,\,\,\,\,\,\,}}}(\Z^n,\Z)\ar[l]&0\ar[l]}.$$
Taking homology of this chain complex we get $\Ext^0(\Z^n,\Z)=\Hom(\Z^n,\Z)$ and $\Ext^i(\Z^n,\Z)=0$ for $i\in\N$. In particular  $\Ext^1(\Z^n,\Z)=0$.
Thus by the lemma, the corollary follows.
\end{proof}

\begin{lem}\label{l:es-ses}
Let 
$$\xymatrix{\ar[r]&A\ar[r]^{f}&B\ar[r]^{g} &C\ar@{>>}[r]^{h}&D\ar[r]&0}$$
be an exact sequence of abelian groups. Then
$$\xymatrix{0\ar[r]&\frac{B}{\im f}\ar@{^{(}->}[r]^{\iota} &C\ar@{>>}[r]^{h}&D\ar[r]&0}$$
is a short exact sequence. In particular,
$$\frac{C}{\iota(\frac{B}{\im f })}\cong D.$$
\end{lem}
\begin{proof}
Since the sequence is exact,
$\im f=\ker g$, and $\im g=\ker h$. Thus
$$\frac{B}{\im f}=\frac{B}{\ker g}\cong \im g\subset C.$$ 
The map induced by $g$ from $\frac{B}{\im f}$ to $C$ is denoted by $\iota$, and it is injective.
Since $\im g=\ker h$, the short sequence is exact. In particular, since $h$ is surjective
$$\frac{C}{\iota(\frac{B}{\im f})}=\frac{C}{\ker h}\cong D.$$
\end{proof}

\section{\textbf{Direct limits of topological spaces and $C^*$-algebras}}\label{a.s:direct-limits-TopCstar}

\begin{pro}\label{p:ind:Xn0-eventually}
Let 
$$X_0\subset X_1\subset X_2\subset \ldots$$
be an increasing sequence of $T_1$-spaces, and equip
$$X:=\bigcup_{n\in\N_0} X_n$$
with the inductive limit topology.
Suppose that $(x_k)_{k=0}^\infty$ is a sequence in $X$ that converges to $x\in X$.
Then there exists an $n_0\in \N_0$ such that $x\in X_{n_0}$ and such that the sequence is eventually in $X_{n_0}$. That is, there exists a $k_0\in \N_0$ such that 
$$x,x_k\in X_{n_0}\qquad \text{ for } k\ge k_0.$$
\end{pro}
\begin{proof}
Define $X_{-1}:=\emptyset$, and assume that $x_k\to x\in X$.
For contradiction, assume that the conclusion of the proposition is false.
Then, without loss of generality, we can assume that there exists a subsequence $(x_{k_\ell})_{\ell=0}^\infty$ such that  
$$x\in X_0\quad \text{and}\quad x_{k_\ell}\in X_\ell\backslash X_{\ell-1},\quad\text{for } \ell \in\N_0,$$
and such that $x_{k_0}\ne x$. (We simply dropped all the spaces not containing an element of the subsequence). 
Then 
\begin{eqnarray*}
&&x_{k_0}\in X_0\\
&&x_{k_1}\in X_1,\quad x_{k_1}\not\in X_0\\
&&x_{k_2}\in X_2,\quad x_{k_2}\not\in X_0\cup X_1,\qquad \text{etc.}
\end{eqnarray*}
Choose a neighborhood $x\in U_0\subset X_0$ open in $X_0$ such that $x_{k_0}\not\in U_0$.
Such neighborhood exists because $X_0$ is a $T_1$-space.
Notice also that $x_{k_\ell}\not\in U_0$ for $\ell\ge 1$ because $x_{k_\ell}\not\in X_0$.
Now, choose $V_1\subset X_1$ open in $X_1$ such that 
$$U_0=V_1\cap X_0.$$
This is possible since $X_0$ has the subspace topology of $X_1$.
Let 
$$U_1:=V_1\backslash \{x_{k_1}\}=V_1\cap \{x_{k_1}\}^c,$$
which is open in $X_1$ since all one point sets in $T_1$-spaces are closed.
Then 
$$U_1\cap X_0=V_1\cap X_0=U_0$$
 since $x_{k_1}\not\in X_0$.
Moreover, $x_{k_\ell}\not\in U_1$ for $\ell\in \N_0$ because $U_1\subset X_1$, because we have removed $x_{k_1}$, and because $x_{k_0}\not\in U_0=U_1\cap X_0$, hence also $x_{k_0}\not\in U_1$ (as $x_{k_0}\in X_0$).
Inductively, we can construct the sets $U_m\subset X_m$, $m\in\N_0$, where
$$\text{$U_m$ is open in $X_m$}$$
and,
\begin{equation}\label{e:Um}
U_{m'}=U_m\cap X_{m'}\quad \forall\,\, m'\le m,
\end{equation}
and such that 
$$x_{k_\ell}\not\in U_m\quad \forall\,\, \ell\in\N_0.$$
Let $U:=\bigcup_{m=0}^\infty U_m$. Then $U$ is open in $X$ because by Eq.~(\ref{e:Um}), $U\cap X_n=U_n$ for all $n\in\N_0$, and $U_n$ is open in $X_n$ for all $n\in\N_0$.
Since 
$$x\in U_0\subset U_1\subset U_2\subset\cdots,$$
 $U$ is a neighborhood of $x$ in $X$. But $x_{k_\ell}\not\in U$ for all $\ell \in \N_0$ because $x_{k_\ell}\not \in U_m$ for all $\ell,m\in \N_0$.  Hence $x_{k_\ell}\not\to x$, a contradiction. Thus the conclusion of the proposition is true.
\end{proof}

\begin{lem}\label{l:Rd-ind}
Let $B_n(0)$ be the open ball in $\R^d$ of radius $n\in \N$, and equip it with the subspace topology of $\R^d$.
Let 
$$X:=\bigcup_{n\in\N} B_n(0)$$
be given the inductive limit topology. Then $X=\R^d$ as topological spaces.
\end{lem}
\begin{proof}
The inclusion map $i:X\hookrightarrow \R^d$, is continuous because the inclusion  $i_n:B_n(0)\hookrightarrow \R^d$  is continuous for all $n\in\N$.
We now show that $i^{-1}$ is continuous.
Let $U\subset X$ be open in $X$. That is, $U\cap B_n(0)$ is open in $\R^d$ for all $n\in\N_0$.
Hence $U=\bigcup_{n\in\N} U\cap B_n(0)$ is open in $\R^d$.
\end{proof}

The following lemma is straightforward and well-known.
\begin{lem}\label{l:seqInSubspace}
Let $Y$ be a subspace of a topological space $X$, and suppose that $x_n\to x$ in $X$ as $n\to \infty$.
If $x_n,x\in Y$ for all $n\in\N$, then $x_n\to x$ in $Y$.
\end{lem}


\begin{pro}\label{ap:Sisdirlim}
Let $R_0\subset R_1\subset R_2\subset \cdots $, be the increasing sequence of \'etale equivalence relations  defined in Section \ref{s:S}, and recall that $R_n\subset R_{n+1}$ is open for $n\in\N_0$. Recall as well that 
 $A_n:=C^*(R_n)$, $n\in\N_0$ is amenable, (thus $C^*(R_n)=C_r^*(R_n)$).   Then
$$C^*(\bigcup_{n=0}^{\infty} R_n)=\lim_{\to} \xymatrix{A_0\ar[r]^{\iota_0}&A_1\ar[r]^{\iota_1}&\ldots&}$$
\end{pro}
\begin{proof}
Consider the following diagram
$$
\xymatrix{
C^*( R_0)\ar@{^{(}->}[r]^{\iota_0}\ar[dr]_{\lambda_0}&C^*( R_1)\ar@{^{(}->}[r]^{\iota_1}\ar[d]^{\lambda_1}&C^*( R_2)\ar@{^{(}->}[r]^{\qquad\iota_3}\ar[dl]^{\lambda_2}&\\
&C^*\Big(\bigcup\limits_{n=0}^{\infty} R_n\Big).
}
$$
where $\lambda_n$ is the inclusion map.
For $f\in C_c(\cup_{n=0}^{\infty}R_n)$, since $\{R_n\}_{n=0}^\infty$ is an open cover, and the support of $f$ is compact, a finite subset covers the support of $f$.
Moreover, since the union is increasing, the support of $f$ is in $R_N$ for some $N\in\N_0$.
Hence $f=\lambda_N(\tilde f)$ for some $\tilde f\in C_c(R_N)$.
By slight abuse of notation, we will write $f\in C_c(R_N)$ instead of $\tilde f\in C_c(R_N)$.

By a similar proof to Remark \ref{r:iota-cont}, $\lambda_n:C^*(R_n)\rightarrow C^*(R'_s)$ is continuous.
By an adaptation of the proof of \cite[Prop.~IV.9]{Gon0}, $\lambda_n$ is injective, which uses $C^*(R_n)=C_r^*(R_n)$ (amenability).
Hence the $C^*$-algebra $C^*(R_n)$, $n\in\N$ can be viewed as a subalgebra of $C^*(R'_s)$.
Then by \cite[Exc.~6.2,~p.104]{Rordam}, the direct limit
$$\dlim(C^*(R_n),\iota_n)=\overline{\bigcup_{n=1}^\infty C^*(R_n)}^{||\cdot||_{C^*(R'_s)}}\subset C^*(R'_s)$$
is a subalgebra.
We now show the inclusion $\supset$.
Let $a\in C^*(R'_s)$. Then there is a sequence $(a_k)_{k=1}^\infty \subset C_c(\bigcup_{n=0}^{\infty}R_n)$ such that 
$$\xymatrix@C=1.5cm{a_k\,\,\ar[r]^{||\cdot||_{C^*(R'_s)}}&\,\,\,a}.$$
By compactness of the support of $a_k$, there exists a $n_k\in \N$ such that $a_k\in C_c(R_{n_k})$.
Hence $(a_k)_{k=1}^\infty\subset \bigcup_{n=1}^\infty C^*(R_{n})$ and thus
$a\in \overline{\bigcup_{n=1}^\infty C^*(R_n)}^{||\cdot||_{C^*(R'_s)}}$.
\end{proof}


\section{\textbf{Compactness in $\R^d$}}
\begin{lem}\label{l:compactK-boundaryU-dist}
Let $U$ be an open subset of $\R^d$, and let its $\R^d$-boundary be non-empty.
Let $K\subset U$ be compact in $U$.
Then $$d(K,\partial U)>0.$$
\end{lem}
\begin{proof}
Note $K\cap \partial U$ is empty because $K$ is a subset of $U$ and $U\cap\partial U=U\cap (\overline{U}\backslash\open{U})=\emptyset$ as $U$ is open.
Suppose for contradiction that $d(K,\partial U)=0$.
Then there is a sequence $(x_n)_{n=1}^\infty\subset K$ such that 
$$d(x_n, \partial U)\to 0\quad \text{as}\quad n\to \infty.$$
Since $K$ is compact, there exists a convergent subsequence  $x_{n_j}\to x\in K$.
Since 
$$d(x_{n_j}, \partial U)=\inf_{u\in \partial U} ||x_{n_j}- u||,$$
we have
$$d(x,\partial U)\le ||x-x_{n_j}||+d(x_{n_j},\partial U)\to 0.$$
Hence $d(x,\partial U)=0$.
Thus there is a sequence $(u_n)_{n=1}^\infty\subset \partial U$ converging to $x$.
Since $\partial U=\overline{U}\cap \overline{\R^d\backslash U}$ is closed, $x\in \partial U$. Thus $x\in K\cap \partial U=\emptyset$, a contradiction.
\end{proof}

\begin{lem}\label{l:X-compact}
Let $X$ be a subset of $\R^d$, and let $K\subset X$.
Then $K$ is $X$-compact if and only if $K$ is $\R^d$-compact.
\end{lem}
\begin{proof}
If $K$ is $X$-compact, then it is $\R^d$-compact because the inclusion map $X\hookrightarrow \R^d$  is continuous. 
If $K$ is $\R^d$-compact then it is $\R^d$-closed and bounded. Since it is $\R^d$-closed, it is complete. 
Since it is bounded, then by the Archimedean property it is totally bounded. Thus $K$ is $X$-compact.
\end{proof}

\section{\textbf{$C_c(G)$}}\label{s:CcX}

Let $G$ be a locally compact groupoid with left Haar system $\{\lambda^u\}_{u\in G^0}$.
Then the set $C_c(G)$ of continuous  $\C$-valued functions with compact support is equipped with the inductive limit topology (cf.~\cite[Example~5.10,~p.118]{Conway}), 
and with operations 
\begin{eqnarray*}
f\star g(x)&:=&\int_{y\in r^{-1}(s(x))} f(xy)g(y^{-1})d\lambda^{s(x)}(y)\\
f^*(x)&:=&\overline{f(x^{-1})}.
\end{eqnarray*}
Under these operations $C_c(G)$ is a $*$-algebra (cf.~\cite[p.~48]{Renault}).

We remark that the names Renault-representation and Renault-bounded (given below) are not standard terminology, but we use them for the sake of brevity.

\begin{defn}[Renault-representation](\cite[p.50]{Renault})
A Renault-representation of $C_c(G)$ on a Hilbert space $H$ is a $*$-homomorphism 
$$L:(C_c(G),\mathrm{ind})\to (B(H),\mathrm{WO})$$
 which is continuous, 
where $C_c(G)$ has the inductive limit topology, and $B(H)$ has the WO-topology (weak-operator topology). Moreover,
the linear span $\{L(f)\xi\mid f\in C_c(G),\xi\in H\}$ is dense in $H$.
\end{defn}

By \cite[Prop.~I.4,~p.50]{Renault},  the topology given by the one-norm $||\cdot||_I$ on $C_c(G)$ is coarser than the inductive limit topology, i.e.
$$\xymatrix{f_n\ar[r]^{ind} &f}\quad\imply\quad ||f_n-f||_I\to 0,$$
where 
$$||f||_I=\max\Big(\sup_{u\in G^0} \int_{x\in r^{-1}(u)} |f(x)|\,d\lambda^u(x)\,,\,\,\,\,\,\,\,\sup_{u\in G^0} \int_{x^{-1}\in r^{-1}(u)} |f(x)|\,d\lambda^u(x^{-1})\Big),$$
for $f\in C_c(G)$. 
Moreover, $||\cdot||_I$ is a $*$-algebra norm on $C_c(G)$.

\begin{defn}[Renault-bounded](cf.~\cite[p.~50-51]{Renault})
A Renault-representation 
$$L:(C_c(G),\mathrm{ind})\to (B(H),\mathrm{WO})$$
 is said to be Renault-bounded if
$$||L(f)||_{B(H)}\le ||f||_I\quad \forall f\in C_c(G),$$
where $||\cdot||_{B(H)}$ is the operator norm of $B(H)$.\\
(Thus $L:(C_c(G),||\cdot||_I)\to (B(H), ||\cdot||_{B(H)})$ is continuous).
\end{defn}

For $f\in C_c(G)$ define
$$||f||:=\sup_{L}||L(f)||_{B(H)},$$
where the supremum is taken over all Renault-bounded Renault-representations $(L,H)$ of $C_c(G)$.
This norm on $C_c(G)$ is a $C^*$-norm, and the $C^*$-algebra $C^*(G)$ of the groupoid $G$ is defined as the completion 
$$C^*(G):=\overline{C_c(G)}^{||\cdot||}.$$
(cf.~\cite[p.51-58]{Renault}). Observe that for $f\in C_c(G)$, 
\begin{equation}\label{e:normLEnormI}
 ||f||\le ||f||_I.
 \end{equation}

\begin{lem} \label{l:RenaultboundedImpliesLbounded}
If 
$$L:(C_c(G),\mathrm{ind})\to (B(H),\mathrm{WO})$$
 is a Renault-bounded Renault-representation, then 
 $$L:(C_c(G),||\cdot||)\to(B(H),||\cdot||_{B(H)})$$ is continuous.
\end{lem}
\begin{proof}
By definition of $||f||$, $f\in C_c(G)$ we have $||L(f)||_{B(H)}\le ||f||$. Hence the linear operator $L(f)$ is bounded (with norm at most 1) and thus continuous.
\end{proof}

\begin{lem}\label{l:LfaithfulreprISrenaultrepr}
If $L:C^*(G)\to B(H)$ is a faithful representation (in the standard way) then $L$ restricted to $C_c(G)$ is a Renault-bounded Renault-representation.
\end{lem}
\begin{proof}
Suppose that a sequence $(f_n)_{n=1}^\infty \subset C_c(G)$ converges to $f\in C_c(G)$ in the inductive limit topology.
Since the I-norm is coarser than the inductive limit topology we have
$$||f_n-f||_I\to 0.$$
Hence by Eq.~(\ref{e:normLEnormI})
$$||f_n-f||\to 0.$$
Since $L$ is faithful, we get by \cite[Theorem~10.7]{Zhu} that 
$$||L(f_n)-L(f)||_{B(H)}=||f_n-f||\to 0.$$
Since convergence in the norm-topology implies convergence in the weak-operator topology we get
$$\xymatrix{L(f_n)\ar[r]^{\mathrm{WO}}& L(f).}$$
Thus $L$ is a Renault-representation.
It is Renault-bounded since 
$$||L(f)||_{B(H)}=||f||\le ||f||_{I},$$
by \cite[Theorem~10.7]{Zhu} and Eq.~(\ref{e:normLEnormI}).

\end{proof}

For convenience to the reader, we provide a proof of the following well-known result:
\begin{thm}[extension by continuity for groupoids]\label{t:tilde-psi-general}
Let $G'$ and $G$ be locally compact groupoids with Haar systems, and assume that $G'$ is second countable and \'etale.
Let $\psi:C_c(G')\to C_c(G)$ be a continuous $*$-homomorphism. 
Then $\psi$ extends to a (continuous) $*$-homomorphism
 $$\tilde\psi:C^*(G')\to C^*(G).$$
\end{thm}
\begin{proof}
Let $(L,H)$ be a faithful representation of $C^*(G)$. 
By Lemma \ref{l:LfaithfulreprISrenaultrepr}, $L$ is a Renault-representation Renault-bounded when restricted to $C_c(G)$. 
Since $\psi$ is continuous in the inductive limit topology and it is a $*$-homomorphism,
$L\circ \psi$ is a Renault-representation.
Then by \cite[Corollary~1.22,~p.~72]{Renault}, $L\circ \psi$ is Renault-bounded.
By Lemma \ref{l:RenaultboundedImpliesLbounded}, 
$$L\circ \psi:(C_c(G'),||\cdot||_{C^*(G')})\to (B(H),||\cdot||_{B(H)})$$ 
is bounded.
Hence we can extend $L\circ \psi$ by continuity to $C^*(G')$ and we denote the extended map by $\widetilde{L\circ\psi}$.
Thus 
$$\tilde\psi= L^{-1}\circ\widetilde{L\circ\psi}$$
is a (continuous) $*$-homomorphism (as $L$ is a $*$-isomorphism onto its image in $B(H)$).
\end{proof}
%

\section{\textbf{On $T+\R^d$ with the subspace topology of $\Omega$}}
\begin{pro}\label{p:TplusxToxNotContinuous}
Let $T+\R^d\subset \Omega$ be given the subspace topology.
Then the map $T+\R^d\to \R^d$ given by 
$$T+x\mapsto x$$
is not continuous.
\end{pro}
\begin{proof}
Let $P_n:=T(B_n(0))$, $n\in \N$, be the patch containing the ball $B_n(0)$ of radius $n$, and note that $P_n\subset P_{n+1}$.
Choose patch $P_n+x_n$ in $T$ such that $||x_n||\ge n$. This can be done by \cite[Thm~2.3,~p.3]{KelPut} and \cite[Cor.~3.5,~p.11]{AP}. 
Since $T$ and $T-x_n$ agree on the patch $P_n$ and thus on $B_n(0)$, we have
$$d(T,T-x_n)\le\frac 1n.$$
Thus $T-x_n$ converges to $T$, but $x_n \not\to 0$ because 
$$||x_n||\to \infty.$$
\end{proof}

\section{\textbf{Projections in $C_r^*(R|_{X_0})$}}
We use the notation of Section \ref{s:S}.
For simplicity of notation, let $T:=T_n$, $R:=R_n$ for fixed $n\in\N_0$.

Let $f\in C_c(R|_{X_0})$ be defined by
$$f(x,y)=\left\{
           \begin{array}{cl}
             1 &\text{\quad  if $x=y=v_i^0$}\\
             0 &\text{\quad  otherwise}\\
           \end{array}
         \right.
         $$
         where $x\sim y$.
         Since $f^*(x,y)=\overline{f(y,x)}=f(y,x)=f(x,y)$, $f$ is self-adjoint.
For $x\sim z$ we have
$$(f\star f)(x,z)=\sum\limits_{x\sim y} f(x,y) f(y,z).$$
$$f(x,y)=\left\{
           \begin{array}{cl}
             1 &\text{\quad  if $x=y=v_i^0$}\\
             0 &\text{\quad  otherwise}\\
           \end{array}
         \right.
         $$
         $$f(y,z)=\left\{
           \begin{array}{cl}
             1 &\text{\quad  if $y=z=v_i^0$}\\
             0 &\text{\quad  otherwise}\\
           \end{array}
         \right.
         $$
Thus $(f\star f)(x,z)\ne0$ only if $x=z=v_i^0$.
In this case,
\begin{eqnarray*}
  (f\star f)(v_i^0,v_i^0)&=&\sum_{y\sim v_i^0} f(v_i^0,y)f(y,v_i^0)\\
&=&f(v_i^0,v_i^0)f(v_i^0,v_i^0)\\
&=&1.
\end{eqnarray*}
Hence, $f\star f=f$. Hence $f$ is a projection in $C_r^*(R|_{X_0})$.

Below we give a different proof of \cite[Prop.~3.7-9]{Gon1}.
 \begin{pro}{\cite[Prop.~3.7-9]{Gon1}}\label{p:a:compacts}
  \begin{enumerate}
  \item $C_r^*(R|_{X_0})\cong \bigoplus\limits_{k=1}^{\mathrm{sV}} K(\ell^2([v_k]))$
\item $C_r^*(R|_{X_1-X_0}) \cong \bigoplus\limits_{k=1}^{\mathrm{sE}} C_0(\open{e}_k, K(\ell^2([e_k])))$
\item $C_r^*(R|_{X_2-X_1})\cong \bigoplus\limits_{k=1}^{\mathrm{sF}} C_0(\open{\sigma}_k, K(\ell^2([\sigma_k])))$
\end{enumerate}
where $\mathrm{sV}, \mathrm{sE}, \mathrm{sF}$ are the number of stable vertices, stable edges, stable faces, respectively, and
$v_k, e_k, \sigma_k$ are representatives of the $R$-equivalence classes.
\end{pro}

\begin{proof}
Let $N:=sV$ be the number of stable vertices, i.e.~the number of $R$-equivalence classes.
Then the underlying space of the 0-skeleton of $T$ is the set of vertices
$$X_0=\{v\in \R^d\mid v\in[v_1]\cup\cdots\cup[v_N]\},$$
where the vertices $v_0,\ldots,v_N\in T$ are (fixed) representatives of the $R$-equivalence classes. 
We remark that each of these equivalence classes has countable many vertices.
The restriction of $R$ to the vertices is 
$$R\mid_{X_0}=\{(x,y)\in\R^d\times\R^d\mid (x,y)\in[v_1]\times[v_1]\,\,\cup\cdots\cup\,\,[v_N]\times[v_N]\}.$$
The $C^*$-algebra  $C_r^*(R|_{X_0})$ is the completion inside $B(\ell^2(X_0))$ of 
  $$\{\lambda(f)\mid f\in C_c(R|_{X_0})\},$$
  where 
  $$(\lambda(f)\xi)(x)= \sum_{y\sim x} f(x,y) \xi(y).$$
For $f\in C_c(R\mid_{X_0})$ let  
$$f_i:=f|_{[v_i]\times [v_i]},\qquad i=1,\ldots,N.$$
Then
$$f=f_1+\cdots +f_N,\qquad \supp(f_i)\subset [v_i]\times [v_i]$$
and
\[\lambda(f)=
\begin{blockarray}{cccccc}
[v_1] & [v_2]  & \cdots &[v_{N-1}] & [v_{N}] \\
\begin{block}{(ccccc)c}
 \overline{\underline{| f_1(x,y)|}} & 0 & 0 & 0 & 0 & [v_1] \\
  0 & \overline{\underline{|f_2(x,y)|}} & 0 & 0 & 0 & [v_2] \\
  0 & 0 & \ddots & 0 & 0 & \vdots \\
  0 & 0 & 0 & \overline{\underline{|f_{N-1}(x,y)|}} & 0 & [v_{N-1}] \\
  0 & 0 & 0 & 0 & \overline{\underline{|f_N(x,y)|}} & [v_N] \\
\end{block}
\end{blockarray}
 \]
\noindent
 We have
 $$C_r^*(R|_{X_0})\cong \bigoplus_{i=1}^{\mathrm{sV}} K(\ell^2([v_i]))$$
 because $\lambda:C_c(R|_{X_0})\hookrightarrow B(\ell^2(X_0))$ is an injective $*$-homomorphism 
 and
  $$C_r^*(R|_{X_0})=\overline{\lambda(C_c(R|_{X_0}))}^{||\cdot||_{B(H)}}=\bigoplus_{i=1}^{sV} K(\ell^2([v_i])),$$
 where $H$ is the Hilbert space $\ell^2(X_0)$. This is fairly easy to show.

Now let $N:=sE$ be the number of stable edges, and $J:=J_n$.
Then the underlying space of the 1-skeleton of $T$ is
$$X_1=\{x\in e\mid e\in [e_1]\cup\cdots\cup[e_N]\},$$
where the edges $e_1,\ldots, e_N\in T$ are (fixed) representatives of the $R$-equivalence classes.
The restriction of $R$ to the 1-skeleton minus its vertices is  
$$R\mid_{X_1-X_0}=\{(x,\gamma_{e,e'}(x))\mid x\in \open{e}, (e,e')\in [e_1]\times[e_1]\,\,\cup\ldots\cup\,\,[e_N]\times [e_N]\},$$
where $\gamma_{e,e'}(x):=x+x_0$, with $x_0$ being the translation of the two edges: $e'=e+x_0$. 
For $f\in C_c(R\mid_{X_1-X_0})$ let  
$$f_i:=f|_{R|_{[e_i]\times [e_i]-X_0}},\qquad i=1,\ldots,N,$$
where 
$$R\mid_{[e_i]\times[e_i]-X_0}=\{(x,\gamma_{e,e'}(x))\mid x\in \open{e}, (e,e')\in [e_i]\times[e_i]\}.$$
Then
$$f=f_1+\cdots +f_N,\qquad \supp(f_i)\subset R|_{[e_i]\times [e_i]-X_0}.$$
Since $f$ has compact support and we have removed the vertices, $f$ is zero in a neighborhood of the vertices.
(Recall that the topology of the graph $\gamma_{\open{e},\open{e}\,\!'}\subset R|_{X_1-X_0}$ is the topology of $\open{e}$, which by definition is homeomorphic to $\R$.)
Thus $f_i$ has compact support as well.
The $C^*$-algebra $C_r^*(R|_{X_1-X_0})$ is the completion inside $\bigoplus_{i=1}^N C_0\big(\open{e}_i,B(\ell^2([e_i]))\big)$ of 
  $$\{\lambda(f)\mid f\in C_c(R|_{X_1-X_0})\},$$
  where 
  $$\lambda(f)=(\lambda(f_1),\ldots,\lambda(f_N)),$$
  and  $\lambda(f_i):\open{e}_i\to B(\ell^2([e_i]))$ is defined on $s\in \open{e}_i$ as
  $$\lambda(f_i)(s)(\xi)(e)=\sum_{e\sim e'}f_i\big(\gamma_{e_i,e}(s),\gamma_{e_i,e'}(s)\big)\xi(e'),$$
Since $f_i$ is non-zero only on a finite number of the graphs $\gamma_{\open{e}_i,\open{e}}$, and it is zero on a neighborhood of its vertices, $\lambda(f_i)$ is zero in a smaller neighborhood of its vertices  as well. 
Thus $\lambda(f_i)$ has also compact support.
 We have
 $$C_r^*(R|_{X_1-X_0})\cong \bigoplus_{i=1}^{\mathrm{sE}} C_0\big(\open{e}_i,K(\ell^2([e_i]))\big)$$
 because $\lambda:C_c(R|_{X_1-X_0})\hookrightarrow \bigoplus_{i=1}^N C_c\big(\open{e}_i,B(\ell^2([e_i]))\big)$ is an injective $*$-homomorphism 
and
  $$C_r^*(R|_{X_1-X_0})=\overline{\bigoplus_{i=1}^{N} C_c(R|_{[e_i]\times[e_i]-X_0})}^{\mathrm{||\cdot||_r}}=
   \bigoplus_{i=1}^{sE}C_0(\open{e_i}, K(\ell^2([v_i]))).$$

The proof of (3) is similar to the proof of (2).
 
\end{proof}

It follows that
\begin{eqnarray*}
  K_0(C_r^*(R|_{X_0}))&\cong& K_0\Big(\bigoplus_{i=1}^{\mathrm{sV}} K(\ell^2([v_i]))\Big)\\
  &=&\bigoplus_{i=1}^{\mathrm{sV}} K_0\Big( K(\ell^2([v_i]))\Big) \cong \Z^{sV},
\end{eqnarray*}
where the equality in the second line follows because the number of stable vertices $sV$ is finite. 

It is well-known that the $K_0$-group  of the compact operators $K(\ell^2([v]))$ is isomorphic to $\Z$ and that it has generator $[p]_0$, for any 1-dimensional projection $p$ in $K(\ell^2([v]))$. Take
\begin{equation}\label{e:fn}
  f_i(x,y)=\left\{
            \begin{array}{cl}
              1 & \quad x=y=v_i^0 \\
              0 & \quad \text{otherwise.} \\
            \end{array}
          \right.
\end{equation}

Now we show that  $\lambda(f_i)$ is a one dimensional projection in $K(\ell^2([v_i]))$:
Let $\xi\in \ell^2([v_i^0])$. For $x\in [v_i^0]$,
\begin{eqnarray*}
  (\lambda(f_i)\xi)(x)&=&\sum_{y\in [v_i^0]} f_i(x,y) \xi(y)\\
  &=&\left\{
            \begin{array}{cc}
              \xi(v_i^0) &\quad x=v_i^0 \\
              0 &\quad x\ne v_i^0 \\
            \end{array}
          \right.
\end{eqnarray*}

So $\lambda(f_i)|_{\ell^2([v_i^0])}$ is a projection onto $\C \delta_{v_i^0}$.
If $\xi\in \ell^2([v_j])$, $j\ne i$, then for $x\in [v_j]$,
$(\lambda(f_i)\xi)(x)=0$ because $x\ne v_i^0$ for all $x\in [v_j^0]$.
Hence $\lambda(f_i)$ is a one dimensional projection in $K(\ell^2([v_i^0])) \,\,\,\subset \bigoplus_{[v_i^0]\in V} K(\ell^2([v_i^0]))$.
Therefore, $$K_0(C_r^*(R|_{X_0}))$$ 
has generators 
$$[f_1]_0, \ldots, [f_n]_0$$
where $f_1,\ldots, f_n$ are defined in (\ref{e:fn}).

\subsection{Exponential map}
Recall that we are using the notation of Section \ref{s:S} and $T:=T_n$ and $R:=R_n$, $n\in \N_0$.
\begin{pro}[{\cite[Proposition~III.28]{Gon0}}]\label{p:delta0-d1}
Assume that the dimension of tiling $T$ is 1.
The exponential map  $\delta^0:\Z^{\sV}\to \Z^{\sE}$ is given by
$$\delta^0(e_1.e_2)=e_1-e_2.$$
\end{pro}
\begin{proof}
Recall that a 1-dimensional projection generates the $K_0$-group of the compact operators $K$.
One can choose any 1-dimensional projection because all are equivalent (cf.~Lemma \ref{l:onedimprojareequiv}). 
Let $v\in T$ be a vertex in the stable vertex $e_1.e_2$;
let $p$ be the 1-dimensional projection in $C=\C_r^*(R|_{X_0})$ given by
$$p(x,y)= \left\{\begin{array}{cc}
    1 & (x,y)=(v,v) \\ 
    0 & else. \\ 
  \end{array}\right.
$$
Let $a\in B$ be the self-adjoint operator given by
$$a(x,y)= \left\{\begin{array}{cc}
    s & x=y=s\in e_1 \\ 
    1-s & x=y=s\in e_2 \\ 
    0 & else. \\ 
  \end{array}\right.
$$
(where we ignore reparametrization of the edge, e.g.~we identify $e_2$ with the unit interval $[0,1]$).
Note that $p$ is a projection that corresponds to the stable edge $e_1.e_2$. Moreover, the support of $p$ and $a$ is on the diagonal of $R$.\\
\centerline{\begin{tikzpicture}[scale=2]
\draw[->] (2,0)--(2.1,0);
\draw[->] (0,2)--(0,2.1);
\node at (2.2,0) {$x$};
\node at (0,2.2) {$y$};

\draw[->] (0,0)--(.5,0);
\draw[->] (1,0)--(1.5,0);
\draw[->] (0,0)--(0,.5);
\draw[->] (0,1)--(0,1.5);
\node at (.5,-.1) {$e_1$};
\node at (1,-.1) {$v$};
\node at (1.5,-.1) {$e_2$};
\node at (-.1,.5) {$e_1$};
\node at (-.1,1) {$v$};
\node at (-.1,1.5) {$e_2$};
\node at (1,0) {\textbullet};
\node at (0,1) {\textbullet};
\draw (0,0) -- (2,0);
\draw[dashed] (0,1) -- (2,1);
\draw[dashed] (0,2) -- (2,2);
\draw (0,-.1) -- (0,2.1);
\draw[dashed] (1,0) -- (1,2);
\draw[dashed] (2,0) -- (2,2);
\draw (-.1,-.1) -- (2.2,2.2);
\draw[red] (0,0) -- (.8,2.5)--(2,2);
\node [red] at (.8,2.5) {\textbullet};
\draw[red,dotted] (1,1)--(.8,2.5);
\node[red] at (.8,2.65) {$1$};
\node[red] at (.4,1.6) {$s$};
\node[red] at (1.4,2.4) {$1-s$};
\node at (2.8,2) {$R$};
\node at (2.25,2.25) {$x=y$};
\node at (1,-.4) {$a(x,y)$};
\end{tikzpicture}}
Then $\tilde p:=(p,0)$ is a projection in the unitization $\tilde C$ of $C$, and $\tilde a:=(a,0)$ is a selfadjoint operator in the unitization $\tilde B$ of $B=C_r^*(R|_{X_1})$.
By \cite[Proposition~12.2.2]{Rordam}, there exists a unique unitary $u\in \tilde J$, where $\tilde J$ is the unitization of $J=C_r^*(R|_{X_1-X_0})$, such that
$$\tilde{i}(u)(x,y)=e^{2\pi i\tilde a}(x,y)$$
and such that
$$\delta^0([\tilde p]_0)=\delta^0([\tilde p]_0-[s(\tilde p)]_0)=[u]_1.$$
and note that the scalar part $s(\tilde p)=0$ is zero. 
We would like to remark that we are using a different convention than R{\o}rdam's (he defines $\delta^0([\tilde p]_0)=-[u]_1$). 
Our $\delta^0$ corresponds to the standard differential defining cellular cohomology.
We evaluate $e^{2\pi i\tilde a}(x,y)$ and get
$$e^{2\pi i\tilde a}(x,y)= \left\{\begin{array}{lc}
    e^{2\pi i s} & s=x=y\in e_1 \\ 
    e^{2\pi i (1-s)}=e^{-2\pi is} & s=x=y\in e_2\\
    1 & x=y\not\in e_1\cup e_2\\
	0&else.
  \end{array}\right.
$$
Note that the evaluation of the exponential function $e^{2\pi i\tilde a}(x,y)$ is done pointwise since $\tilde a$ has support on the diagonal $\Delta_{R}$ of $R$ and 
$C_r^*(\Delta_{R})= C_0(\Delta_{R})$.

We want to describe $u$ in terms of the compacts, that is, up to the isomorphism given in Proposition \ref{p:compacts}. 
But first note that $u$ is in the unitization of $J$, and recall that the unitization of the compacts is $K(H)+\C I_H$.
Suppose that $e_1\not\sim_{R} e_2$ as stable edges (we are only working on the diagonal of $R$).
Then
$$\delta^0([\tilde p]_0)=[e^{2\pi is}\tensor e_{e_1 e_1}+e^{-2\pi is}\tensor e_{e_2 e_2}+ \sum_{e_3\not\sim e_1,e_3\not\sim e_2} 1\tensor e_{e_3 e_3}]_1.$$
$$=[e^{2\pi is}\tensor e_{e_1 e_1}-e^{2\pi is}\tensor e_{e_2 e_2}+ \sum_{e_3\not\sim e_1,e_3\not\sim e_2} 1\tensor e_{e_3 e_3}]_1,$$
(where the second equality is because for unitary u it holds  $0=[uu^{-1}]_1=[u]_1+[u^{-1}]_1$ by Proposition 8.1.4 in \cite{Rordam}).
Remark: $e^{2\pi i s}$ is just a unitary in the unitization of $C_0((0,1))$ i.e.~is in $A:=\{f\in C[0,1]\mid f(0)=f(1)\}$.
$[e^{2\pi is }]_1$ is a generator for $K_1(A)$ (see Bott element in \cite{Rordam} 11.3).
Thus 
$$\delta^0(e_1.e_2)=e_1-e_2.$$
If $e_1\sim_{R} e_2$ as stable edges then we get, by Whitehead's lemma (Lemma 2.1.5 in \cite{Rordam}), that
$$
  \left(\begin{array}{cc}
    e^{2\pi is} & 0 \\ 
    0 & e^{-2\pi is} \\ 
  \end{array}\right)
\sim_h
  \left(\begin{array}{cc}
    1 & 0 \\ 
    0 & 1 \\ 
  \end{array}\right).
$$
Hence
$$\delta^0([\tilde p]_0)=[e^{2\pi is}\tensor e_{e_1 e_1}+e^{-2\pi is}\tensor e_{e_2 e_2}+ \sum_{e_3\not\sim e_1} 1\tensor e_{e_3 e_3}]_1=0.$$
Thus
$$\delta^0(e_1.e_2)=0=e_1-e_2,$$
where the last equality is because $e_1$ and $e_2$ are $R$-equivalent.

\end{proof}

\section{\textbf{Background review}}\label{a:1}
The following lemma is closely related to the splitting lemma.
\begin{lem}
Let $A,B,C$ be $R$-modules, where $R$ is a commutative ring.
Let 
$$\xymatrix{0\ar[r]&A\ar[r]^{i}&B\ar@/_1.0pc/[l]_{r}\ar[r]^{\pi}&C\ar@/_1.0pc/[l]_{s}\ar[r]&0}$$
be a split short exact sequence.
Thus the injective map $i$, the surjetive map $\pi$, the retraction map $r$, and the section map $s$ are all homomorhism, and they satisfy
$$r\circ i= 1_A,\quad \pi\circ s = 1_C, \quad \pi \circ i=0,$$
$$i\circ r+ s\circ\pi= 1_B,\qquad r\circ s=0.$$ 
This is equivalent to $B$ being isomorphic to the direct sum  $A\oplus C$. More precisely, the following diagram commutes
$$\xymatrix{
0\ar[r]&\ar[d]^{1_A}_{\cong}A\ar[r]^{i}&B\ar[d]^{\phi}_{\cong}\ar[r]^{\pi}&C\ar[d]^{1_C}_{\cong}\ar[r]&0\\
0\ar[r]&A\ar[r]^{i'}&A\oplus B\ar[r]^{\pi'}&C\ar[r]&0,
}$$
where $i'(a)=(a,0)$ and $\pi'(a,c)=c$.
\end{lem}
\begin{proof} We will show that $B\cong A\oplus C$.
Define 
$$\phi(b):=(r(b),\pi(b))\qquad \phi^{-1}(a,c):=i(a)+s(c).$$
Then 
$$\phi^{-1}\circ\phi(b)=\phi^{-1}(r(b),\pi(b))=i\circ r(b)+s\circ \pi(b)=b$$
$$\phi\circ\phi^{-1}(a,c)=\phi(i(a)+s(c))=(r(i(a)+s(c)),\pi(i(a)+s(c)))=(a+0,0+c)=(a,c).$$
The first square commutes since
$$\phi\circ i(a)=(r(i(a)),\pi(i(a)))=(a,0)=i'(a).$$
The second square commutes since
$$\pi'\circ\phi(b)=\pi'(r(b),\pi(b))=\pi(b)=1_C\circ\pi(b).$$
Conversely, assume $B\cong A\oplus C$. Define 
$$\xymatrix{0\ar[r]&A\ar[r]^{i}&A\oplus C\ar@/_1.0pc/[l]_{r}\ar[r]^{\pi}&C\ar@/_1.0pc/[l]_{s}\ar[r]&0}$$
as
$$i(a)=(a,0),\quad \pi(a,c)=c,\quad r(a,c)=a,\quad s(c)=(0,c).$$ 
Moreover, we get
$$\ker \pi = \im i,\quad \ker r = \im s,\quad r\circ i=1_A, \quad \pi\circ s=1_C$$
$$i\circ r +s\circ \pi=1_{A\oplus C}.$$
Define $r':=1_A\circ r \circ \phi^{-1}$, $s':=\phi\circ s \circ 1_C$. Then we have the commutative diagram (clockwise and counterclockwise)
$$\xymatrix{
0\ar[r]&\ar[d]^{1_A}A\ar[r]^{i}&\ar@/_1.0pc/[l]_{r}B\ar[d]^{\phi}_{\cong}\ar[r]^{\pi}&C\ar@/_1.0pc/[l]_{s}\ar[d]^{1_C}\ar[r]&0\\
0\ar[r]&A\ar[r]^{i'}&A\oplus B\ar@/^1.0pc/[l]^{r'}\ar[r]^{\pi'}&C\ar@/^1.0pc/[l]^{s'}\ar[r]&0.
}$$
\end{proof}

If $A$ is nuclear and $B$ is any arbitrary $C^*$-algebra, then only the minimal tensor product is possible to make $A\otimes B$ into a $C^*$-algebra.
The unitization of the compacts $K:=K(H)$ is $K(H)+\C I_H$.
Moreover $\widetilde{A\otimes K}\subset \tilde A\otimes \tilde K$.
Equality holds if $A$ or $K$ is unital. ($K$ is unital if $K$ is finite dimensional).
Otherwise the inclusion is proper since $a\otimes 1$ and $1\tensor k$ are not in the unitization of $A\otimes K$.

\begin{lem}\label{l:onedimprojareequiv}
  Let $E$, $F$ be $1$-dimensional projections on a Hilbert space $H$, and let $K(H)$ be the compact operators.
  Then $E\sim F$ in $K(H)$ i.e.~there is a $V\in K(H)$ such that $V^*V=E$ and $VV^*=F$.
\end{lem}
\begin{proof}
  $K:=E(H)$ and $L:=F(H)$ are $1$-dimensional subspaces of $H$. Thus, there is a $V_0:K\to L$ isometry of $K$ onto $L$.
  Extend to $H$ by putting
  $$Vx:= V_0 E x.$$
  Then $V\in B(H)$, $V^*V=E$, $VV^*=F$.
  Since $V(H)=F(H)$ are 1-dimensional, $V$ is finite dimensional, hence compact i.e.~$V\in K(H)$.
\end{proof}

\Addresses


\newpage
\begin{thebibliography}{0}

\bibitem{AP}
 J.~E.~Anderson, I.~F.~Putnam. 
 \emph{Topological Invariants for Substitution Tilings and
their Associated $C^*$-algebras}.
Department of Mathematics and Statistics, University of Victoria,
Victoria B.C.~Canada.  1-45.
(1995)

\bibitem{Arnold}
 D.~M.~Arnold. 
 \emph{ Finite Rank Torsion Free Abelian Groups and Rings}.
 Lecture Notes in Mathematics 931.
 Springer-Verlag Berlin Heidelberg New York.
(1982)


\bibitem{BargeDiamond}
M.~Barge, B.~Diamond.
\emph{Cohomology in one-dimensional substitution tiling spaces.}
Proc.~Amer.~Math.~Soc.~136, no.~6,  2183-2191.
(2008)


\bibitem{BargeDiamondSadun}
M.~Barge, B.~Diamond, J.~Hunton, L.~Sadun.
\emph{Cohomology of substitution tiling spaces.}
Ergodic Theory Dynam.~Systems, 30(6):1607-1627. 
(2010)

\bibitem{Bellissard1} 
J.~Bellissard.
\emph{K-theory of C*-algebras in solid state physics.} 
In Statistical Mechanics
and Field Theory: Mathematical Aspects, eds. T.~C.~Dorlas, N.~M.~Hugenholtz and M.~Winnik. Lecture Notes in Physics, vol. 257, Springer, Berlin. 99-156.
 (1986)
 

\bibitem{Bellissard2}
J.~Bellissard.
\emph{Gap labelling theorems for Schr\"odinger's Operators.}
In From Number Theory to Physics, eds. M.~Waldschmidt, P.~Moussa, J.~M.~Luck and C.~Itzykson, Springer,
Berlin. 539-630.
(1993)


\bibitem{BBG} 
J.~Bellissard, R.~Benedetti, J.~M.~Gambaudo.
\emph{Spaces of tilings, finite telescopic approximations and gap-labelling.}
Comm.~Math.~Phys., 261(1), 1-41. 
(2006)



\bibitem{Blackadar}
B.~Blackadar.
\emph{$K$-Theory for operator algebras}.
MSRI Publ.~5, Cambridge University
Press.
(1998)




\bibitem{CS}
A.~Clark, L.~Sadun.
\emph{When shape matters: deformations of tiling spaces.}
Ergodic Theory Dynam.~Systems, 26 (1), 69-86. 
(2006)

\bibitem{Conway}
J.~B.~Conway.
\emph{A course in functional analysis.}
Springer-Verlag.
(1990)


\bibitem{Dugas}
M.~Dugas.
\emph{Torsion-Free abelian groups defined by an integral matrix.}
Int.~J.~Algebra. Vol 6. no.~1-4, 85-99.
(2012)


\bibitem{Eisenbud}
D.~Eisenbud.
\emph{Commutative algebra, with a view toward algebraic geometry.}
Springer-Verlag.
(1995)

\bibitem{Gahler}
F.~G\"ahler, G.~Maloney.
\emph{Cohomology of one-dimensional mixed substitution tiling spaces} 
Topology Appl.~160, no.~5, 703-719. 
(2013)

\bibitem{Gon0}
D.~Gon\c{c}alves.
\emph{$C^*$-algebras from Substitution Tilings: A New Approach. Ph.D.~Thesis, Univ.~Of Victoria, Victoria, Canada.}
 (2005)

\bibitem{Gon1}
D.~Gon\c{c}alves.
\emph{New $C^*$-algebras from substitution tilings.}
 J.~Operator Th., 2,  57, 391-407.
 (2007)


\bibitem{Gon2}
D.~Gon\c{c}alves.
\emph{On the $K$-theory of the stable $C^*$-algebras from substitution tilings.}
J.~Func.~Anal.~4, 260, 998-1019.
(2011)

\bibitem{Janot}
C.~Janot.
\emph{Quasicrystals: A Primer.}
 Clarendon press, Oxford.
(1994)

\bibitem{JulienSavinien}
A.~Julien, J.~Savinien.
\emph{$K$-theory of the chair tiling via AF-algebras.}
Journal of Geometry and Physics, Vol.~106,  314-326.
(2016)


\bibitem{KPW}
J.~Kaminker, I.~F.~Putnam, M.~Whittaker.
\emph{$K$-Theoretic duality for hyperbolic dynamical system.}
J.~Reine Angew.~Math., 730, 263-299.
(2017)

\bibitem{Kel1}
J.~Kellendonk. 
\emph{Non-commutative geometry of tilings and gap labelling.}
Rev.~Math.~Phys.~7, 1133-1180.
(1995)

\bibitem{Kel2}
J.~Kellendonk. 
\emph{The local structure of tilings and their integer group of coinvariants.}
Commun.~Math.~Phys.~187, 115-157.
(1997)

\bibitem{Kel3}
J.~Kellendonk.
\emph{Integer groups of coinvariants associated to octagonal tilings.}
Fields Insitute Communications 13, 155-169.
(1997)

\bibitem{Kel-PE}
[K] J.~Kellendonk, Pattern-equivariant functions and cohomology, J.~Phys.~A 36 (2003) 5765-5772. 
J.~Kellendonk, Pattern-equivariant functions and cohomology, J.~Phys.~A 36 (2003), no.~21, 5765-5772.


\bibitem{KelPut}
J.~Kellendonk, I.~F.~Putnam.
\emph{Tilings, $C^*$-algebras and $K$-theory}.
Directions in mathematical quasicrystals, CRM Monogr.~Ser., 13, Amer.~Math.~Soc., Providence, RI.
(2000)

\bibitem{KelPut-PE}
J.~Kellendonk, I.~F.~Putnam.
\emph{The Ruelle-Sullivan map for actions of $\R^n$}, 
Math.~Ann.~334, no.~3, 693-711.
(2006)





\bibitem{MRW}
P.~S.~Muhly, J.~N.~Renault, D.~P.~Williams.
\emph{Equivalence and isomorphism for groupoid $C^*$-algebras}.
 J.~Operator Th.~17, 3-22.
(1987)
\bibitem{MRW-continuoustrace}
P.~S.~Muhly, J.~N.~Renault, D.~P.~Williams.
\emph{Continuous-trace groupoid $C^*$-algebras III}.
 Trans.~Amer.~Math.~Soc.~ 348,9, 3621-3641.
(1996)

\bibitem{PutnamH}
I.~F.~Putnam,
\emph{A homology theory for Smale spaces}.
Memoirs A.M.S., Vol 232(1094), 122 pages.
(2014)

\bibitem{PutnamCnotes}
 I.~F.~Putnam. 
 \emph{Lecture notes on $C^*$-algebras}. 
 http://www.math.uvic.ca/faculty/putnam/ln/C*-algebras.pdf.
 (2014)


\bibitem{PutnamSpielberg}
 I.~F.~Putnam, J.~Spielberg. 
 \emph{The structure of $C^*$-algebras associated with hyperbolic
dynamical systems.}
 J.~Func.~Anal., 163:279-299.
(1999)

\bibitem{Putnam96SmaleSpaces}
I.~F.~Putnam,
\emph{$C^*$-algebras from Smale spaces}.
Canad.~J.~Math, 48, 175-195.
(1996)


\bibitem{Reid}
J.~D.~Reid, 
\emph{Abelian groups finitely generated over their endomorphism rings}, 
 Springer-Verlag Lecture Notes 874, 41-52 in Abelian Group Theory (Proceedings of the Oberwolfach Conference).
 (1981). 


\bibitem{Renault}
J.~Renault, 
\emph{A groupoid approach to $C^*$-algebras.}
Lecture Notes in Mathematics, No.793,
Springer-Verlag, Berlin-New York, 
(1980).


\bibitem{Ren-IdealStructure}
J.~Renault,
\emph{The ideal structure of groupoid crossed product $C^*$-algebras.}
J.~Operator Theory 25,3-36.
(1991).
\bibitem{Rieffel}
 M.~A.~Rieffel.
\emph{Induced representations of $C^*$-algebras}, Adv.~Math.~13, 176-257.
(1974). 

\bibitem{Rordam}
M.~R{\o}rdam, F.~Larsen, N.J.~Laustsen. 
\emph{An introduction to $K$-Theory for $C^*$-algebras.}
J.~Operator Theory 25,3-36.
(1991).

\bibitem{Sadun2011}
L.~Sadun.
\emph{Exact regularity and the cohomology of tiling spaces.}
Ergodic Theory and Dynamical Systems 31, 1819-1834.
(2011)

\bibitem{Sadun-PE}
L.~Sadun.
\emph{Pattern-equivariant cohomology with integer coefficients.}
Ergodic Theory and Dynamical Systems 27(6), 1991–1998.
(2007)

\bibitem{Sadun2006}
L.~Sadun.
\emph{Tilings, tiling spaces and topology.}
Philos.~Mag.~86, 875-881 
(2006) 

\bibitem{Sadun}
L.~Sadun. 
\emph{Topology of Tiling Spaces.}
 University Lecture Series Vol.~46, Providence,
Rhode Island.
(2008)


\bibitem{Steinhardt}
P.~J.~Steinhardt, S.~Ostlund.
\emph{The Physics of quasicrystals.}
World Scientific.
(1987)

\bibitem{WaltonHom}
J.~J.~Walton.
\emph{Pattern-equivariant homology}
arxiv.org:1401.8153.
(2014)


\bibitem{Willett}
R.~Willett.
\emph{A non-amenable groupoid whose maximal and reduced $C^*$-algebras are the same.}
arxiv.org:1504.05615.
(2015)


\bibitem{Tu}
J.~L.~Tu.
\emph{La conjecture de Baum-Connes pour les feuilletages moyennables.}
$K$-Theory 17, 215-264.
(1999)

\bibitem{Gustavo}
G.~F.~Valente.
\emph{Cohomologia associada a ladrilhamentos de substitui\c{c}\~{a}o}.
\url{http://mtm.ufsc.br/pos/Gustavo_Felisberto_Valente.pdf}.
Master Thesis. Advisor D.~Gon\c{c}alves
(2013)

\bibitem{WeggeOlsen}
N.E.~Wegge-Olsen.
\emph{ $K$-theory and $C^*$-algebras}.
Oxford university press.
(1993)

\bibitem{Zhu}
K.~Zhu.
\emph{An introduction to operator algebras}.
CRC press.
(1993)



\end{thebibliography}
\end{document}